\documentclass{article}
\usepackage[utf8]{inputenc} 
\usepackage[english]{babel}
\usepackage{geometry}
\usepackage[numbers]{natbib}

\usepackage{amsmath,amsfonts,amssymb,amsthm,mathrsfs,bbm}
\usepackage[hidelinks]{hyperref}
\usepackage{paralist}
\usepackage{enumitem}
\usepackage{textcase} 
\usepackage{titlesec} 
\usepackage{etoolbox} 
\usepackage{float}
\usepackage{graphics}
\usepackage{booktabs}
\usepackage{multirow}
\usepackage{xcolor}

\makeatletter

\numberwithin{equation}{section}
\allowdisplaybreaks

\def\@noindentfalse{\global\let\if@noindent\iffalse}
\def\@noindenttrue {\global\let\if@noindent\iftrue}
\def\@aftertheorem{%
	\@noindenttrue
	\everypar{%
		\if@noindent%
		\@noindentfalse\clubpenalty\@M\setbox\z@\lastbox%
		\else%
		\clubpenalty \@clubpenalty\everypar{}%
		\fi}}

\theoremstyle{plain}
\newtheorem{theorem}{Theorem}[section]
\AfterEndEnvironment{theorem}{\@aftertheorem}
\newtheorem{definition}[theorem]{Definition}
\AfterEndEnvironment{definition}{\@aftertheorem}
\newtheorem{lemma}[theorem]{Lemma}
\AfterEndEnvironment{lemma}{\@aftertheorem}
\newtheorem{corollary}[theorem]{Corollary}
\AfterEndEnvironment{corollary}{\@aftertheorem}
\newtheorem{proposition}[theorem]{Proposition}
\AfterEndEnvironment{proposition}{\@aftertheorem}

\theoremstyle{definition}
\newtheorem{remark}[theorem]{Remark}
\AfterEndEnvironment{remark}{\@aftertheorem}
\newtheorem{example}[theorem]{Example}
\AfterEndEnvironment{example}{\@aftertheorem}

\AfterEndEnvironment{proof}{\@aftertheorem}

\newtheorem{assumption}[theorem]{Assumption}
\AfterEndEnvironment{assumption}{\@aftertheorem}

\titleformat{name=\section}
{\center\small\sc}{\thesection\kern1em}{0pt}{\MakeTextUppercase}
\titleformat{name=\section, numberless}
{\center\small\sc}{}{0pt}{\MakeTextUppercase}

\titleformat{name=\subsection}
{\bf\mathversion{bold}}{\thesubsection\kern1em}{0pt}{}
\titleformat{name=\subsection, numberless}
{\bf\mathversion{bold}}{}{0pt}{}

\titleformat{name=\subsubsection}
{\normalsize\itshape}{\thesubsubsection\kern1em}{0pt}{}
\titleformat{name=\subsubsection, numberless}
{\normalsize\itshape}{}{0pt}{}

\def\note#1{\par\smallskip%
	\noindent\kern-0.01\hsize%
	{\setlength\fboxrule{0pt}\fbox{\setlength\fboxrule{0.5pt}\fbox{%
				\llap{$\boldsymbol\Longrightarrow$ }%
				\vtop{\hsize=0.98\hsize\parindent=0cm\small\rm #1}%
				\rlap{$\enskip\,\boldsymbol\Longleftarrow$}
	}}}%
}

\usepackage{delimset}
\def\given{\mskip 0.5mu plus 0.25mu\vert\mskip 0.5mu plus 0.15mu}
\newcounter{bracketlevel}%
\def\@bracketfactory#1#2#3#4#5#6{%
	\expandafter\def\csname#1\endcsname##1{%
		\global\advance\c@bracketlevel 1\relax%
		\global\expandafter\let\csname @middummy\alph{bracketlevel}\endcsname\given%
		\global\def\given{\mskip#5\csname#4\endcsname\vert\mskip#6}\csname#4l\endcsname#2##1\csname#4r\endcsname#3%
		\global\expandafter\let\expandafter\given\csname @middummy\alph{bracketlevel}\endcsname%
		\global\advance\c@bracketlevel -1\relax%
	}%
}
\def\bracketfactory#1#2#3{%
	\@bracketfactory{#1}{#2}{#3}{relax}{0.5mu plus 0.25mu}{0.5mu plus 0.15mu}
	\@bracketfactory{b#1}{#2}{#3}{big}{1mu plus 0.25mu minus 0.25mu}{0.6mu plus 0.15mu minus 0.15mu}
	\@bracketfactory{bb#1}{#2}{#3}{Big}{2.4mu plus 0.8mu minus 0.8mu}{1.8mu plus 0.6mu minus 0.6mu}
	\@bracketfactory{bbb#1}{#2}{#3}{bigg}{3.2mu plus 1mu minus 1mu}{2.4mu plus 0.75mu minus 0.75mu}
	\@bracketfactory{bbbb#1}{#2}{#3}{Bigg}{4mu plus 1mu minus 1mu}{3mu plus 0.75mu minus 0.75mu}
}
\bracketfactory{clc}{\lbrace}{\rbrace}
\bracketfactory{clr}{(}{)}
\bracketfactory{cls}{[}{]}
\bracketfactory{abs}{\lvert}{\rvert}
\bracketfactory{norm}{\Vert}{\Vert}
\bracketfactory{floor}{\lfloor}{\rfloor}
\bracketfactory{ceil}{\lceil}{\rceil}
\bracketfactory{angle}{\langle}{\rangle}

\newcounter{ctr}\loop\stepcounter{ctr}\edef\X{\@Alph\c@ctr}%
\expandafter\edef\csname s\X\endcsname{\noexpand\mathscr{\X}}
\expandafter\edef\csname c\X\endcsname{\noexpand\mathcal{\X}}
\expandafter\edef\csname b\X\endcsname{\noexpand\boldsymbol{\X}}
\expandafter\edef\csname I\X\endcsname{\noexpand\mathbb{\X}}
\ifnum\thectr<26\repeat

\let\@IE\IE\let\IE\undefined
\newcommand{\IE}{\mathop{{}\@IE}\mathopen{}}
\let\@IP\IP\let\IP\undefined
\newcommand{\IP}{\mathop{{}\@IP}}

\def\^#1{\relax\ifmmode {\mathaccent"705E #1} \else {\accent94 #1}\fi}
\def\~#1{\relax\ifmmode {\mathaccent"707E #1} \else {\accent"7E #1}\fi}
\def\*#1{\relax#1^\ast}
\edef\-#1{\relax\noexpand\ifmmode {\noexpand\bar{#1}} \noexpand\else \-#1\noexpand\fi}
\def\>#1{\vec{#1}}
\def\.#1{\dot{#1}}

\def\atop{\@@atop}

\renewcommand{\leq}{\leqslant}
\renewcommand{\geq}{\geqslant}

\newcommand\indep{\protect\mathpalette{\protect\@indep}{\perp}}
\def\@indep#1#2{\mathrel{\rlap{$#1#2$}\mkern2mu{#1#2}}}

\def\parsetime#1#2#3#4#5#6{#1#2:#3#4}
\def\parsedate#1:20#2#3#4#5#6#7#8+#9\empty{20#2#3-#4#5-#6#7 \parsetime #8}
\def\moddate{\expandafter\parsedate\pdffilemoddate{\jobname.tex}\empty}

\makeatother

\makeatletter

\theoremstyle{definition}

\theoremstyle{remark}

\theoremstyle{definition}

\theoremstyle{plain}

\theoremstyle{plain}

\theoremstyle{plain}

\theoremstyle{plain}

\allowdisplaybreaks[4]

\makeatother

\providecommand{\conditionname}{Condition}
\providecommand{\definitionname}{Definition}
\providecommand{\lemmaname}{Lemma}
\providecommand{\propositionname}{Proposition}
\providecommand{\remarkname}{Remark}
\providecommand{\corollaryname}{Corollary}
\providecommand{\theoremname}{Theorem}

\begin{document}
\title{{Necessary and sufficient condition for CLT of linear spectral statistics of sample correlation matrices}}
\author{Yanpeng Li\thanks{Harbin Institute of Technology. Email: 20230256@hit.edu.cn}, \ Guangming Pan\thanks{Nanyang Technological University. Email: gmpan@ntu.edu.sg}, \ Jiahui Xie\thanks{National University of Singapore. Email: jiahui.xie@u.nus.edu}, \ Wang Zhou\thanks{National University of Singapore. Email: stazw@nus.edu.sg}}
	
	\date{}
	
	\maketitle

\begin{abstract}
{In this paper, we establish the central limit theorem (CLT) for the linear spectral statistics (LSS) of sample correlation matrix $R$, constructed from a $p\times n$ data matrix $X$ with independent and identically distributed (i.i.d.) entries having mean zero, variance one, and infinite fourth moments in the high-dimensional regime $n/p\rightarrow \phi\in \mathbb{R}_+\backslash \{1\}$. We derive a necessary and sufficient condition for the CLT. More precisely, under the assumption that the identical distribution $\xi$ of the entries in $X$ satisfies $\mathbb{P}(|\xi|>x)\sim l(x)x^{-\alpha}$ when $x\rightarrow \infty$ for $\alpha \in (2,4]$, where $l(x)$ is a slowly varying function, we conclude that: (i). When $\alpha\in(3,4]$, the universal asymptotic normality for the LSS of sample correlation matrix holds, with the same asymptotic mean and variance as in the finite fourth moment scenario; (ii) We identify a necessary and sufficient condition $\lim_{x\rightarrow\infty}x^3\mathbb{P}(|\xi|>x)=0$ for the universal CLT; (iii) We establish a local law for $\alpha \in (2, 4]$. Overall, our proof strategy follows the routine of the matrix resampling, intermediate local law, Green function comparison, and characteristic function estimation. In various parts of the proof, we are required to come up with new approaches and ideas to solve the challenges posed by the special structure of sample correlation matrix. For instance, we combine the matrix resampling and cumulant expansion to derive estimates for the odd moments of self-normalized variables. We employ the resolvent expansions and Taylor expansions to manage the column dependence structure in the Green function comparison. Furthermore, we utilize not only the first-order derivative but also the second-order derivative to handle the heavy-tailed part in the characteristic function. Additionally, the fluctuation of the quadratics under heavy-tailed conditions, which is less known even in symmetry cases, demonstrates the robustness of self-normalization and the critical point $\alpha=3$. Our results also demonstrate that the symmetry condition is unnecessary for the CLT of LSS for sample correlation matrix, but the tail index $\alpha$ plays a crucial role in determining the asymptotic behaviors of LSS for $\alpha \in (2, 3)$.
}
\end{abstract}



\tableofcontents
\newpage

\section{Introduction}\label{sec_intro}
\subsection{Problem motivation and related studies}
 Spectral statistics of sample covariance matrix or sample correlation matrix and their limiting behaviors play important roles in multivariate statistics analysis. Many fundamental statistics in the classical multivariate analysis can be   expressed as linear functionals of eigenvalues of covariance matrix or correlation matrix. These linear functionals, commonly referred to as \textit{linear spectral statistics} (LSS) in the literature, play a significant role, particularly in statistical inference problems like, multivariate analysis of variance, multivariate linear models, canonical correlation analysis, and factor analysis. In recent years, beyond the traditional scope of multivariate analysis, the growing volume of data collected from novel sources has sparked increased interest in statistical inference from large data matrices. This trend necessitates the development of new methodologies and techniques to manage the associated high dimensionality and complexity.

As a result, research on LSS for sample covariance matrix and sample correlation matrix under high-dimensional settings has attracted considerable attention and has been widely applied in various domains, including finance, biomedicine, wireless communications, genetic engineering. For a thorough exploration of LSS applications across different matrix models, readers can refer to \cite{bai2010spectral,couillet2011random,johnstone2006high,paul2014random,tulino2004random,yao2015sample}. In historical research, \cite{jonsson1982some} investigated the original problem for LSS using a sequence of Wishart matrices. They established a \textit{central limit theorems} (CLT) for a class of polynomial
functions. Afterward, \cite{bai2008clt} expanded the CLT framework to analytic functions of eigenvalues in large sample covariance matrices. This extension inspired further developments for a variety of matrix models. For instance, \cite{zheng2012central} considered Fisher matrices, \cite{diaconis2001linear} investigated Haar matrices using the moment method, and \cite{bai2005convergence} applied the Stieltjes transform method to Wigner matrices. The central finding across these studies is that the universal CLT for LSS in large sample covariance matrices holds when the randomness satisfies a finite fourth moment condition.

Although the CLT for LSS is universal in the sense that it does not depend on the specific distribution of the data, it does require that the population has at least a finite fourth moment. In practice, having prior knowledge of the population mean, variance, and fourth moment is necessary, which limits the statistical applications of the sample covariance matrix. A natural alternative to address these limitations is the sample correlation matrix. Thanks to its self-normalizing property within columns, the sample correlation matrix does not require knowledge of the first two population moments. This feature makes linear spectral statistics based on the sample correlation matrix more practical than those based on the sample covariance matrix. In addition, the self-normalized sum is expected to be more robust across a broad class of random variables. This raises the natural question: \textit{Can the universal CLT for the LSS of sample correlation matrix be established under much weaker moment conditions?} Indeed, identifying the minimal moment condition under which the CLT holds for the LSS of sample correlation matrix has been a long-standing problem.

The literature on large sample correlation matrices traces back to \cite{jiang2004limiting}, who established a limiting spectral distribution under a finite second moment condition. Later, \cite{cai2011limiting} determined the limiting magnitude for the largest off-diagonal entries in the sample correlation matrix under finite exponential moments. Subsequent work by \cite{bao2012tracy} and \cite{pillai2012edge} focused on the largest eigenvalues, obtaining the Tracy-Widom law under similar conditions. Regarding LSS, \cite{gao2014high} established a CLT under a finite fourth moment condition, as also noted in \cite{yin2023central}. Interestingly, while their proofs rely on the fourth moment, the results do not depend on it (the involvement of the fourth moment in the asymptotic mean and variance in \cite{gao2014high} could be canceled out after a careful examination due to the equation satisfied by the Stieltjes transformation of the Marcenko-Pastur law).
These results significantly enrich the
possible applications of sample correlation matrix and make it more practically feasible to utilize sample correlation matrix. They also offer valuable insights and raise the possibility of establishing a universal CLT under weaker moment conditions. To the best of our knowledge, there are only a few studies on large-dimensional sample correlation matrices derived from heavy-tailed populations. In \cite{heiny2018almost}, the author studied the almost sure convergence for the extreme eigenvalues of sample correlation matrix under certain stringent moment conditions for the self-normalized entries (see, e.g., condition ($C_q$) therein). \cite{heiny2022limiting} established the limiting spectral distribution for sample correlation matrix when the identical distribution $\xi$ of matrix entries meets the condition $\mathbb{P}(|\xi|>x)\sim l(x)x^{-\alpha}$ as $x\rightarrow \infty$, where $\alpha \in (0,2)$ and $l(x)$ is a slowly varying function. Subsequent work by \cite{heiny2023logdet} used Girko's method of perpendiculars and a CLT for martingale differences to show the asymptotic normality of the log-determinant of sample correlation matrix when $\alpha \in (3,4]$ and $\xi$ is symmetric.

 In this paper, we establish a CLT for the LSS of sample correlation matrix without requiring the finite fourth moment.  Notably, we identify a necessary and sufficient condition on the identical distribution of the entries, $\xi$, for the universal CLT of the LSS: $\lim_{x\rightarrow \infty}x^3\mathbb{P}(|\xi|>x)=0$. We begin by introducing our model and main results in Section \ref{sec_model&result}, followed by a brief discussion of our proof strategies and techniques in Section \ref{sec_proofstrategy}.

\subsection{Matrix model and main results}\label{sec_model&result}
Let $\mathbf{x}_1,\dots,\mathbf{x}_n$ be a sequence of independent and identically distributed (i.i.d.) observations of a $p$-dimensional random vector $\mathbf{x}\in\mathbb{R}^p$ whose coordinates are i.i.d.  sampled from a centered random variable $\xi$.  We construct the data matrix as $X=(\mathbf{x}_1,\dots,\mathbf{x}_n)=(X_{ji})_{1\le j\le p;1\le i\le n}$ and denote the sample covariance matrix $S$ (before scaling) and sample correlation matrix $R$ as:
\begin{gather}\label{eq_def_R}
	S=\sum_{i=1}^n\Sigma^{1/2}\mathbf{x}_i\mathbf{x}_i^{*}\Sigma^{1/2}=\Sigma^{1/2}XX^{*}\Sigma^{1/2},~
	R=X^{*}\Sigma^{1/2}(\operatorname{diag}S)^{-1}\Sigma^{1/2}X,
\end{gather}
where $\Sigma$ is a diagonal deterministic population matrix. Due to the self-normalization property, $R$ does not depend on such population matrix $\Sigma$. For simplicity, we assume $\Sigma=I$ without loss of generality, which allows us to simplify $S$ to $XX^*$ and $R$ to $X^{*}(\operatorname{diag}S)^{-1}X$. Further, we write
\begin{equation*}
    R=YY^{*},~Y=X^{*}(\operatorname{diag}S)^{-1/2},
\end{equation*}
to capture the self-normalised random variables $(Y_{ij})_{1\le i\le n,1\le j\le p}$.

Let the empirical spectral distribution (ESD) of $R$ be
\begin{equation*}
	\mu_{n}:=\frac{1}{n}\sum_{i=1}^n\delta_{\lambda_i(R)},
\end{equation*}
where $\lambda_1(R)\ge\lambda_2(R)\ge\dots\ge\lambda_n(R)\ge0$ are the ordered eigenvalues of $R$.
\begin{definition} For any analytic function $f$: $\mathbb{R}\rightarrow\mathbb{R}$, the linear spectral statistics of $R$ is defined as
	\begin{equation}\label{eq_def_LSS}
		\sum_{i=1}^nf(\lambda_i(R))=\operatorname{tr}[f(R)]=n\int f(\lambda)\mu_n(\mathrm{d}\lambda).
	\end{equation}
\end{definition}

\subsubsection{Main assumptions}
We impose our main assumptions as follows,
\begin{assumption}\label{ass_X}
	Assume that $X=(X_{ji})$ in \eqref{eq_def_R} has i.i.d. entries following the distribution of a centered random variable $\xi$ such that $\mathbb{E}\xi^2=1$ and
	\begin{equation*}
	\lim_{x\to \infty}	\frac{\mathbb{P}(|\xi|>x)}{ l(x)/x^\alpha}=1,
	\end{equation*}
	where $\alpha\in(2,4]$,  and $l(x)$ is a slowly varying function in the sense that for all $t>0$
	\begin{equation*}
		\lim_{x \rightarrow \infty} \frac{l(tx)}{l(x)}=1.
	\end{equation*}
\end{assumption}
\begin{assumption}\label{ass_phi}
	For the dimensional parameters, we assume that $n/p\rightarrow\phi\in(0,\infty)\backslash \{1\}$.
\end{assumption}
\begin{remark}
	 Several remarks are in order. Firstly, Assumption \ref{ass_X} only requires the entries $X_{ji}$'s to be centered but not necessarily symmetric. In the literature, \cite{bai2008large} established the convergence of $\mu_n$ to the Mar\v{c}enko-Pastur (MP) law $\mu_{mp,\phi}$ (cf. \eqref{eq_mpdensity}), provided that $\xi$ belongs to the domain of attraction of the normal law, a region that includes our setting \eqref{ass_X}.
 For $\alpha\in(0,2)$ the limit of $\mu_n$ converges to the $\alpha$-heavy Mar\v{c}enko-Pastur law \cite{heiny2022limiting}, which lies beyond the scope of our investigation. As discussed in Section \ref{sec_mainresults}, we only establish a universal CLT for the LSS when $\alpha \in (3,4]$ and observe a transition phenomenon for the universal CLT around the critical point $\alpha=3$. For $\alpha\in(0,3)$, we believe a CLT for the LSS of sample correlation matrix exists but exhibits different variance and appropriate renormalization, similar to \cite{benaych2014central,benaych2016fluctuations}. Secondly, Assumption \ref{ass_phi} excludes the case where $n/p \rightarrow 1$, which can lead to singularities in certain functions, such as the log-determinant of $R$. Establishing the CLT for such functions requires a different methodology. Due to the comprehensive nature of this paper, we will address this problem in our future work.
\end{remark}

\subsubsection{Main results}\label{sec_mainresults}
In this section, we state our main results. We first write the linear spectral statistics as a contour integral involving the Stietjes transform for the test functions in certain domains. To enclose $\operatorname{supp}(\mu_{mp,\phi})$ and avoid the singularity at the origin if necessary, we define the following trajectory,
\begin{gather*}
	\mathcal{C}_1\equiv\mathcal{C}_1(c_1,c_2):=\{z:|z|=c_1,\operatorname{Im}z\ge0,\operatorname{Re}z\ge-c_2\},\\
	\mathcal{C}_2\equiv\mathcal{C}_2(c_1,c_2,C_1):=\{z:\operatorname{Im}z=\sqrt{c_1^2-c_2^2},\; -C_1\le\operatorname{Re}z\le-c_2\},\\
	\mathcal{C}_3\equiv\mathcal{C}_3(C_1,C_2):=\{z:\operatorname{Im}z=C_2,\;-C_1\le\operatorname{Re}z\le C_1\},\\
	\mathcal{C}_4\equiv\mathcal{C}_4(c_1,c_2,C_1,C_2):=\{z:\sqrt{c_1^2-c_2^2}\le\operatorname{Im}z\le C_2,\;\operatorname{Re}z=-C_1\},\\
	\mathcal{C}_5\equiv\mathcal{C}_5(c_1,c_2,C_1,C_2):=\{z:\sqrt{c_1^2-c_2^2}\le\operatorname{Im}z\le C_2,\;\operatorname{Re}z=C_1\},\\
	\mathcal{C}_6\equiv\mathcal{C}_6(c_1,c_2,C_1):=\{z:0\le\operatorname{Im}z\le\sqrt{c_1^2-c_2^2},\;\operatorname{Re}z=C_1\},\\
    \mathcal{C}_7=\{z:0\le\operatorname{Im}z\le c_2,\;\operatorname{Re}z=-C_1\},
\end{gather*}
for some small constants $c_1>c_2>0$ and large constants $C_2>C_1>0$. Then, the contour for the case $\phi\in(0,1)$ is constructed by $\gamma^0:=\mathcal{C}^0\bigcup\bar{\mathcal{C}^0}$, where $\mathcal{C}^0\equiv\mathcal{C}^0(c_1,c_2,C_1,C_2):=\bigcup_{i=1}^6\mathcal{C}_i$. For the case $\phi\in(1,\infty)$, there is no need to eliminate the origin, and thus the contour can be constructed as $\gamma:=\mathcal{C}\bigcup\bar{\mathcal{C}}$, where $\mathcal{C}\equiv\mathcal{C}(c_1,c_2,C_1,C_2):=\bigcup_{i=3}^7\mathcal{C}_i$. With the above configuration, we choose well separated parameters $c_{1i},c_{2i},C_{1i},C_{2i},\;i=1,2,$ such that the following counterclockwise contours are pairwise disjoint,
\begin{gather*}
	\gamma_1^0:=\gamma^0(c_{11},c_{21},C_{11},C_{21}),\quad \gamma_2^0:=\gamma^0(c_{12},c_{22},C_{12},C_{22});\\
	\gamma_1:=\gamma(c_{11},c_{21},C_{11},C_{21}),\quad \gamma_2:=\gamma(c_{12},c_{22},C_{12},C_{22}).
\end{gather*}
Furthermore, denote by  $m(z):=m(z,\phi)$ the Stieltjes transform of $\mu_{mp,\phi}$, and set its companion $\underline{m}(z)=\phi m(z)-(1-\phi)/z$.  In the following theorem, we establish the central limit theorem for the linear spectral statistics defined in \eqref{eq_def_LSS}.
\begin{theorem}[CLT for LSS]\label{thm_main_cltlss}
Let $f$ be analytic inside $\gamma_1^0$ and $\gamma_2^0$. Under Assumption \ref{ass_X} for $\alpha\in (3,4]$ and Assumption \ref{ass_phi}, we have for $\phi\in(0,1)$ that
	\begin{equation}\label{eq_universalclt}
		\frac{\operatorname{tr}f(R)-a_f}{\sigma_f}\xrightarrow{D}\mathrm{N}(0,1),
	\end{equation}
	where
	\begin{equation*}
		a_f=\frac{1}{2\pi \mathrm{i}}\oint_{\gamma_1^0}f(z)\frac{-n+2z(1+zm(z))\underline{m}(z)-\phi^{-1}[z\underline{m}(z)]^{\prime}}{z(1+\phi^{-1}\underline{m}(z))}\mathrm{d}z
	\end{equation*}
	and
	\begin{equation*}
		 \sigma^2_f=\frac{1}{2\pi^2}\oint_{{\gamma}_1^0}\oint_{{\gamma}_2^0}\frac{\partial_{z_2}\left(\frac{z_1m(z_1)-z_2m(z_2)}{z_1-z_2}+\phi^{-1}z_1m(z_1)\underline{m}(z_1)z_2\underline{m}(z_2)\right)}{z_1(1+\phi^{-1}\underline{m}(z_1))}f(z_1)f(z_2)\mathrm{d}z_2\mathrm{d}z_1.
	\end{equation*}
	The same result holds for $\phi\in (1,\infty)$ with $\gamma_1^0$ and $\gamma_2^0$ replaced by $\gamma_1$ and $\gamma_2$, respectively.
\end{theorem}
The necessary part of the established universal CLT can be summarized below.
\begin{theorem}\label{thm_iffalpha}
	Under Assumption \ref{ass_phi}, the universal CLT \eqref{eq_universalclt} holds if and only if
	\begin{equation}\label{eq_criticalcase}
		\lim_{x\rightarrow\infty} x^3\mathbb{P}(|\xi|>x)=0.
	\end{equation}
\end{theorem}

\begin{remark}\label{remark_alpha3}
	It should be noted that \eqref{eq_criticalcase} is critical for the universal CLT (the same asymptotic mean and variance as the finite fourth-moment case), where the asymptotic result doesn't depend on the tail index $\alpha$ and the slowly varying function $l(x)$. This result also suggests that the CLT for the LSS of sample correlation matrices is more robust than for Wigner matrices and sample covariance matrices, where the variance diverges with the index $\alpha$ in the heavy-tailed case, as seen in Wigner matrices \cite{benaych2016fluctuations} and sample covariance matrices \cite{bao2023smallest}. To the best of our knowledge, such critical condition similar to \eqref{eq_criticalcase} only exists in the context of universality for local spectral statistics of Wigner matrices \cite{lee2014necessary} and sample covariance matrices \cite{bao2023smallest,ding2018necessary}. It is the fist time that such critical condition appears in the area of LSS.
\end{remark}
\begin{remark}
 To understand Theorem \ref{thm_main_cltlss}, one might consider a specific example with $f(x)=x^2$ (Schott's statistics). Calculations show that the asymptotic variance is given by $2np^2\beta_4^2+4p^2n^{-2}$ where $\beta_4=\mathbb{E}(Y_{11}^4)$. It is important to note that the term $np^2\beta_{4}^2$ exhibits a phase transition around $\alpha=3$ (depending on the slowly varying function $l(x)$) due to $\beta_{4}\sim C_\alpha n^{-\alpha/2}l(n^{1/2})$ for some constant $C_{\alpha}$ as shown in  Lemma \ref{lem_moment_rates} (see Example \ref{example} for more details).
Secondly, we point out that the symmetry condition for the entries of $X$ is not essential for the universal CLT. As observed in \cite{heiny2023logdet}, simulation results always show that the asymptotic variance of the LSS for the log-determinant is slightly larger in the non-symmetric case than the symmetric case. Consequently, the authors of \cite{heiny2023logdet} conjecture that the symmetry condition is crucial for the logarithmic law, based on their simulation findings.
Indeed, we observe the same phenomenon, but we find that the larger variance in the non-symmetric case is due to the slow convergence of $\mathbb{E}(Y_{11}^4)$, which is highly dependent on the tail distribution of $\xi$, particularly the asymptotic behavior of $\mathbb{E}(\exp(-s\xi^2))$. Both theoretical analysis and simulations indicate that the error shrinks slowly due to the large factor $np^2$, leading to a large variance in numerical results. This fact can also be viewed from another perspective, as demonstrated in the earlier work \cite{gine1997student}, where the authors show that self-normalized sums (e.g., $\sum_{i}Y_{1i}$) remain stochastically bounded when $\xi$ is symmetric around $0$, a result derived from Khinchin's inequality (see Remark 2.7 of \cite{gine1997student}). Consequently, the symmetry condition leads to much faster convergence of $\mathbb{E}(Y_{11}^4)$.
	For the non-symmetry case, Corollary 2.10 in \cite{gine1997student} provides the following necessary and sufficient condition for the sequence of self-normalized sums to remain stochastically bounded,
	\begin{equation*}
		\limsup_{n} n^2\int_{0}^{\delta} [\mathbb{E}(\xi\exp(-s \xi^2))]^2[\mathbb{E}(\exp(-s \xi^2))]^{n-2}\mathrm{d}s<\infty,
	\end{equation*}
	which highlights how non-symmetry affects the convergence of self-normalized sums, as the term $\mathbb{E}(\xi\exp(-s \xi^2))$ equals zero if and only if $\xi$ is symmetric around zero \cite{jonsson2010quadratic}.
\end{remark}

\subsection{Proof strategies}\label{sec_proofstrategy}
The CLT for the LSS of large sample correlation matrices is a fundamental topic in Random Matrix Theory. Most research has focused on the case with a finite fourth-moment condition, which is crucial for ensuring that the self-normalized factor $(\operatorname{diag}S)^{-1}$ can be accurately approximated by the identity matrix \cite{bao2022spectral,heiny2018almost}. However, this is no longer valid in the heavy-tailed context.
Thus, a different technical approach is required to capture the essential elements for obtaining the CLT. Another challenge is that most existing literature only addresses the limiting properties of heavy-tailed randomness under symmetric conditions. Therefore, it is crucial to extend these results to non-symmetric cases or, at the very least, to understand the differences between symmetric and non-symmetric cases.

Intuitively, although the entries in $X$ are heavy-tailed distributed, we believe that most elements in the correlation matrix $R$ exhibit regularity, with only a small portion experiencing significant but controllable fluctuations due to self-normalization properties.  Inspired by this idea and the work \cite{bao2023smallest}, we begin with a resampling for $X$ first.  Via resampling technique, we can decompose the matrix $X^*$ into $L+M+H$, representing the light-tailed, median-tailed, and heavy-tailed components, respectively (see Section \ref{sec_matrixresampling}). As in \cite{bao2023smallest}, the dependence among $L$, $M$ and $H$ is governed by two indicator matrices, $\Psi$ and $\chi$ (refer to \eqref{eq_decomp_X}). As a result, the self-normalized matrix $Y$ can be reformulated as
$Y=L(\operatorname{diag}S)^{-1/2}+\widetilde{H}(\operatorname{diag}S)^{-1/2}$, where $\widetilde{H} = M+H$. We hope that the resampling allows the light-tailed component $L$ to regularize the spectrum of the heavy-tailed component $\widetilde{H}$. Since the key input for the CLT of the LSS involves the estimation of Green functions of $R$, our second step is to establish the local laws for $R = YY^*$. To this end, we compare the original matrix $Y$ with the Gaussian divisible model (GDM) $Y_t=\sqrt{t}W+\widetilde{H}(\operatorname{diag}S)^{-1/2}$, where the parameter $t=n\mathbb{E}|L_{ij}(\operatorname{diag}(S))_{jj}^{-1}|^2$ is chosen so that the second moment of $\sqrt{t}W_{ij}$ matches that of $L_{ij}(\operatorname{diag}S)_{jj}^{-1/2}$, given the indictor variable $\psi_{ij}=0$ (see Section \ref{sec_matrixresampling}). Such a matching idea in the context of random matrices appeared earlier in \cite{tao2011random}. Here $W$ is a Gaussian orthogonal ensemble (GOE) matrix independent of $X$. To implement the aforementioned idea, drawing inspiration from the Dyson-Brownian Motion (DBM) strategy in \cite{bao2023smallest}, we adopt a three-step strategy to derive the local laws for $R$. Firstly, we establish an intermediate local law of $R(\widetilde{H})=\widetilde{H}(\operatorname{diag}S)^{-1}\widetilde{H}^*$. We demonstrate that the spectrum of $R(\widetilde{H})$ is bounded both below and above at a scale $\eta_{*}=n^{-\tau}t$ for some constant $\tau>0$, which is commonly referred to $\eta_{*}$-regularity in the literature. Secondly, leveraging the $\eta_{*}$-regularity, we use the results from \cite{ding2022edge} to show that the $\sqrt{t}W$ component enhances the spectral regularity to the nearly optimal bulk scale $\eta\ge n^{-1+\delta}$, thereby directly implies the bulk universality of the GDM $R(Y_t)=Y_tY_t^{*}$. Thirdly, using the GDM as a starting point for Green function comparison, we establish the local laws for the original matrix $R$. Finally, to obtain the CLT, we need to prove the asymptotic normality of the LSS $\operatorname{tr}f(R)$ (cf. \eqref{eq_def_LSS}). The classical moment method is ineffective due to the poor performance of higher-order moments of the self-normalized random variables in $R$. To overcome this challenge,  we estimate its characteristic function using a careful expansion, showing that the expanded characteristic function is asymptotically equivalent to that of a Gaussian random variable.  This strategy is firstly inspired by the method used in \cite{bao2022spectral} and then modified extensively to adapt our scenario.

In summary, our proof primarily involves the DBM, local laws estimation for the GDM, Green function comparison, and characteristic function estimation. Several arguments adapt the methodologies from \cite{bao2022spectral,bao2023smallest,ding2022edge,hwang2019local}. However, these adaptations are far from straightforward. Below, we summarize some key ideas.

(1) \textit{Intermediate local law}:\quad As in many earlier DBM works, a crucial step prior to analysis is the establishment of intermediate local laws for the heavy-tailed component $\widetilde{H}(\operatorname{diag} S)^{-1/2}$. Specifically, we need to achieve $\eta_{*}$-regularity for the spectral density of $R(\widetilde{H})=\widetilde{H}(\operatorname{diag} S)^{-1}\widetilde{H}^*$ along the support of the MP law. Achieving these results involves two significant challenges.  The first is that local law results for such a heavy-tailed matrix are only available for sample covariance matrices. Specifically, we can only establish bulk regularity and edge regularity for $S(M):=n^{-1}MM^{*}$ (see Lemma \ref{lem_locallaw_corM}). We hope to extend this local scale convergence to the correlation matrix $R(M)$. This involves comparing the scaling difference between $(\operatorname{diag} S)^{-1}$ and $n^{-1}I$ for $R(M)$ and $S(M)$. To do this, we must further decompose the diagonal elements in $n^{-1}S$, observing that most of the diagonal elements of $n^{-1}S$ (say, $(1-\mathrm{o}(1))n$) are sufficiently close to one, while only $\mathrm{o}(n)$ diagonal elements are at a constant or higher order distance from one (see Section \ref{sec_selfnormalisedrandomvariables}). This finding indicates that the local estimations of $R(M)$ and $S(M)$ are close within a controllable margin. For the edge regularity, we employ an eigenvalue sticking argument to show that $\lambda_1(R(M))$ is close to $\lambda_1(R(M^{(\mathrm{Cn})})$ where $M^{(\mathrm{Cn})}$ is the matrix obtained by deleting columns associated with the $\mathrm{o}(n)$ largest diagonal elements of $S$. It is clear that $\lambda_1(R(M^{(\mathrm{Cn})})$ is close to $\lambda_1(S(M^{(\mathrm{Cn})})$. With these discussions, we further obtain the $\eta_*$-regularity for $R(M)$, which serves as an initial estimate for the regularity of $R(\widetilde{H})$. The second challenge is to derive the local laws of $R(\widetilde{H})$ from $R(M)$. It is relatively straightforward to establish the bulk regularity of $R(\widetilde{H})$ (i.e., the averaged local law) since the rank of $H$ is $\mathrm{o}(n)$ conditional on $\Psi$. However, achieving
edge regularity at the right edge of the spectrum is more difficult and cannot always be guaranteed \cite{auffinger2009poisson}.  Fortunately, for our purposes, we do not require the optimal convergence rate at a precise location of the right edge. We only need to demonstrate the square root behavior of the spectral density at a much relaxed level $\eta=n^{-\epsilon}$,  which is adequate for our needs, and ensure the validity of our contour integral in the LSS (maintaining the right support edge at a constant order). Thus, a sufficient small shifting on the right edge is tolerable to obtain the necessary regularity (cf. Proposition \ref{prop_etaregular_corH}).

(2) \textit{Gaussian divisible model}:  With the $\eta_{*}$-regularity on the support confirmed, we turn our attention to the GDM $Y_t=\sqrt{t}W+\widetilde{H}(\operatorname{diag}S)^{-1/2}$ with $1\gg t\gg\sqrt{\eta_{*}}$ as chosen. Using the rectangular free convolution techniques from \cite{capitaine2018spectrum,ding2022edge}, we establish the global law for the GDM by examining the self-consistent equation \eqref{eq_def_LSD1_gdm}. It should be noted that we only need to discuss the problem on our proposed contours in Section \ref{sec_mainresults}.
For the local scale, conditional on the latter part $\widetilde{H}(\operatorname{diag}S)^{-1/2}$, we can establish the averaged local law (cf. Theorem \ref{thm_locallaw_gdm_average}) and entrywise local law (cf. Theorem \ref{thm_locallaw_gdm_entrywise}). The proof of the averaged local law adapts the method from \cite{ding2022edge}. Specifically, for the entrywise local law, conditional on the label matrix $\Psi$, we classify the indices $i,j$ into ``good'' and ``bad'' parts based on $\psi_{ij}$ (see \eqref{eq_indexPsi}). We expect that most indices fall into the ``good'' category, achieving the entrywise bounds $|\mathcal{G}R_{Y_t,ij}(z)-\delta_{ij}m(z)|\prec n^{-c}$ for the Green function $\mathcal{G}R_{Y_t,ij}(z):=(R(Y_t)-zI)^{-1}_{ij}$ with some constant $c>0$. Meanwhile, the $\mathrm{o}(n)$ ``bad'' indices are crudely bounded by a constant due to the construction of the contours.
It is important to highlight that while most methods are borrowed from sample covariance matrices or Wigner matrices, the dependence among the entries of $\widetilde{H}$ requires special attention, resulting in a slower convergence rate of $n^{-c}$ for some constant $c>0$.  These local laws, however, will offer robust tools for the following steps.

(3) \textit{Green function comparison}: In this step, we will extend the local laws for the GDM $R(Y_t)$ to our original matrix $R(Y)$, via a Green function comparison argument. To this end, we define the interpolations
\begin{gather*}
	Y^{\gamma}=\gamma L(\operatorname{diag}S)^{-1/2}+t^{1/2}(1-\gamma^2)^{1/2}W+\widetilde{H}(\operatorname{diag}S)^{-1/2},\quad  R^{\gamma}=Y^{\gamma}(Y^{\gamma})^*,\\
	\mathcal{G}^{\gamma}(z)=(R^{\gamma}-zI)^{-1},\quad m_n^{\gamma}(z)=\frac{1}{n}\operatorname{tr}\mathcal{G}^{\gamma}(z).
\end{gather*}
We aim to show that the averaged local law and entrywise local law for $\mathcal{G}^0(z)$ also hold for $\mathcal{G}^{\gamma}(z)$ for all $\gamma\in[0,1]$. Such an argument can be reduced to controlling the error of the expectation for some smooth function $F$ of the entries of $\mathcal{G}^{\gamma}(z)$ and $\mathcal{G}^0(z)$, for example,
\begin{equation*}
	\sup_{0\le \gamma\le 1}|\mathbb{E}_{\Psi}(F(\operatorname{Im}[\mathcal{G}^{\gamma}(z)]_{ij}))-\mathbb{E}_{\Psi}(F(\operatorname{Im}[\mathcal{G}^{0}(z)]_{ij}))|.
\end{equation*}
The expression above is appropriate for performing a Taylor expansion argument. For example, we interpret $E_{\Psi}(F(\operatorname{Im}[\mathcal{G}^{\gamma}(z)]_{ij}))$ as a multivariate function of the entries of $Y^{\gamma}$ and perform the expansion around each $Y_{ij}^{\gamma}$. Subsequently, by employing a moment matching argument, we can control the error efficiently. This method is exemplified in
 \cite{bao2023smallest}. However, we cannot directly apply the same routine to our model due to the strong correlation within the columns of $Y^{\gamma}$ introduced by $(\operatorname{diag}S)^{-1/2}$. The dependence structure of the correlation matrix introduces two main difficulties. Firstly, performing a Taylor expansion for each entry $Y^{\gamma}_{ij}$ is challenging because the information contained in $Y^{\gamma}_{ij}$ is influenced not only by itself but also by other entries $Y^{\gamma}_{ij}, 1\le i\le n$. Secondly, the strong correlation among entries makes it difficult to take expectations for individual entries $Y^{\gamma}_{ij}$, rendering the moment matching procedure ineffective. To address these issues, we first apply the resolvent expansion to each column $Y_j^{\gamma}$,
\begin{equation*} \mathcal{G}^{\gamma}=\sum_{k=0}^s\mathcal{G}^{\gamma,0}_{(j)}(-Y^{\gamma}_j(Y^{\gamma}_j)^{*}\mathcal{G}^{\gamma,0}_{(j)})^k+\mathcal{G}^{\gamma}(-Y^{\gamma}_j(Y_j^{\gamma})^{*}\mathcal{G}^{\gamma,0}_{(j)})^{s+1},
\end{equation*}
for some fixed integer $s\ge 0$, where $\mathcal{G}^{\gamma,0}_{(j)}$ denotes the corresponding Green function with the $j$-th column in $Y^{\gamma}$ replaced by $0$. We then perform a Taylor expansion for the function $F$ and decompose each $Y^{\gamma}_j$ into the light-tailed and heavy-tailed components derived from the resampling procedures for $X$.  Another barrier is the loss of symmetry, which renders many conventional estimation methods ineffective when comparing the moments. It is obvious that only the even order moments of $Y^{\gamma}_{ij}$'s remain non-degenerate after taking expectation under the symmetry condition. To overcome the sticky non-symmetry, we need improved estimates of the odd order moments of $Y^{\gamma}_{ij}$'s to offset the cumbersome across terms, which can be achieved by a careful examination of the decomposition and the cumulant expansion method. For example, for $\alpha>2$, the classical estimate $\mathbb{E}(Y_{11}Y_{21})=\mathrm{o}(n^{-2})$ by \cite{gine1997student} can be refined to the intermediate $\alpha$-dependent rate $\mathbb{E}(Y_{11}Y_{21})=\mathrm{o}(n^{-\alpha+\epsilon})\mathbbm{1}(2<\alpha<3)+\mathrm{o}(n^{-3})\mathbbm{1}(3\le\alpha<4)$ for some small enough $\epsilon>0$. Then with the large deviation bound for self-normalized random variables and carefully recognizing heavy-tailed entries and light-tailed entries, we eventually have enough convergence rate to control the error for the Green function comparison. One may refer to Section \ref{sec_greenfunctioncomparison} for more details. We believe that our strategies not only effectively address our problem but also have potential applications in other models with strong column dependence or without symmetry. It is noteworthy that the averaged local law for $\mathcal{G}^{\gamma}(z)$ achieves an optimal order of $1/(n\eta)$ for $\alpha\ge 3$, while the entrywise local law exhibits slower convergence than $n^{-1/2}$ due to the heavy-tailed assumption. Fortunately, most entries are stochastically dominated by $n^{-\epsilon}$ for some constant $\epsilon>0$, with only a few entries having constant order within the contours.

(4) \textit{Characteristic function estimation}:  Finally, we are ready to establish the CLT for $\operatorname{tr}f(R)$, which boils down to estimating the characteristic function of the centered statistic $\operatorname{tr}f(R)-\mathbb{E}\operatorname{tr}f(R)$. In practice, the LSS $\operatorname{tr}f(R)$ can be expressed as an integral involving $\operatorname{tr}\mathcal{G}(z)$ over the contours $\bar{\gamma}_{1}^0$ or $\bar{\gamma}_2^0$, using a classical truncation trick for $|\mathrm{Im}z|\ge n^{-K}$, where $K$ is a sufficiently large constant. That is
\begin{equation*}
	\operatorname{tr}f(R)\approx -\frac{1}{2\pi \mathrm{i}}\oint_{\bar{\gamma}_{a}^{0}}\operatorname{tr}\mathcal{G}(z)f(z)\mathrm{d}z =:L_{a}(f),~~a=1,2.
\end{equation*}
We define the characteristic function $\mathcal{T}_{f}(x):=\mathbb{E}(\mathrm{e}^{\mathrm{i}x(L_{1}(f)-\mathbb{E}L_{1}(f))})$, and study the asymptotic behaviors of its derivative
\begin{equation*}
	 (\mathcal{T}_{f}(x))^{\prime}=\frac{-1}{2\pi}\oint_{\bar{\gamma}_{1}^0}\mathbb{E}^{\chi}[\operatorname{tr}\mathcal{G}(z)\mathrm{e}^{\mathrm{i}x(L_{1}(f)-\mathbb{E}L_{1}(f))}]f(z)\mathrm{d}z.
\end{equation*}
We hope that the above equation will mirror the characteristic function of Gaussian random variables. Such a match would demonstrate the asymptotic normality of $\operatorname{tr}f(R)$. As pointed out by \cite{bao2022spectral}, the term $(\mathcal{T}_{f}(x))^{\prime}$ is suitable to the application of the cumulant expansion trick (cf. Lemma \ref{lemma_cumulantexpansion}) since it can be rewritten as $\sum_{i=1}^{n}\sum_{j=1}^{p}\mathbb{E}^{\chi}[Y_{ij}[Y^{*}\mathcal{G}(z)]_{ji} \mathrm{e}^{\mathrm{i}x\langle L_{1}(f)\rangle}]$ for $z\in \bar{\gamma}_1^0$ due to the resolvent identity. The cumulant expansion formula exhibits significant power in studying asymptotic behaviors in random matrix theory, a method proposed in \cite{khorunzhy1996asymptotic,lytova2009central}. For sample covariance matrices, such as $\mathbb{E}(X_{ji}[X^*(XX^*-zI)^{-1}]_{ij})$, the formula can be applied to the entries under the assumption of independence. However, in our correlation matrix case, the strong dependence among $Y_{ij}$ makes the direct application of cumulant expansion formula infeasible. The dependence introduces two major obstacles. Firstly,  the self-normalized factor $(\operatorname{diag}(S))^{-1}$ complicates the partial derivatives of the matrix $Y$ if the cumulant expansion is applied to $Y_{ij}$.  Secondly, the slow convergence of the entrywise local law may not be sufficient to accurately capture the magnitudes of Green function functionals (as noted in Theorem \ref{thm_greenfuncomp_entrywise_uniformbound}) because of the dependence and heavy-tailed conditions. However, as demonstrated in the resampling for $X$, the light-tailed part and the heavy-tailed part are ``independent'' due to the resampling indicators $\psi_{ij}$ and $\chi_{ij}$, which inspires us to apply the cumulant expansions to $L_{ij}, M_{ij}$, respectively. Specifically, we consider the cumulant expansions for
\begin{equation*}
	\mathbb{E}_{\Psi}^{\chi}\left((L_{ij}+M_{ij}+H_{ij})(\operatorname{diag}S)^{-1}_{jj}[Y^*\mathcal{G}(z)]_{ji}\langle\mathrm{e}^{\mathrm{i}x\langle L_{1}(f)\rangle}\rangle\right),
\end{equation*}
with an additional expansion to $(\operatorname{diag}S)^{-1}_{jj}$. We first consider the $L_{ij}+M_{ij}$ part for $\alpha\in (3,4]$. The vital inputs are the respective derivatives of $Y, \mathcal{G}(z)$ and $Y^*\mathcal{G}(z)$ with respect to $M_{ij}$, which have complicated structure due to the self-normalization. Inspired by the sample covariance case \cite{bao2022statistical}, we carefully examine the column dependence and find that the derivatives can be decomposed systematically via the index matrices $E_{ij}, F_{ij}$ into two parts (cf. Section \ref{subsubsec_derivative}), whose main terms alternate according to the order of the derivatives. We then obtain the high-order derivatives via the chain rule. Despite the special structure of the derivatives, the diagonals and off-diagonals in $\mathcal{G}(z)$ also exhibit regularity due to the index matrices. This, together with the averaged and entrywise local laws, simplifies the computations. We note that the estimates are always valid for $\alpha \in (2,4]$, and the strategies are user-friendly for handling the dependence. Then, we turn to estimate the expansion terms involving $M_{ij}$. Here the main difficulty arises from the unknown fluctuation of $(\operatorname{diag}S)^{-1}_{jj}$ and the slow convergence rate of the entrywise local law due to heavy tails, preventing accurate estimates of the entries.  However, since most entries are well-stochastically dominated and only a few have crude constant bounds, we sum over $i, j$ and control the errors by noting that $\operatorname{tr}((\operatorname{diag}S)^{-1}-n^{-1}I)\lesssim n^{-\epsilon}$. The fast convergence of the averaged local law helps exclude the problematic entries, leading to negligible errors. The remaining task is to demonstrate that the $H_{ij}$ component is negligible, which implies that the CLT for the LSS shares the same mean and variance functions as in the case of $\alpha > 4$ (i.e., the finite fourth-moment setting). At this stage, the cumulant expansion is invalid for $H_{ij}$ due to infinite high-order moments. To address this issue, we represent the $H_{ij}$ part as a polynomial function of the quadratic form $Y_j^*\mathcal{G}(z)Y_j$ associated with the label indices $\psi_{ij}$ by the Taylor expansion. The main challenge in bounding these high-order terms up to $\mathrm{o}(1)$ is that we can not apply the concentration of the quadratic forms in \cite{bai2010spectral,gao2014high,heiny2023logdet}. To tackle this, we provide an improved version (cf. Lemma \ref{lemma_momentquadratic}) for $\alpha \in (2,4]$ without symmetry condition,
\begin{equation*}
	\mathbb{E}|Y_j^*AY_j-n^{-1}\operatorname{tr}A|^k\lesssim n\mathbb{E}(Y_{11}^4)+\mathrm{o}(n^{-1/2}),~k\ge 2,
\end{equation*}
where $A$ is an $n\times n$ deterministic symmetric matrix with bounded spectrum. This reveals the concentration rate in terms of the tail distribution of $X_{ji}$ via $\mathbb{E}(Y_{11}^4)$ (cf. Lemma \ref{lem_moment_rates}). This concentration result always works for $\alpha\in (2,4]$, suggesting the robustness of self-normalization compared to the covariance matrix setting. Another barrier is to bound the first-order term which can not be achieved directly due to the dependence between quadratics and $\mathcal{T}_f(x)$. To address this, we take a further derivative of $(\mathcal{T}_f(x))^{\prime}$, which exhibits a well-behaved structure. From this we can obtain the variance of the asymptotic normal distribution by letting $x=0$ (cf. Section \ref{sec_characteristicfunctionestimation}). After a careful examination of the indices, with a combination of the resolvent representation and cumulant expansion, we show that the first-order term is also negligible. Overall, we finally arrive at the desired approximate ODE:
\begin{equation*}
	(\mathcal{T}_{f}(x))^{\prime}=-x\sigma_{f}^2\mathcal{T}_{f}(x)+\mathrm{error},
\end{equation*}
which implies the asymptotic normality of $\operatorname{tr}f(R)$, see Section \ref{sec_characteristicfunctionestimation} for details.

Moreover, a meticulous look at the example of $\operatorname{tr}R^2$ (cf. Example \ref{example}) suggests that the threshold $\alpha=3$ with proper choice of slowly varying function is critical for the universal CLT of the LSS. Via extra arguments, we find a weak third-moment condition \eqref{eq_criticalcase} is necessary and sufficient for the universal CLT of $\operatorname{tr}f(R)$, which relies on the concentration of quadratics and the well-configured properties of $\Psi$. In particular, the asymptotic behaviors of the quadratics are governed by the moments of $Y_{ij}$, heavily influenced by the tail distribution of $\xi$. Consequently, the necessary and sufficient condition \eqref{eq_criticalcase} provides new insights into the ultra heavy-tailed case, such as the sample correlation matrix for $\alpha \in (2,3)$. We believe that the asymptotic normality for the LSS still holds in this case, but the asymptotic mean and variance functions will depend on $\alpha$. We also remark here that our CLT coincides with the asymptotic normality of the log-determinant of $R$ (i.e., $f(x)=\log x$) in \cite{heiny2023logdet} under the symmetry conditions.

\subsection{Organization}
The rest of the paper is organized as follows. Section \ref{sec_pre} collects several preliminary results which are the foundation of our discussion below, whose proofs are deferred to Appendix \ref{prf_sec_pre}.  In Section \ref{sec_secondorderconvergence}, we gradually present our main technical steps to derive the CLT. Specifically, Section \ref{sec_gaussiandivisiblemodel} gives the main results for the Gaussian divisible model via the intermediate local law for the heavy-tailed part on desired domains, from the global structure to the local scale estimations, whose proof will be put into Appendix \ref{app_sec_prf_edge}. Section \ref{sec_greenfunctioncomparison} is devoted to the averaged and entrywise local laws of $R$ by the Green function comparisons, whose proofs are postponed to Appendix \ref{app_sec_Greencom}. Section \ref{sec_characteristicfunctionestimation} illustrates the main idea for estimation of the characteristic function based on our Green function comparison results, and the critical condition for the universal CLT. The proof of the asymptotic normality for our proposed LSS of $R$ is postponed to Appendix \ref{app_sec_charafun}.

\subsection{Convention}
Throughout this article, we set $C,c>0$ to be constants whose value may be different from line to line in request. $\mathbb{Z}$, $\mathbb{Z}_{+}$, $\mathbb{R}$, $\mathbb{R}_{+}$, $\mathbb{C}$, $\mathbb{C}^{+}$
denote the sets of integers, positive integers, real numbers, positive real numbers, complex numbers, and the upper half complex plane respectively. We write $[n]:=\{1,\ldots,n\}$ for $n\in \mathbb{Z}_+$. For any value $a,b\in\mathbb{R}$, $a\land b=\min(a,b)$ and $a\lor b=\max(a,b)$. For any two statements $\mathcal{A},\mathcal{B}$, we denote $\mathcal{A}\land\mathcal{B}$ as the logic $\mathcal{A}$ and $\mathcal{B}$ both hold, and we denote $\mathcal{A}\lor\mathcal{B}$ as the logic $\mathcal{A}$ holds or $\mathcal{B}$ holds. For a complex number $z$, we write $z=\operatorname{Re} z+\mathrm{i}\operatorname{Im} z$ where $\mathrm{i}=\sqrt{-1}.$ For a matrix $A=(A_{ij})\in\mathbb{C}^{n\times p}$, $|A|=(|A_{ij}|)$, ${\rm tr}A$ denotes the trace of $A$, $\|A\|$ denotes the spectral norm of $A$ equal to the largest singular value of $A$ and $\|A\|_{F}$ denotes the Frobenius norm of $A$ equal to $(\sum_{ij}|A_{ij}|^{2})^{1/2}$. For any index sets $\mathrm{I}_r\subset [n]$ and $\mathrm{I}_c\subset [p]$, we use the notation $A^{(\mathrm{I}_r)}$ and $A^{[\mathrm{I}_c]}$ to indicate the matrix $A$ with the corresponding $\{\mathrm{I}_r\}$ rows and $\{\mathrm{I}_c\}$ columns being replaced with zero, respectively. For $n\in\mathbb{Z}_{+}$, ${\rm diag}(a_{1},\dots,a_{n})$ denote the diagonal matrix with $a_{1},\dots,a_{n}$ as its diagonal elements. For two sequences of numbers $\{a_{n}\}_{n=1}^{\infty}$, $\{b_{n}\}_{n=1}^{\infty}$, $a_{n}\asymp b_{n}$ if there exist
constants $C_{1},C_{2}>0$ such that $C_{1}|b_{n}|\le|a_{n}|\le C_{2}|b_{n}|$. We denote $a_n\lesssim b_n$ if there exist constant $C_1>0$ such that $|a_n|\le C_1|b_n|$. Also, we write $a_n\lesssim b_n$ as $b_n\gtrsim a_n$. We say $\mathrm{O}(a_{n})$ and $\mathrm{o}(a_{n})$ in the sense that $|\mathrm{O}(a_{n})/a_{n}|\le C$ with some constant $C>0$ for all large $n$ and $\lim_{n\to\infty}\mathrm{o}(a_{n})/a_{n}=0$. We also use $\mathrm{O}_{\mathbb{P}}$ and $\mathrm{o}_{\mathbb{P}}$ to denote the big and small $\mathrm{O}$ notation in probability. When we write $a_n\ll b_n$ and $a_n\gg b_n$, we mean $|a_n|/b_n\rightarrow0$ and $|a_n|/b_n\rightarrow\infty$ when $n\rightarrow\infty$. $I$ denotes the identity matrix of appropriate size. For a set $A$, $A^{c}$ denotes its complement (with respect to some whole set which is clear in the context). For a measure $\varrho$, ${\rm supp}(\varrho)$ denotes its support. For any finite set $T$, we let $|T|$ denote the cardinality of $T$. For any event $\Xi$, $\mathbbm{1}(\Xi)$ denotes the indicator of the event $\Xi$, equal to $1$ if $\Xi$ occurs and $0$ if $\Xi$ does not occur. For any $a,b\in\mathbb{R}$ with $a\le b$, $\mathbbm{1}_{[a,b]}(x)$ is equal to $1$ if $x\in[a,b]$ and $0$ if $x\notin[a,b]$. For any $x\in \mathbb{R}$, let $(x)_+:=x\mathbbm{1}(x>0)$. For any two random variables $x$ and $y$, we use $x\overset{D}{=}y$ to denote that the distributions of $x$ and $y$ are asymptotically the same. For a sequence of positive random variables $\{y_n\}$, we write $y_{(k)},1\le k\le n$ , as its order statistics such that $y_{(1)}\ge y_{(2)}\ge\dots\ge y_{(n)}>0$. For any number $x>0$, let $\lceil x \rceil$ represent the least integer not less than $x$, and $\lfloor x\rfloor$ denote the largest integer not larger than $x$.

\section{Preliminaries}\label{sec_pre}
In this section, we collect some necessary notation, technical tools, and auxiliary lemmas that we will use throughout the paper.

Recall the definition of $R$ in \eqref{eq_def_R}. We further rewrite $R$ as
\begin{equation*}
	R=YY^{*}=\sum_{j=1}^{p}P_j,
\end{equation*}
where
\begin{gather}\label{eq_def_P}
	Y=X^{*}(\operatorname{diag}S)^{-1/2}=(Y_1,\ldots,Y_p),~ P_j:=Y_jY_j^{*},~ Y_{ij}=X_{ji}/\rho_j
\end{gather}
with $\rho_j=\sqrt{X_{j1}^2+\dots+X_{jn}^2}$. One may observe that $P_j,\;1\le j\le p$ form a sequence of independent random projections. We define the Green function of $R$ as $\mathcal{G}(z):=(R-zI)^{-1}$. Its normalized trace, which is often known as the Stieltjes transform, is denoted as
\begin{equation*}
	m_n(z)=\frac{1}{n}\operatorname{tr}\mathcal{G}(z)=\int\frac{1}{\lambda-z}\mu_n(\mathrm{d}z),\quad z\in\mathbb{C}^{+}.
\end{equation*}
Besides, we denote the density of the \textit{Mar\v{c}enko-Pastur} distribution as $\mu_{mp,\phi}$ (cf. \eqref{eq_mpdensity}) and define its Stieltjes transform as
\begin{equation*}
	m(z):=\int_{\mathbb{R}}\frac{\mu_{mp,\phi}(\mathrm{d}t)}{t-z},\quad z\in\mathbb{C}^{+}.
\end{equation*}

Then, we introduce the notions of stochastic domination and high probability event, which is commonly seen in RMT literature.
\begin{definition}[Stochastic domination]\label{def_SD}
	(i). For two families of nonnegative random variables
	\[
	a=\{a_{n}(t):n\in\mathbb{Z}_{+},t\in T_{n}\},\qquad b=\{b_{n}(t):n\in\mathbb{Z}_{+},t\in T_{n}\},
	\]
	where $T_{n}$ is
	a possibly $n$-dependent parameter set, we say that $A$ is stochastically
	dominated by $B$, uniformly on $t$ if for all (small) $\varepsilon>0$
	and (large) $D>0$ there exists $n_{0}(\varepsilon,D)\in\mathbb{Z}_{+}$
	such that 	as $n\ge n_{0}(\varepsilon,D)$,
	\[
	\sup_{t\in T_{n}}\mathbb{P}\big(a_{n}(t)>n^{\varepsilon}b_{n}(t)\big)\le n^{-D}.
	\]
	If $a$ is stochastically dominated
	by $b$, uniformly on $t$, we use notation $a\prec b$ or $a=\mathrm{O}_{\prec}(b)$.
	Moreover, for some complex family $a$ if $|a|\prec b$ we also write $a=\mathrm{O}_{\prec}(b)$. \\
	(ii). Let $a$ be a family of random matrices and $\zeta$ be a family of nonnegative random variables. Then we denote $a=\mathrm{O}_{\prec}(\zeta)$ if $a$ is dominated by $\zeta$ under weak operator norm sense, i.e. $|\langle\mathbf{v},a\mathbf{w}\rangle|\prec\zeta\|\mathbf{v}\|_2\|\mathbf{w}\|_2$  for any deterministic vectors $\mathbf{v}$ and $\mathbf{w}$.\\
	(iii). For two sequences of numbers $\{a_{n}\}_{n=1}^{\infty}$,
	$\{b_{n}\}_{n=1}^{\infty}$, $a_{n}\prec b_{n}$ if for all $\epsilon>0$, $a_n\le n^{\epsilon}b_n$.
\end{definition}
\begin{definition}[High probability event]\label{def_HPE}
	We say that an $n$-dependent event $\Omega$ holds with high probability if there exists large constant $D>0$ independent of $n$, such that
	\begin{equation*}
		\mathbb{P}(\Omega)\ge 1-n^{-D},
	\end{equation*}
	for all sufficiently large $n$.
\end{definition}

We also introduce the following definitions to relieve the heavy notation of this paper,
\begin{definition}[``Truncated" expectation]\label{def_trucatuedexpectation}
	Let $\chi(x)$ be a smooth cutoff which equals $0$ when $x>2n^K$ and $1$ when $x< n^K$ for some sufficiently large constant $K>0$ and $|\chi^{(k)}(x)|=\mathrm{O}(1)$ for all $k\ge1$. We define the following truncated expectation for any random variable $\zeta$,
	\begin{equation*}
		\mathbb{E}^{\chi}(\zeta):=\mathbb{E}(\zeta\cdot\Xi),
	\end{equation*}
	where
	\begin{equation*}
		\Xi:=\prod_{j=1}^{p}\chi(\rho_j^{-1})\chi(\rho_j)
	\end{equation*}
	is used to control $\rho_j$ and $\rho_j^{-1}$ crudely but deterministically.
\end{definition}

In this paper, we employ several small quantities that are crucial for us to control the convergence rates in our proof. We classify all these small control parameters into the following two groups,
\begin{definition}[Control parameters]\label{def_controlparameter}
	In this paper, we will use the following several independent or related small quantities.
	\begin{itemize}
		\item [(i)] \textit{Global parameters.}  These parameters are used repeatedly throughout the paper.
		\begin{gather*}
			0<\epsilon_l<\min\{\epsilon_h,\epsilon_y,4-\alpha\}/1000,\quad 0<\epsilon_{h}<(\alpha-2)/10\alpha,\\
			0<\epsilon_s\le(\alpha-2)/6\alpha,\quad \beta=1/\alpha+1/2+2\epsilon_s,\quad 0<\epsilon_{y}<(\alpha-2)/10\alpha,\\
			\epsilon_{\alpha}:=\min\{\epsilon_h,\epsilon_s,1-\beta,\epsilon_y\},\quad \epsilon_{\beta}:=\min\{\epsilon_h,\epsilon_y/2\},\quad 0<\epsilon_{\mu}<\epsilon_{\alpha}/2.
		\end{gather*}
		\item [(ii)] \textit{Local parameters.} These parameters are used only within a certain section or paragraph. We use the following rule to name such parameters: $\epsilon_{a,b}>0$ and $\epsilon_{a,b}$ is small under some restrictions. Here the lower index $a$ denotes the section where such parameter appears and $b$ denotes that it is the $b$-th chosen small quantity in Section $a$.
	\end{itemize}
\end{definition}

\subsection{Self-normalized random variables}\label{sec_selfnormalisedrandomvariables}
In this section, we summarise several properties for $X_{ji}$'s, $\rho_j$'s, and $Y_{ij}$'s, which will be used in the remainder of the paper, where the proofs are deferred to Appendix \ref{prf_sec_pre}.  A key observation here is that, under Assumption \ref{ass_X}, the probability distribution of $\xi$ falls within the domain of attraction of the normal law (cf. \cite{gine1997student,levy1935proprietes}). We begin this section by presenting the following typical order of $X_{ji}$'s and $\rho_j$'s, as well as the asymptotic behaviors of $(\operatorname{diag}(S))^{-1}$.
\begin{lemma}\label{lem_priorest_x}
	For any $1\le j\le p$, suppose Assumption \ref{ass_X} holds and denote the order statistics of $X_{ji}$'s by $X_{j(1)}\ge X_{j(2)}\ge\dots\ge X_{j(n)}$. For constants $\epsilon_{2,1}\in(0,1/\alpha)$ and $C,c>1$, we have
	\begin{gather*}
		C^{-1}n^{1/\alpha}\log^{-1}n\le X_{j(1)}\le C n^{1/\alpha}\log n,~~
		X_{j(k)}-X_{j(k+1)}\ge cn^{\epsilon_{2,1}}\log^{-1}n,~ 1\le k\le \lceil n^{1/\alpha-\epsilon}\rceil,\\
		\rho_j\asymp n^{1/2},~
		\sum_{k=1}^{\lceil n^{1/\alpha-\epsilon_{2,1}}\rceil}X_{j(k)}^2=\mathrm{O}(n),
	\end{gather*}
	with probability tending to one.
 Moreover, we have for some constants $\epsilon_{2,2}>0$ and $c>0$ such that
	\begin{equation}\label{eq_rhoj_highpro}
		\mathbb{P}(\rho_j^{2}< \epsilon_{2,2}n)\lesssim \exp(-cn),~\text{for}~1\le j\le p,
	\end{equation}
	and for integer $s\ge 1$, with probability tending to one,
	\begin{equation}\label{eq_tr_SS}
		\operatorname{tr}|(\operatorname{diag}(S))^{-s}-n^{-s}I|
		\lesssim n^{-(s-1)-(\alpha-2)/(6\alpha)}.
	\end{equation}
\end{lemma}

Let $\Omega\equiv\Omega_n$ be the event of $\{X_{ji}\}$ such that the estimates in Lemma \ref{lem_priorest_x} hold. Under the event $\Omega$, we can estimate the magnitude of the coordinates of $Y_j$ as a consequence of Lemma \ref{lem_priorest_x}:
\begin{lemma}\label{lem_entrydistribution}
Under the event $\Omega$, for all $1\le j\le p$ and some constants $\epsilon_{2,3}>0, c_0>2$, there are $n^{1/\alpha-\epsilon_{2,3}}\log^{c_0}n$ elements in $Y_j$ with typical order at least $\mathrm{O}(n^{1/\alpha-1/2}\log n)$. Meanwhile, there are $n-n^{1/\alpha-\epsilon_{2,3}}\log^{c_0}n$ elements with typical order at most $\mathrm{O}(n^{-1/2}\log n)$.
\end{lemma}

In addition to the typical magnitudes of $X_{ij}$'s and $\rho_j$'s, the following results provide a precise description of the moments of self-normalized random variables (see, e.g., \cite{albrecher2007asymptotic,heiny2023logdet}).
\begin{lemma}\label{lem_moment_rates}
	Recall $Y_{ij}$ in \eqref{eq_def_P} and suppose Assumption \ref{ass_X} holds. Consider $k_1,\dots,k_q\ge1$ for $q\in\{1,\dots, n\}$. Then
	\begin{equation*}
		\lim_{n\rightarrow\infty}\frac{n^{N_1(1-\alpha/2)+q\alpha/2}}{l^{q-N_1}(n^{1/2})}\mathbb{E}(Y_{11}^{2k_1}\cdots Y_{q1}^{2k_q})=\frac{(\alpha/2)^{q-N_1}\Gamma(N_1(1-\alpha/2)+q\alpha/2)\prod_{i:k_i\ge2}^q\Gamma(k_i-\alpha/2)}{\Gamma(k_1+\dots+k_q)},
	\end{equation*}
	where $N_1=\#\{1\le i\le q:k_i=1\}$. In particular, we have
	\begin{equation*}
		 \lim_{n\rightarrow\infty}\frac{n^{\alpha/2}}{l(n^{1/2})}\mathbb{E}(Y_{11}^{2k})=\frac{\alpha\Gamma(\alpha/2)\Gamma(k-\alpha/2)}{2\Gamma(k)},\quad k\ge2.
	\end{equation*}
\end{lemma}

We also need the following asymptotic odd moments for self-normalized random variables to tackle the non-symmetry. All the odd moments will not vanish if the symmetry condition is violated, which causes inevitable difficulties concerning the expectation calculation. We emphasize that the classical bound $\mathbb{E}(Y_{11}Y_{21})=\mathrm{o}(n^{-2})$ for $\alpha\ge 2$ (cf. \cite{gine1997student}) is not enough for us to implement the Green function comparison and characteristic function estimation. We refine this to a more precise, $\alpha$-dependent rate.
\begin{lemma}\label{lem_oddmoment_est}
	Let $\alpha\in(2,4]$. Recall $Y_{ij}$ as in \eqref{eq_def_P} and suppose Assumption \ref{ass_X} holds. Consider $k_1,\dots,k_q\ge1$ for $q\in\{1,\dots,n\}$. Then
	\begin{equation*}
		\big|\mathbb{E}(Y_{11}^{k_1}\dots Y_{11}^{k_q})\big|=\mathrm{o}(n^{-q}),
	\end{equation*}
	when $k_i=1$ for at least one $1\le i\le q$. Moreover, we have
 \begin{equation*}
		\mathbb{E}(Y_{11}Y_{21})\lesssim n^{2(1/2-\epsilon_h)(3-\alpha)_{+}-3},
	\end{equation*}
	where $(x)_{+}:=x\mathbbm{1}(x>0)$ for any $x\in\mathbb{R}$.
\end{lemma}

\subsection{Matrix resampling}\label{sec_matrixresampling}
 Inspired by the work \cite{bao2023smallest}, we conduct a decomposition or more precisely a resampling of $X$, and further of the elements of vectors $Y_j$, based on Assumption \ref{ass_X} and Lemma \ref{lem_entrydistribution}. Consider the Bernoulli $0-1$ random variables $\psi_{ij}$ and $\chi_{ij}$,
\begin{equation*}
	\mathbb{P}(\psi_{ij}=1)=\mathbb{P}(|X_{ji}|\ge n^{1/2-\epsilon_h}),\quad \mathbb{P}(\chi_{ij}=1)=\frac{\mathbb{P}(|X_{ji}|\in[n^{-\epsilon_l},n^{1/2-\epsilon_h}])}{\mathbb{P}(|X_{ji}|<n^{1/2-\epsilon_h})}.
\end{equation*}
For any interval $\mathrm{I}\subset\mathbb{R}$, define the random variables $l,m$ and $h$ by
\begin{equation*}
\begin{split}
\mathbb{P}(l_{ij}\in \mathrm{I})&=\frac{\mathbb{P}\big[X_{ji}\in(-n^{-\epsilon_l},n^{-\epsilon_l})\bigcap\mathrm{I}\big]}{\mathbb{P}[|X_{ji}|\le n^{-\epsilon_l}]},\\
\mathbb{P}(m_{ij}\in\mathrm{I})&=\frac{\mathbb{P}\big[X_{ji}\in\big((-n^{1/2-\epsilon_h},-n^{-\epsilon_l}]\bigcup[n^{-\epsilon_l},n^{1/2-\epsilon_h})\big)\bigcap\mathrm{I}\big]}{\mathbb{P}\big[|X_{ji}|\in[n^{-\epsilon_l}, n^{1/2-\epsilon_h})\big]},\\
\mathbb{P}(h_{ij}\in \mathrm{I})&=\frac{\mathbb{P}\big[X_{ji}\in\big((-\infty,-n^{1/2-\epsilon_h})\bigcup(n^{1/2-\epsilon_h},\infty)\big)\bigcap\mathrm{I}\big]}{\mathbb{P}[|X_{ji}|\ge n^{1/2-\epsilon_h}]}.
\end{split}
\end{equation*}
It is easy to see $l_{ij}$ has the same law as $X_{ji}$ conditional on $|X_{ji}|\le n^{-\epsilon_l}$. Similar
statements hold for $m$ and $h$. For each $(i,j)\in[n]\times[p]$, we set
\begin{equation*}
	L_{ij}=(1-\psi_{ij})(1-\chi_{ij})l_{ij},\quad M_{ij}=(1-\psi_{ij})\chi_{ij}m_{ij},\quad H_{ij}=\psi_{ij}h_{ij},
\end{equation*}
where $l,m,h,\psi,\chi$-variables are mutually independent and the indexes $l,m,h$ represent the light-tailed, medium-tailed, heavy-tailed elements of $X^{*}$. Then $X^{*}$ can be represented as
\begin{equation}\label{eq_decomp_X}
	X^{*}=L+\widetilde{H},\quad \widetilde{H}=M+H.
\end{equation}
Consequently, by the decomposition in \eqref{eq_decomp_X}, we have the following decomposition for $Y$,
\begin{equation*}
	Y=L(\operatorname{diag}S)^{-1/2}+\widetilde{H}(\operatorname{diag}S)^{-1/2}.
\end{equation*}
Conditioned on $\psi_{ij}$, we rewrite the sample correlation matrix $R$ as
\begin{equation}\label{eq_def_RL&RH}
	R\equiv R(Y)=R(L)+R(\widetilde{H}):=L(\operatorname{diag}S)^{-1}L^{*}+\widetilde{H}(\operatorname{diag}S)^{-1}\widetilde{H}^{*},
\end{equation}
where we use the notation $R(\cdot)$ to denote the sample correlation matrix with its associated sample.

Further, we also need a parallel decomposition of the entries in $(\operatorname{diag}S)^{-1}$. By Lemma \ref{lem_priorest_x}, we know that $n^{-1}\rho_j^2-1$'s are well concentrated around zero. Thereafter, we define the Bernoulli $0-1$ random variables $\varphi_j$ and $\omega_j$ such that
\begin{equation*}
	\mathbb{P}(\varphi_{j}=1)=\mathbb{P}(|n^{-1}\rho_j^2-1|\ge  n^{\epsilon_h}),\quad \mathbb{P}(\omega_j=1)=\frac{\mathbb{P}(|n^{-1}\rho^2_j-1|\in[n^{-\epsilon_s},n^{\epsilon_h}])}{\mathbb{P}(|n^{-1}\rho_j^2-1|< n^{\epsilon_h})}.
\end{equation*}
Then for any interval $\mathrm{I}\in\mathbb{R}$, define the random variables $ls$, $ms$ and $hs$ by
\begin{equation*}
	\begin{split}
		 \mathbb{P}(ls_{j}\in\mathrm{I})=&\frac{\mathbb{P}\big[(n^{-1}\rho_j^2-1)\in(-n^{-\epsilon_s},n^{-\epsilon_s})\bigcap\mathrm{I}\big]}{\mathbb{P}\big[|n^{-1}\rho_j^2-1|\le n^{-\epsilon_s}\big]},\\
		 \mathbb{P}(ms_j\in\mathrm{I})=&\frac{\mathbb{P}\big[(n^{-1}\rho_j^2-1)\in\big((-n^{\epsilon_{h}},-n^{-\epsilon_s}]\bigcup[n^{-\epsilon_s},n^{\epsilon_{h}})\big)\bigcap\mathrm{I}\big]}{\mathbb{P}[|n^{-1}\rho_j^2-1|\in[n^{-\epsilon_s},n^{\epsilon_{h}})]}, \\
		 \mathbb{P}(hs_{j}\in\mathrm{I})=&\frac{\mathbb{P}\big[(n^{-1}\rho_j^2-1)\in\big((-\infty,-n^{\epsilon_{h}})\bigcup(n^{\epsilon_{h}},\infty)\big)\bigcap\mathrm{I}\big]}{\mathbb{P}\big[|n^{-1}\rho_j^2-1|> n^{\epsilon_{h}}\big]}.
	\end{split}
\end{equation*}
Set $ LS:=\operatorname{diag}(\{(1-\varphi_j)(1-\omega_j)\cdot ls_j\}_{1\le j\le p}), MS:=\operatorname{diag}(\{(1-\varphi_j)\omega_j\cdot ms_j\}_{1\le j\le p})$, and $HS:=\operatorname{diag}(\{\varphi_j\cdot hs_j\}_{1\le j\le p})$.
Then $n^{-1}(\operatorname{diag}S)$ can be rewritten as
\begin{equation}\label{eq_def_LS&HS}
	n^{-1}(\operatorname{diag}S-I)=LS+\widetilde{H}S=LS+MS+HS.
\end{equation}

Before we move ahead, a vital input of this paper is the following observation that controls the number of non-zero $\psi_{ij}$.
\begin{definition}\label{def_psi&pi}
	Define the indicator matrices $\Psi:=(\psi_{ij})\in\mathbb{R}^{n\times p}$ and $\Pi:=\operatorname{diag}[(\varphi+\omega)_j]\in\mathbb{R}^{p\times p}$. We call  $\Psi$ is well configured if there are at most $n^{1-\epsilon_y}$ entries equal to one in $\Psi$. Meanwhile, we say $\Pi$ is well configured if there are at most $n^{\beta}$ entries equal to one in $\Pi$.
\end{definition}
The following lemma shows that $\Psi$ and $\Pi$ are well configured with high probability under Assumption \ref{ass_X}, which is used frequently in the sequel to conduct the conditional argument.
\begin{lemma}\label{lem_wellconfigured}
	For any large constants $D_1,D_2>0$, we have $\mathbb{P}(\Omega_{\Psi}=\{\Psi\text{ is well configured}\})\ge 1-n^{-D_1}$ and $\mathbb{P}(\Omega_{\Pi}=\{\Pi\text{ is well configured}\})\ge 1-n^{-D_2}$.
\end{lemma}
\begin{proof}
	By Assumption \ref{ass_X}, we have
	\begin{equation*}
		\begin{split}
			\mathbb{P}&(\#\{(i,j):|X_{ji}|>n^{1/2-\epsilon_h}\}>n^{1-\epsilon_y})\lesssim\mathbb{P}(|\xi|_{(n^{1-\epsilon_y})}>n^{1/2-\epsilon_h})\\
			 &\lesssim\sum_{k=n^{1-\epsilon_y}}^{n^2}\binom{n^2}{k}n^{-\alpha(1/2-\epsilon_h)k}(1-n^{-1})^{n^2-k}\lesssim\sum_{k=n^{1-\epsilon_y}}^{n^2}n^{(1-\alpha/2+\epsilon_y)k}e^{-n}\lesssim n^{-D},
		\end{split}
	\end{equation*}
where $|\xi|_{(1)}\ge |\xi|_{(2)}\ge\cdots\ge|\xi|_{(np)}$ are the order statistics of $|\xi_1|,\cdots,|\xi_{np}|$.
The second statement follows from \eqref{eq_rho_number_hpe} in Section \ref{prf_sec_pre}.
\end{proof}

\subsection{First order convergence}\label{sec_firstorderconvergence}
In this section, we present the first order asymptotic behavior of $R$,  which is the limiting distribution of $\mu_n$.
 From the definitions \eqref{eq_def_R} and \eqref{eq_def_P}, we know $P_i$ is a projection matrix, whose ESD is trivially
\begin{equation*}
	\nu_i\equiv\nu_i(n)=\frac{1}{n}\delta_1+(1-\frac{1}{n})\delta_0,
\end{equation*}
almost surely. According to Free Probability Theory (cf. \cite{mingo2017free,voiculescu1991limit}), the random matrix $R$ can be regarded as a sum of $p$ random variable $P_i$ with Bernoulli distributions via free additive convolution.
Then, by Theorem 2.3 in \cite{bai2008large} (also Theorem 1.2 in \cite{jiang2004limiting}), under  Assumptions \ref{ass_X} and \ref{ass_phi}, the empirical measure $\mu_n$ converges weakly to the MP law $\mu_{mp,\phi}$, almost surely, with
\begin{equation}\label{eq_mpdensity}
	\mu_{mp,\phi}:=\frac{\sqrt{([(1+\sqrt{\phi})^2-x][x-(1-\sqrt{\phi})^2])_{+}}}{2\pi \phi x}\mathrm{d}x+(1-\phi^{-1})_{+}\delta_0,
\end{equation}
whose Stieltjes transform $m(z)$ satisfies the following equation
\begin{equation}\label{eq_mp}
	z\phi m(z)^2+(z-(1-\phi))m(z)+1=0.
\end{equation}

In the sequel, we define the left edge and right edge of the bulk of $\mu_{mp,\phi}$ as $\lambda_{-}:=(1-\sqrt{\phi})^2$ and $\lambda_{+}:=(1+\sqrt{\phi})^2$ respectively, which mean that $\operatorname{supp}(\mu_{mp,\phi})\subset \{0\}\bigcup [\lambda_{-},\lambda_{+}]$.  We should mention that although the limiting spectral distribution of
$R$ has support on a compact interval, which implies that the number of eigenvalues outside this interval is $\mathrm{o}(n)$ with probability tending to one,  it is insufficient to consider the CLT for the LSS within the interval. The eigenvalues outside this interval may still affect the LSS fluctuations. Therefore, stronger estimates of the extreme eigenvalues of $R$ are required to accurately determine the integration contour. The following lemma gives several key estimates which are useful in our analysis, whose proof is postponed to Appendix \ref{prf_lem_pre_estnorm}.
\begin{lemma}\label{lem_pre_estnorm}
	Under Assumptions \ref{ass_X} and \ref{ass_phi}, we have that $\|R\|$ is bounded with high probability. Further, for any fixed $z\in\gamma_1^0\bigcup\gamma_2^0$, we have $|n^{-1}\operatorname{tr}\mathcal{G}(z)|,\|\mathcal{G}(z)\|\sim1$ with high probability, when $\phi\in(0,1)$. The same bounds hold for $z\in\gamma_1\bigcup\gamma_2$ when $\phi\in (1,\infty)$.
\end{lemma}

Lemma \ref{lem_pre_estnorm} confirms the bounded support of $\mu_{n}$. As a consequence of Lemma \ref{lem_pre_estnorm} and Theorem 2.3 of \cite{bai2008large}, we have the following first order asymptotic result of $\mu_n$,
\begin{theorem}\label{thm_main_globallaw}
	Under Assumptions \ref{ass_X} and \ref{ass_phi}, for any fixed integer $l>0$, as $n\to \infty$,
	\begin{equation*}
		\int x^l\mathrm{d}\mu_n-\int x^l\mathrm{d}\mu_{mp,\phi}\rightarrow 0.
	\end{equation*}
\end{theorem}

\section{Proof sketch}\label{sec_secondorderconvergence}

In this section, we give a detailed proof sketch of Theorem \ref{thm_main_cltlss} and Theorem \ref{thm_iffalpha} as illustrated in the proof strategies (Section \ref{sec_proofstrategy}). The proof of Theorem \ref{thm_main_cltlss} consists of three key steps,
\begin{itemize}
\item[(i)] Gaussian divisible model: With the intuition that the light-tailed part $L\{\operatorname{diag}(S)\}^{-1/2}$ will regularize the spectrum of the heavy-tailed part $\widetilde{H}\{\operatorname{diag}(S)\}^{-1/2}$, we study the spectral distribution of the Gaussian divisible model $R(Y_t)=Y_tY_t^*$ with $Y_t:=\sqrt{t}W+\widetilde{H}\{\operatorname{diag}(S)\}^{-1/2}$ (cf. \eqref{eq_def_Gaussiandivisblemodel}). Precisely, we establish an intermediate entrywise local law for $\widetilde{H}$ (cf. Lemma \ref{lem_locallaw_H}) via $\eta$-regularity (cf. Proposition \ref{prop_etaregular_corH}), and then deduce to the local laws for $R(Y_t)$ (cf. Theorems \ref{thm_locallaw_gdm_average} and \ref{thm_locallaw_gdm_entrywise}, respectively).
\item[(ii)] Green function comparison: Starting from the initial local laws of GDM, we bridge the original sample correlation matrix ${R}$ and the GDM by Green function comparison with interpolation arguments (cf. \eqref{eq_def_interpolations}). With the aid of resolvent expansion (cf. \eqref{eq_prf_Greenfunction_resolvent_expansion}) and moments estimation (cf. Lemma \ref{lem_oddmoment_est}), we derive the entrywise local law and averaged local law for $R$ (cf. Theorems \ref{thm_greenfuncomp_entrywise_uniformbound} and \ref{thm_greenfuncomp_average}, respectively).
\item[(iii)] Characteristic function estimation: Based on the local laws for $R$, we finally demonstrate the asymptotic normality of LSS via the characteristic function estimation with the cumulant expansion method (cf. Lemma \ref{lemma_cumulantexpansion}). More precisely, for $\alpha>3$, with the decomposition in Section \ref{sec_matrixresampling}, we show that the regular part $(L+M)$ performs in the same way as the finite fourth-moment case (cf. Lemma \ref{lem_cumulant_estM}) and the heavy-tailed part $H$ is negligible (cf. Lemma \ref{lem_cumulant_estH}) via the concentration of quadratics (cf. Lemma \ref{lemma_momentquadratic}).
\end{itemize}

Moreover, the proof of Theorem \ref{thm_iffalpha} consists of two parts. On the one hand, we examine a concrete example $f(x)=x^2$ (Schott's statistics; see, e.g., \cite{bao2022spectral}), which suggests the critical case at $\alpha=3$ and thus motivates us to investigate the necessary and sufficient condition for the universal CLT. On the other hand, thanks to the averaged local laws of $R$ for general $\alpha\in (2,4)$, by a careful analysis of the label matrix $\Psi$ and the arguments for characteristic function estimation, we finally arrive at Theorem \ref{thm_iffalpha}, whose proof is deferred to Appendix \ref{sec_alpha3}.

The following three subsections contain the details of the main results. Section \ref{sec_gaussiandivisiblemodel} presents the local laws for GDM. The Green function comparison from GDM to $R$ is demonstrated in Section \ref{sec_greenfunctioncomparison}. Section \ref{sec_characteristicfunctionestimation} is devoted to estimating the characteristic function of LSS. All the detailed proofs are deferred to Appendices \ref{app_sec_prf_edge}, \ref{app_sec_Greencom}, and \ref{app_sec_charafun}, respectively.

\subsection{Gaussian divisible model}\label{sec_gaussiandivisiblemodel}
In this section, we prove the entrywise and averaged local laws for GDM by leveraging the $\eta$-regularity of $R(\widetilde{H})=\widetilde{H}(\operatorname{diag}S)^{-1}\widetilde{H}^*$ (cf. \eqref{eq_def_RL&RH}). We begin by setting up the work domains. Suppose that $\Psi$ and $\Pi$ are well configured. We only need to show the results for the situation $z\in\gamma_1\vee\gamma_2$ while one may follow the same procedures to obtain the same results for $z\in\gamma_1^0\vee\gamma_2^0$. Following the definition of the contours in Section \ref{sec_mainresults}, we will focus on the domain $\mathbf{D}=\mathbf{D}(C_1,C_2,C_3,C_4):=\mathbf{D}_{1}\bigcup\mathbf{D}_{2}$ on $\mathbb{C}_{+}$, where
\begin{equation*}
	\begin{split}
		&\mathbf{D}_{1}:=\{z=E+\mathrm{i}\eta:C_1\le E\le C_2,1/10\le \eta\le C_3\},\\
		&\mathbf{D}_{2}:=\{z=E+\mathrm{i}\eta:E<C_1 ~\text{or}~E>C_2,n^{-2/3-C_4}\le \eta\le C_3\},
	\end{split}
\end{equation*}
with appropriately chosen constants $0<C_1<\lambda_n(R(Y))<\lambda_1(R(Y))<C_2, 0<C_3$ and sufficiently small $C_4>0$. One can see that most of $\gamma_1\lor\gamma_2$ lies in the domain. Moreover, to cover the remaining part of the contours $\gamma_1\lor \gamma_2$ defined in Section \ref{sec_mainresults}, we define
\begin{equation*}
	\mathbf{D}_{3}:=\{z=E+\mathrm{i}\eta:E<C_1~\text{or}~E>C_2,0\le \eta\le n^{-2/3-C_4}\}.
\end{equation*}
To derive the local laws for GDM on the contours $\gamma_{1}\lor \gamma_2$ completely, we will first consider the domain $\mathbf{D}$ and then extend the results to $\mathbf{D}_3$.

Next, we present a simple description of GDM and collect some notation. Consider the Gaussian divisible model
\begin{equation}\label{eq_def_Gaussiandivisblemodel}
    Y_t=\sqrt{t}W+\widetilde{H}\{\operatorname{diag}(S)\}^{-1/2},\quad R({Y_t})=Y_tY_t^{*},
\end{equation}
 where $W=(w_{ij})\in \mathbb{R}^{n\times p}$ is a Gaussian matrix with i.i.d. $\mathrm{N}(0,n^{-1})$ entries. Conditional on any realization of $\widetilde{H}\{\operatorname{diag}(S)\}^{-1/2}$, $Y_t$ is well-known as Gaussian divisible model. Here one may regard $\widetilde{H}\{\operatorname{diag}(S)\}^{-1/2}$ as a deterministic signal part and $\sqrt{t}W$ as random noise. In this section, we choose $t=n\mathbb{E}|L_{ij}\rho_j^{-1}|^2$ with order $t\lesssim n^{-2\epsilon_l}$ by \eqref{eq_rhoj_highpro}. In the sequel, we define the Green function of $R(Y_t)$ as $\mathcal{G}R_{Y_t}(z):=(R({Y_t})-z I)^{-1}$ with its Stieltjes transform given by $m_{n,t}(z)\equiv mR_{Y_t}(z):=n^{-1}\operatorname{tr}(\mathcal{G}R_{Y_t}(z))$. We denote the asymptotic eigenvalue density of $R(Y_t)$ by $\rho_t$, which is characterized by its corresponding Stieltjes transform $m_t:=m_t(z)$ in \eqref{eq_def_LSD1_gdm}. When $t=0$, this reduces to the case of $R(\widetilde{H})$, with the Green function and Stieltjes transform denoted as $\mathcal{G}R_{\widetilde{H}}(z)$ and $mR_{\widetilde{H}}(z):=n^{-1}\operatorname{tr}\mathcal{G}R_{\widetilde{H}}(z)$, respectively. Since our discussion is based on the realization of $\widetilde{H}(\operatorname{diag}S)^{-1/2}$, it follows that the asymptotic version of $m_{n,0}(z)$ is equal to $m_0(z)$. Moreover, denote $m^{(t)}(z):=(1-t)^{-1}m(z/(1-t))$ for any $t>0$, where $m(z)$ represents the Stieltjes transform of the MP law \eqref{eq_mp}.

According to \cite{ding2022edge}, $m_{t}(z)$ is the unique solution of the following equation, for any $t>0$,
\begin{align}\label{eq_def_LSD1_gdm}
    m_{t}(z)=\frac{1}{n}\sum_{i=1}^{n}\frac{1+\phi tm_{t}(z)}{\lambda_{i}(R(\widetilde{H}))-\zeta_t(z)},
\end{align}
subject to the condition of $\operatorname{Im} m_{t}(z)>0$ for any $z\in \mathbb{C}^+$, where
\begin{align}\label{eq_zetarandom}
   \zeta_{t}(z)=(1+\phi t m_{t}(z))^2z-t(1-\phi)(1+\phi t m_{t}(z)).
\end{align}
Define $b_t:=b_t(z)=1+\phi t m_{t}(z)$. It is easy to see that from \eqref{eq_def_LSD1_gdm} $b_t$ satisfies the following equation
\begin{equation*}
    b_t=1+\frac{t\phi}{n}\sum_{i=1}^n\frac{1}{b_t^{-1}\lambda_i(R(\widetilde{H}))-b_tz+t(1-\phi)},
\end{equation*}
and thereafter
\begin{equation}\label{eq_def_zetat}
    \zeta_t:=\zeta_{t}(z)=zb_t^2-tb_t(1-\phi).
\end{equation}
After the above simplification, we can rewrite \eqref{eq_def_LSD1_gdm} as
\begin{equation}\label{eq_def_LSD2_gdm}
    \frac{1}{\phi t}(1-\frac{1}{b_t})=m_0(\zeta_t).
\end{equation}
It follows that
\begin{equation*}
    m_t(\zeta_t)=\frac{m_0(\zeta_t)}{1-\phi tm_0(\zeta_t)}.
\end{equation*}
In free probability theory \cite{capitaine2018spectrum,mingo2017free}, the density $\rho_t$ is the so-called rectangular free convolution of $\rho_0$ with the MP law at time $t$. The above $\zeta_t$ is the subordination function of the rectangular free convolution \cite{capitaine2018spectrum}. One may notice that $m_t(z)$ is an asymptotic version of $mR_{Y_t}(z)$.
Under any realization of $\widetilde{H}\{\operatorname{diag}(S)\}^{-1/2}$, \cite[Lemma 2]{ding2022edge} gives the existence and uniqueness of $m_{t}(z)$. For the convenience of readers, we repeat it here.
\begin{lemma}[Existence and uniqueness of $m_{t}(z)$]\label{lem_existenceuniqueness}
	For any $t>0$. the following properties hold,
    \begin{itemize}
        \item [(i)] There exists a unique solution $m_{t}(z)$ to \eqref{eq_def_LSD1_gdm} satisfying that $\operatorname{Im}m_{t}(z)>0$ and $\operatorname{Im}zm_{t}(z)>0$ for $z\in\mathbb{C}_{+}$.
        \item [(ii)] For all $E\in\mathbb{R}\setminus\{0\}$, $\lim_{\eta\downarrow0}m_{t}(E+\mathrm{i}\eta)$ exists, which will be denoted as $m_t(E)$. The function $m_{t}$ is continuous on $\mathbb{R}\setminus\{0\}$, and $\rho_t(E):=\pi^{-1}\operatorname{Im}m_{t}(E)$ is a continuous probability density function on $\mathbb{R}_{+}$. Moreover, $m_{t}$ is the Stieltjes transform of $\rho_t$. Finally, $m_{t}(E)$ is a solution to \eqref{eq_def_LSD1_gdm} for $z=E$.
        \item [(iii)] For all $E\in\mathbb{R}\setminus\{0\}$, $\lim_{\eta\downarrow0}\zeta_t(E+\mathrm{i}\eta)$ exists, which will be denoted as $\zeta_t(E)$. Moreover, $\operatorname{Im}\zeta_t(z)>0$, if $z\in\mathbb{C}_{+}$.
        \item [(iv)] $\operatorname{Re}(1+\phi tm_{t}(z))>0$ for all $z\in\mathbb{C}_{+}$ and $|m_{t}(z)|\le (\phi t|z|)^{-1/2}$.
    \end{itemize}
\end{lemma}

\subsubsection{$\eta_*$-regularity for $\widetilde{H}$}\label{sec_intermediatelocallaw}
To illustrate the intermediate local law for $R(\widetilde{H})$ (cf. \eqref{eq_def_RL&RH}), we first require the following definition, called $\eta_*$-regularity, which is typical in Dyson Brownian Motion; see, e.g., \cite{ding2022tracy}.
\begin{definition}[$\eta_*$-regularity]\label{def_etaregular}
	Let $\eta_*$ be a parameter satisfying $\eta_*:=n^{-\tau_*}$ for some constant $0<\tau_* \le 2/3$. For $p \times n$ matrix $V$, we say $R(V):=VV^*$ is $\eta_*$-regular if there exist constants $c_V>0$ and $C_V>1$ such that the following properties hold: \\
	(i) For $z=E+\mathrm{i} \eta $ with $\lambda_{-}(R(V))\le E \le \lambda_{+}(R(V))$ and $\eta_*+\sqrt{\eta_*\kappa_0} \le \eta \le 10$,
	\begin{equation*}
		\frac{1}{C_V} \sqrt{\kappa_0+\eta} \le \operatorname{Im} m_V(E+\mathrm{i} \eta) \le C_V \sqrt{\kappa_0+\eta},
	\end{equation*}
	where $\kappa_0:=\min\{|\lambda_{-}(R(V))-E|,|\lambda_{+}(R(V))-E|\}$.

Moreover, for $z=E+\mathrm{i} \eta$ with $ E \in (-\infty,\lambda_{-}(R(V))]\cup [\lambda_+(R(V)),\infty)$ and $\eta_* \le \eta \le 10$,
	\begin{equation*}
		\frac{1}{C_V} \frac{\eta}{\sqrt{\kappa_0+\eta}} \le \operatorname{Im} m_V(E+\mathrm{i} \eta) \le C_V \frac{\eta}{\sqrt{\kappa_0+\eta}} .
	\end{equation*}
	(ii)  $c_V \le \lambda_{-}(R(V))\le \lambda_{+}(R(V)) \le C_V$.\\
	(iii)  $\|R(V)\| \le n^{C_V}$.
\end{definition}

Recalling the definition of control parameters, the following proposition states the $\eta_{*}$-regularity of $R(\widetilde{H})$, whose proof is deferred to Appendix \ref{prf_prop_etaregular_corH}.
\begin{proposition}\label{prop_etaregular_corH}
	Suppose that $\Psi$ and $\Pi$ are well configured. Let  $\eta_{*}=n^{-\epsilon_{\alpha}/6}$. Then $R(\widetilde{H})$ is $\eta_{*}$-regular in the sense of Definition \ref{def_etaregular} with respect to $z\in \mathbf{D}$ with high probability.
\end{proposition}

The following lemma is a direct consequence of $\eta_{*}$-regularity of $R(\widetilde{H})$ (cf. Lemma 3.6 of \cite{ding2022edge}).
\begin{lemma}\label{lem_etaregularityconsequence}Suppose $R(\widetilde{H})$ is $\eta_*$-regular in the sense of Definition \ref{def_etaregular}. Let $\mu^{\widetilde{H}}_{n}$ be the measure associated with $mR_{\widetilde{H}}(z)$. For any fixed integer $a \ge 2$, and any $z=E+\mathrm{i}\eta\in \mathbf{D}\cup\mathbf{D}_3$, we have
	\begin{equation*}
		\int \frac{\mathrm{d}\mu_{n}^{\widetilde{H}}(x)}{|x-E-\mathrm{i}\eta|^a}\sim \frac{\sqrt{\kappa+\eta}}{\eta^{a-1}}\mathbbm{1}(z\in \mathbf{D}_1)+\frac{1}{(\kappa+\eta)^{a-3/2}}\mathbbm{1}(z\in \mathbf{D}_2\cup\mathbf{D}_3),
	\end{equation*}
	where $\kappa=\min\{|E-\lambda_+^{\widetilde{H}}|,|E-\lambda_-^{\widetilde{H}}|\}$ with $\lambda_+^{\widetilde{H}}=\lambda_1(R(\widetilde{H}))$ and $\lambda_-^{\widetilde{H}}=\lambda_n(R(\widetilde{H}))$.
\end{lemma}

\subsubsection{Local laws for GDM}\label{sec_locallaw_GDM}
With the $\eta_*$-regularity of $R(\widetilde{H})$ with high probability, we now turn to our GDM $R(Y_t)$. Recalling the definition of $\Psi$ in Definition \ref{def_psi&pi} and the notation above, we denote the index sets by
\begin{equation}\label{eq_indexPsi}
	\mathrm{I}_r:=\mathrm{I}_r(\Psi):=\{i\in[n]:\sum_{j=1}^p\psi_{ij}\ge1\}, \quad \mathrm{I}_c:=\mathrm{I}_c(\Psi):=\{j\in[p]:\sum_{i=1}^n\psi_{ij}\ge1\},
\end{equation}
which can be interpreted as the rows or columns containing at least one $\psi_{ij}=1$. In addition, we set $\mathrm{T}_r:=[n]\setminus\mathrm{I}_r$ and $\mathrm{T}_c:=[p]\setminus\mathrm{I}_c$. For notational simplicity, for any $z=E+\mathrm{i}\eta\in \mathbf{D}$, we recycle the notion $\kappa=\min\{|E-\lambda_-|,|E-\lambda_+|\}$ where $\lambda_-$ and $\lambda_+$ are respectively defined as the edges of MP law in Section \ref{sec_firstorderconvergence}. In the sequel, we define the control parameter
\begin{equation*}
\Psi_0(z)=\sqrt{\frac{\mathrm{Im} m_t(z)}{n\eta}}+\frac{1}{n\eta}.
\end{equation*}
With the notation above, we now present the averaged and entrywise local laws for our GDM, which will serve as the starting point for Green function comparison.
\begin{theorem}[Averaged local law for GDM]\label{thm_locallaw_gdm_average}
	Suppose that $\Psi$ and $\Pi$ are well configured. We have for $z\in\mathbf{D}_1$,
	\begin{equation}\label{eq_locallaw_gdm_average_in}
		|m_{n,t}(z)-m_t(z)|\prec\frac{1}{n\eta};
	\end{equation}
	and for $z\in\mathbf{D}_2$,
	\begin{equation}\label{eq_locallaw_gdm_average_out}
		|m_{n,t}(z)-m_t(z)|\prec\frac{1}{n(\kappa+\eta)}+\frac{1}{(n\eta)^2\sqrt{\kappa+\eta}}.
	\end{equation}
\end{theorem}
\begin{proof}
	The proof is deferred to Appendix \ref{app_sec_prf_locallaw_gdm_average}.
\end{proof}

\begin{theorem}[Entrywise local law for GDM]\label{thm_locallaw_gdm_entrywise}
    Suppose that $\Psi$ and $\Pi$ are well configured. For any fixed $z\in {\mathbf{D}}$, the following estimates hold under the probability measure $\mathbb{P}_{\Psi}$.
    \begin{equation*}
        \begin{split}
        &|\mathcal{G}R_{Y_t,ij}(z)-\delta_{ij}b_t(z)m^{(t)}(\zeta_t)|\prec \left(t^{-3}\Psi_0(z)+\frac{n^{-\epsilon_{\beta}}}{\kappa^2+\eta^2}\right)\mathbbm{1}_{i\in \mathrm{T}_r~\text{or}~j\in \mathrm{T}_r}+\frac{1-\mathbbm{1}_{i\in \mathrm{T}_r~\text{or}~j\in \mathrm{T}_r}}{\kappa+\eta},  \\
            &|\mathcal{G}R_{Y_t^*,uv}(z)-\delta_{uv}(1+t\underline{m}_t(z))\underline{m}^{(t)}(\zeta_t)|\prec  \left(t^{-3}\Psi_0(z)+\frac{n^{-\epsilon_{\beta}}}{\kappa^2+\eta^2}\right)\mathbbm{1}_{u\in \mathrm{T}_c~\text{or}~v\in \mathrm{T}_c}+\frac{1-\mathbbm{1}_{u\in \mathrm{T}_c~\text{or}~v\in \mathrm{T}_c}}{\kappa+\eta},\\
          &       |[\mathcal{G}R_{Y_t}(z)Y_t]_{iu}|\prec \left(t^{-3}\Psi_0(z)+\frac{n^{-\epsilon_{h}}}{\kappa+\eta}\right)\mathbbm{1}_{i\in \mathrm{T}_r~\text{or}~u\in \mathrm{T}_c}+\frac{1-\mathbbm{1}_{i\in \mathrm{T}_r~\text{or}~u\in \mathrm{T}_c}}{\kappa+\eta}.
        \end{split}
    \end{equation*}
\end{theorem}
\begin{proof}
	The proof is postponed to Appendix \ref{app_sec_prf_locallaw_gdm_entrywise}.
\end{proof}

Now we extend the local law from $\mathbf{D}_2$ to $\mathbf{D}_3$ which is a typical routine in the literature.
\begin{corollary}
Suppose that $\Psi$ and $\Pi$ are well configured. For any fixed $z\in {\mathbf{D}}\cup\mathbf{D}_3$, the following estimates hold under the probability measure $\mathbb{P}_{\Psi}$.
 \begin{equation*}
        \begin{split}
        &|\mathcal{G}R_{Y_t,ij}(z)-\delta_{ij}b_t(z)m^{(t)}(\zeta_t)|\prec n^{-\epsilon_{\beta}}\mathbbm{1}_{i\in \mathrm{T}_r~\text{or}~j\in \mathrm{T}_r}+1(1-\mathbbm{1}_{i\in \mathrm{T}_r~\text{or}~j\in \mathrm{T}_r}),  \\
            &|\mathcal{G}R_{Y_t^*,uv}(z)-\delta_{uv}(1+t\underline{m}_t(z))\underline{m}^{(t)}(\zeta_t)|\prec  n^{-\epsilon_{\beta}}\mathbbm{1}_{u\in \mathrm{T}_c~\text{or}~v\in \mathrm{T}_c}+1(1-\mathbbm{1}_{u\in \mathrm{T}_c~\text{or}~v\in \mathrm{T}_c}),\\
          &       |[\mathcal{G}R_{Y_t}(z)Y_t]_{iu}|\prec n^{-\epsilon_{h}}\mathbbm{1}_{i\in \mathrm{T}_r~\text{or}~u\in \mathrm{T}_c}+1(1-\mathbbm{1}_{i\in \mathrm{T}_r~\text{or}~u\in \mathrm{T}_c}).
        \end{split}
    \end{equation*}
\end{corollary}
\begin{proof}
We shall remark that $b_t=\mathrm{O}(1)$ and $m_t(z)\lesssim 1$ for $z\in \mathbf{D}$. For the case of $\eta\ll 1$ outside the spectrum, (say, $z\in \mathbf{D}_2\cup\mathbf{D}_3$ with $\eta\ll 1$), which implies $\kappa\sim 1$, choosing $\eta_0:=n^{-1/2}\kappa^{1/4}$, by the definition of $\mathrm{D}_{\zeta}$ (cf. \eqref{eq_def_zetadomain}) and $t\lesssim n^{-2\epsilon_l}$, we have
\begin{equation*}
	|\mathcal{G}R_{Y_t,ij}(z)-\delta_{ij}(1+\phi t m_t(z))m^{(t)}(\zeta_t)|\prec n^{-\epsilon_{\beta}}\mathbbm{1}_{i\in \mathrm{T}_r~\text{or}~j\in \mathrm{T}_r}+1(1-\mathbbm{1}_{i\in \mathrm{T}_r~\text{or}~j\in \mathrm{T}_r})
\end{equation*}
if $\eta\ge \eta_0$. Consider $z=E+\mathrm{i}\eta, z_0=E+\mathrm{i}\eta_0$, $0<\eta\le \eta_0$, satisfying
\begin{equation*}
	|m_t(z)-m_t(z_0)|\le C|m_t^{\prime}(z)|\eta_0=Cn^{-1/2}\kappa^{-1/4},
\end{equation*}
where we have used the fact $|m_t^{\prime}(z)|\le C\kappa^{-1/2}$ due to the square root decay behavior of the MP law near $\lambda_+$. It follows that
\begin{equation*}
	|(1+\phi t m_t(z))m^{(t)}(\zeta_t)-(1+\phi t m_t(z_0))m^{(t)}(\zeta_t(z_0))|\prec n^{-1/2}\kappa^{-1/4},
\end{equation*}
due to $\zeta_{t}^{\prime}(z)\sim 1$.
Moreover, noting $\eta\le \eta_0\le \kappa\le |E-\lambda_j(\widetilde{H}_t)|$ for $1\le j\le n$, one can check that
\begin{equation*}
	\begin{split}
		\mathrm{Im}(\mathcal{G}R_{Y_t,ii}(z))&=\sum_{j=1}^{n}\frac{\eta}{(E-\lambda_j(Y_t))^2+\eta^2}|\mathbf{u}_j(k)|^2
		\le 2\sum_{j=1}^{n}\frac{\eta_0}{(E-\lambda_j({Y}_t))^2+\eta_0^2}|\mathbf{u}_j(k)|^2\\
		&\le 2\mathrm{Im}(\mathcal{G}R_{Y_t,ii}(z_0))\prec n^{-1/2}\kappa^{-1/4}
	\end{split}
\end{equation*}
by $\mathrm{Im}((1+\phi t m_t(z_0))m^{(t)}(\zeta_t(z_0)))\sim \mathrm{Im}(m(z_0))\sim \eta_0/\sqrt{\kappa+\eta_0}\sim n^{-1/2}\kappa^{-1/4}$. For the real part, one can verify that
\begin{equation*}
	\begin{split}
		&|\mathrm{Re}(\mathcal{G}R_{Y_t,ii}(z))-\mathrm{Re}(\mathcal{G}R_{Y_t,ii}(z_0))|\le \sum_{j=1}\frac{C(E-\lambda_j(Y_t))(\eta_0^2-\eta^2)}{[(E-\lambda_j(Y_t))^2+\eta^2][(E-\lambda_j(Y_t))^2+\eta_0^2]}|\mathbf{u}_j(k)|^2\\
		&\le \frac{C\eta_0}{\kappa}\sum_{j=1}^{n}\frac{\eta_0}{(E-\lambda_j(Y_t))^2+\eta_0^2}|\mathbf{u}_j(k)|^2
		\le C\frac{\eta_0}{\kappa}\mathrm{Im}(\mathcal{G}R_{Y_t,ii}(z_0))
		\prec n^{-1/2}\kappa^{-1/4},
	\end{split}
\end{equation*}
where $\mathbf{u}_j=(\mathbf{u}_j(k))_{k=1}^{n}, 1\le j\le n$ are the eigenvectors of $Y_tY^*_t$. Thereafter, by polarization and linearity, we conclude that
\begin{equation*}
	|\mathcal{G}R_{Y_t,ij}(z)-\delta_{ij}(1+\phi t m_t(z))m^{(t)}(\zeta_t)|\prec n^{-\epsilon_{\beta}}\mathbbm{1}_{i\in \mathrm{T}_r~\text{or}~j\in \mathrm{T}_r}+1(1-\mathbbm{1}_{i\in \mathrm{T}_r~\text{or}~j\in \mathrm{T}_r})
\end{equation*}
for $0<\kappa\le C, 0<\eta\le C$, which further gives the result for $z\in \mathbf{D}_3$. The estimates of $\mathcal{G}R_{Y_t^*,uv}(z)$ and $[\mathcal{G}R_{Y_t}(z)Y_t]_{iu}$ are analogous, and thus omitted.
\end{proof}
	
\subsection{Green function comparison}\label{sec_greenfunctioncomparison}
In this subsection, we present the Green function comparison from the GDM, say $R(Y_t)=Y_tY_t^{*}$, to our sample correlation matrix $R(Y)=YY^{*}$ on the local scale. At the very beginning, we need some additional notation. Recall
\begin{equation*}
    Y=L(\operatorname{diag}S)^{-1/2}+\widetilde{H}(\operatorname{diag}S)^{-1/2},\quad Y_t=t^{1/2}W+\widetilde{H}(\operatorname{diag}S)^{-1/2},
\end{equation*}
and the parameter $t=n\mathbb{E}|L_{ij}\rho_j^{-1}|^2$. Define for each $\gamma\in[0,1]$ the $n\times n$ matrices
\begin{equation}\label{eq_def_interpolations}
    Y^{\gamma}=\gamma L(\operatorname{diag}S)^{-1/2}+t^{1/2}(1-\gamma^2)^{1/2}W+\widetilde{H}(\operatorname{diag} S)^{-1/2},\quad R^{\gamma}=Y^{\gamma}(Y^{\gamma})^{*}.
\end{equation}
Note that when $\gamma=0$, we have $Y^0=Y_t$. Meanwhile, when $\gamma=1$, $Y^{\gamma}$ recovers $Y$. We define the following Green functions,
\begin{equation}\label{eq_def_greenfunctioncomparison_m}
    \mathcal{G}^{\gamma}(z)=(R^{\gamma}-zI)^{-1},\quad G^{\gamma}=((Y^{\gamma})^{*}Y^{\gamma}-zI)^{-1},\quad m^{\gamma}(z):=m^{\gamma}_n(z)=\frac{1}{n}\operatorname{tr}\mathcal{G}^{\gamma}(z).
\end{equation}
Our comparison results will approximate the entries of $\mathcal{G}^{\gamma}$ by those of $\mathcal{G}^0$ for each $\gamma\in[0,1]$, conditional on $\Psi$ and $\Pi$. So it will be useful to consider the laws of $Y^0$ and $\widetilde{H}(\operatorname{diag}S)^{-1/2}$ conditional on their label $\Psi$ (or $\Pi$ when necessary). Let $\mathbb{P}_{\Psi}$ and $\mathbb{E}_{\Psi}$ denote the respective probability measure and expectation with respect to the joint law of the random variables $\{l_{ij},m_{ij},h_{ij},\psi_{ij},\chi_{ij}\}$ from Section \ref{sec_intro}, conditional on the event $\{\psi_{ij}\}$ given in $\Psi$. In addition, for any $\chi\in\{0,1\}$, we denote $\mathbb{P}_{\Psi}[\cdot|\chi_{ij}]=\mathbb{P}_{\Psi}[\cdot|\chi_{ij}=\chi]$ as the probability measure $\mathbb{P}_{\Psi}$ after conditioning on the event $\chi_{ij}=\chi$. Also, let $\mathbb{E}_{\Psi}[\cdot|\chi_{ij}]=\mathbb{E}[\cdot|\chi_{ij}=\chi]$ denote the associated expectation. Notice that $\mathbb{E}^{\chi}\big(\mathbb{E}_{\Psi}[\cdot|\chi_{ij}]\big)=\mathbb{E}_{\Psi}(\cdot)$, where $\mathbb{E}^{\chi}$ denotes the expectation with respect to the Bernoulli $0-1$ random variable $\chi$. We consider the comparison of entrywise local law and averaged local law respectively.
	
	\subsubsection{Entrywise bounds}
	We first introduce the following shorthand notation: for any $a,b\in[n]$ and $u,v\in[p]$,
	\begin{gather*}
		\mathfrak{R}_{ab}\equiv\mathfrak{R}_{ab}(\Psi):=
		\begin{cases}
			1 \quad \text{if $a$ or $b\in\mathrm{T}_r$,}\\
			t^2 \quad \text{if $a\in\mathrm{I}_r, b\in\mathrm{I}_r$}
		\end{cases},\quad
		\mathfrak{C}_{uv}\equiv\mathfrak{C}_{uv}(\Psi):=
		\begin{cases}
			1 \quad \text{if $u$ or $v\in\mathrm{T}_c$,}\\
			t^2 \quad \text{if $u\in\mathrm{I}_c, v\in\mathrm{I}_c$}
		\end{cases},\\
		\mathfrak{M}_{au}\equiv\mathfrak{M}_{au}(\Psi):=
		\begin{cases}
			1 \quad \text{if $a\in\mathrm{T}_r$ or $u\in\mathrm{T}_c$,}\\
			t^2 \quad \text{if $a\in\mathrm{I}_r. u\in\mathrm{I}_c$}
		\end{cases}.
	\end{gather*}
	The following theorem provides a way to compare conditional expectations of $\mathcal{G}^0$ to those of $\mathcal{G}^{\gamma}$ for any $\gamma\in[0,1]$, which is a key input to derive the local laws.
	\begin{theorem}\label{thm_greenfuncomp_entrywise_differror}
		Let $F:\mathbb{R}\rightarrow\mathbb{R}$ be a function such that
		\begin{equation*}
			\sup_{0\le \mu\le l}F^{(\mu)}(x)\le |x|^{C_0},\quad \sup_{0\le \mu\le l,\; |x|\le 2n^2}F^{(\mu)}(x)\le n^{C_0},
		\end{equation*}
		for real number $l, C_0>0$. For any $0-1$ matrix $\Psi$ and complex number $z$, define for any $a,b\in[n]$ and $u,v\in [p]$,
		\begin{gather*}
			\mathfrak{J}_{0,ab}\equiv\mathfrak{J}_{0,ab}(\Psi,z):=\max_{0\le \mu\le l}\sup_{0\le \gamma\le 1}\mathbb{E}_{\Psi}\big(|F^{(\mu)}[\mathfrak{R}_{ab}(\operatorname{Im}[\mathcal{G}^{\gamma}(z)]_{ab}-\delta_{ab}b_t(z)m^{(t)}(\zeta))]|\big),\\
			\mathfrak{J}_{1,uv}\equiv\mathfrak{J}_{1,uv}(\Psi,z):=\max_{0\le \mu\le l}\sup_{0\le \gamma\le 1}\mathbb{E}_{\Psi}\big(|F^{(\mu)}[\mathfrak{C}_{uv}(\operatorname{Im}[G^{\gamma}(z)]_{uv}-\delta_{uv}(1+t\underline{m}_t(z))\underline{m}^{(t)}(\zeta))]|\big),\\
			\mathfrak{J}_{2,au}\equiv\mathfrak{J}_{2,au}(\Psi,z):=\max_{0\le \mu\le l}\sup_{0\le \gamma\le 1}\mathbb{E}_{\Psi}\big(|F^{(\mu)}(\mathfrak{M}_{au}\operatorname{Im}[\mathcal{G}^{\gamma}(z)Y^{\gamma}]_{au})|\big),
		\end{gather*}
		and $\Omega=\Omega_0\cap\Omega_1\cap\Omega_2\cap\Omega_w\cap\Omega_{\rho}$, $Q_0\equiv Q_0(\varepsilon,z):=1-\mathbb{P}_{\Psi}(\Omega)$ with
		\begin{gather*}
			\Omega_0\equiv\Omega_0(\varepsilon,z):=\{\sup_{a,b\in[n], 0\le\gamma\le 1}|\mathfrak{R}_{ab}([\mathcal{G}^{\gamma}(z)]_{ab}-\delta_{ab}b_t(z)m^{(t)}(\zeta))|\le n^{\varepsilon}\}, \\
			\Omega_1\equiv\Omega_1(\varepsilon,z):=\{\sup_{u,v\in[p], 0\le\gamma\le 1}|\mathfrak{C}_{uv}([G^{\gamma}(z)]_{uv}-\delta_{uv}(1+t\underline{m}_t(z)\underline{m}^{(t)}(\zeta)))|\le n^{\varepsilon}\},\\
			\Omega_2\equiv\Omega_2(\varepsilon,z):=\{\sup_{a\in[n],u\in[p], 0\le\gamma\le 1}|\mathfrak{M}_{au}[\mathcal{G}^{\gamma}(z)Y^{\gamma}]_{au}|\le n^{\varepsilon}\},\\ \Omega_w\equiv\Omega_w(\varepsilon):=\{\sup_{i\in[n],j\in[p]}|w_{ij}|\le n^{-1/2+\varepsilon}t\}, \
			\Omega_{\rho}\equiv\Omega_{\rho}(c):=\{\inf_{j\in[p]}n^{-1}\rho_j^2>\log^{-c}n\}.
		\end{gather*}
		Suppose $\Psi$ and $\Pi$ are well configured. Then there exist sufficiently small positive constants $\varepsilon\le\epsilon_l/100$ and $c_0$, and a large constant $C>0$ such that for $a,b\in[n]$, $u,v\in[p]$ and $(\#_1,\#_2,\#_3)\in$
		\begin{equation*}
			\begin{split}
				 &\Big\{\big(\mathfrak{R}_{ab}(\operatorname{Im}[\mathcal{G}^{\gamma}(z)]_{ab}-\delta_{ab}b_t(z)m^{(t)}(\zeta)),\;\mathfrak{R}_{ab}(\operatorname{Im}[\mathcal{G}^{0}(z)]_{ab}-\delta_{ab}b_t(z)m^{(t)}(\zeta)),\;\mathfrak{J}_{0,ab}\big),\\
				 &\big(\mathfrak{C}_{uv}(\operatorname{Im}[G^{\gamma}(z)]_{uv}-\delta_{uv}(1+t\underline{m}_t(z)\underline{m}^{(t)}(\zeta))),\;\mathfrak{C}_{ab}(\operatorname{Im}[G^{0}(z)]_{uv}-\delta_{uv}(1+t\underline{m}_t(z)\underline{m}^{(t)}(\zeta))),\;\mathfrak{J}_{1,uv}\big),\\
				 &\big(\mathfrak{M}_{au}\operatorname{Im}[\mathcal{G}^{\gamma}(z)Y^{\gamma}]_{au},\;\mathfrak{M}_{au}\operatorname{Im}[\mathcal{G}^{0}(z)Y^0]_{au},\;\mathfrak{J}_{2,au}\big)
				\Big\},
			\end{split}
		\end{equation*}
		we have
		\begin{equation*}
			\sup_{0\le\gamma\le1}|\mathbb{E}_{\Psi}\big(F(\#_1)\big)-\mathbb{E}_{\Psi}\big(F(\#_2)\big)|<Cn^{-c_0}(\#_3+1)+CQ_0n^{C+C_0}.
		\end{equation*}
		The same estimates hold if $\operatorname{Im}$'s are replaced by $\operatorname{Re}$'s.
	\end{theorem}
	\begin{proof}
		The proof of this theorem will be postponed to Appendix \ref{app_prf_thm_greenfuncomp_entrywise_differror}.
	\end{proof}
 With the comparison theorem above and the input from the local laws of GDM provided in Theorem \ref{thm_locallaw_gdm_entrywise}, we have the following entrywise bound.
	\begin{theorem}[Entrywise bound]\label{thm_greenfuncomp_entrywise_uniformbound}
		Recall the definitions of contours $\gamma_1$ and $\gamma_2$. Suppose $\Psi$ and $\Pi$ are well configured. Let $\mathbb{P}_{\Psi}$ be the probability conditioned on the event that $(\psi_{ij})$ is given in $\Psi$. Then we have that for some small constant $c_{l}>0$ and large constant $D>0$,
		\begin{itemize}
			\item [(1)] in the case where $a,b\in\mathrm{T}_r$, $u,v\in\mathrm{T}_c$,
			\begin{equation*}
				\begin{split}
					\mathbb{P}_{\Psi}\big(\sup_{0\le\gamma\le1}&\sup_{z\in\gamma_1\vee\gamma_2}\sup_{a,b\in [n]\atop a,b\in\mathrm{T}_r}|\mathfrak{R}_{ab}([\mathcal{G}^{\gamma}(z)]_{ab}-\delta_{ab}b_t(z)m^{(t)}(\zeta))|\\
					&\vee\sup_{0\le\gamma\le1}\sup_{z\in\gamma_1\vee\gamma_2}\sup_{u,v\in [p]\atop u,v\in\mathrm{T}_c}|\mathfrak{C}_{uv}([G^{\gamma}(z)]_{uv}-\delta_{uv}(1+t\underline{m}_t(z))\underline{m}^{(t)}(\zeta))|\\
					&\vee\sup_{0\le\gamma\le1}\sup_{z\in\gamma_1\vee\gamma_2}\sup_{a\in[n],u\in [p]\atop a\in\mathrm{T}_r,u\in\mathrm{T}_c}|\mathfrak{M}_{au}[\mathcal{G}^{\gamma}(z)Y^{\gamma}]_{au}|\ge n^{-c_l\epsilon_{\alpha}}\big)\le Cn^{-D}.
				\end{split}
			\end{equation*}
			\item [(2)] in the case where $a\in\mathrm{I}_r\vee b\in\mathrm{I}_r$ and $u\in\mathrm{I}_c\vee v\in\mathrm{I}_c$,
			\begin{equation*}
				\begin{split}
					\mathbb{P}_{\Psi}\big(\sup_{0\le\gamma\le1}\sup_{z\in\gamma_1\vee\gamma_2}\sup_{a,b\in [n]\atop a\in\mathrm{I}_r\vee b\in\mathrm{I}_r}&|\mathfrak{R}_{ab}[\mathcal{G}^{\gamma}(z)]_{ab}|\vee\sup_{0\le\gamma\le1}\sup_{z\in\gamma_1\vee\gamma_2}\sup_{u,v\in [p]\atop u\in\mathrm{I}_c\vee v\in\mathrm{I}_c}|\mathfrak{C}_{uv}[G^{\gamma}(z)]_{uv}|\\
					&\vee\sup_{0\le\gamma\le1}\sup_{z\in\gamma_1\vee\gamma_2}\sup_{a\in[n],u\in [p]\atop a\in\mathrm{I}_r\vee u\in\mathrm{I}_c}|\mathfrak{M}_{au}[\mathcal{G}^{\gamma}(z)Y^{\gamma}]_{au}|\ge n^{-c_l\epsilon_{l}}\big)\le Cn^{-D}.
				\end{split}
			\end{equation*}
		\end{itemize}
		The above estimates also hold for $z\in\gamma_1^0\vee\gamma_2^0$.
	\end{theorem}
	\begin{proof}
		The proof of this theorem will be postponed to Appendix \ref{app_prf_thm_greenfuncomp_entrywise_uniformbound}.
	\end{proof}
	
	\subsubsection{Averaged bounds}
	Recall the definitions of $\mathcal{G}^{\gamma},G^{\gamma}$ and $m^{\gamma}(z)$ in \eqref{eq_def_greenfunctioncomparison_m}. We have the following theorem, whose proof is deferred to Appendix \ref{app_prf_thm_greenfuncomp_average}.
	\begin{theorem}[Averaged bound]\label{thm_greenfuncomp_average}
		Suppose that $\Psi$ and $\Pi$ are well configured. Let $z\in\gamma_1\vee\gamma_2$. Then there exists some constant $C_0>0$ such that for $q\ge 3$,
		\begin{itemize}
			\item [(i)] If $\alpha\ge 3$, then we have
			\begin{equation*}
				 \sup_{0\le\gamma\le1}\mathbb{E}_{\Psi}\big(|n\eta(\operatorname{Im}m^{\gamma}(z)-\operatorname{Im}\widetilde{m}_{n,0}(z))|^{2q}\big)\le(1+\mathrm{o}(1))\mathbb{E}_{\Psi}\big(|n\eta(\operatorname{Im}m_{n,0}(z)-\operatorname{Im}\widetilde{m}_{n,0}(z))|^{2q}\big)+n^{-C_0q};
			\end{equation*}
			\item [(ii)] If $\alpha\in(2,3)$, then we have
			\begin{equation*}
				\begin{split}
					 &\sup_{0\le\gamma\le1}\mathbb{E}_{\Psi}\big(|n^{3/2-\alpha/2}\eta(\operatorname{Im}m^{\gamma}(z)-\operatorname{Im}\widetilde{m}_{n,0}(z))|^{2q}\big)\\
					 &\le(1+\mathrm{o}(1))\mathbb{E}_{\Psi}\big(|n^{3/2-\alpha/2}\eta(\operatorname{Im}m_{n,0}(z)-\operatorname{Im}\widetilde{m}_{n,0}(z))|^{2q}\big)+n^{-C_0q},
				\end{split}
			\end{equation*}
		\end{itemize}
		where $\widetilde{m}_{n,t}(z)=n^{-1}\operatorname{tr}(\widetilde{Y}_t(\widetilde{Y}_t)^{*}-zI)^{-1}$ with $\widetilde{Y}_t=\sqrt{t}\widetilde{W}+\widetilde{H}(\operatorname{diag}S)^{-1/2}$. Here, $\widetilde{W}$ is an i.i.d. copy of $W$. Furthermore, the above estimate also holds if $\operatorname{Im}$'s are replaced by $\operatorname{Re}$'s and $z\in\gamma_1^0\vee\gamma_2^0$.
	\end{theorem}
	
	\begin{corollary}\label{col_greenfuncomp_average}
		Suppose that $\Psi$ and $\Pi$ are well configured. Let $z\in\gamma_1\vee\gamma_2$. Then
		\begin{itemize}
			\item [(i)] if $\alpha\ge 3$,
			\begin{equation}\label{eq_averagelocallaw_optimal}
				|m^1(z)-m_t(z)|\prec \frac{1}{n(|E-\lambda_{+}|+\eta)};
			\end{equation}
			\item [(ii)] if $\alpha\in(2,3)$,
			\begin{equation}\label{eq_averagelocallaw_weak}
				|m^1(z)-m_t(z)|\prec \frac{1}{n^{3/2-\alpha/2}(|E-\lambda_{+}|+\eta)}.
			\end{equation}
		\end{itemize}
	\end{corollary}
	\begin{proof}
		By Markov's inequality, Theorems  \ref{thm_locallaw_gdm_average} and \ref{thm_greenfuncomp_average}, one can easily obtain the desired result. We omit further details here.
	\end{proof}
	\begin{remark}
	In the proof of Theorem \ref{thm_greenfuncomp_average}, we find that when $\alpha\ge 3$, we have the optimal convergence rate for $m^1(z)$ in the sense of \eqref{eq_averagelocallaw_optimal}. When $\alpha\in(2,3)$, the current technique only guarantees a weaker convergence rate as in \eqref{eq_averagelocallaw_weak}.
	\end{remark}
	
	\subsection{Characteristic function estimation}\label{sec_characteristicfunctionestimation}
 In this section, we will derive the CLT of the linear spectral statistics for general $f(x)$ when $\alpha\ge 3$. Specifically, we show that the universal CLT of the LSS (the same asymptotic normality as the finite fourth-moment case, cf. \cite{gao2014high}) holds if and only if $\lim_{x\rightarrow\infty} x^3\mathbb{P}(|\xi|>x)=0$, which suggests that the symmetry condition is not necessary under such conditions. This result also illustrates the robustness of sample correlation matrices, while we need a normalized factor to establish the CLT for the LSS under the infinite fourth-moment for sample covariance matrices \cite{bao2023smallest,benaych2016fluctuations}.
	
	\subsubsection{The universal CLT for $\alpha\in (3,4]$}\label{sec_universalclt}
	To establish the CLT for the linear spectral statistics when $\alpha\in (3,4]$, we turn to estimate the characteristic functions of the centered statistics $\operatorname{tr} f(R)-\mathbb{E} \operatorname{tr} f(R)$. By Lemma \ref{lem_pre_estnorm}, the nonzero eigenvalues of $R$ are bounded below and above by positive constants with high probability. By choosing the contours $\gamma_1^0\cup \gamma_2^0$ properly to enclose all eigenvalues with high probability, recalling the definition of $\Xi$ in Section \ref{sec_pre}, $\bar{\gamma}_{1}^0$ and $\bar{\gamma}_2^0$ (say, $|\mathrm{Im} z|\ge n^{-K}$ for some large but fixed $K$), we apply the Cauchy formula to get
\begin{equation*}
\operatorname{tr}f(R)=\frac{-1}{2\pi \mathrm{i}}\oint_{\gamma_a^0}\operatorname{tr}\mathcal{G}(z)f(z)\mathrm{d}z=\frac{-1}{2\pi \mathrm{i}}\oint_{\bar{\gamma}_a^0}\operatorname{tr}\mathcal{G}(z)f(z)\mathrm{d}z \cdot \Xi+\mathrm{O}_{\prec}(n^{-K+1}),~a=1,2.
\end{equation*}
In the sequel, we denote
\begin{equation*}
    L_{a}(f):=\frac{-1}{2\pi \mathrm{i}}\oint_{\bar{\gamma}_a^0}\operatorname{tr}\mathcal{G}(z)f(z)\mathrm{d}z \cdot \Xi.
\end{equation*}
Thanks to the Green function comparison results in Theorems \ref{thm_greenfuncomp_entrywise_uniformbound} and  \ref{thm_greenfuncomp_average}, the Green function $\mathcal{G}(z)$ in the integrand of $L_{a}(f)$ and the derivatives of various functionals of $\mathcal{G}(z)$ always possess deterministic crude bound due to the definition of $\Xi$, which transforms high probability estimates into the expectation estimates by Lemma \ref{lem_pre_estnorm}. Define the characteristic function of $L_{1}(f)$ as
	\begin{equation*}
		\mathcal{T}_{f}(x):=\mathbb{E}[\mathrm{e}^{\mathrm{i}x\langle L_{1}(f)\rangle}],
	\end{equation*}
where $\langle \xi\rangle:=\xi -\mathbb{E}(\xi)$ denotes the centralized random variable, which has the derivative with respect to $x$ as
\begin{equation}\label{eq_characteristicfunctioncomparison_ode}
    \left(\mathcal{T}_{f}(x)\right)^{\prime}=\mathrm{i}\mathbb{E}[\langle L_{1}(f)\rangle\mathrm{e}^{\mathrm{i}x\langle L_{1}(f)\rangle}]=\frac{-1}{2\pi}\oint_{\bar{\gamma}_1^0}\mathbb{E}^{\chi}[\langle\operatorname{tr}\mathcal{G}(z_1)\rangle \mathrm{e}^{\mathrm{i}x\langle L_{1}(f)\rangle}]f(z_1)\mathrm{d}z_1,
\end{equation}
where $\mathbb{E}^{\chi}$ is the ``truncated expectation'' in Definition \ref{def_trucatuedexpectation}.
Thus, the estimate for $ \left(\mathcal{T}_{f}(x)\right)^{\prime}$ boils down to the estimate of $\mathbb{E}^{\chi}[\langle\operatorname{tr}\mathcal{G}(z_1)\rangle \mathrm{e}^{\mathrm{i}x\langle L_{1}(f)\rangle}]$ for $z_1\in \bar{\gamma}_1^0$. Notice the identity
\begin{equation*}
    \operatorname{tr}R\mathcal{G}(z_1)=n+z_1\operatorname{tr}\mathcal{G}(z_1).
\end{equation*}
Then one may easily see that
\begin{equation*}
    \langle\operatorname{tr}R\mathcal{G}(z_1)\rangle=z_1 \langle\operatorname{tr}\mathcal{G}(z_1)\rangle.
\end{equation*}
Therefore, it is sufficient to focus on the estimate of $\mathbb{E}^{\chi}[\langle\operatorname{tr}R\mathcal{G}(z_1)\rangle \mathrm{e}^{\mathrm{i}x\langle L_{1}(f)\rangle}]$ for $z_1\in \bar{\gamma}_1^0$, which is very suitable to start a cumulant expansion as pointed out by \cite{bao2022spectral}, whose various terms can be estimated by $m(z)$. In the following, we aim to establish an approximate ODE:
	\begin{align}\label{eq_cltode}
		\left(\mathcal{T}_{f}(x)\right)^{\prime}=-x \sigma_f^2 \mathcal{T}_{f}(x)+\text{error},
	\end{align}
	from which, we can directly obtain the asymptotic normality of $L_{1}(f)$, as well as $\operatorname{tr}f(R)$.
	
	Now, we turn to investigate $\mathbb{E}^{\chi}[\langle\operatorname{tr}R\mathcal{G}(z_1)\rangle \mathrm{e}^{\mathrm{i}x\langle L_{1}(f)\rangle}]$. Recalling the decomposition of $X$ in Section \ref{sec_matrixresampling} and the definitions of $ly_{ij}$ and $hy_{ij}$ in the proof of Theorem \ref{thm_greenfuncomp_entrywise_uniformbound} and Theorem \ref{thm_greenfuncomp_average} (cf. \eqref{eq_def_ly&hy} for $\gamma=1$), we slightly adjust their notation as follows to adapt for the current case,
	\begin{equation*}
		ly_{ij}=L_{ij}\rho_j^{-1}+M_{ij}\rho_j^{-1},\quad  hy_{ij}=H_{ij}\rho_j^{-1},\quad i\in[n],j\in[p].
	\end{equation*}
	It follows that the entries of the original matrix $Y$ can be decomposed as
	\begin{equation*}
		Y_{ij}=ly_{ij}\mathbbm{1}(\psi_{ij}=0)+hy_{ij}\mathbbm{1}(\psi_{ij}=1),
	\end{equation*}
	where $\Psi=(\psi_{ij})$ is the label matrix by the resampling of $X$. Besides, it is also necessary to classify the label of each column of $Y$. Recall the notation $\mathrm{I}_c$ which summarises the columns $Y_j$'s having at least one $i$ such that $\psi_{ij}=1$ and $\mathrm{T}_c=[p]\setminus\mathrm{I}_c$. Then, we would like to classify the column indexes of $Y$ into two sets $j\in\mathrm{I}_c$ and $j\in\mathrm{T}_c$ according to our label matrix $\Psi$. In the following, we shall consider the conditional expectation on $\Psi$, which can be finally dropped by the law of total expectation with the fact that $\Psi$ is well configured with high probability. We may first observe that we could rewrite $\mathbb{E}_{\Psi}^{\chi}[\langle\operatorname{tr}R\mathcal{G}(z_1)\rangle \mathrm{e}^{\mathrm{i}x\langle L_{1}(f)\rangle}]$ as
	\begin{equation}\label{eq_prf_characteristicfunctioncomparison_expansion_phi}
		\begin{split}
			&\mathbb{E}^{\chi}_{\Psi}[\langle\operatorname{tr}R\mathcal{G}(z_1)\rangle \mathrm{e}^{\mathrm{i}x\langle L_{1}(f)\rangle}]=\mathbb{E}_{\Psi}^{\chi}[\sum_jY_j^{*}\mathcal{G}(z_1)Y_j\langle\mathrm{e}^{\mathrm{i}x\langle L_{1}(f)\rangle}\rangle]\\
			=&\mathbb{E}^{\chi}_{\Psi}[\sum_{i}\sum_{j\in\mathrm{T}_c}Y_{ij}[{Y}^*\mathcal{G}(z_1)]_{ji}\langle \mathrm{e}^{\mathrm{i}x\langle L_{1}(f)\rangle}\rangle]+\mathbb{E}_{\Psi}^{\chi}[\sum_{j\in\mathrm{I}_c}Y_j^{*}\mathcal{G}(z_1)Y_j\langle\mathrm{e}^{\mathrm{i}x\langle L_{1}(f)\rangle}\rangle]\\
			=&\sum_{i}\sum_{j\in\mathrm{T}_c}\mathbb{E}^{\chi}_{\Psi}[ly_{ij}[Y^*\mathcal{G}(z_1)]_{ji}\langle \mathrm{e}^{\mathrm{i}x\langle L_{1}(f)\rangle}\rangle]\cdot \mathbbm{1}({\psi_{ij}=0})+\mathbb{E}_{\Psi}^{\chi}[\sum_{j\in\mathrm{I}_c}Y_j^{*}\mathcal{G}(z_1)Y_j\langle\mathrm{e}^{\mathrm{i}x\langle L_{1}(f)\rangle}\rangle]\\
			=&: \sum_{i}\sum_{j\in\mathrm{T}_c}I_{ij}+II,
		\end{split}
	\end{equation}
	where we have used the identity $\mathbb{E}[\xi_1\langle\xi_2\rangle]=\mathbb{E}[\langle\xi_1\rangle\xi_2]$ in the first line. Roughly speaking, we expect that the first part $I$ recover the universal case and the second part $II$ is negligible, which is motivated by the simple fact that the label matrix $\Psi$ has at most constant non-zero entries with high probability for $\alpha>4$.
	
	Before we state the main result, we introduce the following cumulant expansion formula, which plays a central role in the estimate of expectation, see \cite{bao2022spectral} for details.
	\begin{lemma}[Cumulant expansion formula] \label{lemma_cumulantexpansion}
		For a fixed $\ell \in \mathbb{N}$, let $f \in$ $C^{\ell+1}(\mathbb{R})$. Supposed $\xi$ is a random variable with the $k$-th cumulant $\kappa_k(\xi)$. Then we have
		\begin{align*}
			\mathbb{E}(\xi f(\xi))=\sum_{k=0}^{\ell} \frac{\kappa_{k+1}(\xi)}{k !} \mathbb{E}\big(f^{(k)}(\xi)\big)+\mathbb{E}\left(r_{\ell}(\xi f(\xi))\right),
		\end{align*}
		where the error term $r_{\ell}(\xi f(\xi))$ satisfies
		\begin{align*}
			\left|\mathbb{E}\left(r_{\ell}(\xi f(\xi))\right)\right| \leq C_{\ell} \mathbb{E}\left(|\xi|^{\ell+2}\right) \sup _{|t| \leq s}\left|f^{\ell+1}(t)\right|  +C_{\ell} \mathbb{E}\left(|\xi|^{\ell+2} \mathbbm{1}(|\xi|>s)\right) \sup _{t \in \mathbb{R}}\left|f^{\ell+1}(t)\right|
		\end{align*}
		for any $s>0$ and $C_{\ell}$ such that $C_{\ell} \leq(C \ell)^{\ell} / \ell!$ for some constant $C>0$.
	\end{lemma}

It should be noted that we do not require that the cumulants remain bounded, since the $n$-dependent rates of cumulants can be canceled by rescaling.
With the cumulant expansion formula and the decomposition of $Y_{ij}$ as above, we could estimate the two parts in \eqref{eq_prf_characteristicfunctioncomparison_expansion_phi}. We remark here that most works focus on the cumulant expansion with respect to $X_{ji}$ under the covariance matrix case,  which becomes invalid when applied to $Y_{ij}$ due to the dependence among  $Y_{ij}$. Inspired by \cite{bao2022spectral} and the decomposition of $Y_{ij}$, we consider the cumulant expansion with respect to $L_{ij}+M_{ij}$ for the first part $I_{ij}$ of \eqref{eq_prf_characteristicfunctioncomparison_expansion_phi} up to any fixed order $l\in\mathbb{N}$, due to the truncation and condition on $\psi_{ij}=0$. For the second part $II_{ij}$, thanks to the well-configured label matrix $\Psi$, it can be well estimated after taking the sum over $j\in \mathrm{I}_c$ and utilizing the self-normalized properties. Moreover, we introduce the shorthand notation as follows,
\begin{equation*}
    \sum_{ij}=\sum_{i=1}^{n}\sum_{j=1}^{p},~\partial_{m,ij}=\frac{\partial}{\partial M_{ij}}.
\end{equation*}
	
	In the following, we aim to show the approximated ODE \eqref{eq_cltode}, which means that we can drop the negligible terms with the order of $\mathrm{o}_{\prec}(1)$. However, we do not directly obtain the form of \eqref{eq_cltode} from $\mathbb{E}_{\Psi}^{\chi}\big[\langle \operatorname{tr}R\mathcal{G}(z_1)\rangle e^{\mathrm{i}}x\langle L_1(f)\rangle\big]$ since several quantities in $II$ are very difficult to estimate because of the dependence structure in the Cauchy integral. To solve these difficulties, we use another routine that if one believes \eqref{eq_cltode} is right, then one may take a further derivative of  \eqref{eq_cltode} with respect to $x$ to obtain
	\begin{equation}\label{eq_cha_ode2}
		\begin{split}
			 (\mathcal{T}_{f}(x))^{\prime\prime}=-\sigma_f^2\mathcal{T}_{f}(x)-x\sigma_f^2(\mathcal{T}_{f}(x))^{\prime}+\mathrm{error}=-\sigma_f^2\mathcal{T}_{f}(x)+x^2\sigma_f^4\mathcal{T}_{f}(x)+\mathrm{error}.
		\end{split}
	\end{equation}
	Then, as long as \eqref{eq_cltode} is established,  one may take $x=0$ to recover the desired variance for the CLT. Therefore, we follow the structure of \eqref{eq_cltode} to figure out the main terms. For the error terms, we will take a further derivative on $x$ and set $x=0$ to get a controlled term.
	
	We consider the main part $I_{ij}$ first. By the cumulant expansion formula in Lemma \ref{lemma_cumulantexpansion}, we have
	for $i\in [n],j\in \mathrm{T}_c$
	\begin{equation*}
		\begin{split}
			I_{ij}=\mathbbm{1}({\psi_{ij}=0})\mathbb{E}_{\Psi}^{\chi}[(L_{ij}+M_{ij})\rho_j^{-1}[{Y}^*\mathcal{G}(z_1)]_{ji}\langle \mathrm{e}^{\mathrm{i}x\langle L_{1}(f)\rangle}\rangle].
		\end{split}
	\end{equation*}
	We shall apply the cumulant expansion with respect to $(L_{ij}+M_{ij})$. In the sequel, we reuse the symbol $M_{ij}$ to denote $(L_{ij}+M_{ij})$ without causing any confusion. For simplicity, let $\kappa_{q,m}$ be the $q$-th cumulant of $M_{ij}$. By the definition of $H_{ij}$, one has for $s<\alpha$,
	\begin{equation*}
		\begin{split}
			\mathbb{E}|H_{ij}|^{s}=&\mathbb{E}(\mathbb{E}_{\Psi}(|H_{ij}|))= \mathbb{E}(|h_{ij}|^s)\mathbb{P}(|X_{ji}|>n^{1/2-\epsilon_h})\\
   \lesssim &\int_{n^{1/2-\epsilon_h}}^{\infty} l(x)x^{-\alpha+s-1}\mathrm{d}x \lesssim n^{(1/2-\epsilon_h)(-\alpha+s)},
		\end{split}
	\end{equation*}
	which further implies
	\begin{equation*}
		\kappa_{1,m}=-\mathbb{E}(H_{ij})\lesssim n^{(1/2-\epsilon_h)(-\alpha+1)},
		\kappa_{2,m} =1-\mathbb{E}(H_{ij}^2)-\kappa_{1,m}^2=1-\mathrm{o}(n^{(1/2-\epsilon_h)(-\alpha+2)}).
	\end{equation*}
	Moreover, due to the truncation, one has $\kappa_{q,m}\lesssim n^{(q-\alpha)_{+}(1/2-\epsilon_h)}~\text{for}~q\ge 3$. For  $\alpha>3$, we have
	\begin{equation*}
		 \kappa_{3,m}=\mathbb{E}(M_{ij}^3)-3\mathbb{E}(M_{ij}^2)\mathbb{E}(M_{ij})+2[\mathbb{E}(M_{ij})]^3=\mathbb{E}(\xi^3)-\mathrm{o}(n^{(1/2-\epsilon_h)(-\alpha+3)}),
	\end{equation*}
	which depends on $\mathbb{E}(\xi^3)$ for $\alpha>3$.
	Applying the cumulant expansion for $M_{ij}$ gives
	\begin{equation*}
		\begin{split}
			I_{ij}=&\mathbb{E}_{\Psi}^{\chi}[M_{ij}\rho_j^{-1}[{Y}^*\mathcal{G}(z_1)]_{ji}\langle \mathrm{e}^{\mathrm{i}x\langle L_{1}(f)\rangle}\rangle]\\
			=&\kappa_{1,m}\mathbb{E}_{\Psi}^{\chi}[\rho_j^{-1}[{Y}^*\mathcal{G}(z_1)]_{ji}\langle \mathrm{e}^{\mathrm{i}x\langle L_{1}(f)\rangle}\rangle]+\kappa_{2,m}\mathbb{E}_{\Psi}^{\chi}\left[(\partial_{m,ij}\rho_{j}^{-1})[Y^*\mathcal{G}(z_1)]_{ji}\langle \mathrm{e}^{\mathrm{i}x\langle L_{1}(f)\rangle}\rangle\right]\\
			&+\kappa_{2,m}\mathbb{E}_{\Psi}^{\chi}\left[\rho_j^{-1}\left(\partial_{m,ij}[{Y}^*\mathcal{G}(z_1)]_{ji}\right)\langle \mathrm{e}^{\mathrm{i}x\langle L_{1}(f)\rangle}\rangle\right]+\kappa_{2,m}\mathbb{E}_{\Psi}^{\chi}\left[\rho_j^{-1}[Y^*\mathcal{G}(z_1)]_{ji}\left(\partial_{m,ij}\mathrm{e}^{\mathrm{i}x\langle L_{1}(f)\rangle}\right)\right]\\
			&+\frac{\kappa_{3,m}}{2!}\mathbb{E}_{\Psi}^{\chi}\left[\partial_{m,ij}^{2}\left(\rho_j^{-1}[Y^*\mathcal{G}(z_1)]_{ji}\langle \mathrm{e}^{\mathrm{i}x\langle L_{1}(f)\rangle}\rangle\right)\right]+\frac{\kappa_{4,m}}{3!}\mathbb{E}_{\Psi}^{\chi}\left[\partial_{m,ij}^{3}\left(\rho_j^{-1}[Y^*\mathcal{G}(z_1)]_{ji}\langle \mathrm{e}^{\mathrm{i}x\langle L_{1}(f)\rangle}\rangle\right)\right]\\
			&+\sum_{s\ge 5}^{q}\mathrm{O}(\kappa_{s,m})\mathbb{E}_{\Psi}^{\chi}\left[\partial_{m,ij}^{s-1}\left(\rho_j^{-1}[Y^*\mathcal{G}(z_1)]_{ji}\langle \mathrm{e}^{\mathrm{i}x\langle L_{1}(f)\rangle}\rangle\right)\right]\\
			&+\sum_{s\ge 2}^{q}\mathrm{O}(\kappa_{s,m})\sum_{\substack{s_0+s_1=s-1\\s_1\ge 1}}\mathbb{E}\left[\partial_{m,ij}^{s_0}\left(\rho_j^{-1}[Y^*\mathcal{G}(z_1)]_{ji}\langle \mathrm{e}^{\mathrm{i}x\langle L_{1}(f)\rangle}\rangle\right) \partial_{m,ij}^{s_1}\Xi\right]+\mathbb{E}(R_{m,ij}(q))\\
			 =&:\kappa_{1,m}\mathscr{M}_{0,ij}+\kappa_{2,m}(\mathscr{M}_{1,ij}+\mathscr{M}_{2,ij}+\mathscr{M}_{3,ij})+\kappa_{3,m}E_{1,m}+\kappa_{4,m}E_{2,m}+E_{3,m}+E_{4,m}+\mathbb{E}(R_{m,ij}(q)),
		\end{split}
	\end{equation*}
	where the error term $R_{m,ij}$ is given by
	\begin{equation*}
		\begin{split}
			R_{m,ij}(q)=&C\sup_{M_{ij}\in\mathbb{R}}\left|\partial_{m,ij}^{2q+1}\left(\rho_j^{-1}[Y^*\mathcal{G}(z_1)]_{ji}\langle \mathrm{e}^{\mathrm{i}x\langle L_{1}(f)\rangle}\rangle\Xi\right)\right|\cdot\mathbb{E}(|M_{ij}^{2q+2}|\mathbbm{1}(|M_{ij}|>n^{-\epsilon}))\\
			&+C\mathbb{E}|M_{ij}^{2q+2}|\cdot \sup_{|M_{ij}|\le n^{-\epsilon}}\left|\partial_{m,ij}^{2q+1}\left(\rho_j^{-1}[Y^*\mathcal{G}(z_1)]_{ji}\langle \mathrm{e}^{\mathrm{i}x\langle L_{1}(f)\rangle}\rangle\Xi\right)\right|
		\end{split}
	\end{equation*}
	for some constant $\epsilon>0$.

In the following, we will estimate the terms one by one. Intuitively, we will use the averaged local law in Corollary \ref{col_greenfuncomp_average} and entrywise local law in Theorem \ref{thm_greenfuncomp_entrywise_uniformbound} to substitute the random terms after taking summation over $i,j$ (say, $\operatorname{tr}G(z),\operatorname{tr}\mathcal{G}(z), [G(z)]_{ii}, [\mathcal{G}(z)]_{ii}$) in the expectation above, and control the error into a neglected level. In the sequel, we use the shorthand notation $m(z)$ and $\underline{m}(z)$ to denote the corresponding deterministic quantities of $m_t(z)$ (or $b_{t}(z)m^{(t)}(\zeta)$) and $\underline{m}_t(z)$ (or $(1+t\underline{m}_t(z))\underline{m}^{(t)}(\zeta)$), respectively. As a consequence, due to the fact that $t\sim n^{-2\epsilon_l}$, all the quantities involved can be replaced by $m(z)$ and $\underline{m}(z)$ in the final results. The following lemmas give the estimates of each term in $I_{ij}$ and $II$, whose proofs are deferred to Section \ref{subsubsec_prf_cumulantexpansionL} and Section \ref{subsec_cumulant_estH}, respectively. To shorten the formulae, we will remove the indexes $i$ and $j$ in $I_{ij}$ and $II$ when there is no confusion.

	\begin{lemma}\label{lem_cumulant_estM}
		For the expansion of $I_{ij}$, we have
		\begin{equation*}
			\begin{split}
				\sum_{ij}\kappa_{1,m}\mathscr{M}_{0,ij}\lesssim & ~n^{-c},
				\sum_{ij}\mathscr{M}_{1,ij}=-\mathbb{E}_{\Psi}^{\chi}\left((1+z_1m(z_1))\langle \mathrm{e}^{\mathrm{i}x\langle L_{1}(f)\rangle}\rangle\right)+\mathrm{O}_{\prec}(n^{-\epsilon_{\alpha}}),\\
				 \sum_{ij}\mathscr{M}_{2,ij}=&-\mathbb{E}_{\Psi}^{\chi}\left[\left(\partial_{z_1}(z_1m(z_1))+(\phi^{-1}z_1\underline{m}(z_1)+\mathrm{O}_{\prec}(n^{-\epsilon_{\alpha}}))\operatorname{tr}\mathcal{G}(z_1)\right)\langle \mathrm{e}^{\mathrm{i}x\langle L_{1}(f)\rangle}\rangle\right]\\
				&-\mathbb{E}_{\Psi}^{\chi}\left(\phi^{-1}z_1m(z_1)\underline{m}(z_1)(1+2z_1\underline{m}(z_1))\langle \mathrm{e}^{\mathrm{i}x\langle L_{1}(f)\rangle}\rangle\right)+\mathrm{O}_{\prec}(n^{-\epsilon_{\alpha}}),\\
				 \sum_{ij}\mathscr{M}_{3,ij}=&\frac{x}{\pi}\oint_{\bar{\gamma}_2^0}\left[\partial_{z_2}\left(\frac{z_1m(z_1)-z_2m(z_2)}{z_1-z_2}+\phi^{-1}z_1m(z_1)\underline{m}(z_1)z_2\underline{m}(z_2)\right)\right]f(z_2)\mathrm{d}z_2\\
&\times\mathbb{E}_{\Psi}^{\chi}[\mathrm{e}^{\mathrm{i}x\langle L_{1}(f)\rangle}]+\mathrm{O}_{\prec}(n^{-c}),
			\end{split}
		\end{equation*}
		and
		\begin{equation*}
			\begin{split}
				\sum_{ij}(\kappa_{3,m}E_{1,m}+\kappa_{4,m}E_{2,m})&\lesssim \mathrm{O}_{\prec}(n^{-\epsilon_{\alpha}})\mathbb{E}_{\Psi}^{\chi}\left[z_1\operatorname{tr}\mathcal{G}(z_1)\langle \mathrm{e}^{\mathrm{i}x\langle L_{1}(f)\rangle}\rangle\right]+\mathrm{O}_{\prec}(n^{-c}), \\
				E_{3,m}+E_{4,m}&=\mathrm{O}_{\prec}(n^{-D}), \sum_{ij}\mathbb{E}(R_{m,ij}(q))=\mathrm{O}_{\prec}(n^{-c}),
			\end{split}
		\end{equation*}
		for some constant $c>0$ and any fixed $D>0$.
	\end{lemma}
	\begin{lemma}\label{lem_cumulant_estH}
		For $\alpha\in (3,4]$, we have
		\begin{equation*}
			II=\sum_{j\in\mathrm{I}_c}\mathbb{E}_{\Psi}^{\chi}(Y_j^*\mathcal{G}(z_1)Y_j\langle \mathrm{e}^{\mathrm{i}x\langle L_{1}(f)\rangle}\rangle)
			\lesssim  \mathrm{o}(n^{-c}).
		\end{equation*}
	\end{lemma}
	With the estimates above, we now proceed to the proof of the asymptotic normality of $L_{1}(f)$. With an additional negligible factor $n^{-c}$ for some small constant $c>0$, recalling \eqref{eq_characteristicfunctioncomparison_ode} and \eqref{eq_prf_characteristicfunctioncomparison_expansion_phi}, for $\alpha\in (3,4]$, we get
	\begin{equation*}
		\begin{split}
			&\mathbb{E}^{\chi}_{\Psi}[\langle\operatorname{tr}R\mathcal{G}(z_1)\rangle \mathrm{e}^{\mathrm{i}x\langle L_{1}(f)\rangle}]
			 =-\phi^{-1}\mathbb{E}_{\Psi}^{\chi}\left[(z_1\underline{m}(z_1)+\mathrm{O}_{\prec}(n^{-\epsilon_{\alpha}}))\operatorname{tr}\mathcal{G}(z_1)\langle \mathrm{e}^{\mathrm{i}x\langle L_{1}(f)\rangle}\rangle\right]\\
			 &+\frac{x}{\pi}\oint_{\bar{\gamma}_2^0}{\partial_{z_2}\left(\frac{z_1m(z_1)-z_2m(z_2)}{z_1-z_2}+\phi^{-1}z_1m(z_1)\underline{m}(z_1)z_2\underline{m}(z_2)\right)}f(z_2)\mathrm{d}z_2\cdot\mathbb{E}_{\Psi}^{\chi}[\mathrm{e}^{\mathrm{i}x\langle L_{1}(f)\rangle}]+\mathrm{O}_{\prec}(n^{-c}),
		\end{split}
	\end{equation*}
	which together with the trivial identity $\langle\operatorname{tr}R\mathcal{G}(z_1)\rangle=z_1 \langle\operatorname{tr}\mathcal{G}(z_1)\rangle$ gives
	\begin{eqnarray*}
	&&	\mathbb{E}^{\chi}_{\Psi}[\langle\operatorname{tr}\mathcal{G}(z_1)\rangle \mathrm{e}^{\mathrm{i}x\langle L_{1}(f)\rangle}]\\
&=&\oint_{\bar{\gamma}_2^0}\frac{x}{\pi}\frac{{\partial_{z_2}\left(\frac{z_1m(z_1)-z_2m(z_2)}{z_1-z_2}+\phi^{-1}z_1m(z_1)\underline{m}(z_1)(z_2\underline{m}(z_2))\right)}}{z_1(1+\phi^{-1}\underline{m}(z_1))}f(z_2)\mathrm{d}z_2\cdot\mathbb{E}_{\Psi}^{\chi}[\mathrm{e}^{\mathrm{i}x\langle L_{1}(f)\rangle}]+\mathrm{O}_{\prec}(n^{-c}).
	\end{eqnarray*}
	Combining  \eqref{eq_characteristicfunctioncomparison_ode}, we obtain
	\begin{eqnarray*}
	&&	\left(\mathcal{T}_{f}(x)\right)^{\prime}+\mathrm{O}_{\prec}(n^{-c})\\
&=&-x\frac{1}{2\pi^2}\oint_{\bar{\gamma}_1^0}\oint_{\bar{\gamma}_2^0}\frac{\partial_{z_2}\left(\frac{z_1m(z_1)-z_2m(z_2)}{z_1-z_2}+\phi^{-1}z_1m(z_1)\underline{m}(z_1)(z_2\underline{m}(z_2))\right)}{z_1(1+\phi^{-1}\underline{m}(z_1))}f(z_1)f(z_2)\mathrm{d}z_2\mathrm{d}z_1\cdot\mathcal{T}_{f}(x),
	\end{eqnarray*}
	which concludes the proof of the approximate ODE, and thus the asymptotic normality of $\operatorname{tr}f(R)$.
	
	It remains to estimate the expectation of $\operatorname{tr}f(R)$. Since the asymptotic variance is of order one, we shall estimate $\mathbb{E}(\operatorname{tr}f(R))$ up to constant order. Recalling the relation
	\begin{equation*}
		\mathbb{E}(\operatorname{tr}f(R))=\mathbb{E}(L_{a}(f))+\mathrm{O}_{\prec}(n^{-K+1})=-\frac{1}{2\pi \mathrm{i}}\oint_{\bar{\gamma}_1^{0}}\mathbb{E}(\operatorname{tr}\mathcal{G}(z))f(z)\mathrm{d}z\cdot\Xi+\mathrm{O}_{\prec}(n^{-K+1})
	\end{equation*}
	and the identity $\operatorname{tr}(R\mathcal{G}(z))=n+z_1\operatorname{tr}\mathcal{G}(z)$, we shall apply the cumulant expansion formula to $\mathbb{E}_{\Psi}^{\chi}[ \operatorname{tr}(R\mathcal{G}(z))]$, which differs from \eqref{eq_prf_characteristicfunctioncomparison_expansion_phi} in the term $\langle \mathrm{e}^{\mathrm{i}x\langle L_{1}(f)\rangle}\rangle$. In particular, for $\alpha\in (3,4]$,  we have
	\begin{equation*}
		\begin{split}
			&\mathbb{E}_{\Psi}^{\chi}(\operatorname{tr}(R\mathcal{G}(z)))\\
=&-\phi^{-1}(z\underline{m}(z))^{\prime}-\phi^{-1}\mathbb{E}_{\Psi}^{\chi}\left[(z\underline{m}(z)+\mathrm{O}_{\prec}(n^{-c}))\operatorname{tr}\mathcal{G}(z)\right]+2(1+zm(z))z\underline{m}(z)+\mathrm{O}_{\prec}(n^{-c}),
		\end{split}
	\end{equation*}
	which combined with $\operatorname{tr}(R\mathcal{G}(z))=n+z\operatorname{tr}\mathcal{G}(z)$ gives
	\begin{equation*}
		 \mathbb{E}_{\Psi}^{\chi}(\operatorname{tr}\mathcal{G}(z))=\frac{-n+2z(1+zm(z))\underline{m}(z)-\phi^{-1}[z\underline{m}(z)]^{\prime}}{z(1+\phi^{-1}\underline{m}(z))}+\mathrm{O}_{\prec}(n^{-c}),
	\end{equation*}
	as presented in Theorem \ref{thm_main_cltlss}.
	
\subsubsection{The critical case for $\alpha=3$}\label{sec_criticalcase}
For the critical case $\alpha=3$, the point lies in the asymptotic behavior of the slowly varying function $l(x)$ as $x\to \infty$. We begin with a simple instance, Schott's statistics $f(x)=x^2$, whose details are postponed to Part III of Appendix \ref{sec_alpha3}.
\begin{example}[Schott's statistics]\label{example}
Suppose Assumptions \ref{ass_X} and \ref{ass_phi} hold. We have
\begin{equation*}
    \mathbb{E}(\operatorname{tr}(R^2))=p+np(p-1)\beta_{2}^2+n(n-1)p(p-1)\beta_{1,1}^2,
\end{equation*}
and
\begin{equation}\label{eq_example_variance}
    \operatorname{Var}(\operatorname{tr}R^2)=2np^2\beta_{4}^2+4p^2n^{-2}+\mathrm{o}(1).
\end{equation}
where $\beta_{k_1,\ldots,k_q}:=\mathbb{E}(Y_{11}^{k_1}\cdots Y_{q1}^{k_q})$.
\end{example}
The term $4p^2n^{-2}$ contributes to the universal variance of $\operatorname{tr}R^2$ (cf. \cite[Corollary 1.14]{bao2022spectral}), while the extra term $2np^2\beta_{4}^2\asymp n^{3-\alpha}(l(n^{1/2}))^2 $ by Lemma \ref{lem_moment_rates}, is negligible if and only if \eqref{eq_criticalcase} holds.

To this end, we give a formal proof of Theorem \ref{thm_iffalpha}. Given the critical condition \eqref{eq_criticalcase}, we expect that the main part still follows from the $I$-part, while we need to control the $II$-part. Thanks to the concentration of the quadratics, we can show that the $II$ part is negligible whenever \eqref{eq_criticalcase} holds. In contrast, this heavy part does not vanish if
	\begin{equation*}
		\limsup_{x\rightarrow \infty}x^3\mathbb{P}(|\xi|>x)\neq 0,
	\end{equation*}
	which further implies the universal CLT fails. All the proofs are deferred to Appendix \ref{sec_alpha3}.

\section{Appendix}\label{app}

\subsection{Proof of the results in Section \ref{sec_pre}}\label{prf_sec_pre}
\subsubsection{Proof of Lemma \ref{lem_priorest_x}}\label{prf_lem_priorest_x}
The proofs of the first two statements can be found in Section A.2 of \cite{ding2023extreme}. For the third statement, we observe that for any $1\le j\le p$ and $\epsilon_{4,1}>0$ that
	\begin{equation*}
		\begin{split}
			&\mathbb{P}(|\frac{1}{n}\sum_{i=1}^nx_{ji}^2-1|>\epsilon_{4,1})\\
&\le n\mathbb{P}(|x_{11}|>\delta\sqrt{n})+\mathbb{P}(|\sum_{i=1}^nx_{ji}^2\mathbbm{1}(|x_{11}|\le\delta\sqrt{n})-n\mathbb{E}x_{11}^2\mathbbm{1}(|x_{11}|\le\delta\sqrt{n})|>\frac{n\epsilon_{4,1}}{2})\\
			&\le n\mathbb{P}(|x_{11}|>\delta\sqrt{n})+\frac{4}{\epsilon_{4,1}^2n}\mathbb{E}(x_{11}^4\mathbbm{1}(|x_{11}|\le\delta\sqrt{n}))\le n\mathbb{P}(|x_{11}|>\delta\sqrt{n})+\frac{4\delta^2}{\epsilon_{4,1}^2}\mathbb{E}(x_{11}^2\mathbbm{1}(|x_{11}|\le\delta\sqrt{n})).
		\end{split}
	\end{equation*}
Setting $n^{1/\alpha-1/2}\ll\delta\ll\epsilon_{4,1}\ll1$ and letting $n\rightarrow\infty$, we have $n^{-1}\sum_ix_{ji}^2=1+\mathrm{O}_{\mathbb{P}}(n^{1/\alpha-1/2+\epsilon_{4,2}})$ for some properly chosen $\epsilon_{4,2}>0$. Then the typical order of $\rho_j$ follows. Consequently, the fourth statement is a direct corollary of the third statement.

For \eqref{eq_rhoj_highpro}, let $C$ be a constant such that $\mathbb{P}(|\xi|\ge C)=1/2$. For constant $\epsilon_{2,2}>0$ such that $\epsilon_{2,2}/C^2<1/10$, one has
	\begin{equation*}
		\begin{split}
			\mathbb{P}(n^{-1}\rho_j^2<\epsilon_{2,2})\le \mathbb{P}(\sum_{i=1}^n\mathbbm{1}_{|X_{ji}}|\ge C\le n\epsilon_{2,2}/C^2)\lesssim \binom{n}{n\epsilon_{2,2}/C^2}2^{-n}\le \exp(-cn),
		\end{split}
	\end{equation*}
	noting that $\binom{n}{k}\le (en/k)^{k}$.
	
	For \eqref{eq_tr_SS}, recalling \eqref{eq_rhoj_highpro} and the control parameters defined in Definition \ref{def_controlparameter}, choosing $\delta=n^{1/2\alpha-1/4}$ above, we have for some $0<\epsilon_s<(1/2-1/\alpha)/2$,
	\begin{equation*}
		\mathbb{P}(|n^{-1}\rho_j^2-1|>n^{-\epsilon_s})\le n^{1/\alpha-1/2+2\epsilon_s}.
	\end{equation*}
	It follows that
	\begin{equation*}
		\begin{split}
			\mathbb{P}&(\#\{j:|n^{-1}\rho^2_j-1|>n^{-\epsilon_s}\}>n^{\beta})\lesssim \mathbb{P}(|n^{-1}\rho^2_{(n^{\beta})}-1|>n^{-\epsilon_s})\\
			&\lesssim\sum_{k=n^{\beta}}^{n}\binom{n}{k}n^{(1/\alpha-1/2+2\epsilon_s)k}(1-n^{1/\alpha-1/2})^{n-k}\\
			&\lesssim\sum_{k=n^{\beta}}^{n} n^{-\beta k}n^{(1/\alpha+1/2+2\epsilon_s)k}e^{-(n-k)/n^{1/\alpha-1/2}}\le n^{-D}
		\end{split}
	\end{equation*}
	for some $\beta=1/\alpha+1/2+2\epsilon_s<1$, where $\rho_{(1)}\geq \rho_{(2)}\geq \cdots\geq \rho_{(p)}$ are the order statistics of $\rho_1,\cdots,\rho_p$. We choose $\epsilon_s=(1/2-1/\alpha)/3$, so $\beta=1/\alpha+1/2+(1-2/\alpha)/3$. Thus, we have
	\begin{equation}\label{eq_rho_number_hpe}
		\mathbb{P}(\#\{j:|n^{-1}\rho_{j}^2-1|>n^{-(1/2-1/\alpha)/3}\}\le n^{\beta})\le n^{-D},
	\end{equation}
	where $D>0$ is any fixed constant. With this observation, consider the diagonal matrix $n(\operatorname{diag}(S))^{-1}-I=\operatorname{diag}(n/\rho_1^2-1,\ldots,n/\rho_{p}^2-1)$. We have
	\begin{equation*}
		\begin{split}
			 \operatorname{tr}|(\operatorname{diag}(S))^{-1}-n^{-1}I|&=n^{-1}\operatorname{tr}|(n^{-1}\operatorname{diag}(S)-I)(n^{-1}\operatorname{diag}(S))^{-1}|\\
			&\lesssim \frac{n-n^{\beta}}{n}\cdot n^{-\epsilon_s}+\frac{n^{\beta}}{n}\lesssim n^{-(\alpha-2)/(6\alpha)},
		\end{split}
	\end{equation*}
	where we have used the high probability estimate $\rho_j^{-2}\lesssim n^{-1}$ by \eqref{eq_rhoj_highpro}.
	Moreover, for  $s\ge 1$, the induction gives $\operatorname{tr}|(\operatorname{diag}(S))^{-s}-n^{-s}I|
	\lesssim n^{-(s-1)-(\alpha-2)/(6\alpha)}$.
\qed

\subsubsection{Proof of Lemma \ref{lem_oddmoment_est}}

The first statement follows from a combination of Lemma 3.1 and Theorem 3.3 in \cite{gine1997student}. We omit further details here.

 Now we present the estimate of $\mathbb{E}(Y_{11}Y_{21})$ via the cumulant expansion method (cf. Lemma \ref{lemma_cumulantexpansion}). Throughout this subsection, for simplification, we omit the constant factor which does not affect the typical order. Recalling the decomposition \eqref{eq_decomp_X} and the derivatives (cf. Section \ref{subsubsec_derivative} below), we have
	\begin{equation*}
		\begin{split}
			 \mathbb{E}(Y_{11}Y_{21})=\mathbb{E}\left(\frac{M_{1j}M_{2j}}{\rho_j^2}\right)+\mathbb{E}\left(\frac{H_{1j}(M_{2j}+H_{2j})+H_{2j}M_{1j}}{\rho_j^2}\right),
		\end{split}
	\end{equation*}
	where we reuse the notation $M$ to denote the part of $L+M$ in \eqref{eq_decomp_X} (also as illustrated below \eqref{eq_cha_ode2}). Moreover, we borrow the notation of $\kappa_{q,m}$ to denote the $q$-th cumulant of $M_{ij}$, and apply the estimates in Section \ref{sec_characteristicfunctionestimation}. Invoking the estimate $\mathbb{E}|H_{ij}|^{s}\lesssim n^{(1/2-\epsilon_h)(-\alpha+s)}$ for any $s<\alpha$, one has
	$\mathbb{E}(H_{1j}H_{2j}\rho_j^{-2})\lesssim n^{2(-\alpha+1)(1/2-\epsilon_h)-1}$ by $\rho_{j}^{-1}\lesssim n^{-1/2}$. By $\kappa_{1,m}=-\mathbb{E}(H_{ij})\lesssim n^{(1/2-\epsilon_h)(-\alpha+1)}$ and the cumulant expansion w.r.t. $M_{1j}$, one has
	\begin{equation*}
		\begin{split}
			\mathbb{E}(H_{1j}M_{2j}\rho_j^{-2})\lesssim &\kappa_{1,m}\mathbb{E}(H_{1j}\rho_j^{-2})+\kappa_{2,m}\mathbb{E}(H_{1j}M_{2j}\rho_j^{-4})+\kappa_{3,m}\mathbb{E}(H_{1j}\rho_j^{-4}+H_{1j}M_{2j}^2\rho_j^{-6})+\operatorname{Res}_1\\
			\lesssim &n^{2(-\alpha+1)(1/2-\epsilon_h)-1}+n^{(-\alpha+1)(1/2-\epsilon_h)-2}+n^{(3-\alpha)_{+}(1/2-\epsilon_h)+(-\alpha+1)(1/2-\epsilon_h)-2}+\operatorname{Res}_1\\
			\lesssim & n^{(-\alpha+1)(1/2-\epsilon_h)-2}\mathbbm{1}(\alpha> 3)+n^{2(-\alpha+1)(1/2-\epsilon_h)-1}\mathbbm{1}(\alpha\le 3),
		\end{split}
	\end{equation*}
	where $\operatorname{Res}_1$ denotes the higher-order terms in the cumulant expansion which can be bounded by
	\begin{equation*}
		\operatorname{Res}_1\lesssim \kappa_{4,m}\mathbb{E}(H_{1j}M_{2j}\rho_j^{-6})+\kappa_{5,m}\mathbb{E}(H_{1j}\rho_j^{-6})\lesssim n^{(1/2-\epsilon_h)(6-2\alpha)-3}
	\end{equation*}
	in a similar way due to the truncation and the estimate $\kappa_{q,m}\lesssim n^{(1/2-\epsilon_h)(q-\alpha)}$ for integer $q\ge 4$. The first term can be estimated by the cumulant expansion w.r.t. $M_{1j}$ as
	\begin{equation*}
		\begin{split}
			\mathbb{E}\left(\frac{M_{1j}M_{2j}}{\rho_j^2}\right)\lesssim &\kappa_{1,m}\mathbb{E}\left(M_{2j}\rho_j^{-2}\right) +\kappa_{2,m}\mathbb{E}\left(M_{1j}M_{2j}\rho_{j}^{-4}\right)+\frac{\kappa_{3,m}}{2!}\mathbb{E}(M_{2j}\rho_{j}^{-4}+M_{1j}^2M_{2j}\rho_{j}^{-6})\\
			&+\frac{\kappa_{4,m}}{3!}\mathbb{E}(M_{1j}M_{2j}\rho_{j}^{-6}-M_{1j}^3M_{2j}\rho_{j}^{-8})+\operatorname{Res}_2
		\end{split}
	\end{equation*}
 where $\operatorname{Res}_2$ can be bounded similarly as $\operatorname{Res}_2\lesssim n^{(q-\alpha)_+(1/2-\epsilon_h)-(q+2)/2}$. Moreover, by repeating the cumulant expansion and a similar argument handling the high-order terms, we have
	\begin{equation*}
		\begin{split}
			\mathbb{E}\left(M_{2j}\rho_j^{-2}\right) =&\kappa_{1,m}\mathbb{E}\left(\rho_j^{-2}\right)+\kappa_{2,m}\mathbb{E}\left(-2M_{2j}\rho_j^{-4}\right)+\frac{\kappa_{3,m}}{2!}\mathbb{E}(-2\rho_j^{-4}+8M_{2j}^2\rho_j^{-6})+\mathrm{o}(n^{-3/2})\\
			 =&\kappa_{1,m}\mathbb{E}(\rho_j^{-2})+\kappa_{2,m}\mathbb{E}(-2M_{2j}\rho_j^{-4})-\kappa_{3,m}\mathbb{E}(\rho_j^{-4})+\mathrm{o}(n^{-3/2}).
		\end{split}
	\end{equation*}
	By the estimates for $\kappa_{q,m}, 1\le q\le 3$ and another cumulant expansion for $\mathbb{E}(M_{2j}\rho_j^{-4})$
    \begin{equation*}
		\mathbb{E}(M_{2j}\rho_j^{-4})\lesssim n^{(-\alpha+1)(1/2-\epsilon_h)-2}+ n^{(3-\alpha)_{+}(1/2-\epsilon_h)-3}+\kappa_{2,m}\mathbb{E}(-2M_{2j}\rho_j^{-6})\lesssim n^{-5/2},
	\end{equation*}
    we get
    \begin{equation*}
		\mathbb{E}(M_{2j}\rho_j^{-2})\lesssim n^{(3-\alpha)_{+}(1/2-\epsilon_h)-2}+n^{(-\alpha+1)(1/2-\epsilon_h)-1}+\kappa_{2,m}\mathbb{E}(-2M_{2j}\rho_j^{-4}).
	\end{equation*}

     This implies $\mathbb{E}(M_{2j}\rho_j^{-2})\lesssim n^{(3-\alpha)_{+}(1/2-\epsilon_h)-2}+n^{(-\alpha+1)(1/2-\epsilon_h)-1}$.
	Moreover, we observe the pattern that such estimation can be generalized to any integer $q\ge 3$ as
	\begin{equation*}
		\begin{split}
			\mathbb{E}(M_{2j}\rho_j^{-2q})\lesssim n^{(3-\alpha)_{+}(1/2-\epsilon_h)-q-1}+n^{(-\alpha+1)(1/2-\epsilon_h)-q}.
		\end{split}
	\end{equation*}
	Thus, all the odd terms are bounded by $n^{(-\alpha+1)(1/2-\epsilon_h)-2}\mathbbm{1}(\alpha> 3)+n^{2(-\alpha+1)(1/2-\epsilon_h)-1}\mathbbm{1}(\alpha\le 3)$.
	For the even terms, we have
	\begin{equation*}
		\begin{split}
			\mathbb{E}\left(\frac{M_{1j}M_{2j}}{\rho_j^4}\right)\lesssim &\kappa_{1,m}\mathbb{E}\left(M_{2j}\rho_j^{-4}\right) +\kappa_{2,m}\mathbb{E}\left(M_{1j}M_{2j}\rho_{j}^{-6}\right)+\frac{\kappa_{3,m}}{2!}\mathbb{E}(M_{2j}\rho_{j}^{-6}+M_{1j}^2M_{2j}\rho_{j}^{-8})\\
			 &+\frac{\kappa_{4,m}}{3!}\mathbb{E}(M_{1j}M_{2j}\rho_{j}^{-8}+M_{1j}^3M_{2j}\rho_{j}^{-10})+\mathrm{O}_{\prec}(n^{(q-\alpha)_{+}(1/2-\epsilon_h)-(q+2)/2}).
		\end{split}
	\end{equation*}
     Repeat the cumulant expansion and the estimations for $\mathbb{E}(M_{2j}\rho_j^{-2q})$, we have
	\begin{equation*}
		\begin{split}
			\mathbb{E}(M_{1j}M_{2j}\rho_{j}^{-4})\lesssim n^{2(-\alpha+1)(1/2-\epsilon_h)-2}+ n^{2(3-\alpha)_{+}(1/2-\epsilon_h)-3}.
		\end{split}
	\end{equation*}
	This can be also extended to
	$\mathbb{E}(M_{1j}M_{2j}\rho_{j}^{-2q})\lesssim n^{2(-\alpha+1)(1/2-\epsilon_h)-q}+ n^{2(3-\alpha)_{+}(1/2-\epsilon_h)-q-1}$ for general integer $q\ge 3$. Thus, we have
	\begin{equation}\label{eq_m1m2estimate}
		\mathbb{E}\left(\frac{M_{1j}M_{2j}}{\rho_j^2}\right)=n^{(-\alpha+1)(1/2-\epsilon_h)-2}\mathbbm{1}(\alpha> 3)+n^{2(-\alpha+1)(1/2-\epsilon_h)-1}\mathbbm{1}(\alpha\le 3).
	\end{equation}
	Therefore, we get $\mathbb{E}(Y_{11}Y_{21})\lesssim n^{(-\alpha+1)(1/2-\epsilon_h)-2}\mathbbm{1}(\alpha> 3)+n^{2(-\alpha+1)(1/2-\epsilon_h)-1}\mathbbm{1}(\alpha\le 3)$ for general $\alpha\in (2,4]$, which improves the classical bound for $\mathbb{E}(Y_{11}Y_{21})=\mathrm{o}(n^{-2})$.
 \qed

\subsubsection{Proof of Lemma \ref{lem_pre_estnorm}}\label{prf_lem_pre_estnorm}
Before presenting the proof, we collect some notation that will be used in the sequel. We denote the sample correlation matrix with data matrix $M$ as $R(M)=M(\operatorname{diag}S)^{-1}M^*$ (cf. \eqref{eq_def_RL&RH}). Define its Green function as  $\mathcal{G}R_{M}(z)=(R(M)-zI)^{-1}$ and the corresponding Stieltjes transform as $mR_{M}(z)=n^{-1}\operatorname{tr}\mathcal{G}R_{M}(z)$. Also, we will use the sample covariance matrix $\mathcal{S}(M):=n^{-1}MM^{*}$ whose Green function and Stieltjes transform are $\mathcal{G}S_{M}(z)=(\mathcal{S}(M)-zI)^{-1}$, $mS_{M}(z)=n^{-1}\operatorname{tr}\mathcal{G}S_{M}(z)$, respectively.

For the first statement, similarly to \eqref{eq_def_RL&RH} conditioned on $\psi_{ij}$, we rewrite
\begin{equation*}
	R(Y)=R(L+M)+R(H):=(L+M)(\operatorname{diag}S)^{-1}(L+M)^*+H(\operatorname{diag}S)^{-1}H^*.
\end{equation*}
To show $\|R(Y)\|\le C$ for some large constant $C>0$ with high probability, it suffices to bound $\lambda_{1}(R(L+M))$ and $\lambda_1(R(H))$. Consider $R(H)$ first. We know that the number of non-zero entries of $H$ is at most $n^{1-\epsilon_y}$ with high probability since $\Psi$ and $\Pi$ are well configured ( Lemma \ref{lem_wellconfigured}). Let $q$ be some integer such that $q/(\log n)^2\rightarrow \infty$ as $n\rightarrow \infty$. We have $\mathbb{E}_{\Psi}(\lambda_1(R(H)))^q\le \mathbb{E}_{\Psi}\operatorname{tr}((R(H))^q)$, which can be estimated as
\begin{equation*}
\begin{split}
    \mathbb{E}_{\Psi}\operatorname{tr}((R(H))^q)&=\sum^{p}_{j}\sum_{i}^{n}\mathbb{E}_{\Psi}(H_{i_1j_q}\rho_{j_q}^{-1}H_{i_1j_1}\rho_{j_1}^{-1}\cdots H_{i_{q}j_{q-1}}\rho_{j_{q-1}}^{-1}H_{i_qj_{q}}\rho_{j_q}^{-1})\\
    &\lesssim \sum_{r=1}^{q}(n^{1-\epsilon_y})^r(n^{-(N_1(1-\alpha/2)+r\alpha/2)}+n^{-r})\\
    &\lesssim n^{-\epsilon_y},
\end{split}
\end{equation*}
where we have used Lemmas \ref{lem_moment_rates} and \ref{lem_oddmoment_est}, $r$ denotes the number of different couples of the subscripts, and $N_1=\#\{1\le i\le r,k_{i}=1\}$. Thereafter, one has
\begin{equation*}
    \mathbb{P}_{\Psi}(\lambda_{1}(R(H))>2)\le 2^{-q}\mathbb{E}_{\Psi}\operatorname{tr}((R(H))^q)\lesssim n^{-D}
\end{equation*}
for any large constant $D>0$, which implies that $\|R(H)\|\le 2$ with high probability conditioned on $\Psi$.

Now we consider $\lambda_{1}(R(L+M))$, which is more involved. In the following, we aim to show
\begin{equation*}
    |\lambda_{1}(R(L+M))-(1-\widetilde{t})\lambda_+|\prec n^{-2\epsilon_{h}}+n^{-2/3},
\end{equation*}
where $\widetilde{t}=1-\mathbb{E}|L_{ij}+M_{ij}|^2=\mathrm{o}(n^{-1})$.
From \cite[Theorem 2.9]{hwang2019local}, we first observe that
    \begin{equation}\label{eq_lambda1_SLM}
        |\lambda_1(\mathcal{S}(L+M))-(1-\widetilde{t})\lambda_{+}|\prec n^{-2\epsilon_{h}}+n^{-2/3}.
    \end{equation}
Denote $\mathrm{Cn}:=\mathrm{Cn}(\widetilde{H}S)$ as the number of non-zero columns of $\widetilde{H}S$ and $\mathrm{Ch}:=\mathrm{Ch}(HS)$ as the number of non-zero columns of $HS$ (recall the definitions of $\widetilde{H}S$ and $HS$ in \eqref{eq_def_LS&HS}). By Lemma \ref{lem_priorest_x}, $|\mathrm{Cn}|\le n^{\beta}$ with high probability. Moreover, repeating similar arguments in Lemma \ref{lem_priorest_x}, we have $|\mathrm{Ch}|\le n^{1-\epsilon_h}$ with high probability. Then by the Cauchy interlacing theorem, we have
    \begin{equation*}
        \lambda_1(R(M^{(\mathrm{Cn})}))\le \lambda_1(R(M^{(\mathrm{Ch})}))\le \lambda_1(R(\widetilde{M}^{(\mathrm{Ch})})),
    \end{equation*}
    where $M^{(\mathrm{Cn})}$ is a copy of $M$ such that its $k$-th column is same as the $k$-th column of $M$ if the $k$-th column of $\widetilde{H}S$ is zero, otherwise replacing its $k$-th column with zero if the $k$-th column of $\widetilde{H}S$ is non-zero. Similarly, $M^{(\mathrm{Ch})}$ has the same $k$-th column as $M$ if the $k$-th column of $HS$ is zero, otherwise replacing its $k$-th column with zero if the $k$-th column of $HS$ is non-zero.
     And $\widetilde{M}$ denotes the random matrix with the same size as $M$ such that its entries satisfy $\{\widetilde{m}_{ij}\in\mathrm{I}\}=\{X_{ji}\in\big((-n^{1/2-\epsilon_h}\cdot n^{\epsilon_h}, -n^{-\epsilon_l}]\bigcup [n^{-\epsilon_l}, n^{1/2-\epsilon_h}\cdot n^{\epsilon_h})\big)\bigcap\mathrm{I}\}$. One may easily see that $\widetilde{M}$ has more elements than $M$, thus $M$ is a submatrix of $\widetilde{M}$. On the one hand, one may observe that
    \begin{equation*}
        \begin{split}
            \|R(M^{(\mathrm{Cn})})-\mathcal{S}(M^{(\mathrm{Cn})})\|&\le \|n^{-1}M^{(\mathrm{Cn})}\big(LS\cdot(n^{-1}\operatorname{diag}S)^{-1}\big)(M^{(\mathrm{Cn})})^{*}\|\prec n^{-\epsilon_s}.
        \end{split}
    \end{equation*}
    It follows from \eqref{eq_lambda1_SLM} that
    \begin{equation*}
        |\lambda_1(R(M^{(\mathrm{Cn})}))-(1-t)\lambda_{+,\phi_n}|\prec n^{-\epsilon_s}+n^{-2/3},
    \end{equation*}
    with modified ratio $\phi_n=(p-|\mathrm{Cn}|)/n$. In addition, for matrix $R(\widetilde{M}^{(\mathrm{Ch})})$, we have
    \begin{equation*}
        \widetilde{M}^{(\mathrm{Ch})}(\operatorname{diag} S)^{-1}(\widetilde{M}^{(\mathrm{Ch})})^{*}=\frac{1}{n} \widetilde{M}^{(\mathrm{Ch})}(MS+LS+I)^{-1}(\widetilde{M}^{(\mathrm{Ch})})^{*}.
    \end{equation*}
    The deterministic bound for the elements in $MS$ implies that the entries in $\{\mathrm{Cn}\}\setminus\{\mathrm{Ch}\}$ columns of $\widetilde{M}^{(\mathrm{Ch})}$ remain bounded support $\mathrm{O}(n^{-\epsilon_{\alpha}})$. Therefore, deploying the same argument as that in \cite[Theorem 2.9]{hwang2019local}, we have
    \begin{equation*}
        |\lambda_1(R(\widetilde{M}^{(\mathrm{Ch})}))-(1-t)\lambda_{+,\phi^{\prime}_n}|\prec n^{-\epsilon_{\alpha}}+n^{-2/3},
    \end{equation*}
    with another modified ratio $\phi^{\prime}_n=(p-|\mathrm{Ch}|)/n$. Noticing that $|\phi-\phi_n|+|\phi-\phi^{\prime}_n|=\mathrm{O}(n^{-\epsilon_s})$,  we obtain the limit for $\lambda_1(R(M^{(\mathrm{Ch})}))$ as follows,
    \begin{equation}\label{eq_eigenconvergence_locallaw_RMremoved}
        |\lambda_1(R(M^{(\mathrm{Ch})}))-(1-t)\lambda_{+}|\prec n^{-\epsilon_s}+n^{-2/3}.
    \end{equation}

    On the other hand, we observe that
    \begin{equation*}
        R(M)-R(M^{(\mathrm{Ch})})=n^{-1}M^{[\mathrm{Ch}]}(n^{-1}\operatorname{diag}S)^{-1}(M^{[\mathrm{Ch}]})^{*},
    \end{equation*}
    where $M^{[\mathrm{Ch}]}$ is another copy of $M$ such that its $k$-th column is same as the one of $M$ if the $k$-th column of $HS$ is non-zero, and its $k$-th column is zero if the $k$-th column of $HS$ is zero. Notice that $n^{-1}(M^{[\mathrm{Ch}]})^{*}M^{[\mathrm{Ch}]}$ is of rank $|\mathrm{Ch}|$. We further rewrite the right hand side of the above equation as
    \begin{equation*}
        R(M)=R(M^{(\mathrm{Ch})})+n^{-1}\sum_{i\in\{\mathrm{Ch}\}}(HS+I)_i^{-1}\mathfrak{m}_i\mathfrak{m}_i^{*},
    \end{equation*}
    where $\mathfrak{m}_i$  denotes the $i$th column of $M$. Define
    \begin{equation*}
        \mathbf{H}_{\tau}(\lambda):=\lambda I-\big(R(M^{(\mathrm{Ch})})+\tau\times n^{-1}\sum_{i\in\{\mathrm{Ch}\}}(HS+I)_i^{-1}\mathfrak{m}_i\mathfrak{m}_i^{*}\big).
    \end{equation*}
    It is easy to see that the eigenvalues of $R(M^{(\mathrm{Ch})})+\tau\times n^{-1}\sum_{i\in\{\mathrm{Cn}\}}(HS+I)_i^{-1}\mathfrak{m}_i\mathfrak{m}_i^{*}$ is continuous in $\tau$. We are going to prove the following statement,
    \begin{equation}\label{eq_eigenstick_locallaw_corM}
        \mathbb{P}(\operatorname{det}(\mathbf{H}_{\tau}(\lambda))\neq 0,\forall t\in[0,1])=1-\mathrm{o}(1),\quad \mu=(1-t)\lambda_{+}+n^{-\epsilon_{\mu}},
    \end{equation}
    for some small value $\epsilon_{\mu}>0$ (cf. Definition \ref{def_controlparameter}). We know that if \eqref{eq_eigenstick_locallaw_corM} holds, then the largest eigenvalue of $R(M^{(\mathrm{Ch})})+\tau\times n^{-1}\sum_{i\in\{\mathrm{Ch}\}}(HS+I)_i^{-1}\mathfrak{m}_i\mathfrak{m}_i^{*}$ will not cross the boundary $\lambda_1(R(M^{(\mathrm{Ch})}))\pm n^{-\epsilon_{\mu}}$. Hence $\lambda_1(R(M))$ is sticking to $\lambda_1(R(M^{(\mathrm{Ch})}))$ with a rate smaller than $n^{-\epsilon_{\mu}}$.

Now, we prove \eqref{eq_eigenstick_locallaw_corM}. From \eqref{eq_lambda1_SLM} and \eqref{eq_eigenconvergence_locallaw_RMremoved}, we know that the eigenvalues of $R(M^{(\mathrm{Ch})})$ are separated with order $n^{-\epsilon_{\alpha}}$, so for $\epsilon_{\mu}<\epsilon_{\alpha}/2$,
\begin{equation*}
	\mathbb{P}(|\lambda_j(R(M^{(\mathrm{Ch})}))-\mu|>n^{-\epsilon_{\mu}})=1-\mathrm{o}(1).
\end{equation*}
Therefore, $\mu$ is not an eigenvalue of $R(M^{(\mathrm{Ch})})$, and we have
\begin{equation*}
	\operatorname{det}(\mathbf{H}_{\tau}(\mu))=\operatorname{det}(\mu-R(M^{(\mathrm{Ch})}))\operatorname{det}(1-\tau n^{-1}\sum_{i\in\{\mathrm{Ch}\}}(HS+I)_i^{-1}\mathfrak{m}_i\mathfrak{m}_i^{*}\mathcal{G}R_{{M^{(\mathrm{Ch})}}}(\mu)),
\end{equation*}
where $\mathcal{G}R_{{M^{(\mathrm{Ch})}}}(\mu)=(R(M^{(\mathrm{Ch})})-\mu I)^{-1}$. Let $z=(1-t)\lambda_{+}+\mathrm{i}\eta$ for some $\eta>0$.
Note that for each $i$, $(HS+I)_i^{-1}\mathfrak{m}_i\mathfrak{m}_i^{*}\mathcal{G}R_{M^{(\mathrm{Ch})}}(z)=(HS+I)_i^{-1}\mathfrak{m}_i\mathfrak{m}_i^{*}\mathcal{G}R^{(i)}_{M^{(\mathrm{Ch})}}(z)$, since the $i$-th column in $M^{(\mathrm{Ch})}$ is zero. Therefore, we have
\begin{equation*}
	\begin{split}
		\|\tau n^{-1}\sum_{i\in\{\mathrm{Ch}\}}(HS+I)_i^{-1}\mathfrak{m}_i\mathfrak{m}_i^{*}\mathcal{G}R_{M^{(\mathrm{Ch})}}(z)\|
		\le \|(HS+I)^{-1}\|\big\|n^{-1}\sum_{i\in\{\mathrm{Ch}\}}\mathfrak{m}_i\mathfrak{m}_i^{*}\big\|\|\mathcal{G}R_{{M^{(\mathrm{Ch})}}}(z)\|\le C_1
	\end{split}
\end{equation*}
for some constant $0<C_1<1$, observing that $|(hs_i+1)^{-1}|$ can be chosen sufficiently small, and the operator norm of $n^{-1}\sum_{i\in\{\mathrm{Ch}\}}\mathfrak{m}_i\mathfrak{m}_i^{*}$ can be controlled by \eqref{eq_lambda1_SLM} and $\|\mathcal{G}R_{{M^{(\mathrm{Ch})}}}(z)\|$ is bounded.

Now we consider the difference $\mathcal{G}R_{{M^{(\mathrm{Ch})}}}(\mu)-\mathcal{G}R_{{M^{(\mathrm{Ch})}}}(z)$.
Let $\nu_{k}$ be the eigenvector of $R(M^{(\mathrm{Ch})})$ corresponding to the $k$-th eigenvalue  $\lambda^{\prime}_{k}:=\lambda_{k}(R(M^{(\mathrm{Ch})}))$. Note that for $\lambda^{\prime}_{k}$ being greater than some threshold and $z^{\prime}:=\lambda^{\prime}_{k}+\mathrm{i}n^{-1+\delta}$, we have $|n^{-1}\mathfrak{m}_i^{*}\mathcal{G}R_{{M^{(\mathrm{Ch})}}}(z^{\prime})\mathfrak{m}_i|\le C_2$ with high probability. Moreover,
\begin{equation*}
	 \begin{split}|\operatorname{Im}n^{-1}\mathfrak{m}_i^{*}\mathcal{G}R_{{M^{(\mathrm{Ch})}}}(z^{\prime})\mathfrak{m}_i|=(\operatorname{Im}z^{\prime})\sum_{j}\frac{\langle n^{-1/2}\mathfrak{m}_i,\nu_j\rangle^2}{|\lambda^{\prime}_{j}-z^{\prime}|^2}
		\ge (\operatorname{Im}z^{\prime})^{-1}\langle n^{-1/2}\mathfrak{m}_i,\nu_k\rangle^2.
	\end{split}
\end{equation*}
The right hand side has the upper bound $\langle n^{-1/2}\mathfrak{m}_i,\nu_k\rangle^2\prec n^{-1}$. Let $k^{\prime}$ be the largest $k$ satisfying this condition. Then we have $k^{\prime}\asymp cn$ for some constant $c$. The eigenvalue rigidity for $\mathcal{S}(M^{(\mathrm{Ch})})$ implies that $\lambda^{\prime}_k-(1-t)\lambda_{+}\asymp (kn^{-2\epsilon_{\alpha}})$ for any $k\ge n^{\epsilon_{4,1}}$ with $\epsilon_{4,1}$ sufficiently small. Recall $z=(1-t)\lambda_{+}+\mathrm{i}\eta$, for $\eta=n^{-\epsilon_{\alpha}}$. Then for each $i$,
\begin{equation*}
	\begin{split}
		&|n^{-1}\mathfrak{m}_i^{*}\big(\mathcal{G}R_{M^{(\mathrm{Ch})}}(\mu)-\mathcal{G}R_{M^{(\mathrm{Ch})}}(z)\mathfrak{m}_i\big)|=\sum_{k}\langle n^{-1/2}\mathfrak{m}_i,\nu_k\rangle^2\Big|\frac{1}{\mu-\lambda^{\prime}_k}-\frac{1}{z-\lambda^{\prime}_k}\Big|\\
		&\le \sum_{k}\langle n^{-1/2}\mathfrak{m}_i,\nu_k\rangle^2\Big(\frac{\eta}{\big((1-t)\lambda_{+}-\lambda^{\prime}_k\big)^2+\eta^2}+\frac{(1+\mathrm{o}(1))\eta^2}{|\lambda^{\prime}_k-\mu||((1-t)\lambda_{+}-\lambda_k^{\prime})^2+\eta^2|}\Big).
	\end{split}
\end{equation*}
We will split the right hand side summation into three parts.  Firstly,
\begin{equation*}
	\begin{split}
		&\sum_{k\le n^{\epsilon_{4,1}}}\langle n^{-1/2}\mathfrak{m}_i,\nu_k\rangle^2\Big(\frac{\eta}{\big((1-t)\lambda_{+}-\lambda^{\prime}_k\big)^2+\eta^2}+\frac{(1+\mathrm{o}(1))\eta^2}{|\lambda^{\prime}_k-\mu||((1-t)\lambda_{+}-\lambda_k^{\prime})^2+\eta^2|}\Big)\\
		&\prec n^{\epsilon_0}n^{-1}(\eta n^{2\epsilon_{\alpha}}+n^{\epsilon_{\mu}})\le n^{-1+\epsilon_{\mu}+\epsilon_{4,1}}.
	\end{split}
\end{equation*}
Secondly,
\begin{equation*}
	\begin{split}
		&\sum_{n^{\epsilon_{4,1}}\le k\le k^{\prime}}\langle n^{-1/2}\mathfrak{m}_i,\nu_k\rangle^2\Big(\frac{\eta}{\big((1-t)\lambda_{+}-\lambda^{\prime}_k\big)^2+\eta^2}+\frac{(1+\mathrm{o}(1))\eta^2}{|\lambda^{\prime}_k-\mu||(\lambda-\lambda_k^{\prime})^2+\eta^2|}\Big)\\
		&\prec \sum_{n^{\epsilon_{4,1}}\le k\le k^{\prime}}n^{-\epsilon_{\alpha}}n^{-1}(\frac{k}{n^{2\epsilon_{\alpha}}})^{-2}\le n^{-1+3\epsilon_{\alpha}}\sum_{1\le k\le cn}k^{-2}\le n^{-1+3\epsilon_{\alpha}}.
	\end{split}
\end{equation*}
Thirdly,
\begin{equation*}
	\begin{split}
		\sum_{k^{\prime}\le k}\langle n^{-1/2}\mathfrak{m}_i,\nu_k\rangle^2\Big(\frac{\eta}{\big((1-t)\lambda_{+}-\lambda^{\prime}_k\big)^2+\eta^2}+\frac{(1+\mathrm{o}(1))\eta^2}{|\lambda^{\prime}_k-\mu||(\lambda-\lambda_k^{\prime})^2+\eta^2|}\Big)\\
		\prec \sum_{k}\eta^{-1}\langle n^{-1/2}\mathfrak{m}_i,\nu_k\rangle\prec n^{-1+\epsilon_{\alpha}}.
	\end{split}
\end{equation*}
Therefore, we have
\begin{equation*}
	\| \tau n^{-1}\sum_{i\in\{\mathrm{Cn}\}}(HS+I)_i^{-1}\mathfrak{m}_i\mathfrak{m}_i^{*}\big(\mathcal{G}R_{{M^{(\mathrm{Ch})}}}(\mu)-\mathcal{G}R_{{M^{(\mathrm{Ch})}}}(z)\big)\|\prec n^{-1+2\epsilon_{\mu}}.
\end{equation*}
Combining the above estimates, with probability tending to one, we have
\begin{equation*}
	\operatorname{det}\Big(1-\tau n^{-1}\sum_{i\in\{\mathrm{Ch}\}}(HS+I)_i^{-1}\mathfrak{m}_i\mathfrak{m}_i^{*}\mathcal{G}R_{{M^{(\mathrm{Ch})}}}(\mu)\Big)\neq 0,\;\forall\tau\in[0,1].
\end{equation*}
This concludes the first statement.

Following similar arguments as above, one can show that
\begin{equation*}
	|\lambda_{n}(R(L+M))-(1-\widetilde{t})\lambda_{-}|\prec n^{-\epsilon_s}+n^{-2/3}
\end{equation*}
by noticing $|\lambda_{n}(\mathcal{S}(L+M))-(1-\widetilde{t})\lambda_{-}|\prec n^{-2\epsilon_h}+n^{-2/3}$ via the techniques from \cite[Theorem 2.9]{hwang2019local} and \cite[Lemma 2.7]{bao2023smallest}. By the Cauchy interlacing theorem, one can easily get
\begin{equation*}
	\lambda_{n}(R(Y^{[\mathrm{I}_c]}))\le \lambda_n(R(Y))\le \lambda_{n-|\mathrm{I}_r|}(R(Y^{[\mathrm{I}_r]}))
\end{equation*}
which can be read as
\begin{equation*}
	\lambda_{n}(R((L+M)^{[\mathrm{I}_c]}))\le \lambda_n(R(Y))\le \lambda_{n-|\mathrm{I}_r|}(R((L+M)^{[\mathrm{I}_r]}))
\end{equation*}
since $R(Y^{[\mathrm{I}_c]})=R((L+M)^{[\mathrm{I}_c]})$ and $R(Y^{[\mathrm{I}_r]})=R((L+M)^{[\mathrm{I}_r]})$. Applying similar methods with the modified parameters $\phi_n=(n/(p-|\mathrm{I}_c|))$ and $\phi_n=((n-|\mathrm{I}_r|)/p)$ respectively leads to
\begin{equation*}
	|\lambda_n(R(Y))-(1-\widetilde{t})\lambda_{-}|\prec n^{-\epsilon_s}
\end{equation*}
by $\widetilde{t}=\mathrm{o}(n^{-1})$. Thus, together with $\|R(Y)\|\le C$ with high probability, one get that there exists some constants $0<a<b$ such that $\operatorname{supp}(R(Y))\subset \{0\}\bigcup [a,b]$ with high probability.

Using the result that $\|R(Y)\|\le C$ with high probability, and invoking the construction of the contours in Section \ref{sec_mainresults}, we have
$\|\mathcal{G}(z)\|\sim 1$ with high probability. Similarly to Lemma 4.2 of \cite{bao2022spectral}, one can show $|n^{-1}\operatorname{tr}\mathcal{G}(z)|\sim 1$ for any fixed $z\in \gamma_1\lor \gamma_2$. The same results hold for $z\in \gamma_1^0\lor \gamma_2^0$.

\subsection{Proofs for Gaussian divisible model}\label{app_sec_prf_edge}
This section is devoted to the proofs of the averaged and entrywise local laws for GDM. As well illustrated in Section \ref{sec_secondorderconvergence} and the proof strategies sketched in Section \ref{sec_proofstrategy}, we begin with the $\eta$-regularity of $R(\widetilde{H})$ in Section \ref{prf_prop_etaregular_corH}. Additionally, we introduce some basic notation and preliminary estimates for GDM in Section \ref{sec_basicnotion_GDM}. Subsequently, we present an intermediate entrywise local law for $\widetilde{H}$ in Section \ref{app_sec_prf_locallaw_H}. Finally, we derive the averaged and entrywise local laws in Sections \ref{app_sec_prf_locallaw_gdm_average} and \ref{app_sec_prf_locallaw_gdm_entrywise}, respectively.

\subsubsection{Proof of Proposition \ref{prop_etaregular_corH}}\label{prf_prop_etaregular_corH}
To prove Proposition \ref{prop_etaregular_corH}, we first give a simple estimate of $mR_{M}$, which generalizes \cite[Theorem 2.7]{hwang2019local} and \cite[Theorem 2.9]{hwang2019local} from covariance matrix to correlation matrix. Then we utilize the well configured property of the label matrix $\Psi$ to derive the $\eta_*$-regularity of $R(\widetilde{H})$. Recall the definition $m^{(t)}(z):=(1-t)^{-1}m(z/(1-t))$ for  $t=n\mathbb{E}|L_{ij}\rho^{-1}_j|^2$.
\begin{lemma}\label{lem_locallaw_corM}
	Fix $z\in\mathbf{D}$. We have that
	\begin{equation*}
		|mR_{M}(z)-m^{(t)}(z)|\prec n^{-\epsilon_{\alpha}}\eta^{-2}+n^{-\epsilon_{h}}+(n\eta)^{-1}.
	\end{equation*}
\end{lemma}
\begin{proof}
 For the random matrix $M$ with bounded support, it follows from \cite[Theorem 2.7]{hwang2019local} that
    \begin{equation*}
        |mS_{M}(z)-m^{(\widetilde{t})}(z)|\prec n^{-\epsilon_{h}}+(n\eta)^{-1},
    \end{equation*}
    where $\widetilde{t}=1-\mathbb{E}|M_{ij}|^2$. In addition, notice that $|\widetilde{t}-t|=\mathrm{o}(n^{-1})$, which implies
    $|m^{(t)}(z)-m^{(\widetilde{t})}(z)|\le (n\eta)^{-1}$. Also notice that $R(M)=M{(\operatorname{diag}S)}^{-1}M^{*}$ differs from $\mathcal{S}(M)=n^{-1}MM^{*}$ only in the scaling. Thereafter, one may consider the difference
	\begin{equation*}
		|mR_{M}(z)-mS_{M}(z)|=|n^{-1}\operatorname{tr}\mathcal{G}R_{M}(z)-n^{-1}\operatorname{tr}\mathcal{G}S_{M}(z)|.
	\end{equation*}
	  Using the resolvent identity $A^{-1}-B^{-1}=B^{-1}(B-A)A^{-1}$ for any square matrices $A$ and $B$ of the same size, we obtain that  \begin{equation}\label{eq_averagelocallaw_diff_mRmmSm}
		\begin{split}
			&|mR_{M}(z)-mS_{M}(z)|\\
			=&|n^{-1}\operatorname{tr}\Big\{\big(R(M)-zI\big)^{-1}\big(R(M)-\mathcal{S}(M)\big)\big(\mathcal{S}(M)-zI\big)^{-1}\Big\}|\\
			\le& \frac{1}{\eta^2}\operatorname{tr}\Big|\big(n^{-1}M((\operatorname{diag}S)^{-1}-n^{-1}I)M^{*}\big)\Big|\\
			\le& \frac{1}{\eta^2}\operatorname{tr}|(\operatorname{diag}S)^{-1}-n^{-1}I|\cdot\|n^{-1}M^{*}M\|\lesssim n^{-\epsilon_{\alpha}}\eta^{-2},
		\end{split}
	\end{equation}
 where in the last step, we have used \eqref{eq_tr_SS} and the result that
    \begin{equation*}
        |\|n^{-1}M^{*}M\|-(1-\widetilde{t})\lambda_+|\prec n^{-2\epsilon_h}+n^{-2/3}
    \end{equation*}
    by \cite[Theorem 2.9]{hwang2019local}. Thereafter,  \eqref{eq_averagelocallaw_diff_mRmmSm} reads
    \begin{equation*}
           |mR_{M}(z)-mS_{M}(z)|\prec n^{-\epsilon_{\alpha}}\eta^{-2},
    \end{equation*}
    which further implies the desired statement.
\end{proof}

Now, we can proceed to the proof of Proposition \ref{prop_etaregular_corH}.\\
\textit{Proof of Proposition \ref{prop_etaregular_corH}}:
Suppose $\Psi$ and $\Pi$ are well configured and invoke the definition of the domain $\mathbf{D}$. We firstly consider $(i)$ in Definition \ref{def_etaregular}. Let $\mu^{\widetilde{H}}_{n}$ and $\mu_{n}^{M}$ be the empirical spectral distributions of $R(\widetilde{H})$ and $R({M})$, respectively. By the rank inequality, $|\mu_n^{\widetilde{H}}-\mu_{n}^M| \le 2\operatorname{rank}({H}) / n$, we have
\begin{equation*}
\left|\operatorname{Im} mR_{{\widetilde{H}}}(z)-\operatorname{Im} mR_{M}(z)\right| \le \int\left|\frac{\eta}{(\lambda-E)^2+\eta^2}\big(\mu^{\widetilde{H}}_{n}-\mu^{M}_{n}\big)(\mathrm{d} \lambda)\right|\lesssim n^{-\epsilon_y}\eta^{-1},
\end{equation*}
which combined with Lemma \ref{lem_locallaw_corM} gives
\begin{equation*}
\left|\operatorname{Im} mR_{{\widetilde{H}}}(z)-\operatorname{Im} m^{(t)}(z)\right| \prec n^{-\epsilon_{y}}\eta^{-1}+n^{-\epsilon_{\alpha}}\eta^{-2}+n^{-\epsilon_{h}}+(n\eta)^{-1}.
\end{equation*}
Thus, for any fixed $z\in \mathbf{D}$ with $\eta_*\le \eta\le C$, the first statement follows from the fact that
$m^{(t)}(z)\sim \sqrt{\kappa_0+\eta}$ for $\lambda_{-}(R(\widetilde{H}))\le E\le \lambda_{+}(R(\widetilde{H}))$ and $m^{(t)}(z)\sim \eta/\sqrt{\kappa_0+\eta}$ for $E\le \lambda_{-}(R(\widetilde{H}))$ or $ E\ge \lambda_{+}(R(\widetilde{H}))$.

Next, statement $(ii)$ of Definition \ref{def_etaregular} follows from similar arguments in Lemma \ref{lem_pre_estnorm}.

As for statement $(iii)$ in Definition \ref{def_etaregular}, we can easily select another large $C_{\widetilde{H}}$ based on the rank of $\widetilde{H}$, for example $C_{\widetilde{H}}=50$. This concludes the proof of Proposition \ref{prop_etaregular_corH}.   \qed

\subsubsection{Preliminaries for GDM}\label{sec_basicnotion_GDM}
In this subsection, we introduce some basic notation and tools that will be used in the subsequent proof. Firstly, define $\widetilde{H}_t=\widetilde{H}(\operatorname{diag}(S))^{-1/2}/\sqrt{1-t}$ and assume that
    \begin{equation}\label{eq_def_svd_tildeH}
        \widetilde{H}_t=\frac{1}{\sqrt{1-t}}\mathrm{U}_1\mathbf{E}(\widetilde{H})\mathrm{U}_2^{*}
    \end{equation}
    is the singular value decomposition of $\widetilde{H}_t$, where $\mathbf{E}(\widetilde{H})$ is an $n\times p$ rectangular matrix with
    \begin{equation*}
        \mathbf{E}:=\mathbf{E}(\widetilde{H})=(D,0),\quad D^2=\operatorname{diag}(d_1,\dots,d_n),
    \end{equation*}
    with $\sqrt{d_1}\ge\sqrt{d_2}\ge\dots\ge\sqrt{d_n}\ge0$ being the singular values of $\widetilde{H}_t$. Due to the rotationally invariant property of Gaussian matrix, we may write
    \begin{equation}\label{eq_svd_Yt}
        Y_t=\sqrt{t}W+\widetilde{H}_0\overset{d}{=}\mathrm{U}_1Y_t\mathrm{U}_2^{*},\quad \widetilde{Y}_t:=\sqrt{t}W+\mathbf{E}.
    \end{equation}
For any deterministic matrix $A\in\mathbb{C}^{n\times p}$, define its linearisation matrix $L(A)$ as
\begin{equation}\label{eq_def_linearisation}
    L(A):=
    \begin{pmatrix}
        0 &A\\
        A^{*} &0
    \end{pmatrix}.
\end{equation}
Let $\mathcal{L}(A,z)=(z^{1/2}L(A)-z)^{-1}$ be the resolvent of $L(A)$. By the Schur complement formula, we have
\begin{equation}\label{eq_def_resolvent_L}
    \mathcal{L}(A,z)=
    \begin{pmatrix}
        \mathcal{G}(A,z) &z^{-1/2}\mathcal{G}(A,z)A\\
        z^{-1/2}A^{*}\mathcal{G}(A,z) &\mathcal{G}(A^{*},z)
    \end{pmatrix}.
\end{equation}
For the consistency of notation, we remark that
\begin{equation*}
    \mathcal{L}(\widetilde{Y}_t,z)=
    \begin{pmatrix}
        \mathcal{G}R_{\widetilde{Y}_t}(z) &z^{-1/2}\mathcal{G}R_{\widetilde{Y}_t}(z)Y_t\\
        z^{-1/2}\widetilde{Y}_t^{*}\mathcal{G}R_{\widetilde{Y}_t}(z) &GR_{\widetilde{Y}_t}(z)
    \end{pmatrix}
    =\begin{pmatrix}
        \mathcal{G}R_{\widetilde{Y}_t}(z) &z^{-1/2}\widetilde{Y}_tGR_{\widetilde{Y}_t}(z)\\
        z^{-1/2}GR_{\widetilde{Y}_t}(z)\widetilde{Y}^{*}_t &GR_{\widetilde{Y}_t}(z)
    \end{pmatrix},
\end{equation*}
where $\mathcal{G}R_{\widetilde{Y}_t}(z)=(R(\widetilde{Y}_t)-zI)^{-1}$ and $GR_{\widetilde{Y}_t}(z)=(R(\widetilde{Y}_t^*)-zI)^{-1}$. In addition, by the singular value decomposition in \eqref{eq_svd_Yt}, we have
\begin{equation*}
    \mathcal{L}(Y_t,z)\overset{d}{=}\begin{pmatrix}
        \mathrm{U}_1 &0\\
        0 &\mathrm{U}_2
    \end{pmatrix}
    \mathcal{L}(\widetilde{Y}_t,z)
    \begin{pmatrix}
        \mathrm{U}^{*}_1 &0\\
        0 &\mathrm{U}^{*}_2
    \end{pmatrix}.
\end{equation*}

    We define the asymptotic limit of $\mathcal{L}(\widetilde{Y}_t,z)$ as
    \begin{equation*}
        \Pi_{\widetilde{Y}_t}(z):=
        \begin{pmatrix}
            \frac{-(1+\phi tm_{t})}{z(1+\phi tm_{t})(1+t\underline{m}_{t}-\mathbf{E}\mathbf{E}^{*})}\quad &\frac{-z^{-1/2}}{z(1+\phi tm_{t})(1+t\underline{m}_{t}-\mathbf{E}\mathbf{E}^{*})}\mathbf{E}\\
            \mathbf{E}^{*}\frac{-z^{-1/2}}{z(1+\phi tm_{t})(1+t\underline{m}_{t}-\mathbf{E}\mathbf{E}^{*})}\quad &\frac{-(1+ t\underline{m}_{t})}{z(1+\phi tm_{t})(1+t\underline{m}_{t}-\mathbf{E}\mathbf{E}^{*})}
        \end{pmatrix}.
    \end{equation*}
    We will use the following index sets
    \begin{equation*}
        \mathcal{I}_1:=\{1,\dots,n\},\quad\mathcal{I}_2:=\{n+1,\dots, n+p\},\quad\mathcal{I}:=\mathcal{I}_1\bigcup\mathcal{I}_2.
    \end{equation*}
    And in the sequel, we will use the Latin letters $i,j\in\mathcal{I}_1$, Greek letters $\mu,\nu\in\mathcal{I}_2$ and $\mathfrak{a},\mathfrak{b}\in\mathcal{I}$. For simplicity, given a vector $\mathbf{u}\in\mathbb{C}^{\mathcal{I}_{1,2}}$, we always identify it with its natural embedding in $\mathbb{C}^{\mathcal{I}}$. For example, we shall recognize $\mathbf{u}\in\mathbb{C}^{\mathcal{I}_1}$ with $(\mathbf{u}^{*},\mathbf{0}_p^{*})^{*}$. For an $\mathcal{I}\times\mathcal{I}$ matrix, we define the $2\times 2$ minor as
    \begin{equation*}
        A_{[ij]}:=
        \begin{pmatrix}
            A_{ij}\quad &A_{i\bar{j}}\\
            A_{\bar{i}j}\quad &A_{\bar{i}\bar{j}}
        \end{pmatrix},
    \end{equation*}
    where $\bar{i}:=i+n\in\mathcal{I}_2$, $\bar{j}:=j+n\in\mathcal{I}_2$. Moreover, for $\mathfrak{a}\in\mathcal{I}\setminus\{i,\bar{i}\}$, we denote
    \begin{equation*}
        A_{[i]\mathfrak{a}}=
        \begin{pmatrix}
            A_{i\mathfrak{a}}\\
            A_{\bar{i}\mathfrak{a}}
        \end{pmatrix},\quad
        A_{\mathfrak{a}[i]}=(A_{\mathfrak{a}i}, A_{\mathfrak{a}\bar{i}}).
    \end{equation*}
\begin{lemma}[Resolvent]\label{lem_resolvent}
The following identities hold.
\begin{itemize}
    \item [(i)] For $i\in\mathcal{I}_1$, we have
    \begin{eqnarray*}
     &&   (\mathcal{L}_{[ii]}(\widetilde{Y}_t,z))^{-1}=
        \begin{pmatrix}
            -z &z^{1/2}(d_i^{1/2}+\sqrt{t}W_{i\bar{i}})\\
            z^{1/2}(d_i^{1/2}+\sqrt{t}W_{i\bar{i}}) &-z
        \end{pmatrix}\\
     &&\qquad\qquad \qquad\qquad \qquad \qquad \qquad -
        zt\begin{pmatrix}
            (W^{*}\mathcal{L}^{(i)}W)_{ii} &(W^{*}\mathcal{L}^{(i)}W^{*})_{i\bar{i}}\\
            (W\mathcal{L}^{(i)}W)_{\bar{i}i} &(W\mathcal{L}^{(i)}W^{*})_{\bar{i}\bar{i}}
        \end{pmatrix}.
    \end{eqnarray*}
    For $n+1\le \mu\le n+p$, we have
    \begin{equation}
        \frac{1}{\mathcal{L}_{\mu\mu}(\widetilde{Y}_t,z)}=-z-zt\big(W\mathcal{L}^{(\mu)}W^{*}\big)_{\mu\mu}.
    \end{equation}
    \item [(ii)] For $i\neq j\in\mathcal{I}_1$, we have
    \begin{eqnarray*}
        \mathcal{L}_{[ii]}(\widetilde{Y}_t,z)&=&-z^{1/2}\sqrt{t}\mathcal{L}_{[ii]}
        \begin{pmatrix}
            (W^{*}\mathcal{L}^{(i)})_{ij} &(W^{*}\mathcal{L}^{(i)})_{i\bar{j}}\\
            (W\mathcal{L}^{(i)})_{\bar{i}j} &(W\mathcal{L}^{(i)})_{\bar{i}\bar{j}}
        \end{pmatrix}\\
       &=&-z^{1/2}\sqrt{t}
        \begin{pmatrix}
            (\mathcal{L}^{(j)}W)_{ij} &(\mathcal{L}^{(j)}W^{*})_{i\bar{j}}\\
            (\mathcal{L}^{(j)}W)_{\bar{i}j} &(\mathcal{L}^{(j)}W^{*})_{\bar{i}\bar{j}}
        \end{pmatrix}\mathcal{L}_{[jj]}.
    \end{eqnarray*}
    For $n+1\le\mu\neq\nu\le p+n$, we have
    \begin{equation*}
        \mathcal{L}_{\mu\nu}(\widetilde{Y}_t,z)=-z^{1/2}\sqrt{t}\mathcal{L}_{\mu\mu}(W\mathcal{L}^{(\mu)})_{\mu\nu}=-z^{1/2}\sqrt{t}(\mathcal{L}^{(\nu)}W^{*})_{\mu\nu}\mathcal{L}_{\nu\nu}=zt\mathcal{L}_{\mu\mu}\mathcal{L}_{\nu\nu}^{(\mu)}(W\mathcal{L}^{(\mu\nu)}W^{*})_{\mu\nu}.
    \end{equation*}
    \item [(iii)] For $i\in\mathcal{I}_1$ and $n+1\le \mu\le p+n$, we have
    \begin{equation*}
        \mathcal{L}_{[i]\mu}=-z^{1/2}\sqrt{t}\mathcal{L}_{[ii]}
        \begin{pmatrix}
            (W^{*}\mathcal{L}^{(i\bar{i}\mu)})_{i\mu}\\
            (W\mathcal{L}^{(\bar{i}i\mu)})_{i\mu}
        \end{pmatrix}
        -zt\mathcal{L}_{[ii]}\mathcal{L}_{[\mu\mu]}^{(i)}
        \begin{pmatrix}
            (W^{*}\mathcal{L}^{(i\bar{i}\mu)}W^{*})_{i\mu}\\
            (W\mathcal{L}^{(\bar{i}i\mu)}W^{*})_{\bar{i}\mu}
        \end{pmatrix}.
    \end{equation*}
    \item [(vi)] For $\mathfrak{a}\in\mathcal{I}$ and $\mathfrak{b},\mathfrak{c}\in\mathcal{I}\setminus\{\mathfrak{a}\}$,
    \begin{equation*}
\mathcal{L}_{\mathfrak{b}\mathfrak{c}}=\mathcal{L}_{\mathfrak{b}\mathfrak{c}}^{(\mathfrak{a})}+\frac{\mathcal{L}_{\mathfrak{b}\mathfrak{a}}\mathcal{L}_{\mathfrak{a}\mathfrak{c}}}{\mathcal{L}_{\mathfrak{a}\mathfrak{a}}},\quad \frac{1}{\mathcal{L}_{\mathfrak{b}\mathfrak{b}}}=\frac{1}{\mathcal{L}_{\mathfrak{b}\mathfrak{b}}^{(\mathfrak{a})}}-\frac{\mathcal{L}_{\mathfrak{b}\mathfrak{a}}\mathcal{L}_{\mathfrak{a}\mathfrak{b}}}{\mathcal{L}_{\mathfrak{b}\mathfrak{b}}\mathcal{L}_{\mathfrak{b}\mathfrak{b}}^{(\mathfrak{a})}\mathcal{L}_{\mathfrak{a}\mathfrak{a}}}.
    \end{equation*}
    For $i\in\mathcal{I}_1$ and $\mathfrak{a},\mathfrak{b}\in\mathcal{I}\setminus\{i,\bar{i}\}$, we have
    \begin{equation*}
        \mathcal{L}_{\mathfrak{a}\mathfrak{b}}=\mathcal{L}^{(i)}_{\mathfrak{a}\mathfrak{b}}+\mathcal{L}_{\mathfrak{a}[i]}\mathcal{L}_{[ii]}^{-1}\mathcal{L}_{[i]\mathfrak{b}},\quad \mathcal{L}^{-1}_{\mathfrak{a}\mathfrak{a}}=(\mathcal{L}^{(i)}_{\mathfrak{a}\mathfrak{a}})^{-1}-\mathcal{L}^{-1}_{\mathfrak{a}\mathfrak{a}}(\mathcal{L}^{(i)}_{\mathfrak{a}\mathfrak{a}})^{-1}\mathcal{L}_{\mathfrak{a}[i]}\mathcal{L}^{-1}_{[ii]}\mathcal{L}^{-1}_{[i]\mathfrak{a}}.
    \end{equation*}
    where we repeatedly use the notation $(\bar{i})$ for some upper index to denote the element $\bar{i}=i+n\in\mathcal{I}_2$.
\end{itemize}
\end{lemma}

In the sequel, we define the random error,
\begin{equation*}
    \Lambda_t(z):=|m_{n,t}(z)-m_t(z)|=\phi^{-1}|\underline{m}_{n,t}(z)-\underline{m}_{t}(z)|,
\end{equation*}
the random control parameters,
\begin{equation*}
    \Psi_{\Lambda_t}(z):=\sqrt{\frac{\operatorname{Im}m_{t}(z)+\Lambda_t(z)}{n\eta}}+\frac{1}{n\eta},
\end{equation*}
and
\begin{equation*}
    \Lambda^o_{t}:=\max_{i\neq j\in\mathcal{I}_1}\|\mathcal{L}^{-1}_{[ii]}\mathcal{L}_{[ij]}(\mathcal{L}_{[jj]}^{(i)})^{-1}\|+\max_{\mu\neq\nu\ge n+1}\|\mathcal{L}^{-1}_{\mu\mu}\mathcal{L}_{\mu\nu}(\mathcal{L}_{\nu\nu}^{(\mu)})^{-1}\|+\max_{i\in\mathcal{I}_1,\mu\ge n+1}\|\mathcal{L}^{-1}_{[ii]}\mathcal{L}_{[i]\mu}(\mathcal{L}_{\mu\mu}^{(i)})^{-1}\|,
\end{equation*}
which control the size of off-diagonal entries. We introduce the random variable from conditional expectation as
\begin{equation*}
    Z_{a}:=(1-\mathbb{E}_{a})(\mathcal{L}_{aa})^{-1},\quad a\in\mathcal{I},
\end{equation*}
where $\mathbb{E}_{a}[\cdot]:=\mathbb{E}[\cdot|L^{(a)}(\widetilde{Y}_t)]$ is the conditional expectation over the randomness of the $a$-th row and column of $L(\widetilde{Y}_t)$. Here the index $a$ can be regarded as $[i], i\in\mathcal{I}_1$ or $\mu,\mu\in\mathcal{I}_2$. Then by Lemma \ref{lem_resolvent}, we have
\begin{gather*}
    Z_{\mu}=zt(\mathbb{E}_{\mu}-1)(W\mathcal{L}^{(\mu)}W^{*})_{\mu\mu}=zt\sum_{i,j\in\mathcal{I}_1}\mathcal{L}^{(\mu)}_{ij}(\frac{1}{n}\delta_{ij}-w_{i\mu}w_{j\mu}),\\
    Z_{[i]}=\sqrt{zt}
    \begin{pmatrix}
        0 &W_{i\bar{i}}\\
        W_{i\bar{i}} &0
    \end{pmatrix}
    +zt
    \begin{pmatrix}
        \sum_{\mu,\nu\in\mathcal{I}_2}\mathcal{L}_{\mu\nu}^{(i)}(n^{-1}\delta_{\mu\nu}-w_{i\mu}w_{i\nu}) &\sum_{j\in\mathcal{I}_1,\mu\in\mathcal{I}_2}\mathcal{L}_{\mu j}^{(i)}w_{i\mu}w_{j\bar{i}}\\
        \sum_{j\in\mathcal{I}_1,\mu\in\mathcal{I}_2}\mathcal{L}_{j\mu}^{(i)}w_{i\mu}w_{j\bar{i}} &\sum_{j,k\in\mathcal{I}_1}\mathcal{L}_{jk}^{(i)}(n^{-1}\delta_{jk}-w_{j\bar{i}}w_{k\bar{i}})
    \end{pmatrix}.
\end{gather*}

We have the following preliminary bounds, which rely heavily on the large deviation inequality for Gaussian vectors.
\begin{lemma}\label{lem_basic_locallaw_preliminarybound_Z}
    Define the event $\Xi_0:=\{\Lambda_t=\mathrm{O}(1)\}$. For $z\in\mathbf{D}_2$, we have
    \begin{equation*}
        \|Z_{[i]}\|+|Z_{\mu}|\prec t\Psi_{\Lambda_t}+t^{1/2}n^{-1/2},\quad i\in\mathcal{I}_1, \ \mu\ge n+1,
    \end{equation*}
    and
    \begin{equation*}
        \mathbbm{1}(\Xi_0)\Lambda_t^o\prec t\Psi_{\Lambda_t}.
    \end{equation*}
    Moreover, for any constant $c_0>0$, we have
    \begin{equation*}
        \mathbbm{1}(\eta\ge c_0)(\|Z_{[i]}\|+|Z_{\mu}|+\Lambda_t^o)\prec t^{1/2}n^{-1/2},\quad i\in\mathcal{I}_1,\mu\ge n+1.
    \end{equation*}
\end{lemma}
\begin{proof}
    The proof of this lemma is the same as the one in \cite[Lemma 4.7]{ding2022edge}, where we observe that $\max_i\{d_i\}<\infty$ and $\mathbbm{1}(\Xi_0)|m_{n,t}|=\mathrm{O}(1)$. We omit further details here.
\end{proof}

Recalling the notation $\rho_t, m_t(z), \zeta_{t}(z)$ and $b_t(z)$ in Section \ref{sec_locallaw_GDM}, now we present some preliminary properties for $m_t(z)$ and $\zeta_t(z)$ with $z\in \mathbf{D}$, based on the analysis of rectangular free convolution, that will be used in the remainder of this section. Notice that from \eqref{eq_def_LSD2_gdm} and the definition of $\zeta_t$, \eqref{eq_zetarandom} is equivalent to
\begin{equation}\label{eq_def_LSD3_gdm}
    \Phi_t(\zeta_t(z))=z,
\end{equation}
where $ \Phi_t$ is an analytic function on $\mathbb{C}_{+}$ defined as
\begin{equation}\label{eq_psizetarandom}
    \Phi_t(\zeta):=\zeta (1-\phi t m_{n,0}(\zeta))^2+(1-\phi)t(1-\phi t m_{n,0}(\zeta)), ~\zeta\in \mathbb{C}_+.
\end{equation}
Notice that $\zeta_{t}(z)$ and $\Phi_t(\zeta)$ are random variables due to the randomness of $\widetilde{H}(\operatorname{diag}S)^{-1/2}$.
Recalling the definition of $\zeta_t(z)$ in \eqref{eq_def_zetat}, similarly to \cite{bao2023smallest}, we may consider the following domain
\begin{equation}\label{eq_def_zetadomain}
\mathrm{D}_{\zeta}(c_1,c_2,c_3,c_4)=\mathrm{D}_1\cup \mathrm{D}_2\subset \mathbb{C}^+
\end{equation}
for $\zeta_t(z)$ with $z\in {\mathbf{D}}$ where
\begin{equation*}
\begin{split}
	&\mathrm{D}_1(c_1,c_2,c_3):=\{z=E+\mathrm{i}\eta:c_1\le E\le c_2,1/20\le \eta\le c_3\},\\
	&\mathrm{D}_{2}(c_1,c_2,c_3,c_4):=\{z=E+\mathrm{i}\eta:E<c_1 ~\text{or}~E>c_2,n^{-2/3-c_4} t\le \eta\le c_3\}.
	\end{split}
\end{equation*}
It can be verified that there exist appropriate constants $0<c_1<\lambda_{n}(R(\widetilde{H}))<\lambda_{1}(R(\widetilde{H}))<c_2,0<c_3,0<c_4$ such that $\zeta_{t}(z)\in \mathrm{D}_{\zeta}(c_1,c_2,c_3,c_4)$ for $z\in {\mathbf{D}}$ by $t\lesssim n^{-2\epsilon_l}$ and $|m_t(z)|\lesssim 1$.

\

\noindent\textbf{I. Algebraic structure for $\Phi_t(\zeta)$}

Denote the support of $\rho_t$ as $\operatorname{supp}(\rho_t)$ and $\lambda_{\pm,t}$ as the respective left and right-most edges. It has been shown in \cite[Proposition 3]{vallet2012improved} that the support and edges of $\operatorname{supp}(\rho_t)$ can be completely characterized by $\Phi_t$ on $\mathbb{R}$.
\begin{lemma}\label{lem_extremaproposition}
	Fix any $t>0$. The function $\Phi_t(x)$ on $\mathbb{R} \backslash\{0\}$ admits $2 q$ positive local extrema counting multiplicities for some $q \in \mathbb{N}$. The preimages of these extrema are denoted by $\zeta_{1,-}(t)<0<\zeta_{1,+}(t) \le \zeta_{2,-}(t) \le \zeta_{2,+}(t) \le \cdots \le \zeta_{q,-}(t) \le \zeta_{q,+}(t)$, and they belong to the set $\left\{\zeta \in \mathbb{R}: 1-\phi t m_{n,0}\left(\zeta\right)>0\right\}$. Moreover, we have $\lambda_{+, t}=\Phi_t\left(\zeta_{q,+}(t)\right)$ and $\zeta_{q,-}(t)<\lambda_1(R(\widetilde{H}))<\zeta_{q,+}(t)$.
\end{lemma}
Observe that by solving $b_t$ in \eqref{eq_def_zetat}, we have
\begin{equation*}
	b_t=\frac{t(1-\phi)+\sqrt{t^2(1-\phi)^2+4\zeta_t z}}{2z}.
\end{equation*}
Thereafter, we may define a multivariate function $F_t(z,\zeta_t)$ by
\begin{equation}\label{eq_def_F_t}
F_t(z,\zeta_t)=1+\frac{t(1-\phi)-\sqrt{
			t^2(1-\phi)^2+4\zeta_tz}}{2\zeta_t}-\phi tm_t(\zeta_t).
\end{equation}
From \eqref{eq_def_LSD2_gdm}, one may find that the pair $(z,b_t)$ is a solution to \eqref{eq_def_LSD2_gdm} if only if $(z,\zeta_t)$ is a solution to \eqref{eq_def_F_t}. And $\Phi_t(\zeta_t(z))=z$ is the same as $F_t(z,\zeta_t)=0$. From Lemma \ref{lem_extremaproposition}, we have the following consequence on the edges of $\operatorname{supp}(\rho_t)$; see, e.g., \cite[Lemma 3.4]{ding2022edge}.
\begin{lemma}
	Denote $a_{k,\pm}(t):=\Phi_t(\zeta_{k,\pm}(t)), 1\le k\le q$. Then $(a_{k,\pm}(t),\zeta_{k,\pm}(t))$ are real solutions of
	\begin{equation*}
		F_t(z,\zeta)=0,\quad \frac{\partial F_t}{\partial \zeta}(z,\zeta)=0.
	\end{equation*}
\end{lemma}

\

\noindent\textbf{II. Spectral behavior of $\zeta_t(z)$}

Denote $\zeta_{\pm,t}=\zeta_{t}(\lambda_{\pm,t})$ and define the centered factors as
\begin{equation*}
	\xi_t(z):=\zeta_t(z)-\lambda_{+}^{\widetilde{H}}, \quad \xi_{+}(t):=\zeta_{+,t}- \lambda_{+}^{\widetilde{H}}.
\end{equation*}
The following lemma provides a preliminary estimate for the locations of $\zeta_{+,t}$ and $\lambda_{+}^{\widetilde{H}}$ (recall the definition $\lambda_{+}^{\widetilde{H}}:=\lambda_1(R(\widetilde{H}))$ in Lemma \ref{lem_etaregularityconsequence}).
\begin{lemma}\label{lem_preliminaryestimate}
	Suppose $\Psi$ and $\Pi$ are well configured. Then $\xi_{+}(t)\ge 0$ and $\xi_{+}(t)\sim t^2$.
\end{lemma}
\begin{proof}
	By Lemma \ref{lem_extremaproposition}, $\zeta_{+,t}=\zeta_{q,+}(t)\ge \lambda_{1}(R(\widetilde{H}))=\lambda_{+}^{\widetilde{H}}$. Moreover, Lemma \ref{lem_extremaproposition} implies that $\Phi_{t}(\zeta_{+,t})$ is the only local extrema on the interval $[\lambda_+^{\widetilde{H}},\infty)$. Thus, $\Phi_{t}^{\prime}(\zeta_{+,t})=0$, which gives
	\begin{align*}
		(1-\phi t m_{n,0}(\zeta_{+,t}))^2-2\zeta_{+,t}\phi t m_{n,0}^{\prime}(\zeta_{+,t})(1-\phi t m_{n,0}(\zeta_{+,t}))-(1-\phi)\phi t^2m_{n,0}^{\prime}(\zeta_{+,t})=0.
	\end{align*}
	Hence, we have
	\begin{align*}
		\phi t m_{n,0}^{\prime}(\zeta_{+,t})=\frac{(1-\phi t m_{n,0}(\zeta_{+,t}))^2}{2\zeta_{+,t}(1-\phi t m_{n,0}(\zeta_{+,t}))+(1-\phi)t}\sim 1,
	\end{align*}
	by (iv) of Lemma \ref{lem_existenceuniqueness}, which further gives $m_{n,0}^{\prime}(\zeta_{+,t})\sim t^{-1}$. This further gives $ \xi_+(t)=\zeta_{+,t}-\lambda_+^{\widetilde{H}}\sim t^2$
	by Lemma \ref{lem_etaregularityconsequence} as desired.
\end{proof}

Note that $\xi_t(z)$ is a complex variable while $\xi_{+}(t)$ is real. We may further denote $\xi_t(z):=\alpha(z)+\mathrm{i}\beta(z), \; z\in\mathbb{C}_{+}$. The following lemma gives the approximation of $\xi_t(z)$ on the real line, which is analog to \cite[Lemma 3.9]{ding2022edge} and \cite[Lemma 3.10]{ding2022edge}. Define $\kappa:=\min\{|E-\lambda_{+,t}|,|E-\lambda_{-,t}|\}$ for $z=E+\mathrm{i}\eta \in\mathbf{D}$.
\begin{lemma}[$\xi_t(z)$ on real line]\label{lem_estxi_realline}
	Fix $c_1>0$. For $-3c_1/4\le \alpha(E)\le \xi_{+}(t)$ and $c_1/8\le E\le \lambda_{+,t}$, we have
	\begin{equation*}
		|\alpha(E)-\xi_{+}(t)|\sim|E-\lambda_{+,t}|, \quad \beta(E)\sim t|\alpha(E)-\xi_{+}(t)|^{1/2}.
	\end{equation*}
\end{lemma}
\begin{proof}
	Suppose $\xi_t$ is sufficiently close to $\xi_{+}$, say $|\xi_t-\xi_{+}|\le \delta t^2$ for some small $\delta>0$. One may expand $\Psi_t(\xi)$ around $\xi_{+}$ in \eqref{eq_def_LSD3_gdm} to get
	\begin{equation*}
		 E-\lambda_{+,t}=\Psi_t(\xi_t)-\Psi_t(\xi_{+})=\frac{\Psi^{\prime\prime}(\xi_{+})}{2}(\xi_t-\xi_{+})^2+\frac{\Psi^{(3))}(\xi_{+})}{6}(\xi_t-\xi_{+})^3+\mathrm{O}(t^{-6}|\xi_t-\xi_{+}|^4).
	\end{equation*}
	By Lemma \ref{lem_preliminaryestimate}, one may easily check that $\alpha(E)\gtrsim t^2\gg\eta_{*}$. Then, we obtain from Lemma \ref{lem_etaregularityconsequence} that for any fixed $k\ge1$,
	\begin{equation*}
		|t m_0^{(k)}(\zeta_t)|\lesssim t^{-(2k-2)}, \quad |\xi_t-\xi_{+}|\le\delta t^2.
	\end{equation*}
	Moreover, from Lemma \ref{lem_preliminaryestimate}, we have $|t m_0^{(k)}(\zeta_{+,t})|\sim t^{-(2k-2)}$. Then, one can check that $\Phi^{\prime\prime}(\xi_{+})\sim t^{-2}, |\Phi_t^{(k)}(\xi_t)|\lesssim t^{-(2k-2)}$. By the above estimates and inverting the expansion for $(\xi_t-\xi_{+})$, we have
	\begin{equation*}
		 \xi_t-\xi_{+}=\sqrt{\frac{2(E-\lambda_{+,t})}{\Phi^{\prime\prime}_t(\xi_{+})}}\Big(1-\frac{\Phi^{(3)}_t(\xi_{+})}{3\Phi^{\prime\prime}_t(\xi_{+})}(\xi_t-\xi_{+})+\mathrm{O}(t^{-4}|\xi_t-\xi_{+}|^2)\Big).
	\end{equation*}
	Plugging the above equation into $(\xi_t-\xi_{+})$ on the right-hand side again and taking the real and imaginary parts, we obtain that
	\begin{equation*}
		|\alpha(E)-\xi_{+}|\sim|E-\lambda_{+,t}|,\quad \beta(E)\sim t|E-\lambda_{+,t}|^{1/2}\sim t|\alpha(E)-\xi_{+}|^{1/2}.
	\end{equation*}
	
	On the other hand, for the case $\xi_{+}(t)-\alpha(E)\ge\delta t^2$, the estimate $\beta(E)\sim t|\alpha(E)-\xi_{+}(t)|^{1/2}$ follows from an indirect argument. One may find the details in the proof of \cite[Lemma 3.9]{ding2022edge}.  In order to get $|\alpha(E)-\xi_{+}(t)|\sim|E-\lambda_{+,t}|$, it suffices to show that
	\begin{equation*}
		\frac{\mathrm{d}\alpha(E)}{\mathrm{d}E}\ge0,\quad \frac{\mathrm{d}\alpha(E)}{\mathrm{d}E}\sim0.
	\end{equation*}
	Actually, from \eqref{eq_def_LSD3_gdm}, we have
	\begin{equation*}
		 \frac{\mathrm{d}\alpha(E)}{\mathrm{d}E}=\operatorname{Re}\frac{1}{\Phi^{\prime}_t(\zeta_t)}=\frac{\operatorname{Re}\Phi^{\prime}_t(\zeta_t)}{|\Phi^{\prime}_t(\zeta_t)|^2}.
	\end{equation*}
By a direct calculation using Lemma \ref{lem_etaregularityconsequence} and the bound for $\beta(E)$ above, one may find that $\operatorname{Re}\Phi^{\prime}_t(\zeta_t)\sim 1$.
	A direct consequence of the above estimate is that $|\Phi^{\prime}_t(\zeta_t)|\gtrsim 1$. Now we utilize the expression for $\Phi^{\prime}_t(\zeta_t)$ by \eqref{eq_psizetarandom}, together with $(iv)$ of Lemma \ref{lem_existenceuniqueness}, the estimate $|tm^{\prime}_{n,0}(\zeta_t)|\lesssim1$ from Lemma \ref{lem_etaregularityconsequence}, and the bound for $\beta(E)$ above to find that $|\Phi^{\prime}_t(\zeta_t)|\lesssim1$. The details can be found in the proof of \cite[Lemma 3.10]{ding2022edge}. So we can complete the proof here.
\end{proof}

The following lemma is analogous to Lemma \ref{lem_estxi_realline} on the region $\xi_t(z)\in\mathbb{C}_{+}$.
\begin{lemma}[$\xi_t(z)$ on upper complex plane]\label{lem_estxi_complex}
	Suppose $-3c_1/4\le \alpha\le \xi_{+}-\delta t^2$ for fixed $c_1>0$ and constant $\delta>0$. Let $|E-\lambda_{+,t}|\le c_1/2$ and $0\le\eta\le 10$. Then we have
	\begin{equation*}
		\beta(E, \eta)\ge c_{\delta}t|E-\lambda_{+,t}|^{1/2},
	\end{equation*}
	for some constant $c_{\delta}>0$.
\end{lemma}
\begin{proof}
	The proof of the above lemma can be found in \cite[Lemma 3.11]{ding2022edge}, we omit further details here.
\end{proof}

The following lemma borrowed from \cite[Lemma 3.13]{ding2022edge} is useful to establish the local law for GDM.
\begin{lemma}\label{lem_est_xit&xi+}
	If $|z-\lambda_{+,t}|\le\delta t^2$ for some sufficiently small constant $\delta>0$, then we have
	\begin{equation*}
		t\sqrt{|E-\lambda_{+,t}|+\eta}\sim|\xi_t-\xi_{+}|,
	\end{equation*}
	which implies
	\begin{equation*}
		|\Phi^{\prime}_t(\zeta_t)|\sim\min\{1,\frac{\sqrt{|E-\lambda_{+,t}|+\eta}}{t}\}.
	\end{equation*}
\end{lemma}
We note that all the results can be extended to the left part $\lambda_{-}^{\tilde{H}}$ in a similar way, and thus we omit the details.

\

\noindent\textbf{III. Useful estimates.}

The following lemmas are direct results of Lemma \ref{lem_etaregularityconsequence} and the global structure \eqref{eq_def_LSD2_gdm}, whose proofs can be found in \cite[Lemmas 3.15-3.17]{ding2022edge} with Lemma \ref{lem_etaregularityconsequence}, $(iv)$ of Lemma \ref{lem_existenceuniqueness} and \eqref{eq_zetarandom} being the main inputs in our setting. Here we omit further details.

\begin{lemma}\label{lem_prior_locallaw1}
	Fix any constant $C_1>0$. For $z\in\mathbf{D}$ with $\kappa\le C_1t^2$, we have
	\begin{gather*}
		\min_{1\le i\le n}|\lambda_i(R(\widetilde{H}))-\zeta_t|\gtrsim t^2+\eta+t\operatorname{Im}m_t(z);\\
		\int\frac{\mathrm{d}\mu_n^{\widetilde{H}}(x)}{|x-\zeta_t|^a}\lesssim\frac{t+\sqrt{\kappa+\eta}}{(t^2+\operatorname{Im}\zeta_t)^{a-1}},\quad \text{for any fixed $a\ge2$;}\\
		t+\sqrt{\kappa+\eta}\lesssim t+\operatorname{Im}m_t(z).
	\end{gather*}
\end{lemma}
\begin{lemma}\label{lem_prior_locallaw2}
	Fix any constant $C_1>0$. For $z\in\mathbf{D}$ with $\kappa\ge C_1t^2$, we have
	\begin{gather*}
		\min_{1\le i\le n}|\lambda_i(R(\widetilde{H}))-\zeta_t|\gtrsim t^2+\kappa+\eta;\\
		\int\frac{\mathrm{d}\mu_n^{\widetilde{H}}(x)}{|x-\zeta_t|^a}\lesssim\frac{1}{(\kappa+\eta+t^2)^{a-3/2}}, \quad \text{for any fixed $a\ge2$.}
	\end{gather*}
\end{lemma}
\begin{lemma}\label{lem_prior_locallaw3}
	If $\kappa+\eta\le c_1t^2$ for some sufficiently small constant $c_1>0$, then we have
	\begin{equation*}
		|m^{\prime\prime}_t(\zeta_t)|\sim t^{-3}.
	\end{equation*}
\end{lemma}

Moreover, with the $\eta_*$-regularity of $\widetilde{H}$, the Gaussian part will regularize the spectrum of the heavy-tailed part, which is a consequence of Proposition \ref{prop_etaregular_corH} and Lemma \ref{lem_estxi_realline}.
\begin{lemma}\label{lem_squareroot}
	Suppose $\Psi$ and $\Pi$ are well configured. We have, for $z=E+\mathrm{i}\eta\in \mathbf{D}$,
	\begin{equation*}
		\rho_t\sim \sqrt{(E-\lambda_{+,t})_+}~\text{for}~\lambda_{+,t}-3c/4\le E\le \lambda_{+,t}+3c/4,
	\end{equation*}
	and
	\begin{equation*}
		\operatorname{Im} m_{t}(z) \sim
			\sqrt{\kappa+\eta}\mathbbm{1}(z\in \mathbf{D}_1)+
			\frac{\eta}{\sqrt{\kappa+\eta}}\mathbbm{1}(z\in \mathbf{D}_2).
	\end{equation*}
\end{lemma}

\subsubsection{Intermediate entrywise local law for $\widetilde{H}$}\label{app_sec_prf_locallaw_H}

With the aid of notation and preliminary estimates above, recalling the $\eta_*$-regularity of $R(\widetilde{H})$ and the definitions of $\mathrm{T}_r,\mathrm{T}_c$, now we present an intermediate entrywise local law for $R(\widetilde{H})$, which is mostly similar to the heavy-tailed case for sample covariance matrices proved in \cite[Section A.2]{bao2023smallest} and \cite[Section 6]{pillai2014universality}. The difference is that we shall handle the dependence among the entries $\widetilde{h}_{ik}$ due to the self-normalization, which is well-tackled by large deviation result for the self-normalized heavy-tailed part (cf. Lemma \ref{lem_largedeviation_H}). Recall the definitions of $\kappa$ and $\epsilon_{\beta}$.
\begin{lemma}\label{lem_locallaw_H}
	Suppose that $\Psi$ and $\Pi$ are well configured. For any fixed $\zeta\in \mathrm{D}_{\zeta}$ in \eqref{eq_def_zetadomain}, we have
	\begin{equation}\label{eq_bound_locallawH}
		|\mathcal{G}R_{\widetilde{H},ij}(\zeta)-\delta_{ij}m^{(t)}(\zeta)|\prec \left(\frac{\Psi_0(w)}{\sqrt{\kappa+\eta}}+\frac{n^{-\epsilon_{\beta}}}{(\kappa+ \eta)^{3/2}}\right)\mathbbm{1}_{i,j\in \mathrm{T}_r}+\frac{1}{\kappa+\eta}(1-\mathbbm{1}_{i,j\in \mathrm{T}_r}),
	\end{equation}
	\begin{equation}\label{eq_bound_locallawH2}
		|\mathcal{G}R_{\widetilde{H}^*,uv}(\zeta)-\delta_{uv}\underline{m}^{(t)}(\zeta)|\prec \left(\frac{\Psi_0(w)}{\sqrt{\kappa+\eta}}+\frac{n^{-\epsilon_{\beta}}}{(\kappa+ \eta)^{3/2}}\right)\mathbbm{1}_{u,v\in \mathrm{T}_c}+\frac{1}{\kappa+\eta}(1-\mathbbm{1}_{u,v\in \mathrm{T}_c}),
	\end{equation}
	where $w=\zeta/(1-t)$, $\Psi_0(w)=\sqrt{\frac{\mathrm{Im} m(w)}{n\eta}}+(n\eta)^{-1}$, $\underline{m}^{(t)}(\zeta)=\phi m^{(t)}(\zeta)-(1-\phi)/\zeta$.
\end{lemma}

For $z\in \mathbf{D}$, recall the definition of the rescaled matrix $\widetilde{H}_t:=\widetilde{H}(\operatorname{diag}(S))^{-1/2}/\sqrt{1-t}$ and set $w:=\zeta_t(z)/(1-t)$, which is rewritten as $w=E+\mathrm{i}\eta$.
It suffices to show
\begin{equation}\label{eq_bound_GRH}
	|\mathcal{G}R_{\widetilde{H}_t,ij}(w)-\delta_{ij}m(w)|\prec \left(\frac{\Psi_0(w)}{\sqrt{\kappa+\eta}}+\frac{n^{-\epsilon_{\beta}}}{(\kappa+ \eta)^{3/2}}\right)\mathbbm{1}_{i,j\in \mathrm{T}_r}+\frac{1}{\kappa+\eta}(1-\mathbbm{1}_{i,j\in \mathrm{T}_r})
\end{equation}
for any fixed $z\in \mathbf{D}$, where $\epsilon_{\beta}$ is defined in Definition \ref{def_controlparameter}. The case for $\mathcal{G}R_{\widetilde{H}_t^*,ij}(w)$ is analogous, and thus omitted. Recalling $\zeta\in \mathrm{D}_{\zeta}$, it is not hard to verify that the following crude bound holds with high probability,
\begin{equation}\label{eq_crudebound_GRH}
	|\mathcal{G}R_{\widetilde{H}_t,ij}(w)|\le  \frac{1}{\kappa+\eta}.
\end{equation}
Throughout the proof, for simplicity, denote $\widetilde{H}_t:=(\widetilde{h}_{ij})$, $\mathfrak{h}_{i}=(\widetilde{h}_{i1},\ldots,\widetilde{h}_{ip})^*$. Define the parameter
\begin{equation*}
	\varphi:=(\log n)^{\log \log n}.
\end{equation*}
Now we state the proof of \eqref{eq_bound_GRH}, which relies on ``self-consistent equations" below (cf. Section 6 of \cite{pillai2014universality} or Section A.2 of \cite{bao2023smallest}),
\begin{equation}\label{eq_formula_GRHw}
	\mathcal{G}R_{\widetilde{H}_t,ii}(w)=\frac{1}{-w-wn^{-1}\sum_{j=1}^{p}\mathcal{G}R_{(\widetilde{H}_t^{(i)})^*,jj}(w)-U_i(w)},
\end{equation}
where $U_i(w):=w\langle\mathfrak{h}_i,\mathcal{G}R_{(\widetilde{H}_t^{(i)})^{*}}(w)\mathfrak{h}_i\rangle-wn^{-1}\operatorname{tr}\mathcal{G}R_{(\widetilde{H}_t^{(i)})^{*}}(w)$. Define the following quantities:
\begin{equation*}
	\Lambda_{d}(w):=\max_{i\in \mathrm{T}_r}|\mathcal{G}R_{\widetilde{H}_t,ii}(w)-m(w)|,\quad \Lambda_{o}(w):=\max_{i\ne j, i,j\in \mathrm{T}_r}|\mathcal{G}R_{\widetilde{H}_t,ij}(w)|,\quad\Lambda(w):=|m_{\widetilde{H}_t}(w)-m(w)|,
\end{equation*}
where the subscripts refer to the ``diagonal" and ``off-diagonal" matrix elements, respectively. Define the following events:
\begin{equation*}
	\begin{split}
		\Omega_o(w,K):&=\{\Lambda_{o}(w)\ge K\Theta\},
		\Omega_d(w,K):=\{\max_{i\in \mathrm{T}_r}|\mathcal{G}R_{\widetilde{H}_t,ii}(w)-m_{\widetilde{H}_t}(w)|)\ge K\Theta\},\\
		\Omega_U(w,K):&=\{\max_{i\in \mathrm{T}_r}|U_{i}(w)|\ge K\Theta\},
		\Omega(w,K):=\Omega_o(w,K)\cup \Omega_d(w,K)\cup \Omega_U(w,K),\\
		\mathbf{B}(w):&=\{\Lambda_{d}(w)+\Lambda_{o}(w)>(\log n)^{-1}\},
		\mathsf{H}(w,K):=\Omega^{c}(w,K)\cup \mathbf{B}(w),
	\end{split}
\end{equation*}
where $K>0$ and the control parameter $\Theta$ is given by
\begin{equation*}
	\Theta=\Theta(w):=\sqrt{\frac{\operatorname{Im} m(w)+\Lambda(w)}{n\eta}}+\frac{n^{-\epsilon_{\beta}}}{\kappa+\eta}.
\end{equation*}
\begin{lemma}\label{lem_locallaw_H_lemma1}
	Suppose $\Psi$ and $\Pi$ are well configured. There exists a constant $C>0$ such that the event $\bigcap_{\zeta\in \mathrm{D}_{\zeta}}\mathsf{H}(w,\varphi^C)$ holds with high probability.
\end{lemma}
\begin{proof}
	Firstly, noting $\Psi$ and $\Upsilon$ are well configured, by the notation above, it suffices to show that $\Omega_o^{c}(w, K)\cup \mathbf{B}(w)$, $\Omega_{d}^{c}(w, K)\cup \mathbf{B}(w)$ and $\Omega_U^{c}(w, K)\cup \mathbf{B}(w)$ hold with high probability for any $w=w(\zeta)$ with $\zeta\in \mathrm{D}_{\zeta}$ and $K=\varphi^C$, which together with a standard lattice argument give the desired result. In the following, we shall show that the three events hold with high probability.
	
	For the first event $\Omega_o^{c}(w, K)\cup \mathbf{B}(w)$, for $i\ne j, i,j\in \mathrm{T}_r$, we have $|\mathcal{G}R_{\widetilde{H}_t,ij}(w)|\sim 1$ on $\mathbf{B}^c(w)$ since $|m(w)|\sim 1$, which together with the resolvent identity gives
	\begin{equation*}
		 \mathcal{G}R_{\widetilde{H}_t^{(i)},jj}(w)=\mathcal{G}R_{\widetilde{H}_t,jj}(w)-\frac{\mathcal{G}R_{\widetilde{H}_t,ji}(w)\mathcal{G}R_{\widetilde{H}_t,ij}(w)}{\mathcal{G}R_{\widetilde{H}_t,ii}(w)}\sim 1.
	\end{equation*}
	Recall the resolvent identity
	\begin{equation*}
		\mathcal{G}R_{\widetilde{H}_t,ij}(w)=w\mathcal{G}R_{\widetilde{H}_t,ii}(w)\mathcal{G}R_{\widetilde{H}_t^{(i)},jj}(w)\langle \mathfrak{h}_i,\mathcal{G}R_{(\widetilde{H}_t^{(ij)})^*}(w)\mathfrak{h}_j\rangle, ~\text{for}~i\ne j,
	\end{equation*}
	which further gives
	\begin{equation*}
		\Lambda_o(w) \lesssim \max_{i\ne j, i,j\in \mathrm{T}_r}\left|\sum_{1\le k,l\le p}\widetilde{h}_{ik}\mathcal{G}R_{(\widetilde{H}_t^{(ij)})^*,kl}(w)\widetilde{h}_{jl}\right|.
	\end{equation*}
	By applying the large deviation results for $\widetilde{h}_{ik}$ in Lemma \ref{lem_largedeviation_H}, we have
	\begin{equation*}
		\left|\sum_{1\le k,l\le p}\widetilde{h}_{ik}\mathcal{G}R_{(\widetilde{H}_t^{(ij)})^*,kl}(w)\widetilde{h}_{jl}\right|\le \varphi^C\left(n^{-\epsilon_{h}}\max_{k,l}|\mathcal{G}R_{(\widetilde{H}_t^{(ij)})^*,kl}(w)|+\frac{1}{n}\big(\sum_{k,l}|\mathcal{G}R_{(\widetilde{H}_t^{(ij)})^*,kl}(w)|^2\big)^{1/2}\right)
	\end{equation*}
	with high probability. Recall the resolvent identities
	\begin{equation}\label{eq_resolvent_h}
		\begin{split}
			 \sum_{k,l}|\mathcal{G}R_{(\widetilde{H}_t^{(ij)})^*,kl}(w)|^2&=\frac{\sum_{k}\mathrm{Im}\mathcal{G}R_{(\widetilde{H}_t^{(ij)})^*,kk}(w)}{\eta},\\
			\sum_{k}\mathcal{G}R_{(\widetilde{H}_t^{(ij)})^*,kk}(w)&=\sum_{l}\mathcal{G}R_{\widetilde{H}_t^{(ij)},ll}(w)+\frac{n-p+2}{w}.
		\end{split}
	\end{equation}
	Therefore, by \eqref{eq_crudebound_GRH}, \eqref{eq_resolvent_h}, and $\mathrm{Im}(w^{-1})=\eta |w|^{-2}\sim \eta$, we get that on $\mathbf{B}^c$, the following estimate holds with high probability:
	\begin{equation*}
		\Lambda_o(w)\le \varphi^C\left(\frac{n^{-\epsilon_{h}}}{\kappa+\eta}+\sqrt{\frac{\mathrm{Im}~ m(w)+\Lambda(w)+\Lambda_o^2(w)}{n\eta}+\frac{n^{-\epsilon_{y}}}{\kappa^2+\eta^2}+\frac{1}{n}}\right),
	\end{equation*}
	which further implies $\Lambda_o(w)=o(1)$ for $\zeta\in \mathrm{D}_{\zeta}$. Thus we have $ \Lambda_o(w) \le \varphi^C \Theta$ on $\mathbf{B}^c$ with high probability, which implies that the first event $\Omega_o^{c}(w, K)\cup \mathbf{B}(w)$ holds with high probability for $K=\varphi^C$.

	Now we turn to the event $\Omega_U^{c}(w, K)\cup \mathbf{B}(w)$ for any fixed $w=w(\zeta)$ with $\zeta\in \mathrm{D}_{\zeta}$. Noticing the definition of $U_{i}(w)$ for $i\in \mathrm{T}_r$, one can apply the large deviation result in Lemma \ref{lem_largedeviation_H} by the above similar argument above to obtain the desired result, that is $|U_i(w)|\le \varphi^C \Theta$ with high probability for $i\in \mathrm{T}_r$.
	
	It remains to consider the event $\Omega_d^{c}(w, K)\cup \mathbf{B}(w)$. For $i\in \mathrm{T}_r$,
	\begin{equation*}
		\begin{split}
			\max_{i}|\mathcal{G}R_{\widetilde{H}_t,ii}(w)-m_{\widetilde{H}_t}(w)|&\le n^{-1}\sum_{j=1}^{n}\max_{j\ne i} |\mathcal{G}R_{\widetilde{H}_t,ii}(w)-\mathcal{G}R_{\widetilde{H}_t,jj}(w)|\\
			&\le \max_{i\ne j, j\in \mathrm{T}_r} |\mathcal{G}R_{\widetilde{H}_t,ii}(w)-\mathcal{G}R_{\widetilde{H}_t,jj}(w)|+\varphi^{C}n^{-\epsilon_{y}}\frac{1}{\kappa+\eta}.
		\end{split}
	\end{equation*}
	On the event $\mathbf{B}^c$, for $i\ne j$, $j\in \mathrm{T}_r$, by \eqref{eq_formula_GRHw}, we have
	\begin{equation*}
		\begin{split}
			|\mathcal{G}R_{\widetilde{H}_t,ii}(w)-\mathcal{G}R_{\widetilde{H}_t,jj}(w)|
			\lesssim \max_{i\in \mathrm{T}_r}|U_{i}(w)|+\Lambda_o^2(w)+n^{-\epsilon_{y}}\frac{1}{\kappa^2+\eta^2},
		\end{split}
	\end{equation*}
	where the last inequality follows from the resolvent identity
	$\sum_{k}\mathcal{G}R_{(\widetilde{H}_t^{(i)})^{*},kk}(w)=\sum_{l}\mathcal{G}R_{\widetilde{H}_t^{(i)},kk}(w)+(n-p+1)w^{-1}$ and the estimates above. Thus, the event $\Omega_d^{c}(w, K)\cup \mathbf{B}(w)$ holds with high probability. So we can complete the proof of this proposition.
\end{proof}

In the case of $\eta\sim 1$, the following result holds without assuming $\mathbf{B}^c$.
\begin{lemma}\label{lem_locallaw_H_lemma2}
	Suppose $\Psi$ and $\Pi$ are well configured. For $w:=E+\mathrm{i}\eta$ with $\eta>C^{\prime}$ for some constant $C^{\prime}>0$, there exists a constant $C$ such that the event $\Omega^{c}(w,\varphi^C)$ holds with high probability.
\end{lemma}
\begin{proof}
	For $w:=E+\mathrm{i}\eta$ with $\eta>C^{\prime}$, we have $|\mathcal{G}R_{\widetilde{H}_t,ij}(w)|\le \frac{1}{C^{\prime}}$. Similarly to the proof of Lemma \ref{lem_locallaw_H_lemma1}, one can use
	 the large deviation results for $\widetilde{h}_{ik}$ in Lemma \ref{lem_largedeviation_H} and the crude bound for $\eta\sim 1$ to obtain that $\Omega_{o}^c(w,\varphi^C)$ and $\Omega_{U}^c(w,\varphi^C)$ hold with high probability. It remains to consider the event $\Omega^c_{d}(w,\varphi^C)$, which can be handled similarly using the identity
	\begin{equation*}
		\sum_{k}\mathcal{G}R_{(\widetilde{H}_t^{(i)})^{*},kk}(w)-\sum_{k}\mathcal{G}R_{(\widetilde{H}_t^{(j)})^{*},kk}(w)=\mathrm{O}(\eta^{-1})
	\end{equation*}
	due to Cauchy's interlacing theorem of eigenvalues.
\end{proof}

For a function $u(w)$, define the following deviation function:
\begin{equation}\label{eq_def_deviance}
	D(u(w),w):=(u^{-1}(w)+\phi w u(w))-(m^{-1}(w)+\phi w m(w)).
\end{equation}
It is clear that $D(m(w),w)=0$. We aim to show $|D(m_{\widetilde{H}_t}(w),w)|\approx 0$, which further implies $\Lambda(w)=|m_{\widetilde{H}_t}(w)-m(w)|\approx 0$. The following lemma provides a precise bound for the deviation function $D(u(w),w)$ on the event $\mathsf{H}(w,\varphi^C)$, whose proof is the same as that of Lemma A.6 in \cite{bao2023smallest}, thus is omitted.
\begin{lemma}\label{lem_locallaw_H_lemma3}
	On the event $\mathsf{H}(w,\varphi^C)$, we have
	\begin{equation*}
		|D(m_{\widetilde{H}_t}(w),w)|\le \mathrm{O}(\varphi^{2C}\Theta^2)+\infty \mathbbm{1}_{\mathbf{B}(w)}.
	\end{equation*}
\end{lemma}

Recalling $w=\zeta/\sqrt{1-t}$ and $w=E+\mathrm{i}\eta$, we also collect the stability result of $D(m_{\widetilde{H}_t}(w),w)$ as follows, see Section 6 of \cite{pillai2014universality} for more details.
\begin{lemma}\label{lem_locallaw_H_lemma4}
	Let $C,C^{\prime}>0$ be constants, and define the event
	\begin{equation*}
		A\subset \bigcap_{\zeta\in \mathrm{D}_{\zeta}}\mathsf{H}(w,\varphi^C)\cap \bigcap_{\zeta\in \mathrm{D}_{\zeta},\eta=C^{\prime}}\mathbf{B}^{c}(w).
	\end{equation*}
	Suppose that, in $A$, we have $|D(m_{\widetilde{H}_t}(w),w)|\le \delta(w)+\infty\mathbbm{1}_{\mathbf{B}(w)}$, where $\delta:\mathbb{C}\mapsto \mathbb{R}_+$ is a continuous function, decreasing in $\mathrm{Im}w$ and $|\delta(w)|\le (\log n)^{-8}$. Then for all $w=\zeta/\sqrt{1-t}$ with $\zeta\in \mathrm{D}_{\zeta}$, we have
	\begin{equation*}
		|m_{\widetilde{H}_t}(w)-m(w)|\lesssim (\log n) \frac{\delta(w)}{\sqrt{|E-\lambda_+|+\eta+\delta(w)}}~~\text{in}~A,
	\end{equation*}
	and
	\begin{equation*}
		A\subset \bigcap_{\zeta\in \mathrm{D}_{\zeta}}\mathbf{B}^c(w).
	\end{equation*}
\end{lemma}
\begin{proof}
	The proof is mostly similar to Lemma A.7 of \cite{bao2023smallest} and thus we do not reproduce here.
\end{proof}

Thereafter, the proof of \eqref{eq_bound_GRH} follows from similar arguments leading to Proposition A.8 in \cite{bao2023smallest}, so we collect the main steps here for completeness.
In the following, we aim to show the event
\begin{equation*}
	\bigcap_{\zeta\in \mathrm{D}_{\zeta}}\{\Lambda_o(w)+\Lambda_{d}(w)\le \varphi^C \left(\sqrt{\frac{\mathrm{Im}m(w)}{n\eta}}+\frac{n^{-\epsilon_{\beta}}}{\kappa+ \eta}+\frac{1}{n\eta}\right)\frac{1}{\sqrt{k+\eta}}\}
\end{equation*}
holds with high probability. Consider the event
\begin{equation*}
	A=A_0\cap \bigcap_{\zeta\in \mathrm{D}_{\zeta},\eta=C^{\prime}}\mathbf{B}^{c}(w),~\text{where}~A_0=\bigcap_{\zeta\in \mathrm{D}_{\zeta}}\Xi(w,\varphi^C).
\end{equation*}
Thus $A$ holds with high probability by Lemma \ref{lem_locallaw_H_lemma1} and Lemma \ref{lem_locallaw_H_lemma2}. For $\zeta\in \mathrm{D}_{\zeta}$, we have
\begin{equation*}
	\Theta\lesssim \varphi^C \left(\sqrt{\frac{\mathrm{Im} m(w)+1}{n\eta}}+\frac{n^{-\epsilon_{\beta}}}{\kappa+ \eta}\right),
\end{equation*}
which is decreasing in $\eta$. Thereafter, by Lemmas \ref{lem_locallaw_H_lemma3} and \ref{lem_locallaw_H_lemma4}, on the event $A$,
\begin{equation*}
	\Lambda(w)\lesssim \frac{\delta(w)}{\sqrt{|E-\lambda_+|+\eta+\delta(w)}},~\text{with}~\delta(w):=\varphi^{2C} \left(\frac{\mathrm{Im} m(w)+1}{n\eta}+\frac{n^{-2\epsilon_{\beta}}}{\kappa^2+ \eta^2}\right),
\end{equation*}
where $A\subset \bigcap_{\zeta\in \mathrm{D}_{\zeta}}\mathbf{B}^c(w)$. This further implies that $A\in \Omega^c(w,\varphi^C)$ for any $w=w(\zeta)$ with $\zeta\in \mathrm{D}_{\zeta}$. Thus, one can check that
\begin{equation*}
	\Lambda(w)\prec (n\eta)^{-1/4}+\frac{n^{-\epsilon_{\beta}}}{\kappa+\eta}  .
\end{equation*}
Using the ``self-improving'' trick gives $\Lambda(w)\prec n^{-\epsilon_{\beta}}(\kappa+\eta)^{-1}+(n\eta)^{-1}$. Thus, we conclude that the following estimates hold with high probability,
\begin{equation*}
	\Lambda_o(w)\le \varphi^C \left(\sqrt{\frac{\mathrm{Im}m(w)}{n\eta}}+\frac{n^{-\epsilon_{\beta}}}{\kappa+ \eta}+\frac{1}{n\eta}\right)
\end{equation*}
and
\begin{equation*}
	\Lambda_d(w)\le \max_{i\in \mathrm{T}_r}|\mathcal{G}R_{\widetilde{H}_t,ii}(w)-m_{\widetilde{H}_t}(w)|+\Lambda(w)\le \varphi^C \left(\sqrt{\frac{\mathrm{Im}m(w)}{n\eta}}+\frac{n^{-\epsilon_{\beta}}}{\kappa+ \eta}+\frac{1}{n\eta}\right) \frac{1}{\sqrt{\kappa+\eta}}.
\end{equation*}
Hence we have finished the proof of \eqref{eq_bound_GRH}.   \qed

The following large deviation result for the self-normalized heavy-tailed part is an extension of Corollary 25 of \cite{aggarwal2019bulk}, which relies on the improved moment estimates for the self-normalized random variable via the cumulant expansion method (cf. Lemma \ref{lem_oddmoment_est}).
\begin{lemma}[Large deviation for $\widetilde{h}_{ij}$]\label{lem_largedeviation_H}
	For $i,j\in \mathrm{T}_{r}$, let $\{A_k\}_{1\le k\le p}$ and $\{A_{kl}\}_{1\le k,l\le p}$ be sequences of real numbers. Recall $\epsilon_h$ in Definition \ref{def_controlparameter}. There exists a constant $\nu$ such that
\begin{equation}\label{eq_hlargedeviation1}
    \mathbb{P}\left[\left|\sum_{1\le k\le p}\widetilde{h}_{ik}A_{k}\right|\ge \varphi^{C}(n^{-\epsilon_{h}}\max_{k}|A_{k}|+(n^{-1}\sum_{k}|A_{k}|^2)^{1/2})\right]\le \exp (-\nu (\log n)^{\xi}),
\end{equation}
\begin{equation}\label{eq_hlargedeviation2}
    \mathbb{P}\left[\left|\sum_{1\le k\ne l\le p}\widetilde{h}_{ik}A_{kl}\widetilde{h}_{jl}\right|\ge \varphi^{C}(n^{-\epsilon_{h}}\max_{k\ne l}|A_{kl}|+n^{-1}(\sum_{k\ne l}|A_{kl}|^2)^{1/2})\right]\le \exp (-\nu (\log n)^{\xi}),
\end{equation}
and
\begin{equation}\label{eq_hlargedeviation3}
    \mathbb{P}\left[\left|\sum_{1\le k,l\le p}\widetilde{h}_{ik}A_{kl}\widetilde{h}_{jl}\right|\ge \varphi^{C}(n^{-\epsilon_{h}}\max_{k,l}|A_{kl}|+n^{-1}(\sum_{k,l}|A_{kl}|^2)^{1/2})\right]\le \exp (-\nu (\log n)^{\xi})
\end{equation}
for any $2\le \xi\le \log \log n$.
\end{lemma}
\begin{proof}
Following the similar argument leading to Corollary 25 of \cite{aggarwal2019bulk}, for $i,j\in \mathrm{T}_{r}$, we have $t\sim n^{-2\epsilon_{l}}$ and $|\widetilde{h}_{ik}|\lesssim M_{ik}\rho_k^{-1}$. Applying the similar arguments resulting in Lemma \ref{lem_oddmoment_est}, one can get, for $i,j\in \mathrm{T}_r$,
\begin{equation*}
    |\mathbb{E}(\widetilde{h}_{ik})|\lesssim |\mathbb{E}(M_{ik}\rho_k^{-1})|\lesssim \kappa_{1,m}\mathbb{E}(\rho_k^{-1})+\kappa_{2,m}\mathbb{E}(M_{ik}\rho_{k}^{-3})+\kappa_{3,m}\mathbb{E}(\rho_{k}^{-3})\lesssim n^{-1/2-(1/2-\epsilon_h)(\alpha-1)}\lesssim n^{-1-\epsilon_h}
\end{equation*}
by the cumulant expansion method, where we used the simple bounds $\kappa_{1,m}\lesssim n^{(1/2-\epsilon_h)(-\alpha+1)}$ and $\kappa_{3,m}\lesssim n^{(1/2-\epsilon_h)(3-\alpha)_+}$, and the definition of control parameters in Definition \ref{def_controlparameter}. For general integer $2\le q\le \varphi^C$ and constant $C>0$, we have
	\begin{equation*}
		\mathbb{E}(|\widetilde{h}_{ik}|^{q})\lesssim \mathbb{E}(|\widetilde{h}_{ik}|^{2})n^{-(q-2)\epsilon_h} \lesssim C^{q}n^{-1}n^{-(q-2)\epsilon_h}
	\end{equation*}
 by $|M_{ik}|\le n^{1/2-\epsilon_h}$ and the high probability bound $\rho_k^{-1}\lesssim n^{-1/2}$ in \eqref{eq_rhoj_highpro}. Thus, \eqref{eq_hlargedeviation1} and \eqref{eq_hlargedeviation2} follow from Corollary 25 of \cite{aggarwal2019bulk}.

  Now we proceed to \eqref{eq_hlargedeviation3}, which is more involved than \cite{aggarwal2019bulk} due to the dependence in our setting. Define $A_{d}:=\max_{k}|A_{kk}|$ and $A_{o}:=\max_{k\ne l}|A_{kl}|$. Note
	\begin{equation}\label{eq_concentration_hkl0}
		\left|\sum_{k,l}\widetilde{h}_{ik}A_{kl}\widetilde{h}_{jl}\right|\le \left|\sum_{k}\widetilde{h}_{ik}A_{kk}\widetilde{h}_{jk}\right|+\left|\sum_{k\ne l}\widetilde{h}_{ik}A_{kl}\widetilde{h}_{jl}\right|,
	\end{equation}
	which represents the diagonal part and the off-diagonal part, respectively. Notice the second part has been estimated in \eqref{eq_hlargedeviation2}. For the first part, when $i, j \in \mathrm{T}_r$ and $i\ne j$, the random variables $\widetilde{h}_{ik}$ and $\widetilde{h}_{jk}$ are dependent due to the same self-normalization factor $\rho_k$. Invoking Lemma \ref{lem_moment_rates}, we have $|\widetilde{h}_{ik}\widetilde{h}_{jk}|\le C|m_{ik}m_{jk}|\rho_k^{-2}\le Cn^{-2\epsilon_h}$ by \eqref{eq_rhoj_highpro}. Moreover, one can derive that
 \begin{equation*}
     |\mathbb{E}(\widetilde{h}_{ik}\widetilde{h}_{jk})|\lesssim n^{-2}\lesssim n^{-1-2\epsilon_h}
 \end{equation*}
via the cumulant expansion method (cf. \eqref{eq_m1m2estimate}). Thereafter, for the diagonal part, applying the first inequality and noting the random variables $\widetilde{h}_{ik}\widetilde{h}_{jk}$ are independent over $1\le k\le p$, we have
	\begin{equation*}
		\mathbb{E}\left(\left|\sum_{k}\widetilde{h}_{ik}A_{kk}\widetilde{h}_{jk}\right|^{2q}\right)\le (Cq)^{2q}\left(\frac{A_d}{n^{2\epsilon_{h}}}+\left(\frac{1}{n^2}\sum_{k}|A_{kk}|^2\right)^{1/2}\right)^{2q}\le (Cq)^{2q}A_dn^{-2\epsilon_{h}}.
	\end{equation*}
This implies
	\begin{equation}\label{eq_concentration_hkl1}
		\mathbb{P}\left(\left|\sum_{k}\widetilde{h}_{ik}A_{kk}\widetilde{h}_{jk}\right|\ge \varphi^C n^{-2\epsilon_{h}}A_d\right)\le \exp(-\nu (\log n)^{\xi})
	\end{equation}
	by a high moment Markov inequality for $q=\nu (\log n)^{\xi}$. Therefore, combining \eqref{eq_hlargedeviation2}, \eqref{eq_concentration_hkl0} and \eqref{eq_concentration_hkl1} gives the desired result.
\end{proof}

\subsubsection{Proof of Theorem \ref{thm_locallaw_gdm_average}}\label{app_sec_prf_locallaw_gdm_average}
In this subsection, we focus on the proof of the averaged local law. The argument is mostly similar to the proof of \cite[Theorem 2.7]{ding2022edge}. We highlight the main steps and differences here. We divide our proof into two parts according to whether $\kappa\le Ct^2$ or $\kappa>Ct^2$ for some $C>0$.
The following lemma gives the local laws for $\kappa\le Ct^2$.
\begin{lemma}\label{lem_app_averagelocallaw_D1}
	Suppose the assumptions in Theorem \ref{thm_locallaw_gdm_average} hold. Then, \eqref{eq_locallaw_gdm_average_in} holds uniformly on $z\in\mathbf{D}_1$ with $\kappa\le Ct^2$.
\end{lemma}

Before proving Lemma \ref{lem_app_averagelocallaw_D1}, we need the following stability properties of $\Phi_t(\zeta)$,
\begin{lemma}\label{lem_app_averagelocallaw_D1_stable1}
	For any variable $\vartheta\in\mathbb{C}_{+}$ satisfying $ |\vartheta-m_t|\le (\log n)^{-2}(t+\operatorname{Im}m_t(z))$
	and $u\equiv u(\vartheta):=(1+\phi t\vartheta)^2z-(1+\phi t\vartheta)t(1-\phi)$, we have
	\begin{equation*}
		|\Phi_t^{\prime}(u)|\sim\min\{1,\frac{\sqrt{\kappa+\eta}}{t}\},
	\end{equation*}
	and for $k=2,3$,
	\begin{equation*}
		|\Phi_t^{(k)}(u)|\lesssim\frac{t^2+t\sqrt{\kappa+\eta}}{(t^2+\operatorname{Im}\zeta_t)^k}.
	\end{equation*}
	Moreover, if $\kappa+\eta\le c_1t^2$ for some small enough constant $c_1>0$, then we have $|\Phi_t^{(2)}(u)|\sim t^{-2}$.
\end{lemma}
\begin{proof}
	Firstly, one may observe $ |u-\zeta_t|\lesssim t|\vartheta-m_t|\lesssim (\log n)^{-2}\min_i|\lambda_i(R(\widetilde{H}))-\zeta_t|,$
	by Lemma \ref{lem_prior_locallaw1}. Note that we have $\min_i|\lambda_i(R(\widetilde{H}))-\zeta_t|\le \xi_t(z)$. In region $|\kappa+\eta|\le \delta t^2$, one has $|\xi_t-\xi_{+}|\sim t\sqrt{\kappa+\eta}$ and $\xi_{+}\sim t^2$. It indicates that $u$ is close to $\zeta_t$. Then we may deploy the same proof as that of Lemma \ref{lem_est_xit&xi+} to obtain the estimation for the first derivative for $\Phi^{\prime}_t(u)$.
	
	For the second result, as discussed above, we have the prior estimate $|u-\zeta_t|\lesssim t^2$. Using Lemma \ref{lem_prior_locallaw1} and the definition of $m_0(z)$ (notice that it only involves the density of $\mu_n^{\widetilde{H}}$ near the edge), we obtain for any $k\ge 1$,
	\begin{equation*}
		|m^{(k)}_0(u)|\lesssim\frac{t+\sqrt{\kappa+\eta}}{(t^2+\operatorname{Im}\zeta_t)},
	\end{equation*}
	where we replace $u$ with $\zeta_t$ in the denominator via the prior bound of $|u-\zeta_t|$. Then by the definition of $\Phi_t$ in \eqref{eq_psizetarandom}, we can directly get that
	\begin{equation*}
		\Phi^{\prime\prime}_t(u)=-2\phi tm^{\prime\prime}_0(u)u(1-\phi tm_0(u))-4\phi tm^{\prime}_0(u)(1-\phi tm_0(u))+2u\big(\phi tm^{\prime}_0(u)\big)^2-\phi(1-\phi)t^2m^{\prime\prime}_0(u).
	\end{equation*}
	Using the estimates $|m^{(k)}_0(u)|$, we could obtain the second result. Besides, by Lemma \ref{lem_prior_locallaw3} and the above expression of $\Phi^{\prime\prime}_t(u)$, we obtain the third statement.
\end{proof}

In order to prove Lemma \ref{lem_app_averagelocallaw_D1}, it is necessary to analyze the following self-consistent equation, which gives a stable structure of \eqref{eq_zetarandom}. We define
\begin{equation*}
	u_t\equiv u(m_{n,t}):=(1+\phi tm_{n,t}(z))^2z-t(1-\phi)(1+\phi tm_{n,t}(z)),
\end{equation*}
and the event
\begin{equation*}
	\Xi_{1}:=\{\Lambda_t\le \frac{t+\operatorname{Im}m_{t}(z)}{(\log n)^2}\}.
\end{equation*}
Note that under the event $\Xi_1$, the condition in Lemma \ref{lem_app_averagelocallaw_D1_stable1} holds for $u_t$. Besides, by the definition of $\Pi_{\widetilde{Y}_t}(z)$, it is easy to obtain via the Schur complement formula that
\begin{equation*}
	\Pi_{[ii]}(z)=
	\begin{pmatrix}
		-z(1+t\underline{m}_t) &z^{1/2}d_i^{1/2}\\
		z^{1/2}d_i^{1/2} &-z(1+\phi tm_t)
	\end{pmatrix}^{-1}.
\end{equation*}

The following lemma comes from \cite[Lemma 4.10]{ding2022edge}.
\begin{lemma}[Stability of $\Phi_t$]\label{lem_stability_Phi(t)}
	For any constant $\epsilon_1>0$ and $z\in\mathbf{D}$ with $\kappa<Ct^2$, we have
	\begin{equation*}
		|\Phi_t^{\prime}(\zeta_t)(u_t-\zeta_t)+\frac{1}{2}\Phi^{\prime\prime}_t(\zeta_t)(u_t-\zeta_t)^2|\le \frac{\Lambda_t^2}{\log n}\frac{t^2}{t^2+\operatorname{Im} \zeta_t}+n^{\epsilon_1}\frac{t^2+t\sqrt{\kappa+\eta}}{t^2+\operatorname{Im}\zeta_t}t\Psi_{\Lambda_t}+n^{\epsilon_1}\frac{n^{-1}t}{t^2+\operatorname{Im}\zeta_t},
	\end{equation*}
	with high probability on $\Xi_1$. Moreover, we have the finer estimate: for any constant $\epsilon_1>0$,
	\begin{equation*}
		|\Phi_t^{\prime}(\zeta_t)(u_t-\zeta_t)+\frac{1}{2}\Phi^{\prime\prime}_t(\zeta_t)(u_t-\zeta_t)^2|\le \frac{\Lambda_t^2}{\log n}\frac{t^2}{t^2+\operatorname{Im} \zeta_t}+n^{\epsilon_1}t\|[Z]\|+n^{\epsilon_1}\frac{t(t\Psi_{\Lambda_t}^2+n^{-1})}{t^2+\operatorname{Im}\zeta_t},
	\end{equation*}
	with high probability, where $[Z]:=n^{-1}\sum_{i\in\mathcal{I}_1}\Pi_{[ii]}Z_{[i]}\Pi_{[ii]}$.
\end{lemma}
\begin{proof}
	The proof of this lemma is mostly similar to the one in \cite[Lemma 4.10]{ding2022edge}, so we only write down the key procedures. It is essential to obtain the structure of \eqref{eq_def_LSD3_gdm} for $u_t$ and control the error terms. Under the event $\Xi_1$, we firstly observe from Lemma \ref{lem_resolvent} that \begin{equation}\label{eq_stability_Phi_1}
		(\mathcal{L}_{\mu\mu})^{-1}=-z-z\phi tm_t+\varepsilon_{\mu},~(\mathcal{L}_{[ii]})^{-1}=
		\begin{pmatrix}
			-z(1+t\underline{m}_t) &z^{1/2}d_i^{1/2}\\
			z^{1/2}d_i^{1/2} &-z(1+\phi tm_t)
		\end{pmatrix}
		+\varepsilon_{[i]},
	\end{equation}
	where $\varepsilon_{\mu}:=Z_{\mu}+z\phi t(m_{n,t}-m_{n,t}^{(\mu)})+z\phi t(m_t-m_{n,t})$ and $\varepsilon_{[i]}:=Z_{[i]}+z\phi t\Big((m_{n,t}-m_{n,t}^{(i)})+(m_t-m_{n,t})\Big)\begin{pmatrix}
		1 &0\\
		0 &1
	\end{pmatrix}$.
	Observe that $|m_{n,t}-m_{n,t}^{(i)}|\le c(n\eta)^{-1}$ for some constant $c>0$. Then by Lemma \ref{lem_basic_locallaw_preliminarybound_Z}, we obtain that
	\begin{equation}\label{eq_stability_bound_epsiloni&epsilonmu}
		\|\varepsilon_{[i]}\|+|\varepsilon_{\mu}|\lesssim t\Lambda_{t}+\mathrm{O}_{\prec}(t\Psi_{\Lambda_t}+t^{1/2}n^{-1/2}).
	\end{equation}
	Recall that under the event $\Xi_1$, we have the upper bound of $\Lambda_t$, which gives $\|\varepsilon_{[i]}\|/|d_i-\zeta_t|=\mathrm{O}_{\prec}((\log n)^{-2})$. Thereafter, we have $\|\Pi_{[ii]}\varepsilon_{[i]}\|=\mathrm{O}((\log n)^{-2})$.
	Taking the inverse of $(\mathcal{L}_{[ii]})^{-1}$ gives
	\begin{equation}\label{eq_stability_expression_L}
		\mathcal{L}_{[ii]}=\Pi_{[ii]}(1+\mathrm{O}((\log n)^{-2}),
	\end{equation}
	with high probability on the event $\Xi_1$. On the other hand, we have
	\begin{equation*}
		\mathbbm{1}(\Xi_1)|m_{n,t}^{(i)}-m_t|\le \frac{c}{n\eta}+\frac{t+\operatorname{Im}m_t}{(\log n)^2}\le 2\frac{t+\operatorname{Im}m_t}{(\log n)^2}\lesssim (\log n)^{-2}.
	\end{equation*}
	Then similar procedures give the approximation of the minor matrix of $\mathcal{L}_{[jj]}$,
	\begin{equation}\label{eq_stability_expression_L(i)}
		\mathcal{L}_{[jj]}^{(i)}=\Pi_{[jj]}^{(i)}(1+\mathrm{O}((\log n)^{-2}).
	\end{equation}
	
	In addition, we may also write \eqref{eq_stability_Phi_1} as the following empirical version
	\begin{equation*}
		(\mathcal{L}_{[ii]})^{-1}=
		\begin{pmatrix}
			-z(1+t\underline{m}_{n,t}) &z^{1/2}d_i^{1/2}\\
			z^{1/2}d_i^{1/2} &-z(1+\phi tm_{n,t})
		\end{pmatrix}
		+\varepsilon_{n,[i]},
	\end{equation*}
	where
	\begin{equation}\label{eq_stability_bound_varepsilon(n)}
		\varepsilon_{n,[i]}:=Z_{[i]}+z\phi t(m_{n,t}-m_{n,t}^{(i)})
		\begin{pmatrix}
			1 &0\\
			0 &1
		\end{pmatrix}
		\prec t\Psi_{\Lambda_t}+\frac{t^{1/2}}{n^{1/2}}.
	\end{equation}
	Again, taking the matrix inverse of the above equation, we have
	\begin{equation}\label{eq_stability_expression_L&Pi}
		\mathcal{L}_{[ii]}=\Pi_{n,[ii]}-\Pi_{n,[ii]}\varepsilon_{n,[i]}\Pi_{n,[ii]}+\mathrm{O}(\|\Pi_{n,[ii]}\|^3\|\varepsilon_{n,[i]}\|^2),
	\end{equation}
	where
	\begin{equation*}
		\Pi_{n,[ii]}(z):=
		\begin{pmatrix}
			-z(1+t\underline{m}_{n,t}) &z^{1/2}d_i^{1/2}\\
			z^{1/2}d_i^{1/2} &-z(1+\phi tm_{n,t})
		\end{pmatrix}.
	\end{equation*}
	It can be seen that the difference between $\Pi_{n,[ii]}$ and $\Pi_{[ii]}$ lies in the empirical $m_{n,t}$ and the limit $m_t$, so on the event $\Xi_1$ we have
	\begin{equation}\label{eq_stability_difference_Pi(n)&Pi}
		\|\Pi_{n,[ii]}-\Pi_{[ii]}\|\lesssim\frac{t\Lambda_t}{|d_i-\zeta_t|^2}\lesssim\frac{t\Lambda_t}{(t^2+\operatorname{Im}\zeta_t)}, \quad \|\Pi_{n,[ii]}\|\lesssim\|\Pi_{[ii]}\|\lesssim\frac{1}{|d_i-\zeta_i|},
	\end{equation}
	where we have used Lemma \ref{lem_prior_locallaw1} for the denominator.
	
	Besides, from Lemma \ref{lem_resolvent}, we have $\mathcal{L}_{[jj]}=\mathcal{L}_{[jj]}^{(i)}+\mathcal{L}_{[ji]}\mathcal{L}_{[ii]}^{-1}\mathcal{L}_{[ij]}$, which further gives
	\begin{equation*}
		\begin{split}
			 |m_{n,t}-m_{n,t}^{(i)}|\lesssim\frac{1}{n}\sum_{j\in\mathcal{I}_1}\|\mathcal{L}_{[ii]}^{-1}\mathcal{L}_{[ij]}\|\lesssim\frac{1}{n}\|\mathcal{L}_{[ii]}\|+\frac{1}{n}\sum_{j\neq i}(\Lambda_t^o)^2\|\mathcal{L}_{[ii]}\|\mathcal{L}_{[jj]}^{(i)}\|^2\lesssim\frac{1}{t^2+\operatorname{Im}\zeta_t}(n^{-1}+t\Psi_{\Lambda_t}^2),
		\end{split}
	\end{equation*}
	where we have used the bound for $\Lambda_{t}^o$ under the event $\Xi_0$ in Lemma \ref{lem_basic_locallaw_preliminarybound_Z}, \eqref{eq_stability_expression_L} and \eqref{eq_stability_expression_L(i)}.
	
	Taking average over $i$ of \eqref{eq_stability_expression_L&Pi} and using \eqref{eq_stability_bound_varepsilon(n)}, \eqref{eq_stability_difference_Pi(n)&Pi}, Lemma \ref{lem_prior_locallaw1} for $a=3$ give
	\begin{equation*}
		 \frac{1}{n}\sum_{i\in\mathcal{I}_1}\mathcal{L}_{[i]}=\frac{1}{n}\sum_{i\in\mathcal{I}_1}\Pi_{n,[i]}-[Z]+\mathrm{O}_{\prec}\Big(\frac{1}{t^2+\operatorname{Im}\zeta_t}\big[\Lambda_{t}(t\Psi_{\Lambda_t}+\frac{t^{1/2}}{n^{1/2}})+t\Psi_{\Lambda_t}^2+\frac{1}{n}\big]\Big).
	\end{equation*}
	We should notice that the $(1,1)$-th entry of the above expression gives that
	\begin{equation}\label{eq_stability_equation_m(n,t)}
		m_{n,t}=\frac{1}{n}\sum_{i=1}^n\frac{1+\phi tm_{n,t}}{d_i-(1+\phi tm_{n,t})^2z+t(1-\phi)(1+\phi t m_{n,t})}+\mathrm{O}_{\prec}(\mathcal{E}),
	\end{equation}
	where
	\begin{equation*}
		 \mathcal{E}:=\|[Z]\|+\frac{1}{t^2+\operatorname{Im}\zeta_t}\Big(\Lambda_{t}(t\Psi_{\Lambda_t}+\frac{t^{1/2}}{n^{1/2}})+t\Psi_{\Lambda_t}^2+\frac{1}{n}\Big).
	\end{equation*}
	Notice that the right-hand side of \eqref{eq_stability_bound_varepsilon(n)} is exactly the definition of $m_t(u_t)$ with $t=0$. Therefore, we can rewrite \eqref{eq_stability_bound_varepsilon(n)} as $ m_{n,t}=(1+\phi tm_{n,t})m_0(u_t)+\mathrm{O}_{\prec}(\mathcal{E})$.
	Plugging it into the definition of $u_t$, we get $ z=\Phi_t(u_t)+\mathrm{O}_{\prec}(t\mathcal{E})$.
	Through Taylor's expansion, comparing \eqref{eq_stability_equation_m(n,t)} with $z=\Phi_{t}(\zeta_t)$ leads to
	\begin{equation}\label{eq_stability_expansion_Phi}
		|\Phi_t^{\prime}(\zeta_t)(u_t-\zeta_t)+\frac{1}{2}\Phi^{\prime\prime}_t(\zeta_t)(u_t-\zeta_t)^2|\lesssim \frac{t^2+t\sqrt{\kappa+\eta}}{(t^2+\operatorname{Im}\zeta_t)^3}t^3\Lambda_{t}^3+\mathrm{O}_{\prec}(t\mathcal{E}),
	\end{equation}
	where we have used the bounds for $\Phi_{t}^{(k)}(\cdot)$ in Lemma \ref{lem_app_averagelocallaw_D1_stable1}. For the error term, we have
	\begin{equation*}
		t\mathcal{E}\prec t\|[Z]\|+\frac{t}{t^2+\operatorname{Im}\zeta_t}\big(n^{\epsilon}(t\Psi_{\Lambda_t}^2+n^{-1})+n^{-\epsilon}t\Lambda_t^2\big),
	\end{equation*}
	for any constant $\epsilon>0$ since $t^{1/2}\Lambda_t(t^{1/2}\Psi_{\Lambda_t}+n^{-1/2})\le n^{-\epsilon}t\Lambda_t^2+n^{\epsilon}(t^{1/2}\Psi_{\Lambda_t}+n^{-1/2})^2$.
	On the other hand, under the event $\Xi_1$ and noticing the fact $t^2+\operatorname{Im}\zeta_t\gtrsim t^2+\eta+t\operatorname{Im}m_t$, we have
	\begin{equation*}
		\frac{t^2+t\sqrt{\kappa+\eta}}{(t^2+\operatorname{Im}\zeta_t)^3}t^3\Lambda_t^3\lesssim \frac{t^3\Lambda_t^2}{(t^2+\operatorname{Im}\zeta_t)^2}\frac{t+\operatorname{Im}m_t}{(\log n)^2}\lesssim\frac{\Lambda_t^2}{(\log n)^2}\frac{t^2}{t^2+\operatorname{Im}\zeta_t}.
	\end{equation*}
	Combining the above two estimates with \eqref{eq_stability_expansion_Phi}, we conclude the second statement in Lemma \ref{lem_stability_Phi(t)}.
	
	For the first statement, we rewrite
	\begin{equation*}
		[Z]=\frac{1}{n}\sum_{i\in\mathcal{I}_1}\Pi_{[ii]}Z_{[i]}^1\Pi_{[ii]}+\frac{1}{n}\sum_{i\in\mathcal{I}_1}\Pi_{[ii]}Z_{[i]}^2\Pi_{[ii]},
	\end{equation*}
	where $Z_{[i]}^1$ is the first term in the definition of $Z_{[i]}$ while $Z_{[i]}^2$ is the second term. From Lemma \ref{lem_basic_locallaw_preliminarybound_Z} or the proof of \cite[Lemma 4.7]{ding2022edge}, we already have the bounds $\|Z_{[i]}^1\|\prec t^{1/2}n^{-1/2}$ and $\|Z_{[i]}^2\|\prec t\Psi_{\Lambda_t}$. Then by the large deviation bound for Gaussian random vectors and Lemma \ref{lem_prior_locallaw1}, we have
	\begin{equation*}
		\begin{split}
		&\frac{1}{n}\sum_{i\in\mathcal{I}_1}\Pi_{[ii]}Z_{[i]}^1\Pi_{[ii]}\prec n^{-1}\big(\frac{1}{n}\sum_i\frac{t}{|d_i-\zeta_t|^4}\big)^{1/2}\lesssim\frac{1}{n(t^2+\operatorname{Im}\zeta_t)},\\
		 &\frac{1}{n}\sum_{i\in\mathcal{I}_1}\Pi_{[ii]}Z_{[i]}^2\Pi_{[ii]}\prec\frac{1}{n}\sum_i\frac{t\Psi_{\Lambda_t}}{|d_i-\zeta_t|^2}\prec\frac{t^2+t\sqrt{\kappa+\eta}}{t^2+\operatorname{Im}\zeta_t}\Psi_{\Lambda_t}.
		\end{split}
	\end{equation*}
	Noticing $\Psi_{\Lambda_t}\ll t+\sqrt{\kappa+\eta}$ for $\kappa\le Ct^2$ and using the second statement complete the proof of the first statement.
\end{proof}

Combining Lemma \ref{lem_prior_locallaw1}, Lemma \ref{lem_app_averagelocallaw_D1_stable1} and Lemma \ref{lem_stability_Phi(t)}, we can conclude the proof of Lemma \ref{lem_app_averagelocallaw_D1} by similar procedures  in \cite{ding2022edge} via the following weak averaged local law and fluctuation averaging.
\begin{lemma}[Weak averaged local law]\label{lem_app_weakaveragelocallaw}
	There exists a constant $C>0$ such that the following estimates hold for $z\in\mathbf{D}$ with $\kappa\le ct^2$: if $\kappa+\eta\ge Ct^2$, then $\Lambda_{t}\prec \Psi_0(z)=:\sqrt{\frac{\operatorname{Im}m_t}{n\eta}}+\frac{1}{n\eta}$;
	if $\kappa+\eta\le Ct^2$, then $\Lambda_{t}\prec t^{2/3}(n\eta)^{-1/3}$.
\end{lemma}
\begin{lemma}[Fluctuation averaging]\label{lem_app_fluctuationaveraging}
	Suppose that $\Lambda_t\prec\omega$, where $\omega$ is a positive, $n$-dependent deterministic function on $\mathbf{D}$ with $\kappa\le ct^2$ satisfying that
	\begin{equation*}
		\frac{1}{n\eta}\le \omega\le \frac{t+\operatorname{Im}m_t}{(\log n)^2}.
	\end{equation*}
	Then for $z\in\mathbf{D}$ with $\kappa\le ct^2$ we have
	\begin{equation*}
		|[Z]|\prec \frac{1}{n\eta}\frac{\operatorname{Im}m_t+\omega}{t+\operatorname{Im}m_t+\omega}.
	\end{equation*}
\end{lemma}
\begin{proof}
	The proof of the above two lemmas can be found in \cite{ding2022edge}. For the proof of the weak averaged local law, the key input is the self-consistent equation $\Phi_t(\zeta_t)=z$ with the bounds in Lemma \ref{lem_stability_Phi(t)} and deterministic estimates in Lemma \ref{lem_app_averagelocallaw_D1_stable1}. For the proof of the fluctuation averaging, it is a standard procedure to handle Green functions with Gaussian randomness. Finally, the proof of Lemma \ref{lem_app_averagelocallaw_D1} relies on the above lemmas. The details are the same as the ones in Section 4 in \cite{ding2022edge}.
\end{proof}

Next, we establish the averaged local law on the region $\mathbf{D}$ with $\kappa> ct^2$. 
We first have the following result which is parallel to Lemma \ref{lem_app_averagelocallaw_D1_stable1}.
\begin{lemma}\label{lem_app_averagelocallaw_D2_stable1}
	For any variable $\vartheta\in\mathbb{C}_{+}$ satisfying $|\vartheta-m_t|\le t(\log n)^{-2}$
	and $u=(1+\phi t\vartheta)^2z-(1+\phi t\vartheta)t(1-\phi)$, we have
	\begin{equation*}
		|\Phi^{\prime}_t(u)|\sim 1,\quad |\Phi^{(2)}_t(u)|\lesssim\frac{t}{(t^2+\kappa+\eta)^{3/2}}.
	\end{equation*}
\end{lemma}
\begin{proof}
	The proof is similar to that of Lemma \ref{lem_app_averagelocallaw_D1_stable1}. The details can be found in \cite[Lemma 4.14]{ding2022edge}, so we omit them.
\end{proof}
The following lemma is a counterpart of Lemma \ref{lem_stability_Phi(t)}, where we update the definition of $\Xi_1$ with $\kappa>ct^2$ to
\begin{equation*}
	\Xi_1:=\{\Lambda_t\le \frac{t}{(\log n)^2}\}.
\end{equation*}
\begin{lemma}
	For $z\in\mathbf{D}$ with $\kappa>ct^2$, we have for any constant $\epsilon>0$,
	\begin{equation*}
		|\Phi_t^{\prime}(\zeta_t)(u_t-\zeta_t)|\lesssim\Lambda_t^2+n^{\epsilon}t\Psi_{\Lambda_t}+\frac{n^{\epsilon}t}{n(\kappa+\eta)},
	\end{equation*}
	with high probability on the event $\Xi_1$. Moreover, we have the finer estimate: for any constant $\epsilon>0$,
	\begin{equation*}
		 |\Phi^{\prime}_t(\zeta_t)(u_t-\zeta_t)|\lesssim\Lambda_t^2+n^{\epsilon}t\|[Z]\|+\frac{n^{\epsilon}t}{(n\eta)^2\sqrt{\kappa+\eta}}+\frac{n^{\epsilon}t}{n(\kappa+\eta)},
	\end{equation*}
	with high probability.
\end{lemma}
\begin{proof}
	From Lemma \ref{lem_prior_locallaw2},
	the error term in \eqref{eq_stability_equation_m(n,t)} will be updated to
	\begin{equation*}
		 \mathcal{E}=\|[Z]\|+\frac{t}{(t^2+\kappa+\eta)^{3/2}}\Big(\Lambda_t\big(t\Psi_{\Lambda_t}+\frac{t^{1/2}}{n^{1/2}}\big)+t\Psi_{\Lambda_t}^2+\frac{1}{n}\Big).
	\end{equation*}
	Thereafter, a Taylor expansion for $z=\Phi_t(\zeta_t)$ up to the first derivative gives for any constant $\epsilon>0$,
	\begin{equation*}
		|\Phi_t^{\prime}(\zeta_t)(u_t-\zeta_t)|\lesssim\frac{t^3}{(t^2+\kappa+\eta)^{3/2}}\Lambda_t^2+\mathrm{O}_{\prec}(t\mathcal{E})\lesssim \Lambda_t^2+\mathrm{O}_{\prec}\Big(t\|[Z]\|+\frac{n^{\epsilon}t}{(t^2+\kappa+\eta)^{3/2}}\big(t^2\Psi_{\Lambda_t}^2+\frac{t}{n}\big)\Big),
	\end{equation*}
	with high probability. Then by the definition of $\Psi_{\Lambda_t}$, the estimation of $\operatorname{Im}m_t(z)$ in Lemma \ref{lem_squareroot} and $\kappa\ge C_1t^2$, we can conclude the second statement.
	
	For the first statement, we can repeat the argument to bound $[Z]$  in the proof of Lemma \ref{lem_stability_Phi(t)}. Therefore, we can complete the proof.
\end{proof}

The rest of this part is a regular routine to obtain the weak local law and perform fluctuation averaging. One may refer to \cite[Lemmas 4.16-4.17]{ding2022edge} for further details. We omit them here.

\subsubsection{Proof of Theorem \ref{thm_locallaw_gdm_entrywise}}\label{app_sec_prf_locallaw_gdm_entrywise}
Recall the definition of $b_t$ and $\zeta_t$ below \eqref{eq_zetarandom}. By the relationship of $R(Y_t)$ and $R(Y_t^*)$, we have
\begin{equation*}
	\underline{m}_{t}(z)=\phi m_{t}(z)-(1-\phi)/z,\quad \underline{m}^{(t)}(\zeta)=\phi m^{(t)}(\zeta)-(1-\phi)/\zeta.
\end{equation*}
Similarly to \cite[Theorem 2.7]{ding2022edge}, conditional on $\widetilde{H}(\operatorname{diag}S)^{-1/2}$, we have the following result.
\begin{theorem}\label{thm_difference_GYtGH}
	Suppose that the assumptions in Theorem \ref{thm_locallaw_gdm_entrywise} hold. Then the following estimates hold uniformly on $z\in\mathbf{D}$,
	\begin{gather*}
		|\big(\mathcal{G}R_{Y_t}(z)-b_t\mathcal{G}R_{\widetilde{H}}(\zeta_t(z))\big)_{ij}|\prec t^{-3}\Big(\sqrt{\frac{\operatorname{Im} m_{t}}{n\eta}}+\frac{1}{n\eta}\Big)+\frac{t^{-7/2}}{n^{1/2}},\\
		|\big(\mathcal{G}R_{Y_t}(z)-(1+t\underline{m}_{t}(z))\mathcal{G}R_{\widetilde{H}}(\zeta_t(z))\big)_{ij}|\prec t^{-3}\Big(\sqrt{\frac{\operatorname{Im}m_{t}}{n\eta}}+\frac{1}{n\eta}\Big)+\frac{t^{-7/2}}{n^{1/2}},\\
		|\big(\mathcal{G}R_{Y^{*}_t}(z)Y_t-\mathcal{G}R_{\widetilde{H}^{*}}(\zeta_t(z))\widetilde{H}_0\big)_{ij}|\prec t^{-3}\Big(\sqrt{\frac{\operatorname{Im} m_{t}}{n\eta}}+\frac{1}{n\eta}\Big)+\frac{t^{-7/2}}{n^{1/2}},\\
		|\big(Y^{*}_t\mathcal{G}R_{Y_t}(z)-\widetilde{H}^{*}\mathcal{G}R_{\widetilde{H}}(\zeta_t(z))\big)_{ij}|\prec t^{-3}\Big(\sqrt{\frac{\operatorname{Im} m_{t}}{n\eta}}+\frac{1}{n\eta}\Big)+\frac{t^{-7/2}}{n^{1/2}}.
	\end{gather*}
\end{theorem}
\begin{proof}
	The proof is almost the same as the one of \cite[Theorem 2.7]{ding2022edge}, so we only point out the necessary steps and differences.
	By \eqref{eq_svd_Yt}, in order to prove the entrywise local law of $R(Y_t)$, it suffices to prove an anisotropic local law for the resolvent of $R(\widetilde{Y}_t)$.
	Recall the error parameter $\Psi_0(z)=\sqrt{\frac{\operatorname{Im}m_{t}}{n\eta}}+\frac{1}{n\eta}$.
	Similarly to \cite[Theorem 2.10]{bao2023smallest}, we shall prove the following bound
	\begin{equation}\label{eq_prf_locallaw_gdm_anisotropic}
		|\mathbf{u}^{*}\big(\mathcal{G}R_{\widetilde{Y}_t}(z)-\Pi_{Y_t}(z)\big)\mathbf{v}|\prec t^{-3}\Psi_0(z)+\frac{t^{-7/2}}{n^{1/2}},
	\end{equation}
	for any deterministic unit vector $\mathbf{u},\mathbf{v}\in\mathbb{C}^{n+p}$. The following entrywise version of \eqref{eq_prf_locallaw_gdm_anisotropic} is needed,
	\begin{equation}\label{eq_prf_locallaw_gdm_anisotropic_prior1}
		|\big(\mathcal{G}R_{\widetilde{Y}_t}(z)-\Pi_{Y_t}(z)\big)_{\mathfrak{a},\mathfrak{b}}|\prec t^{-3}\Psi_0(z)+\frac{t^{-7/2}}{n^{1/2}},
	\end{equation}
	for any $\mathfrak{a},\mathfrak{b}\in\mathcal{I}$. The proof for \eqref{eq_prf_locallaw_gdm_anisotropic_prior1} is based on the observation in the proof of Lemma \ref{lem_stability_Phi(t)},
	\begin{equation*}
		(\mathcal{L}_{\mu\mu})^{-1}=-z-z\phi tm_t+\epsilon_{\mu},~(\mathcal{L}_{[ii]})^{-1}=\Pi_{[ii]}^{-1}+\epsilon_{[i]}.
	\end{equation*}
	Recall the bound in \eqref{eq_stability_bound_epsiloni&epsilonmu} that $\|\epsilon_{[i]}\|+|\epsilon_{\mu}|\prec t\Psi_0(z)+t^{1/2}n^{-1/2}$, where at this time $\Lambda_t\prec \Psi_0(z)$. Moreover, by Lemma \ref{lem_prior_locallaw1}, we have $\min_i|d_i-\zeta_t|\gtrsim t^2$. It follows that
	\begin{gather*}
		\mathcal{L}_{\mu\mu}=\Pi_{\mu\mu}+\mathrm{O}_{\prec}(t^{-3}+t^{-7/2}n^{-1/2}),~
		\mathcal{L}_{[ii]}=\Pi_{[ii]}(1+\mathrm{O}_{\prec}(t^{-3}+t^{-7/2}n^{-1/2})).
	\end{gather*}
	This gives the diagonal bounds for $|\mathcal{G}R_{\widetilde{Y}_t}(z)-\Pi_{\widetilde{Y}_t}(z)|$. For the off-diagonal entries, combining the above estimates and the bound for $\Lambda_t^o$ in Lemma \ref{lem_basic_locallaw_preliminarybound_Z} with $\Psi_{\Lambda_t}(z)$ being updated to $\Psi_0(z)$, we have the desired results. Therefore, we obtain \eqref{eq_prf_locallaw_gdm_anisotropic_prior1}.
	
	Then for general $\mathbf{u}$ and $\mathbf{v}$, due to the polarization identity $\mathbf{u}^{*}A\mathbf{v}=\frac{1}{2}(\mathbf{u}+\mathbf{v})^{*}A(\mathbf{u}+\mathbf{v})-\frac{1}{2}(\mathbf{u}-\mathbf{v})^{*}A(\mathbf{u}-\mathbf{v})$, it suffices to consider the case $\mathbf{u}=\mathbf{v}.$ Recalling the notation $u_{[i]}=(u_i,u_{\bar{i}})^{*}$, we have
	\begin{equation*}
		\begin{split}
			|\mathbf{u}^{*}(\mathcal{G}R_{\widetilde{Y}_t}(z)-\Pi_{\widetilde{Y}_t}(z))\mathbf{v}|\prec t^{-3}&\Psi_0(z)+\frac{t^{-7/2}}{n^{1/2}}+|\sum_{i\neq j\in\mathcal{I}_1}u_{[i]}^{*}\mathcal{G}R_{\widetilde{Y}_t,[ij]}u_{[j]}|\\
			&+|\sum_{\mu\neq \nu\ge 2n+1}u_{\mu}\mathcal{G}R_{\widetilde{Y}_t,\mu\nu}u_{\nu}|+2|\sum_{i\in\mathcal{I}_1,\mu\ge 2n+1}u_{[i]}^{*}\mathcal{G}R_{\widetilde{Y}_i,[i]\mu}u_{\mu}|.
		\end{split}
	\end{equation*}
	Therefore, it suffices to prove the following high moment bounds for any fixed $a\in\mathbb{N}$,
	\begin{gather*}
		\mathbb{E}|\sum_{i\neq j\in\mathcal{I}_1}u_{[i]}^{*}\mathcal{G}R_{\widetilde{Y}_t,[ij]}u_{[j]}|^{2a}\prec (t^{-3}\Psi_0(z)+\frac{t^{-7/2}}{n^{1/2}})^{2a},
	\end{gather*}
	as well as the same bound for the remaining two terms. The proof for the above estimates is based on a polynomialization method developed in \cite{bao2023smallest}, which involves combinatorial results. The key point in the proof is that $\max_i|d_i|=\mathrm{O}(1)$ which can be ensured by Proposition \ref{prop_etaregular_corH} for any realization of $\widetilde{H}_0$.
\end{proof}

By Lemma \ref{lem_locallaw_H} and the large deviation results for $\widetilde{h}_{ik}$ in Lemma \ref{lem_largedeviation_H}, we can get the following estimates for entries of $\mathcal{G}R_{\widetilde{H}}(\zeta)\widetilde{H}$ indexed by $\mathrm{T}_r$ and $\mathrm{T}_c$.
\begin{proposition}\label{prop_largedeviation_GHH}
	Let $\zeta=E+\mathrm{i}\eta\in \mathbb{C}_+$. For $i \in \mathrm{T}_r$, we have
	\begin{equation*}
		\begin{split}
			|[\mathcal{G}R_{\widetilde{H}}(\zeta)\widetilde{H}]_{ij}|&\prec \left(n^{-\epsilon_{h}}\max_{k}|\mathcal{G}R_{(\widetilde{H}^{(i)})^*,kj}(\zeta)|+\left(\frac{\mathrm{Im}\mathcal{G}R_{(\widetilde{H}^{(i)})^*,jj}(\zeta)}{n\eta}\right)^{1/2}\right)\\
			&\times \left(1+|\zeta|\cdot|\mathcal{G}R_{\widetilde{H},ii}(\zeta)|\cdot \Big(n^{-\epsilon_{h}}\max_{k,l}|\mathcal{G}R_{(\widetilde{H}^{(i)})^*,kl}(\zeta)|+\big(\frac{\sum_{k}\mathrm{Im}\mathcal{G}R_{(\widetilde{H}^{(i)})^*,kk}(\zeta)}{n^2\eta}\big)^{1/2}\Big)\right)
		\end{split}
	\end{equation*}
	For $j\in \mathrm{T}_c$, we have
	\begin{equation*}
		\begin{split}
			|[\mathcal{G}R_{\widetilde{H}}(\zeta)\widetilde{H}]_{ij}|&\prec \left(n^{-\epsilon_{h}}\max_{k}|\mathcal{G}R_{(\widetilde{H}^{(j)})^*,ik}(\zeta)|+\left(\frac{\mathrm{Im}\mathcal{G}R_{(\widetilde{H}^{(j)})^*,ii}(\zeta)}{n\eta}\right)^{1/2}\right)\\
			&\times \left(1+|\zeta|\cdot|\mathcal{G}R_{\widetilde{H}^*,jj}(\zeta)|\cdot \Big(n^{-\epsilon_{h}}\max_{k,l}|\mathcal{G}R_{(\widetilde{H}^{(j)})^*,kl}(\zeta)|+\big(\frac{\sum_{k}\mathrm{Im}\mathcal{G}R_{(\widetilde{H}^{(j)})^*,kk}(\zeta)}{n^2\eta}\big)^{1/2}\Big)\right)
		\end{split}
	\end{equation*}
\end{proposition}
\begin{proof}
	The proof follows from similar arguments leading to Proposition A.12 of \cite{bao2023smallest} with the large deviation results for $\widetilde{h}_{ik}$ in Lemma \ref{lem_largedeviation_H},  thus we do not reproduce here.
\end{proof}

The following lemma is a direct consequence of Lemma \ref{lem_locallaw_H} and Proposition \ref{prop_largedeviation_GHH}.
\begin{lemma}\label{lem_locallaw_H_GHH}
	Suppose that the assumptions in Lemma \ref{lem_locallaw_H} hold. The following estimate holds with respect to the probability measure $\mathbb{P}_{\Psi}$,
	\begin{equation*}
		|[\mathcal{G}R_{\widetilde{H}}(\zeta)\widetilde{H}]_{ij}|\prec \left(\sqrt{\frac{\mathrm{Im}m(w)}{n\eta}}+\frac{n^{-\epsilon_{h}}}{\kappa+ \eta}\right)\mathbbm{1}_{i\in \mathrm{T}_r~\text{or}~j\in \mathrm{T}_c}+\frac{1}{\kappa+\eta}(1-\mathbbm{1}_{i\in \mathrm{T}_r~\text{or}~j\in \mathrm{T}_c}).
	\end{equation*}
\end{lemma}

By Lemma \ref{lem_largedeviation_H} and Proposition \ref{prop_largedeviation_GHH}, we get the following estimates, which provide improved bounds for the off-diagonal entries of the Green function,
\begin{lemma}\label{lem_improved_locallaw_H}
	Suppose that the assumptions in Lemma \ref{lem_locallaw_H} hold. The following estimates hold with respect to the probability measure $\mathbb{P}_{\Psi}$.
	\begin{gather*}
		|\mathcal{G}R_{\widetilde{H},ij}(\zeta)|\prec \frac{n^{-\epsilon_{h}}}{\kappa^2+\eta^2}\mathbbm{1}_{i\in \mathrm{T}_r~\text{or}~j\in \mathrm{T}_r}+\frac{1}{\kappa+\eta}(1-\mathbbm{1}_{i\in \mathrm{T}_r~\text{or}~j\in \mathrm{T}_r})~\text{for}~i\ne j,\\
		|\mathcal{G}R_{\widetilde{H}^*,uv}(\zeta)|\prec \frac{n^{-\epsilon_{h}}}{\kappa^2+\eta^2}\mathbbm{1}_{u\in \mathrm{T}_c~\text{or}~v\in \mathrm{T}_c}+\frac{1}{\kappa+\eta}(1-\mathbbm{1}_{u\in \mathrm{T}_c~\text{or}~v\in \mathrm{T}_c})~\text{for}~u\ne v.
	\end{gather*}
\end{lemma}
\begin{proof}
	The crude bounds follow from \eqref{eq_crudebound_GRH}. We only show the estimate of $\mathcal{G}R_{\widetilde{H},ij}(\zeta)$ for $i\ne j$ and $i\in \mathrm{T}_r$, since other cases can be handled analogously. By the resolvent identity, we have
	\begin{equation*}
		|\mathcal{G}R_{\widetilde{H},ij}(\zeta)|=|\mathcal{G}R_{\widetilde{H},ii}(\zeta)|\cdot \left|\sum_{k,l}\widetilde{h}_{ik}\mathcal{G}R_{(\widetilde{H}^{(i)})^*,kl}(\zeta)\widetilde{h}_{jl}\right|.
	\end{equation*}
	The estimate of the second term is analogous to Lemma \ref{lem_locallaw_H_lemma1}, which relies on the representation
	\begin{equation*}
		 \sum_{l}\mathcal{G}R_{(\widetilde{H}^{(i)})^*,kl}(\zeta)\widetilde{h}_{jl}=[\mathcal{G}R_{(\widetilde{H}^{(i)})^*}(\zeta)(\widetilde{H}^{(i)})^*]_{kj},
	\end{equation*}
	and the large deviation results for $\widetilde{h}_{ik}$ in Lemma \ref{lem_largedeviation_H}. It is elementary to show that, for $\zeta\in \mathrm{D}_{\zeta}$,
	\begin{equation*}
		\max_{k,j}|[\mathcal{G}R_{(\widetilde{H}^{(i)})^*}(\zeta)(\widetilde{H}^{(i)})^*]_{kj}|\le \|\mathcal{G}((\widetilde{H}^{(i)})^*,\zeta)((\widetilde{H}^{(i)})^*)\| \le \max_{1\le k\le p}|d_k/(d_k^2-\zeta)|\lesssim \frac{1}{\kappa+\eta}
	\end{equation*}
	and
	\begin{equation*}
		\begin{split}
			\sum_{k}[\mathcal{G}R_{(\widetilde{H}^{(i)})^*}(\zeta)(\widetilde{H}^{(i)})^*]_{kj}^2 &\le \max_{1\le k\le p}\operatorname{diag}(\widetilde{H}^{(i)}\mathcal{G}R_{(\widetilde{H}^{(i)})^*}(\zeta)\mathcal{G}R_{(\widetilde{H}^{(i)})^*}(\zeta)(\widetilde{H}^{(i)})^*)\\
			&\le \|\mathcal{G}R_{(\widetilde{H}^{(i)})^*}(\zeta)(\widetilde{H}^{(i)})^*\|^2
			\le \max_{1\le k\le p}\left|\frac{d_k}{d_k^2-\zeta}\right|^2\lesssim \frac{1}{\kappa^{2}+\eta^2} ,
		\end{split}
	\end{equation*}
	where $d_1\ge d_2\ge \cdots\ge d_p$ are the singular values of $\widetilde{H}^{(i)}$.
	Specifically, we get
	\begin{equation*}
		 \begin{split}&\left|\sum_{k,l}\widetilde{h}_{ik}\mathcal{G}R_{(\widetilde{H}^{(i)})^*,kl}(\zeta)\widetilde{h}_{jl}\right|=\left|\sum_{k}\widetilde{h}_{ik}[\mathcal{G}R_{(\widetilde{H}^{(i)})^*}(\zeta)(\widetilde{H}^{(i)})^*]_{kj}\right|
			\lesssim n^{-\epsilon_{h}}(\kappa+\eta)^{-1},
		\end{split}
	\end{equation*}
	by Lemma \ref{lem_largedeviation_H}, $t\lesssim n^{-2\epsilon_{l}}$ and the above estimates, which further implies the desired results.
\end{proof}
Now we proceed to the proof of Theorem \ref{thm_locallaw_gdm_entrywise}, which relies on Theorem \ref{thm_difference_GYtGH}, Lemma \ref{lem_locallaw_H_GHH} and Lemma \ref{lem_improved_locallaw_H}.
For $z\in \mathbf{D}$, we have $\zeta_{t}(z)\in \mathrm{D}_{\zeta}$ with high probability, which further together with Lemma \ref{lem_locallaw_H}
and Theorem \ref{thm_difference_GYtGH} gives
\begin{equation*}
	|\mathcal{G}R_{Y_t,ij}(z)-\delta_{ij}(1+\phi t m_t(z))m^{(t)}(\zeta_t)|\prec \big(t^{-3}\Psi_0(z)+\frac{n^{-{\epsilon_{\beta}}}}{\kappa^2+\eta^2}\big)\mathbbm{1}_{i\in \mathrm{T}_r~\text{or}~j\in \mathrm{T}_r}+\frac{1}{\kappa+\eta}(1-\mathbbm{1}_{i\in \mathrm{T}_r~\text{or}~j\in \mathrm{T}_r}),
\end{equation*}
by $t\lesssim n^{-2\epsilon_{l}}$ and $b_t=\mathrm{O}(1)$ for $z\in \mathbf{D}$. The other estimates follow from similar arguments leading to Lemma \ref{lem_locallaw_H} and Theorem \ref{thm_difference_GYtGH}. So we can complete the proof of Theorem \ref{thm_locallaw_gdm_entrywise}. \qed

\subsection{Proofs for Green function comparisons}\label{app_sec_Greencom}
This subsection aims to prove the Green function comparisons from GDM $R(Y_t)$ to $R$. The general scheme of proof is analogous to the previous work (cf. \cite{bao2023smallest}) except for the difficulties due to dependence and non-symmetry in our setting. We highlight the main techniques to tackle these issues as follows,
\begin{itemize}
    \item[(i)] Resolvent expansion by columns: The
    dependence due to the factor $(\operatorname{diag}S)^{-1/2}$ makes the Taylor expansion uninformative and the moment matching procedure ineffective. We adopt the resolvent expansion of Green functions to each column $Y_j$ (cf. \eqref{eq_prf_Greenfunction_resolvent_expansion}), which makes the expectation w.r.t. individuals accessible.
    \item[(ii)] Assemble patterns classification: The loss of symmetry leads to various terms in the Taylor expansion and resolvent expansion, which can not be estimated directly. Thanks to Lemma \ref{lem_oddmoment_est}, with the aid of decomposition (cf. \eqref{eq_decomp_X}) and proper assemble patterns classification (cf. Definition \ref{def_prf_Greenfunctioncomparison_encode_beta}), we finally arrive at the desired bounds via a careful analysis of the different odd-moment cases.
\end{itemize}
The rest of this subsection is arranged as follows: Section \ref{sec_Green_pre} introduces preliminaries for Green function comparison, including some additional notation, and formal descriptions of the resolvent expansion by columns, as well as the assemble patterns. The subsequent three subsections \ref{app_prf_thm_greenfuncomp_entrywise_differror}, \ref{app_prf_thm_greenfuncomp_entrywise_uniformbound} and \ref{app_prf_thm_greenfuncomp_average} present the proofs of Theorems \ref{thm_greenfuncomp_entrywise_differror}, \ref{thm_greenfuncomp_entrywise_uniformbound} and \ref{thm_greenfuncomp_average}, respectively.

\subsubsection{Preliminaries for Green function comparison}\label{sec_Green_pre}

In this section, we prepare several tools and notation to simplify the description of the Green function comparison technique. We divide this section into three parts: Matrix replacement notation, Expansion notation, and Assemble patterns classification. We begin by introducing the matrix replacement notation that will be used in the Green function comparison.

\

\textbf{I. Matrix replacement notation.}

We denote the standard basis vector  as $\mathbf{e}_{j}\in\mathbb{R}^n$, with its $j$-th element being one while other elements being zero. For any matrix $A\in\mathbb{R}^{n\times p}$ and vector $\varpi\in\mathbb{R}^{n}$, we define the column replacement matrix $A^{\varpi}_{(j)}\equiv A_{(j)}(\varpi)\in\mathbb{R}^{n\times p}$ as
\begin{equation*}
	[A_{(j)}(\varpi)]_{ab}:=
	\begin{cases}
		\varpi_{a1}\quad \text{if $j=b$}\\
		A_{ab}\quad \text{if $j\neq b$}
	\end{cases},
	\quad a\in[n],\; b\in[p].
\end{equation*}
That is, we replace the $j$-th column in $A$ with $\varpi$.
Define the interpolation matrix $Y^{\gamma}$ by $Y_{ij}^\gamma=ly_{ij}\cdot\mathbbm{1}(\psi_{ij}=0)+hy_{ij}\cdot\mathbbm{1}(\psi_{ij}=1)$, where,
\begin{equation}\label{eq_def_ly&hy}
	\begin{split}
	ly_{ij}\equiv ly_{ij}(\gamma,w_{ij}):=\gamma(1-\chi_{ij})l_{ij}\cdot\rho^{-1}_j+\chi_{ij}m_{ij}\cdot\rho^{-1}_j+(1-\gamma^2)^{1/2}t^{1/2}w_{ij},\\
	hy_{ij}\equiv hy_{ij}(\gamma,w_{ij}):=h_{ij}\cdot\rho^{-1}_j+(1-\gamma^2)^{1/2}t^{1/2}w_{ij}, \quad i\in[n],\; j\in[p].
	\end{split}
\end{equation}
Thus, $Y^{\gamma}$ serves as an interpolation between the original matrix $Y$ (at $\gamma=1$) and the GDM matrix $Y_t$ (at $\gamma=0$). Let $Y_j^{\gamma}$ be the $j$-th column of $Y^{\gamma}$ and decompose $Y_j^{\gamma}$ into the light-tailed ensemble $\mathfrak{ly}_j$ and heavy-tailed ensemble $\mathfrak{hy}_j$ according to the label matrix $\Psi$. Specifically, $Y_j^\gamma=\mathfrak{ly}_j+\mathfrak{hy}_j$, where $\mathfrak{ly}_j:= (ly_{1j}\cdot\mathbbm{1}(\psi_{1j}=0),\dots,ly_{nj}\cdot\mathbbm{1}(\psi_{nj}=0))^{*}$ and $\mathfrak{hy}_j:=(hy_{1j}\cdot\mathbbm{1}(\psi_{1j}=1),\dots,hy_{nj}\cdot\mathbbm{1}(\psi_{nj}=1))^{*}$.
Moreover, we denote the resolvent of sample correlation matrix $R(Y^{\gamma,\varpi}_{(j)})$ with $Y^{\gamma,\varpi}_{(j)}=(Y^{\gamma})_{(j)}(\varpi)$ as $\mathcal{G}^{\gamma,\varpi}_{(j)}(z):=(R(Y^{\gamma,\varpi}_{(j)})-zI)^{-1}$.

With the decomposition of $Y_j^{\gamma}$ above, now we turn to conduct several expansions around the Green function $\mathcal{G}_{(j)}^{\gamma,\varpi}(z)$.

\

\textbf{II. Expansion notation.}

As demonstrated in previous works on Green function comparison, the analysis ultimately reduces to estimating the expectation of certain functions associated with the Green functions (cf. \cite[Section 5]{bao2023smallest}). To achieve a precise estimation of this expectation,  we introduce the following resolvent expansion with respect to each column. Recall the linearisation matrix $L(A)$ in \eqref{eq_def_linearisation} with its resolvent $\mathcal{L}(A,z)$. Then for any matrix $B=A+\Delta\in\mathbb{R}^{n\times p}$, we have for any fixed integer $s\ge 0$,
\begin{equation}\label{eq_resolvent_perturbation_expansion}
	 \mathcal{L}(A,z)=\sum_{j=0}^s\big(\mathcal{L}(B,z)L(z^{1/2}\Delta)\big)^j\mathcal{L}(B,z)+\big(\mathcal{L}(B,z)L(z^{1/2}\Delta)\big)^{s+1}\mathcal{L}(A,z).
\end{equation}
In the current situation, recall $Y^{\gamma}=Y_{(j)}^{\gamma,0}+Y^{\gamma}_j\mathbf{e}_j^{*}$.
By \eqref{eq_resolvent_perturbation_expansion}, we have,
\begin{equation}\label{eq_prf_Greenfunction_resolvent_expansion}
	\begin{split}
		 \mathcal{G}^{\gamma}=&\sum_{k=0}^{s}\mathcal{G}_{(j)}^{\gamma,0}(-Y^{\gamma}_j(Y^{\gamma}_j)^{*}\mathcal{G}_{(j)}^{\gamma,0})^k+\mathcal{G}^{\gamma}(-Y^{\gamma}_j(Y^{\gamma}_j)^{*}\mathcal{G}_{(j)}^{\gamma,0})^{s+1},\\
		 \mathcal{G}^{\gamma}Y^{\gamma}=&\big(\sum_{k=0}^{s}\mathcal{G}_{(j)}^{\gamma,0}(-Y^{\gamma}_j(Y^{\gamma}_j)^{*}\mathcal{G}_{(j)}^{\gamma,0})^k+\mathcal{G}^{\gamma}(-Y^{\gamma}_j(Y^{\gamma}_j)^{*}\mathcal{G}_{(j)}^{\gamma,0})^{s+1}\big)\cdot(Y_{(j)}^{\gamma,0}+Y^{\gamma}_j\mathbf{e}_j^{*}).
	\end{split}
\end{equation}

Denote $Y_j^{\gamma}=(y_{1j},\ldots,y_{nj})^*$. By direct expansion, we have
\begin{equation*}
	\begin{split}
		 [\mathcal{G}^{\gamma}]_{ab}=&[\mathcal{G}_{(j)}^{\gamma,0}]_{ab}+\sum_{u_{11},u_{12}}(-1)y_{u_{11}j}y_{u_{12}j}[\mathcal{G}_{(j)}^{\gamma,0}]_{au_{11}}[\mathcal{G}_{(j)}^{\gamma,0}]_{u_{12}b}\\
		 &+\sum_{u_{21},u_{22},u_{23},u_{24}}(-1)^2y_{u_{21}j}y_{u_{22}j}y_{u_{23}j}y_{u_{24}j}[\mathcal{G}_{(j)}^{\gamma,0}]_{au_{21}}[\mathcal{G}_{(j)}^{\gamma,0}]_{u_{22}u_{23}}[\mathcal{G}_{(j)}^{\gamma,0}]_{u_{24}b}\\
        &+\sum_{u_{31},u_{32},\dots, u_{36}}(-1)^3y_{u_{31}j}y_{u_{32}j}\cdots y_{u_{36}j}[\mathcal{G}_{(j)}^{\gamma,0}]_{au_{31}}[\mathcal{G}_{(j)}^{\gamma,0}]_{u_{32}u_{33}}\cdots [\mathcal{G}_{(j)}^{\gamma,0}]_{u_{36}b}\\
        &+\cdots\\
        &+\sum_{u_{s1},u_{s2},\cdots, u_{s(2s)}}(-1)^sy_{u_{s1}j}y_{u_{s2}j}\cdots y_{u_{s(2s)}j}[\mathcal{G}_{(j)}^{\gamma,0}]_{au_{s1}}[\mathcal{G}_{(j)}^{\gamma,0}]_{u_{s2}u_{s3}}\cdots [\mathcal{G}_{(j)}^{\gamma,0}]_{u_{s(2s)}b}\\
		&+\sum_{u_{s1},u_{s2},\cdots, u_{s(2s+2)}}(-1)^{s+1}y_{u_{s1}j}y_{u_{s2}j}\cdots y_{u_{s(2s+2)}j}[\mathcal{G}_{(j)}^{\gamma}]_{au_{s1}}[\mathcal{G}_{(j)}^{\gamma,0}]_{u_{s2}u_{s3}}\cdots [\mathcal{G}_{(j)}^{\gamma,0}]_{u_{s(2s+2)}b}.
	\end{split}
\end{equation*}
Similarly, we can expand $[\mathcal{G}^{\gamma}Y^{\gamma}]_{ab}$ as
\begin{equation*}
\begin{split}
    &[\mathcal{G}^{\gamma}Y^{\gamma}]_{ab}\\
    =&[\mathcal{G}^{\gamma}Y^{\gamma,0}_{(j)}]_{ab}\cdot \mathbbm{1}(b\neq j)+[\mathcal{G}^{\gamma}Y_j^{\gamma}\mathbf{e}_j^{*}]_{aj}\\
    =&[(\sum_{k=0}^s\mathcal{G}_{(j)}^{\gamma,0}(-Y_j^{\gamma}(Y_j^{\gamma})^{*})\mathcal{G}_{(j)}^{\gamma,0})^kY_{(j)}^{\gamma,0}]_{ab}\cdot\mathbbm{1}(b\neq j)+[(\mathcal{G}^{\gamma}(-Y_j^{\gamma}(Y_j^{\gamma})^{*})\mathcal{G}_{(j)}^{\gamma,0})^{s+1}Y_{(j)}^{\gamma,0}]_{ab}\cdot\mathbbm{1}(b\neq j)\\
    &+[(\sum_{k=0}^s\mathcal{G}_{(j)}^{\gamma,0}(-Y_j^{\gamma}(Y_j^{\gamma})^{*})\mathcal{G}_{(j)}^{\gamma,0})^kY_j^{\gamma}\mathbf{e}_j^{*}]_{aj}+[(\mathcal{G}^{\gamma}(-Y_j^{\gamma}(Y_j^{\gamma})^{*})\mathcal{G}_{(j)}^{\gamma,0})^{s+1}Y_j^{\gamma}\mathbf{e}_j^{*}]_{aj}\\
    =&\mathbbm{1}(b\neq j)\cdot\Big\{[\mathcal{G}_{(j)}^{\gamma,0}Y_{(j)}^{\gamma,0}]_{ab}+\sum_{u_{11},u_{12}}(-1)y_{u_{11}j}y_{u_{12}j}[\mathcal{G}_{(j)}^{\gamma,0}]_{au_{11}}[\mathcal{G}_{(j)}^{\gamma,0}Y_{(j)}^{\gamma,0}]_{u_{12}b}+\cdots\\
    &+\sum_{u_{s1},u_{s2},\cdots, u_{s(2s+2)}}(-1)^{s+1}y_{u_{s1}j}y_{u_{s2}j}\cdots y_{u_{s(2s+2)}j}[\mathcal{G}^{\gamma}]_{au_{s1}}[\mathcal{G}_{(j)}^{\gamma,0}]_{u_{s2}u_{s3}}\cdots [\mathcal{G}_{(j)}^{\gamma,0}Y_{(j)}^{\gamma,0}]_{u_{s(2s+2)}b}\Big\}\\
    &+\Big\{\sum_{u_{01}}y_{u_{01}j}[\mathcal{G}_{(j)}^{\gamma,0}]_{au_{01}}+\sum_{u_{11},u_{12},u_{13}}(-1)y_{u_{11}j}y_{u_{12}j}y_{u_{13}j}[\mathcal{G}_{(j)}^{\gamma,0}]_{au_{11}}[\mathcal{G}_{(j)}^{\gamma,0}]_{u_{12}u_{13}}+\cdots\\
    &+\sum_{u_{s1},u_{s2},\dots,u_{s(2s+3)}}(-1)^{s+1}y_{u_{s1}j}y_{u_{s2}j}\cdots y_{u_{s(2s+2)}j}y_{u_{s(2s+3)}j}[\mathcal{G}^{\gamma}]_{au_{s1}}[\mathcal{G}_{(j)}^{\gamma,0}]_{u_{s2}u_{s3}}\cdots[\mathcal{G}_{(j)}^{\gamma,0}]_{u_{s(2s+2)}u_{s(2s+3)}}\Big\}.
\end{split}
\end{equation*}

In order to reduce the long expression of the expansion of $[\mathcal{G}^{\gamma}]_{ab}$ and $[\mathcal{G}^{\gamma}Y^{\gamma}]_{ab}$, we define the following notation.
\begin{equation}\label{eq_Greenfunction_def_calg}
\begin{split}
    &\mathfrak{g}\big[(Y_j^{\gamma},\mathrm{k}_1)\otimes(\mathcal{G}_{(j)}^{\gamma,0},\mathrm{k}_2)\otimes(\mathcal{G}_{(j)}^{\gamma,0}Y_{(j)}^{\gamma,0},\mathrm{k}_3)\otimes(\mathcal{G}^{\gamma},\mathrm{k}_4)\big]\\
    &:=\sum_{u_1\dots u_{\mathrm{k}_1}}(-1)^{\mathrm{k}_1-\mathrm{k}_3}y_{u_1j}\cdots y_{u_{\mathrm{k}_1}j}[\mathcal{G}^{\gamma}]_{au_1}[\mathcal{G}_{(j)}^{\gamma,0}]_{u_2u_3}\cdots [\mathcal{G}_{(j)}^{\gamma,0}Y_{(j)}^{\gamma,0}]_{u_{\mathrm{k}_1}b},
\end{split}
\end{equation}
where the notation in $\mathfrak{g}\big[\cdot\big]$ has the meaning that, $(Y_j^{\gamma},\mathrm{k}_1)$:  entries of $Y^{\gamma}_j$ multiplied by themselves $\mathrm{k}_1$ times; $(\mathcal{G}_{(j)}^{\gamma,0},\mathrm{k}_2)$:  entries of $\mathcal{G}_{(j)}^{\gamma,0}$ multiplied by themselves $\mathrm{k}_2$ times, and so on for $(\mathcal{G}_{(j)}^{\gamma,0}Y_{(j)}^{\gamma,0},\mathrm{k}_3)$ and $(\mathcal{G}^{\gamma},\mathrm{k}_4)$. We should notice that the parameters $\mathrm{k}_1,\mathrm{k}_2,\mathrm{k}_3,\mathrm{k}_4$ have the relationship
\begin{equation*}
    \mathrm{k}_1=
    \begin{cases}
        2(\mathrm{k}_2+\mathrm{k}_3+\mathrm{k}_4-1)\quad \text{if $\mathrm{k}_1$ is even};\\
        2(\mathrm{k}_2+\mathrm{k}_3+\mathrm{k}_4-1)+1\quad \text{if $\mathrm{k}_1$ is odd},
    \end{cases}
\end{equation*}
and $\mathrm{k}_3,\mathrm{k}_4$ can only be one or zero. Then one may simplify the expansion of $[\mathcal{G}^{\gamma}]_{ab}$ and $[\mathcal{G}^{\gamma}Y^{\gamma}]_{ab}$ as
\begin{equation}\label{eq_resolventexpansion_G}
\begin{split}
    [\mathcal{G}^{\gamma}]_{ab}&=[\mathcal{G}_{(j)}^{\gamma,0}]_{ab}+\sum_{\mathrm{k}_1=2,\mathrm{k}_1\text{\;is even\;}}^{2s}\mathfrak{g}\big[(Y_j^{\gamma},\mathrm{k}_1)\otimes(\mathcal{G}_{(j)}^{\gamma,0},\mathrm{k}_1/2+1)\otimes(\mathcal{G}_{(j)}^{\gamma,0}Y_{(j)}^{\gamma,0},0)\otimes(\mathcal{G}^{\gamma},0)\big]\\
    &+\mathfrak{g}\big[(Y_j^{\gamma},2s+2)\otimes(\mathcal{G}_{(j)}^{\gamma,0},s+1)\otimes(\mathcal{G}_{(j)}^{\gamma,0}Y_{(j)}^{\gamma,0},0)\otimes(\mathcal{G}^{\gamma},1)\big].
\end{split}
\end{equation}
\begin{equation}\label{eq_resolventexpansion_GY}
    \begin{split}
        [\mathcal{G}^{\gamma}Y^{\gamma}]_{ab}&=\Big\{[\mathcal{G}^{\gamma,0}_{(j)}Y_{(j)}^{\gamma,0}]_{ab}+\sum_{\mathrm{k}_1=2,\mathrm{k}_1\text{\;is even\;}}^{2s}\mathfrak{g}\big[(Y_j^{\gamma},\mathrm{k}_1)\otimes(\mathcal{G}_{(j)}^{\gamma,0},\mathrm{k}_1/2)\otimes(\mathcal{G}_{(j)}^{\gamma,0}Y_{(j)}^{\gamma,0},1)\otimes(\mathcal{G}^{\gamma},0)\big]\\
        &+\mathfrak{g}\big[(Y_j^{\gamma},2s+2)\otimes(\mathcal{G}_{(j)}^{\gamma,0},s+1)\otimes(\mathcal{G}_{(j)}^{\gamma,0}Y_{(j)}^{\gamma,0},1)\otimes(\mathcal{G}^{\gamma},1)\big]\Big\}\cdot\mathbbm{1}(b\neq j)\\
        &+\sum_{\mathrm{k}_1=1, \mathrm{k}_1\text{\;is odd\;}}^{2s+1}\mathfrak{g}\big[(Y_j^{\gamma},\mathrm{k}_1)\otimes(\mathcal{G}_{(j)}^{\gamma,0},(\mathrm{k}_1-1)/2+1)\otimes(\mathcal{G}_{(j)}^{\gamma,0}Y_{(j)}^{\gamma,0},0)\otimes(\mathcal{G}^{\gamma},0)\big]\\
        &+\mathfrak{g}\big[(Y_j^{\gamma},2s+3)\otimes(\mathcal{G}_{(j)}^{\gamma,0},s+1)\otimes(\mathcal{G}_{(j)}^{\gamma,0}Y_{(j)}^{\gamma,0},0)\otimes(\mathcal{G}^{\gamma},1)\big].
    \end{split}
\end{equation}

It should be noted that all other ensembles formed by $\mathcal{G}^{\gamma}$ and $Y^{\gamma}$ can be expanded similarly. With the above resolvent expansion, estimating the expectation of the Green functions can be divided into two independent parts: $Y_j^{\gamma}$ and $\mathcal{G}_{(j)}^{\gamma,0}$. However, the expansion is more complicated than before due to various terms in $\mathfrak{g}$ and the convolutions. In order to obtain concrete estimates for the expectation for the Green functions, we need to trace the patterns in the products of $(Y_j^{\gamma},\mathrm{k}_1)$'s. For this purpose, we introduce a notation called \textit{assemble patterns classification} to systematically record the products of the entries of $Y_j^{\gamma}$ and $\mathcal{G}_{(j)}^{\gamma,0}$.

\

\textbf{III. Assemble patterns classification}
\begin{definition}[Assemble patterns classification]\label{def_prf_Greenfunctioncomparison_encode_beta}
Define the following assemble patterns for the entries of $Y^{\gamma}_j$ as
	\begin{equation*}
	\beta^{(K)}_{k_1k_2\dots k_q}=y_{ij}^{k_1}y_{u_2 j}^{k_2}\cdots y_{u_q j}^{k_q}\cdot\mathbbm{1}(u_q\neq u_{q-1}\neq\cdots\neq u_2\neq i),\quad K=k_1+k_2+\dots+k_q.
	\end{equation*}
	For the entries of Green function $[\mathcal{G}_{(j)}^{\gamma,0}]$, we define the following assemble patterns accompanying $\beta^{(K)}_{k_1k_2\dots k_q}$,
	\begin{equation*}
	\begin{split}
	\mathfrak{G}^{(K)}_{(j),k_1k_2\dots k_q}&=\sum_{\stackrel{u_q\ne \cdots\ne u_2\ne i}{(v_1,\ldots,v_{K})\in \mathbb{Q}^{\mathrm{permutation}}_{k_1k_2\ldots k_q}}}[\mathcal{G}_{(j)}^{\gamma,0}]_{ib}[\mathcal{G}_{(j)}^{\gamma,0}]_{av_1}\cdots[\mathcal{G}_{(j)}^{\gamma,0}]_{v_{K-1}v_{K}},\\
	\widetilde{\mathfrak{G}}^{(K)}_{(j),k_1k_2\dots k_q}&=\sum_{\stackrel{u_q\ne \cdots\ne u_2\ne i}{(v_1,\ldots,v_{K})\in \mathbb{Q}^{\mathrm{permutation}}_{k_1k_2\ldots k_q}}}[\mathcal{G}_{(j)}^{\gamma,0}]_{iv_1}\cdots[\mathcal{G}_{(j)}^{\gamma,0}]_{v_{K-1}b}[\mathcal{G}_{(j)}^{\gamma,0}]_{av_{K}},
	\end{split}
	\end{equation*}
where the set $\mathbb{Q}^{\mathrm{permutation}}_{k_1k_2\ldots k_q}$ contains all the permutations of the set $\mathbb{Q}_{k_1k_2\ldots k_q}$ with
	\begin{equation*}
	\mathbb{Q}_{k_1k_2\ldots k_q}=\{\underbrace{i, \ldots, i}_{k_1}, \underbrace{u_2, \ldots, u_2}_{k_2}, \ldots, \underbrace{u_q, \ldots, u_q}_{k_{q}}\}.
	\end{equation*}
\end{definition}
Notice that the upper index $K$ should satisfy $K=\mathrm{k}_1$ as in \eqref{eq_Greenfunction_def_calg}. Then by Definition \ref{def_prf_Greenfunctioncomparison_encode_beta}, for example, we have
\begin{equation*}
    \mathbb{E}\Big(\mathfrak{g}\big[(Y_j^{\gamma},K))\otimes(\mathcal{G}_{(j)}^{\gamma,0},(K-1)/2+1)\otimes(\mathcal{G}_{(j)}^{\gamma,0}Y_{(j)}^{\gamma,0},1)\otimes(\mathcal{G}^{\gamma},0)\big]\cdot\mathbbm{1}(\mu_1=i)\Big)=\mathbb{E}\beta_{k_1k_2\dots k_q}^{(K)}\cdot \widetilde{\mathfrak{G}}_{(j),k_1k_2\dots k_q}^{(K/2+1/2)},
\end{equation*}
where the expectation is with respect to the randomness of $Y_j^{\gamma}$. Here, in Definition \ref{def_prf_Greenfunctioncomparison_encode_beta}, we are particularly concern about quantity $y_{ij}$ since it plays a key role of moment comparison in each $Y_{ij}^{\gamma}$. For simplification, we can assume that the lower indexes $u_2,u_3,\dots u_q$ of $\beta^{(K)}_{k_1k_2\dots k_q}$ are $1,2,\dots,i-1,i+1,\dots q$ respectively, without causing further confusion since $y_{ij},1\le i\le n$ have the same distribution. Moreover, in the sense of expectation, by the independence of $Y^{\gamma}_j$ and $\mathcal{G}_{(j)}^{\gamma,0}$, we can treat the expectation of $\beta_{k_1,k_2,\dots,k_q}^{(k)}$ and $\mathfrak{G}_{(j),k_1,k_2,\dots,k_q}^{(K)}$ (or $\widetilde{\mathfrak{G}}_{(j),k_1,k_2,\dots,k_q}^{(K)}$) respectively, cf., $\mathbb{E}(\sum_{u_2\ne i}y_{u_2j}[\mathcal{G}_{(j)}^{\gamma,0}]_{ib}[\mathcal{G}_{(j)}^{\gamma,0}]_{au_2})=\mathbb{E}(y_{u_2j}\mathbbm{1}(u_2\ne j))\mathbb{E}(\sum_{u_2\ne i}[\mathcal{G}_{(j)}^{\gamma,0}]_{ib}[\mathcal{G}_{(j)}^{\gamma,0}]_{au_2})$. It should be pointed out that $\beta^{(K)}_{k_1k_2\dots k_q}$ identifies the patterns $\mathfrak{G}^{(K)}_{(j),k_1k_2\dots k_q}$ and $\widetilde{\mathfrak{G}}^{(K)}_{(j),k_1k_2\dots k_q}$ via the index set $\mathbb{Q}_{k_1k_2\ldots k_q}$.
The notation $\beta^{(K)}_{k_1k_2\dots k_q}$ and $\mathfrak{G}^{(K)}_{(j),k_1k_2\dots k_q}$ enables us to label several quantities in the expansion of $\mathcal{G}^{\gamma}$, $\mathcal{G}^{\gamma}Y^{\gamma}$ or their products. For example, in \eqref{eq_prf_Greenfunctioncomparison_expansion_f} below, one term $\operatorname{Im}\big[[\mathcal{G}_{(j)}^{\gamma,0}]_{ib}[\mathcal{G}_{(j)}^{\gamma,0}Y^{\gamma}_j]_{aj}\big]$ in the expansion will give the following two patterns
$\beta_{1}^{(1)}\mathfrak{G}^{(1)}_{(j),1}$ or $\beta_{01}^{(1)}\mathfrak{G}^{(1)}_{(j),01}$, say,
\begin{equation*}
\operatorname{Im}\big[[\mathcal{G}_{(j)}^{\gamma,0}]_{ib}[\mathcal{G}_{(j)}^{\gamma,0}Y^{\gamma}_j]_{aj}\big]=\operatorname{Im}\big[\beta_{1}^{(1)}\mathfrak{G}^{(1)}_{(j),1}+\beta_{01}^{(1)}\mathfrak{G}^{(1)}_{(j),01}\big].
\end{equation*}
Similarly, another term  $\operatorname{Im}\big[[\mathcal{G}_{(j)}^{\gamma,0}]_{ib}[\mathcal{G}_{(j)}^{\gamma,0}Y^{\gamma}_j(Y^{\gamma}_j)^{*}\mathcal{G}_{(j)}^{\gamma,0}Y^{\gamma}_j]_{aj}\big]$ means the following patterns
\begin{equation*}
\begin{split}
&\operatorname{Im}\big[[\mathcal{G}_{(j)}^{\gamma,0}]_{ib}[\mathcal{G}_{(j)}^{\gamma,0}Y^{\gamma}_j(Y^{\gamma}_j)^{*}\mathcal{G}_{(j)}^{\gamma,0}Y^{\gamma}_j]_{aj}\big]=\\
&\operatorname{Im}\big[\beta_{3}^{(3)}\mathfrak{G}^{(3)}_{(j),3}+\beta_{21}^{(3)}\mathfrak{G}^{(3)}_{(j),21}+\beta_{111}^{(3)}\mathfrak{G}^{(3)}_{(j),111}+\beta_{12}^{(3)}\mathfrak{F}^{(3)}_{(j),12}+\beta_{03}^{(3)}\mathfrak{G}^{(3)}_{(j),03}+\beta_{021}^{(3)}\mathfrak{G}^{(3)}_{(j),021}+\beta_{0111}^{(3)}\mathfrak{G}^{(3)}_{(j),0111}\big],
\end{split}
\end{equation*}
and the term $\operatorname{Im}\big[[\widetilde{\mathfrak{G}}_{(j)}^{\gamma,0}Y^{\gamma}_j(Y^{\gamma}_j)^{*}\mathcal{G}_{(j)}^{\gamma,0}]_{ib}[\mathcal{G}_{(j)}^{\gamma,0}Y^{\gamma}_j]_{aj}\big]$ gives
\begin{equation*}
\begin{split}
&\operatorname{Im}\big[[\mathcal{G}_{(j)}^{\gamma,0}Y^{\gamma}_j(Y^{\gamma}_j)^{*}\mathcal{G}_{(j)}^{\gamma,0}]_{ib}[\mathcal{G}_{(j)}^{\gamma,0}Y^{\gamma}_j]_{aj}\big]=\\
&\operatorname{Im}\big[\beta_{3}^{(3)}\widetilde{\mathfrak{G}}^{(3)}_{(j),3}+\beta_{21}^{(3)}\widetilde{\mathfrak{G}}^{(3)}_{(j),21}+\beta_{111}^{(3)}\widetilde{\mathfrak{G}}^{(3)}_{(j),111}+\beta_{12}^{(3)}\widetilde{\mathfrak{G}}^{(3)}_{(j),12}+\beta_{03}^{(3)}\widetilde{\mathfrak{G}}^{(3)}_{(j),03}+\beta_{021}^{(3)}\widetilde{\mathfrak{G}}^{(3)}_{(j),021}+\beta_{0111}^{(3)}\widetilde{\mathfrak{G}}^{(3)}_{(j),0111}\big].
\end{split}
\end{equation*}
With the assemble patterns classification, we can figure out various terms in the resolvent expansion and combine the moment bounds in Lemmas \ref{lem_moment_rates} and \ref{lem_oddmoment_est}. Now we turn to the formal proofs of Theorems \ref{thm_greenfuncomp_entrywise_differror}-\ref{thm_greenfuncomp_average}.

\subsubsection{Proof of Theorem \ref{thm_greenfuncomp_entrywise_differror}}\label{app_prf_thm_greenfuncomp_entrywise_differror}
This subsection concerns the proof for Theorem \ref{thm_greenfuncomp_entrywise_differror}.
We only show the case where $(\#_1,\#_2,\#_3)=(\mathfrak{R}_{ab}(\operatorname{Im}[\mathcal{G}^{\gamma}(z)]_{ab}-\delta_{ab}b_t(z)m^{(t)}(\zeta)),\mathfrak{R}_{ab}(\operatorname{Im}[\mathcal{G}^0(z)]_{ab}-\delta_{ab}b_t(z)m^{(t)}(\zeta)),\mathfrak{J}_{0,ab})$ for $a\in\mathrm{T}_r$ or $b\in\mathrm{T}_r$, while the other cases can be proved similarly. Moreover, in the proof below, we will neglect the deterministic terms $\delta_{ab}b_t(z)m^{(t)}(\zeta)$ for simplicity. One may observe that such simplification will not influence the proof since $\delta_{ab}b_t(z)m^{(t)}(\zeta)=\mathrm{O}(1)$ is smaller than the prior upper bound $n^{\epsilon}$. In the sequel, we may omit the dependence on $z$  in notation for simplification whenever there is no confusion. Observe that
\begin{equation*}
	\begin{split}
		&\frac{\partial \mathbb{E}_{\Psi}\big(F([\operatorname{Im}\mathcal{G}^{\gamma}]_{ab})\big)}{\partial \gamma}=-\sum_{i,j}\mathbb{E}_{\Psi}\Big[F^{(1)}\big(\operatorname{Im}[\mathcal{G}^{\gamma}]_{ab}\big)\operatorname{Im}\big([\mathcal{G}^{\gamma}]_{ib}[\mathcal{G}^{\gamma}Y^{\gamma}]_{aj}\big)\Big(L_{ij}\rho_j^{-1}-\frac{\gamma t^{1/2}w_{ij}}{(1-\gamma^2)^{1/2}}\Big)\Big]\\
		 &-\sum_{i,j}\mathbb{E}_{\Psi}\Big[F^{(1)}\big(\operatorname{Im}[\mathcal{G}^{\gamma}]_{ab}\big)\operatorname{Im}\big([\mathcal{G}^{\gamma}]_{ai}[(Y^{\gamma})^{*}\mathcal{G}^{\gamma}]_{jb}\big)\Big(L_{ij}\rho_j^{-1}-\frac{\gamma t^{1/2}w_{ij}}{(1-\gamma^2)^{1/2}}\Big)\Big]:=-\sum_{i,j}\Big[(\mathcal{E}_1)_{ij}+(\mathcal{E}_2)_{ij}\Big].
	\end{split}
\end{equation*}
Therefore, it suffices to show that the following bound holds,
\begin{equation}\label{eq_prf_Greenfunctioncomparison_equivbound}
	\sum_{i,j}\Big[|(\mathcal{E}_1)_{ij}|+|(\mathcal{E}_2)_{ij}|\Big]\le \frac{C}{(1-\gamma^2)^{1/2}}\big(n^{-c_0}(\mathfrak{J}_{0,ab}+1)+Q_0n^{C_0+C}\big),
\end{equation}
for some constant $C>0$. We shall only derive the estimate for $(\mathcal{E}_1)_{ij}$'s while the same estimates for $(\mathcal{E}_2)_{ij}$'s can be obtained similarly.

Define the shorthand notation $\mathfrak{F}_{(j)}(\varpi)=\mathfrak{f}_{(j)}(\varpi)\mathfrak{i}_{(j)}(\varpi)$, where
\begin{eqnarray*}
&&	\mathfrak{f}_{(j)}(\varpi)=F^{(1)}\big(\operatorname{Im}[\mathcal{G}^{\gamma,\varpi}_{(j)]}]_{ab}\big),~ \mathfrak{i}_{(j)}(\varpi)=\operatorname{Im}\big([\mathcal{G}^{\gamma,\varpi}_{(j)]}]_{ib}[\mathcal{G}^{\gamma,\varpi}_{(j)}Y^{\gamma,\varpi}_{(j)}]_{aj}\big),\\
&&	 \widetilde{\mathfrak{i}}{(j)}(\varpi)=\operatorname{Im}\big([\mathcal{G}^{\gamma,\varpi}_{(j)]}]_{ai}[(Y^{\gamma,\varpi}_{(j)})^{*}\mathcal{G}^{\gamma,\varpi}_{(j)]}]_{jb}\big).
\end{eqnarray*}
Then for any $i\in[n], j\in [p]$, invoking $L_{ij}=0$ if $\psi_{ij}=1$, we rewrite $(\mathcal{E}_1)_{ij}$ as
\begin{equation*}
	\begin{split}(\mathcal{E}_1)_{ij}&=\mathbb{E}_{\Psi}\Big[\mathfrak{F}_{(j)}\big(Y^{\gamma}_{j}\big)\Big(L_{ij}\rho_j^{-1}-\frac{\gamma t^{1/2}w_{ij}}{(1-\gamma^2)^{1/2}}\Big)\cdot\big(\mathbbm{1}(\psi_{ij}=0)+\mathbbm{1}(\psi_{ij}=1)\big)\Big]\\
		&=\mathbb{E}_{\Psi}\Big[\mathfrak{F}_{(j)}(Y^{\gamma}_j)\Big(L_{ij}\rho_j^{-1}-\frac{\gamma t^{1/2}w_{ij}}{(1-\gamma^2)^{1/2}}\Big)\cdot\mathbbm{1}(\psi_{ij}=0)\Big]\\
&\qquad \qquad -\frac{\gamma}{(1-\gamma^2)^{1/2}}t^{1/2}\mathbb{E}_{\Psi}\Big[w_{ij}\mathfrak{F}_{(j)}(Y^{\gamma}_j)\cdot\mathbbm{1}(\psi_{ij}=1)\Big]\\
		&=:(\mathcal{E}_{11})_{ij}-(\mathcal{E}_{12})_{ij}.
	\end{split}
\end{equation*}

Let us consider $(\mathcal{E}_{12})_{ij}$ first. Using Gaussian integration by part on $w_{ij}$
gives that
\begin{equation*}
	\begin{split}
		|(\mathcal{E}_{12})_{ij}|=&\Big|\frac{\gamma t^{1/2}}{(1-\gamma^2)^{1/2}n}\mathbb{E}_{\Psi}\Big[\partial_{w_{ij}}\mathfrak{F}_{(j)}(Y^{\gamma}_j)\cdot\mathbbm{1}(\psi_{ij}=1)\Big]\Big|\\
		\le&\frac{\gamma t^{1/2}}{(1-\gamma^2)^{1/2}n}\mathbb{E}_{\Psi}\Big[\big(|\partial_{w_{ij}}\mathfrak{F}_{(j)}(Y^{\gamma}_j)\cdot\mathbbm{1}(\Omega)|+|\partial_{w_{ij}}\mathfrak{F}_{(j)}(Y^{\gamma}_j)\cdot\mathbbm{1}(\Omega^c)|\big)\cdot\mathbbm{1}(\psi_{ij}=1)\Big].
	\end{split}
\end{equation*}
Notice that
\begin{equation*}
	\partial_{w_{ij}}\mathfrak{F}_{(j)}(Y^{\gamma}_j)=\partial_{w_{ij}}\mathfrak{f}_{(j)}(Y^{\gamma}_j)\cdot \mathfrak{i}_{(j)}(Y^{\gamma}_j)+\mathfrak{f}_{(j)}(Y^{\gamma}_j)\cdot\partial_{w_{ij}} \mathfrak{i}_{(j)}(Y^{\gamma}_j),
\end{equation*}
where
\begin{equation*}
	\begin{split}
		 \partial_{w_{ij}}\mathfrak{f}_{(j)}(Y^{\gamma}_j)=&-(1-\gamma^2)^{1/2}t^{1/2}F^{(2)}\big(\operatorname{Im}[\mathcal{G}^{\gamma}]_{ab}\big)\big(\mathfrak{i}_{(j)}(Y^{\gamma}_j)+\widetilde{\mathfrak{i}}_{(j)}(Y^{\gamma}_j)\big),\\
		 \partial_{w_{ij}}\mathfrak{i}_{(j)}(Y^{\gamma}_j)=&-(1-\gamma^2)^{1/2}t^{1/2}\operatorname{Im}\Big([\mathcal{G}^{\gamma}]_{ii}[(Y^{\gamma})^{*}\mathcal{G}^{\gamma}]_{jb}[\mathcal{G}^{\gamma}Y^{\gamma}]_{aj}+[\mathcal{G}^{\gamma}Y^{\gamma}]_{ij}[\mathcal{G}^{\gamma}]_{ib}[\mathcal{G}^{\gamma}Y^{\gamma}]_{aj}\\
		&\quad-[\mathcal{G}_{(j)}^{\gamma}]_{ib}[\mathcal{G}^{\gamma}]_{ai}
		 +[\mathcal{G}^{\gamma}]_{ib}[\mathcal{G}^{\gamma}]_{ai}[(Y^{\gamma})^{*}\mathcal{G}^{\gamma}Y^{\gamma}]_{jj}+[\mathcal{G}^{\gamma}]_{ib}[\mathcal{G}^{\gamma}Y^{\gamma}]_{aj}[\mathcal{G}^{\gamma}Y^{\gamma}]_{ij}\Big).
	\end{split}
\end{equation*}
Also note that we have $i\in\mathrm{I}_r$ and $j\in\mathrm{I}_c$  when $\psi_{ij}=1$. Then we observe that $\mathbbm{1}(\Omega)\mathbbm{1}(\psi_{ij}=1)|\mathfrak{i}_{(j)}(Y^{\gamma}_j)|\le n^{2\varepsilon}t^{-2}$ and $\mathbbm{1}(\Omega)\mathbbm{1}(\psi_{ij}=1)|\partial_{w_{ij}}\mathfrak{i}_{(j)}(Y^{\gamma}_j)|\le n^{3\varepsilon}t^{-7/2}$ since $a\in\mathrm{T}_r$ or $b\in\mathrm{T}_r$. Then, there exists a large constant $K_1>0$ such that
\begin{equation*}
	\begin{split}
		 |(\mathcal{E}_{12})_{ij}|\lesssim&\frac{\mathbbm{1}(\psi_{ij}=1)}{n^{1-3\varepsilon}t^3}\mathbb{E}_{\Psi}\Big[|F^{(1)}(\operatorname{Im}[\mathcal{G}^{\gamma}]_{ab})|\Big]+\frac{\mathbbm{1}(\psi_{ij}=1)}{n^{1-4\varepsilon}t^3}\mathbb{E}_{\Psi}\Big[|F^{(2)}(\operatorname{Im}[\mathcal{G}^{\gamma}]_{ab})|\Big]\\
		 &+\frac{t^{1/2}\mathbbm{1}(\psi_{ij}=1)}{n^{1-3\varepsilon}t^3}\mathbb{E}_{\Psi}\Big[|\partial_{w_{ij}}\mathfrak{F}_{(j)}(Y^{\gamma}_j)\mathbbm{1}(\Omega^c)|\Big]\\
		\lesssim& \frac{\mathbbm{1}(\psi_{ij}=1)}{n^{1-3\varepsilon}t^3}\mathbb{E}_{\Psi}\Big[|F^{(1)}(\operatorname{Im}[\mathcal{G}^{\gamma}]_{ab})|\Big]+\frac{\mathbbm{1}(\psi_{ij}=1)}{n^{1-4\varepsilon}t^3}\mathbb{E}_{\Psi}\Big[|F^{(2)}(\operatorname{Im}[\mathcal{G}^{\gamma}]_{ab})|\Big]+n^{K_1}Q_0\mathbbm{1}(\psi_{ij}=1),
	\end{split}
\end{equation*}
where in the last step, we used the assumption of $F^{(\mu)}(x)$ and the crude bound $|\operatorname{Im}\mathcal{G}^{\gamma}|\lesssim\eta^{-1}$ for $\eta>n^{-1}$. Recall that $\sum_{i,j}\mathbbm{1}(\psi_{ij}=1)\lesssim n^{1-\epsilon_y}$ and $t\gg n^{-\epsilon_{\alpha}}$, so we can choose $\varepsilon<\epsilon_{\alpha}/8$ such that
\begin{equation*}
	 |(\mathcal{E}_{12})_{ij}|\lesssim\frac{\mathbbm{1}(\psi_{ij}=1)}{n^{1-\epsilon_{\alpha}/2}}\mathbb{E}_{\Psi}\Big[|F^{(1)}(\operatorname{Im}[\mathcal{G}^{\gamma}]_{ab})|+|F^{(2)}(\operatorname{Im}[\mathcal{G}^{\gamma}]_{ab})|\Big]+n^{K_1}Q_0\mathbbm{1}(\psi_{ij}=1).
\end{equation*}

Next, we consider $(\mathcal{E}_{11})_{ij}$.  The calculation for this part is relatively complicated and we will heavily rely on the preliminaries defined in Section \ref{sec_Green_pre}. The first task is to replace $\mathfrak{F}_{(j)}(Y_j^{\gamma})$ with $\mathfrak{F}_{(j)}(0)$. In order to do this, we begin with $\mathfrak{f}_{(j)}(Y^{\gamma}_j)$ first. Basically, we need to perform the Taylor expansion in $\mathfrak{f}_{(j)}(Y^{\gamma}_j)$ via our assumption of $F$, based on the resolvent expansion for the Green functions. One may see from \eqref{eq_resolventexpansion_G} that
\begin{equation*}
	\begin{split}
		&\mathfrak{f}_{(j)}(Y^{\gamma}_j)=F^{(1)}\big(\operatorname{Im}[\mathcal{G}_{(j)}^{\gamma,0}]_{ab}+\sum_{k=2,k\text{\;is even\;}}^{2s}\mathfrak{g}\big[(Y_j^{\gamma},k)\otimes(\mathcal{G}_{(j)}^{\gamma,0},k/2+1)\otimes(\mathcal{G}_{(j)}^{\gamma,0}Y_{(j)}^{\gamma,0},0)\otimes(\mathcal{G}^{\gamma},0)\big]\\
        &+\mathfrak{g}\big[(Y_j^{\gamma},2s+2)\otimes(\mathcal{G}_{(j)}^{\gamma,0},s+1)\otimes(\mathcal{G}_{(j)}^{\gamma,0}Y_{(j)}^{\gamma,0},0)\otimes(\mathcal{G}^{\gamma},1)\big]\big)\\
	=&F^{(1)}\big(\operatorname{Im}[\mathcal{G}_{(j)}^{\gamma,0}]_{ab}\big)\\
&+\sum_{l=1}^{s_l}\frac{1}{l!}F^{(1+l)}\big(\operatorname{Im}[\mathcal{G}_{(j)}^{\gamma,0}]_{ab}\big)\cdot\Big(\sum_{k=2, k\text{\;is even\;}}^{2s}\mathfrak{g}\big[(Y_j^{\gamma},k)\otimes(\mathcal{G}_{(j)}^{\gamma,0},k/2+1)\otimes(\mathcal{G}_{(j)}^{\gamma,0}Y_{(j)}^{\gamma,0},0)\otimes(\mathcal{G}^{\gamma},0)\big]\\
        &+\mathfrak{g}\big[(Y_j^{\gamma},2s+2)\otimes(\mathcal{G}_{(j)}^{\gamma,0},s+1)\otimes(\mathcal{G}_{(j)}^{\gamma,0}Y_{(j)}^{\gamma,0},0)\otimes(\mathcal{G}^{\gamma},1)\big]\Big)^l+F^{(s_l+2)}_{\operatorname{res}}\\
		=&F^{(1)}\big(\operatorname{Im}[\mathcal{G}_{(j)}^{\gamma,0}]_{ab}\big)\\
&+\sum_{l=1}^{s_l}\frac{1}{l!}F^{(1+l)}\big(\operatorname{Im}[\mathcal{G}_{(j)}^{\gamma,0}]_{ab}\big)\cdot\Big(\sum_{k=2,k\text{\;is even\;}}^{2s}\mathfrak{g}\big[(Y_j^{\gamma},k)\otimes(\mathcal{G}_{(j)}^{\gamma,0},k/2+1)\otimes(\mathcal{G}_{(j)}^{\gamma,0}Y_{(j)}^{\gamma,0},0)\otimes(\mathcal{G}^{\gamma},0)\big]\Big)^l\\
		&+\sum_{l=1}^{s_l}\sum_{l_1=0}^{l-1}\frac{1}{l_1!}\binom{l}{l_1}F^{(1+l)}\big(\operatorname{Im}[\mathcal{G}_{(j)}^{\gamma,0}]_{ab}\big)\\
&\cdot\Big(\sum_{k=2,k\text{\;is even\;}}^{2s}\mathfrak{g}\big[(Y_j^{\gamma},k)\otimes(\mathcal{G}_{(j)}^{\gamma,0},k/2+1)\otimes(\mathcal{G}_{(j)}^{\gamma,0}Y_{(j)}^{\gamma,0},0)\otimes(\mathcal{G}^{\gamma},0)\big]\Big)^{l_1}\\
        &\cdot\big(\mathfrak{g}\big[(Y_j^{\gamma},2s+2)\otimes(\mathcal{G}_{(j)}^{\gamma,0},s+1)\otimes(\mathcal{G}_{(j)}^{\gamma,0}Y_{(j)}^{\gamma,0},0)\otimes(\mathcal{G}^{\gamma},1)\big]\big)^{l-l_1}+F^{(s_l+2)}_{\operatorname{res}}\\
		=&: \mathfrak{f}_{(j)}(0)+\mathfrak{f}^1_{(j)}(0)+\mathfrak{f}^r_{(j)}(0)+F^{(s_l+2)}_{\operatorname{res}},
	\end{split}
\end{equation*}
where $F_{\operatorname{res}}^{(s_l+2)}$ denotes the residual term after the Taylor expansion for $F$. And $\mathfrak{f}_{(j)}(0):=F^{(1)}(\operatorname{Im}[\mathcal{G}_{(j)}^{\gamma,0}]_{ab})$,
\begin{equation*}
\mathfrak{f}^1_{(j)}(0):=\sum_{l=1}^{s_l}\frac{1}{l!}F^{(1+l)}\big(\operatorname{Im}[\mathcal{G}_{(j)}^{\gamma,0}]_{ab}\big)\cdot\Big(\sum_{k=2,e}^{2s}\mathfrak{g}\big[(Y_j^{\gamma},k)\otimes(\mathcal{G}_{(j)}^{\gamma,0},k/2+1)\otimes(\mathcal{G}_{(j)}^{\gamma,0}Y_{(j)}^{\gamma,0},0)\otimes(\mathcal{G}^{\gamma},0)\big]\Big)^l,
\end{equation*}
\begin{equation*}
\begin{split}
    \mathfrak{f}^r_{(j)}(0):=&\sum_{l_1=1}^{s_l}\sum_{l_2=1}^{l_1}\frac{1}{l_1!}\binom{l_1}{l_2}F^{(1+l_1)}\big(\operatorname{Im}[\mathcal{G}_{(j)}^{\gamma,0}]_{ab}\big)\\
    &\cdot\Big(\sum_{k=2,e}^{2s}\mathfrak{g}\big[(Y_j^{\gamma},k)\otimes(\mathcal{G}_{(j)}^{\gamma,0},k/2+1)\otimes(\mathcal{G}_{(j)}^{\gamma,0}Y_{(j)}^{\gamma,0},0)\otimes(\mathcal{G}^{\gamma},0)\big]\Big)^{l_2}\\
        &\cdot\big(\mathfrak{g}\big[(Y_j^{\gamma},2s+2)\otimes(\mathcal{G}_{(j)}^{\gamma,0},s+1)\otimes(\mathcal{G}_{(j)}^{\gamma,0}Y_{(j)}^{\gamma,0},0)\otimes(\mathcal{G}^{\gamma},1)\big]\big)^{l_1-l_2},
\end{split}
\end{equation*}
where we used $k=2,e$ to denote $k=2, k\text{\;is even\;}$, and $k=1,o$ means $k=1, k\text{\;is odd\;}$, to shorten the notation.

Also, we can obtain the parallel result for $\mathfrak{i}_{(j)}(Y^{\gamma}_j)$ (notice that $b=j$ at this time),
\begin{equation*}
	\begin{split}
		\mathfrak{i}_{(j)}(Y^{\gamma}_j) =&\operatorname{Im}\Big\{\big([\mathcal{G}_{(j)}^{\gamma,0}]_{ib}+\sum_{k=2,e}^{2s}\mathfrak{g}\big[(Y_j^{\gamma},k)\otimes(\mathcal{G}_{(j)}^{\gamma,0},k/2+1)\otimes(\mathcal{G}_{(j)}^{\gamma,0}Y_{(j)}^{\gamma,0},0)\otimes(\mathcal{G}^{\gamma},0)\big]\\
    &+\mathfrak{g}\big[(Y_j^{\gamma},2s+2)\otimes(\mathcal{G}_{(j)}^{\gamma,0},s+1)\otimes(\mathcal{G}_{(j)}^{\gamma,0}Y_{(j)}^{\gamma,0},0)\otimes(\mathcal{G}^{\gamma},1)\big]\big)\\
    &\cdot\big(\sum_{k=1,o}^{2s+1}\mathfrak{g}\big[(Y_j^{\gamma},k)\otimes(\mathcal{G}_{(j)}^{\gamma,0},(k-1)/2+1)\otimes(\mathcal{G}_{(j)}^{\gamma,0}Y_{(j)}^{\gamma,0},0)\otimes(\mathcal{G}^{\gamma},0)\big]\\
        &+\mathfrak{g}\big[(Y_j^{\gamma},2s+3)\otimes(\mathcal{G}_{(j)}^{\gamma,0},s+1)\otimes(\mathcal{G}_{(j)}^{\gamma,0}Y_{(j)}^{\gamma,0},0)\otimes(\mathcal{G}^{\gamma},1)\big]\big)\Big\}\\
  =&\operatorname{Im}\Big[[\mathcal{G}_{(j)}^{\gamma,0}]_{ib}[\mathcal{G}_{(j)}^{\gamma,0}Y^{\gamma}_j]_{aj}+\sum_{k=3,o}^{2s+1}\mathfrak{g}\big[(Y_j^{\gamma},k)\otimes(\mathcal{G}_{(j)}^{\gamma,0},(k-1)/2+2)\otimes(\mathcal{G}_{(j)}^{\gamma,0}Y_{(j)}^{\gamma,0},0)\otimes(\mathcal{G}^{\gamma},0)\big]\\
  &+\sum_{k=2,e}^{2s}\mathfrak{g}\big[(Y_j^{\gamma},k+1)\otimes(\mathcal{G}_{(j)}^{\gamma,0},k/2+2)\otimes(\mathcal{G}_{(j)}^{\gamma,0}Y_{(j)}^{\gamma,0},0)\otimes(\mathcal{G}^{\gamma},0)\big]\\
  &+\sum_{k=5,o}^{4s+1}\mathfrak{g}\big[(Y_j^{\gamma},k)\otimes(\mathcal{G}_{(j)}^{\gamma,0},k-1)\otimes(\mathcal{G}_{(j)}^{\gamma,0}Y_{(j)}^{\gamma,0},0)\otimes(\mathcal{G}^{\gamma},0)\big]\\
  &+\sum_{k=0,e}^{2s}\mathfrak{g}\big[(Y_j^{\gamma},k+2s+3)\otimes(\mathcal{G}_{(j)}^{\gamma,0},k/2+s+2)\otimes(\mathcal{G}_{(j)}^{\gamma,0}Y_{(j)}^{\gamma,0},0)\otimes(\mathcal{G}^{\gamma},1)\big]\\
&+\sum_{k=1,o}^{2s+1}\mathfrak{g}\big[(Y_j^{\gamma},k+2s+2)\otimes(\mathcal{G}_{(j)}^{\gamma,0},(k-1)/2+s+2)\otimes(\mathcal{G}_{(j)}^{\gamma,0}Y_{(j)}^{\gamma,0},0)\otimes(\mathcal{G}^{\gamma},1)\big]\\
&+\mathfrak{g}\big[(Y_j^{\gamma},4s+5)\otimes(\mathcal{G}_{(j)}^{\gamma,0},2s+2)\otimes(\mathcal{G}_{(j)}^{\gamma,0}Y_{(j)}^{\gamma,0},0)\otimes(\mathcal{G}^{\gamma},2)\big]\Big].
	\end{split}
\end{equation*}
Actually, one may draw several observations from the above expression. First, the expansion of $\mathfrak{i}_{(j)}(Y_j^{\gamma})$ can be divided into two groups according to whether the parameter $\mathrm{k}_4$ in $(\mathcal{G}^{\gamma},\mathrm{k}_4)$ is zero or not. Second, we can find that the parameters $\mathrm{k}_1$ in $(Y_j^{\gamma},\mathrm{k}_1)$ are all odd numbers. Following the notation in the expansion of $\mathfrak{f}_{(j)}(Y_j^{\gamma})$, we can write
\begin{equation*}
    \mathfrak{i}_{(j)}(Y_j^{\gamma})=\mathfrak{i}_{(j)}^1(0)+\mathfrak{i}_{(j)}^r(0),
\end{equation*}
where
\begin{equation*}
\begin{split}
    \mathfrak{i}_{(j)}^1(0):=&\operatorname{Im}\Big[\sum_{k=1,o}^{2s+1}\mathfrak{g}\big[(Y_j^{\gamma},k)\otimes(\mathcal{G}_{(j)}^{\gamma,0},(k-1)/2+2)\otimes(\mathcal{G}_{(j)}^{\gamma,0}Y_{(j)}^{\gamma,0},0)\otimes(\mathcal{G}^{\gamma},0)\big]\\
    &+\sum_{k=5,o}^{4s+1}\mathfrak{g}\big[(Y_j^{\gamma},k)\otimes(\mathcal{G}_{(j)}^{\gamma,0},k-1)\otimes(\mathcal{G}_{(j)}^{\gamma,0}Y_{(j)}^{\gamma,0},0)\otimes(\mathcal{G}^{\gamma},0)\big]\Big],
\end{split}
\end{equation*}
\begin{equation*}
    \begin{split}
        \mathfrak{i}_{(j)}^r(0):=&\operatorname{Im}\Big[\sum_{k=0,e}^{2s}2\mathfrak{g}\big[(Y_j^{\gamma},k+2s+3)\otimes(\mathcal{G}_{(j)}^{\gamma,0},k/2+s+2)\otimes(\mathcal{G}_{(j)}^{\gamma,0}Y_{(j)}^{\gamma,0},0)\otimes(\mathcal{G}^{\gamma},1)\big]\\
        &+\mathfrak{g}\big[(Y_j^{\gamma},4s+5)\otimes(\mathcal{G}_{(j)}^{\gamma,0},2s+2)\otimes(\mathcal{G}_{(j)}^{\gamma,0}Y_{(j)}^{\gamma,0},0)\otimes(\mathcal{G}^{\gamma},2)\big]\Big].
    \end{split}
\end{equation*}

Now, by the expansion of $\mathfrak{f}_{(j)}(Y_j^{\gamma})$ and $\mathfrak{i}_{(j)}(Y_j^{\gamma})$, we have
\begin{equation}\label{eq_prf_Greenfunctioncomparison_expansion_f}
	\begin{split}
        \mathfrak{F}_{(j)}(Y^{\gamma}_j)=&\mathfrak{f}_{(j)}(Y^{\gamma}_j)\mathfrak{i}_{(j)}(Y^{\gamma}_j)\\
        =&\big(\mathfrak{f}_{(j)}(0)+\mathfrak{f}_{(j)}^1(0)+\mathfrak{f}_{(j)}^r(0)+F_{res}^{(s_l+2)}\big)\times\big(\mathfrak{i}_{(j)}^1(0)+\mathfrak{i}_{(j)}^r(0)\big)\\
        =&(\mathfrak{f}_{(j)}(0)+\mathfrak{f}_{(j)}^1(0))\times\mathfrak{i}_{(j)}^1(0)+\big(\mathfrak{f}_{(j)}(0)+\mathfrak{f}_{(j)}^1(0)+\mathfrak{f}_{(j)}^r(0)+F_{res}^{(s_l+2)}\big)\times\mathfrak{i}_{(j)}^r(0)\\
        :=&\mathfrak{M}_{(j)}+\mathfrak{R}_{(j)},
	\end{split}
\end{equation}
where we summarised the expansion of $\mathfrak{F}_{(j)}(Y_j^{\gamma})$ into two parts in the last step, the main part $\mathfrak{M}_{(j)}$ and the residual part $\mathfrak{R}_{(j)}$.

Now we consider the residual term $\mathfrak{R}_{(j)}$ first. As one can see from the expression of $(\mathcal{E}_1)_{ij}$, $\mathfrak{R}_{(j)}$ is the total error, that is
\begin{equation*}
    (\mathcal{E}_{11})_{ij}^{Res}:=\mathbbm{1}(\Psi_{ij}=0)\mathbb{E}_{\Psi}\big[\mathfrak{R}_{(j)}\big(L_{ij}\rho_j^{-1}-\frac{\gamma t^{1/2}w_{ij}}{(1-\gamma^2)^{1/2}}\big)\big].
\end{equation*}
In order to control this error term, we need to prepare several prior knowledge. The first concerned thing is the rough bound for  $[\mathcal{G}_{(j)}^{\gamma,0}]_{ab}$. We will establish a perturbation bound from $[\mathcal{G}_{(j)}^{\gamma}]_{ab}$ via the resolvent expansion of $\mathcal{L}(Y^{\gamma},z)$. Observe from \eqref{eq_prf_Greenfunction_resolvent_expansion} that $[\mathcal{G}^{\gamma}]_{ab}=[\mathcal{G}_{(j)}^{\gamma,0}]_{ab}+\big[\mathcal{G}^{\gamma}(-Y^{\gamma}_j(Y^{\gamma}_j)^{*}\mathcal{G}_{(j)}^{\gamma,0})\big]_{ab}$ and
	\begin{equation*}
		\big[\mathcal{G}^{\gamma}Y^{\gamma}_j(Y^{\gamma}_j)^{*}\mathcal{G}_{(j)}^{\gamma,0}\big]_{ab}\lesssim \big[\mathcal{G}^{\gamma}\|Y^{\gamma}_j(Y^{\gamma}_j)^{*}\mathcal{G}_{(j)}^{\gamma,0}\|\big]_{ab}\lesssim[\mathcal{G}^{\gamma}]_{ab}|(Y^{\gamma}_j)^{*}\mathcal{G}_{(j)}^{\gamma,0}Y^{\gamma}_j|.
	\end{equation*}
	By Lemma \ref{lem_largedeviation_H}, in the quadratic form $(Y^{\gamma}_j)^{*}\mathcal{G}_{(j)}^{\gamma,0}Y^{\gamma}_j$,  we obtain that
	\begin{equation*}
		|\mathfrak{ly}_{j}^{*}\mathcal{G}_{(j)}^{\gamma,0}\mathfrak{ly}_j|\prec n^{-c}\max_{a,b}|[\mathcal{G}_{(j)}^{\gamma,0}]_{ab}|.
	\end{equation*}
	Besides, we have $|\mathfrak{hy}_{j}^{*}\mathcal{G}_{(j)}^{\gamma,0}\mathfrak{hy}_j|\lesssim C_1\max_{a,b}|[\mathcal{G}_{(j)}^{\gamma,0}]_{ab}|$ for some $C_1>0$. Thereafter, under $\Omega$, we use the above estimates to obtain that
 \begin{equation*}
		[\mathcal{G}_{(j)}^{\gamma,0}]_{ab}\lesssim (C_1+\mathrm{O}_{\prec}(n^{-c}))[\mathcal{G}^{\gamma}]_{ab}.
	\end{equation*}
 Repeating the above procedure we can control the elements in $\mathcal{L}(Y_{(j)}^{\gamma,0},z)$ from $\mathcal{L}(Y^{\gamma},z)$ on $\Omega$. Therefore, all the Green function entries, say $[\mathcal{G}_{(j)}^{\gamma,0}]_{ab}$, $[\mathcal{G}^{\gamma}]_{ab}$ are finite. Secondly, we observe the following deterministic inequality for any $k_2,\dots,k_q\ge2$,
	\begin{equation}\label{eq_prf_greenfunctioncomparison_deterministicbound_ly}
		\begin{split}
			|\sum_{u_1\neq u_2\neq\cdots\neq u_q}y_{u_1j}^{k_1}y_{u_2j}^{k_2}\cdots y_{u_qj}^{k_q}|&\lesssim|\sum_{u_1}y_{u_1j}^2|\cdots|\sum_{u_q}y_{u_qj}^2|\lesssim 1,
		\end{split}
	\end{equation}
 and
 \begin{equation*}
    |\sum_{u_1\neq u_2}y_{u_1j}y_{u_2j}|\lesssim|\sum_{u_1}y_{u_1j}^2|^{1/2}|\sum_{u_2}y_{u_2j}^2|^{1/2}\lesssim1
 \end{equation*}
since $(\sum_iy_{ij}^2)^l\le1$ for any integer $l>0$ and $|y_{ij}|\le1$. Moreover, we can easily obtain the deterministic bound $\mathbbm{1}(\Omega)|hy_{ij}|\lesssim 1$ and $\mathbbm{1}(\Omega)|ly_{ij}|\lesssim n^{-\epsilon_{\alpha}}\log^{c}n$.

 Now, with the above preparation, we are ready to estimate $|(\mathcal{E}_{11})_{ij}^{Res}|$. Under the event $\Omega$, there exists a large constant $K_2>0$ such that
	\begin{equation*}
		\begin{split}
            |(\mathcal{E}_{11})_{ij}^{Res}|\le&\mathbbm{1}(\psi_{ij}=0)\mathbb{E}_{\Psi}\Big[\Big|\mathfrak{R}_{(j)}\big(L_{ij}\rho_j^{-1}-\frac{\gamma t^{1/2}w_{ij}}{(1-\gamma^2)^{1/2}}\big)\Big|\cdot\mathbbm{1}(\Omega)\Big]\\
			&+\mathbbm{1}(\psi_{ij}=0)\mathbb{E}_{\Psi}\Big[\Big|\mathfrak{R}_{(j)}\big(L_{ij}\rho_j^{-1}-\frac{\gamma t^{1/2}w_{ij}}{(1-\gamma^2)^{1/2}}\big)\Big|\cdot\mathbbm{1}(\Omega^c)\Big].
		\end{split}
	\end{equation*}
        Note that the second term on the right hand side has the upper bound $n^{K_2}Q_0\cdot\mathbbm{1}(\psi_{ij}=0)$. For the first term, we use the following key inputs discussed above:
        \begin{itemize}
            \item [(1)] All the Green function entries have the upper bound with order at most $\mathrm{O}(1)$.
            \item [(2)] The convergence rate comes from the cell $((Y_j^{\gamma})^{*}\mathcal{G}_{(j)}^{\gamma,0}Y_j^{\gamma})$ in the expansion of $\mathcal{G}^{\gamma}Y^{\gamma}$, where we have $|\mathfrak{ly}_j^{*}\mathcal{G}_{(j)}^{\gamma,0}\mathfrak{ly}_j|\prec n^{-c}$.
            \item [(3)] We need to deal with the remaining unpleasant quantities such as
            $$h_{u_1j}^{k_1}\dots h_{u_qj}^{k_q}\rho_j^{-(k_1+\dots+k_q)}[\mathcal{G}_{(j)}^{\gamma,0}]_{u_1u_1}\dots [\mathcal{G}_{(j)}^{\gamma,0}]_{u_qu_q}.$$
            \item [(4)] For the rest quantities, we can use the estimate \eqref{eq_prf_greenfunctioncomparison_deterministicbound_ly}.
        \end{itemize}
       Recall that $Y_j^{\gamma}=\mathfrak{ly}_j+\mathfrak{hy}_j$ and the leading terms in $(Y_j^{\gamma})^{*}\mathcal{G}_{(j)}^{\gamma,0}Y_j^{\gamma}$ are $\mathfrak{ly}_j^{*}\mathcal{G}^{\gamma,0}_{(j)}\mathfrak{ly}_j$ and $\mathfrak{hy}_j^{*}\mathcal{G}^{\gamma,0}_{(j)}\mathfrak{hy}_j$. We decompose $\mathfrak{R}_{(j)}$ into $\mathfrak{R}_{(j)}=\mathfrak{R}_{(j),good}+\mathfrak{R}_{(j),bad}$ based on the cell $(Y_j^{\gamma})^{*}\mathcal{G}_{(j)}^{\gamma,0}Y_j^{\gamma}$. In $\mathfrak{R}_{(j),good}$, there are enough light-tailed quadratic form ensembles $|\mathfrak{ly}_j^{*}\mathcal{G}_{(j)}^{\gamma,0}\mathfrak{ly}_j|\prec n^{-c}$.  Therefore, choosing large enough $s>0$ with $k\ge 2s+2$ one has for $\mathfrak{R}_{(j),good}$,
	\begin{equation*}
		\begin{split}
			&\mathbbm{1}(\psi_{ij}=0)\mathbb{E}_{\Psi}\Big[\Big|\mathfrak{R}_{(j),good}\big(L_{ij}\rho_j^{-1}-\frac{\gamma t^{1/2}w_{ij}}{(1-\gamma^2)^{1/2}}\big)\Big|\cdot\mathbbm{1}(\Omega)\Big]\\
			&\le\frac{n^{C_1\epsilon}\log^{(s+1)c}n\mathbbm{1}(\psi_{ij}=0)}{t^{s+5+(s+1)(1+s_l)}n^{1/2+\epsilon_l+(2s+2)\epsilon_{\alpha}}}\lesssim n^{-3}\cdot \mathbbm{1}(\psi_{ij}=0),
		\end{split}
	\end{equation*}
 where we used the deterministic bound for $L_{ij}\rho_j^{-1}$ and $\gamma t^{1/2}w_{ij}/(1-\gamma^2)^{1/2}$.

Next, we turn to $\mathfrak{R}_{(j),bad}$. For simplicity, we only discuss the worst situation in this case. It turns out that the leading ensembles in $\mathfrak{R}_{(j),bad}$ will take the form $$|h_{u_1j}^{k_1}h_{u_2j}^{k_2}\dots h_{u_qj}^{k_q}\rho_j^{-(k_1+\dots+k_q)}[\mathcal{G}_{(j)}^{\gamma,0}]_{u_1u_1}\dots [\mathcal{G}_{(j)}^{\gamma,0}]_{u_qu_q}|$$ with $k=k_1+k_2+\dots+k_q=s$. One may see that at this time there is no enough light-tailed element ($ly_{ij}$) producing sufficiently fast convergence rate. However, we observe from the expression in $\mathfrak{i}^r_{(j)}(0)$ that there is be at least one additional term $[\mathcal{G}_{(j)}^{\gamma,0}Y_j^{\gamma}]_{aj}$ accompanying the cell $((Y_j^{\gamma})^{*}\mathcal{G}_{(j)}^{\gamma,0}Y_j^{\gamma})^k$. Then, either one can extract one more $L_{ij}\rho_j^{-1}$ factor from $[\mathcal{G}_{(j)}^{\gamma,0}Y_j^{\gamma}]_{aj}$ to obtain that
\begin{equation*}
	\begin{split}
		&\mathbbm{1}(\psi_{ij}=0)\mathbb{E}_{\Psi}\Big[\Big|L_{ij}\rho_j^{-1}[\mathcal{G}_{(j)}^{\gamma,0}]_{ai}h_{u_1j}^{k_1}h_{u_2j}^{k_2}\dots h_{u_qj}^{k_q}\rho_j^{-(k_1+\dots+k_q)}\\
&\qquad \qquad \cdot [\mathcal{G}_{(j)}^{\gamma,0}]_{u_1u_1}\dots [\mathcal{G}_{(j)}^{\gamma,0}]_{u_qu_q}\big(L_{ij}\rho_j^{-1}-\frac{\gamma t^{1/2}w_{ij}}{(1-\gamma^2)^{1/2}}\big)\Big|\cdot\mathbbm{1}(\Omega)\Big]\\
		&\lesssim\mathbbm{1}(\psi_{ij}=0)\mathbb{E}_{\Psi}\Big[|\big(h_{u_1j}\rho_j^{-1}\big)^{k_1}|\cdot(L^2_{ij}\rho_j^{-2}+|\frac{\gamma t^{1/2}w_{ij}}{(1-\gamma^2)^{1/2}}\cdot L_{ij}\rho_j^{-1}|)\cdot\mathbbm{1}(\Omega)\Big]\\
		&\lesssim\mathbbm{1}(\psi_{ij}=0)\Big(\mathbb{E}_{\Psi}\Big[|(h_{u_1j}\rho_j^{-1})^{k_1}|\cdot L^2_{ij}\rho_j^{-2}\cdot\mathbbm{1}(\Omega)\Big]+\mathbb{E}_{\Psi}|(h_{u_1j}\rho_j^{-1})^{k_1}\cdot\mathrm{O}_{\prec}(n^{-1})\Big)\\
        &\lesssim n^{-\alpha/2+1-\epsilon_y},
    \end{split}
\end{equation*}
where in the second step, we have used the fact that the random variables $L_{ij}\rho_j^{-1}$ and $-t^{1/2}w_{ij}$ are independent with second moment matching, and we choose $k_1\ge 4$ which is ensured by the fact that there are at most finitely many components of $Y^{\gamma}_j$ being heavy tailed with high probability. Or we get the additional $\sum_{k\neq i}[\mathcal{G}^{\gamma}]_{ak}y_{kj}$ such that
\begin{equation*}
    \sum_{ij}\mathbbm{1}(\psi_{ij}=0)\mathbb{E}_{\Psi}\big[\sum_{k\neq i}|[\mathcal{G}^{\gamma}]_{ik}y_{kj}L_{ij}\rho_j^{-1}(h_{u_1j}\rho_j^{-1})^{k_1}\big]\lesssim n^{-\alpha/2+3/2-\epsilon_y}.
\end{equation*}
Then using a parallel argument for each item in $\mathfrak{R}_{(j),bad}$ and choosing $s_1=2s+3$, $s_l=5$, $\epsilon<\epsilon_l/8$, $s>C_0/4+2/\epsilon_{\alpha}$ and noticing $t\gg n^{-\epsilon_{\alpha}/8}$, we have
\begin{equation*}
    \mathbbm{1}(\psi_{ij}=0)\mathbb{E}_{\Psi}\Big[\Big|\mathfrak{R}_{(j),bad}\big(L_{ij}\rho_j^{-1}-\frac{\gamma t^{1/2}w_{ij}}{(1-\gamma^2)^{1/2}}\Big)\Big|\cdot \mathbbm{1}(\Omega)\big]\lesssim n^{-\alpha/2+3/2-\epsilon_y}.
\end{equation*}
Therefore,
\begin{equation*}
    \sum_{i,j}|(\mathcal{E}_{11})_{ij}^{Res}|\lesssim n^{-\alpha/2+3/2-\epsilon_y}\cdot\mathbbm{1}(\psi_{ij}=0)+n^{K_2}Q_0\cdot\mathbbm{1}(\psi_{ij}=0).
\end{equation*}

Now we consider the remainder term in $(\mathcal{E}_{11})_{ij}$,
\begin{equation*}
    \mathbbm{1}(\Psi_{ij}=0)\mathbb{E}\big[\mathfrak{M}_{(j)}\big(L_{ij}\rho_j^{-1}-\frac{\gamma t^{1/2}w_{ij}}{(1-\gamma^2)^{1/2}}\big)\big].
\end{equation*}
Recall the notation in the assemble patterns classification in Definition \ref{def_prf_Greenfunctioncomparison_encode_beta}. We will use $\beta^{(K)}_{k_1k_2\dots k_q}$, $\mathfrak{G}^{(K)}_{(j),k_1k_2\dots k_q}$, and $\widetilde{\mathfrak{G}}_{(j),k_1k_2\ldots k_q}^{(K)}$ to label each term in $\mathfrak{M}_{(j)}$. In the sequel, we focus on $\mathfrak{G}_{(j),k_1k_2\ldots k_q}^{(K)}$ since $\widetilde{\mathfrak{G}}_{(j),k_1k_2\ldots k_q}^{(K)}$ can be handled similarly. Therefore, we will always let $\mathfrak{G}_{(j),k_1k_2\ldots k_q}^{(K)}$ denote the assemble patterns of the entries of Green function for simplicity without causing confusion.
Another important property for $\beta^{(K)}_{k_1k_2\dots k_q}$ is that except $\beta_{k_1}^{(K)}$, each additional lower index $u_i, 2\le i\le q$ will be accompanied by the corresponding Green function entries $[\mathcal{G}_{(j)}^{\gamma,0}]_{u_ib}$ or $[\mathcal{G}_{(j)}^{\gamma,0}]_{au_i}$ in only two ways which are from two different kinds of cells in \eqref{eq_prf_Greenfunctioncomparison_expansion_f}, $[(Y_j^{\gamma})^{*}\mathcal{G}_{(j)}^{\gamma,0}Y_j^{\gamma}]_{11}$ and $[\mathcal{G}_{(j)}^{\gamma,0}Y_j^{\gamma}(Y_j^{\gamma})^{*}\mathcal{G}_{(j)}^{\gamma,0}]_{ib}$.
By Lemmas \ref{lem_moment_rates}-\ref{lem_oddmoment_est}, we have
\begin{equation*}
		 \mathbb{E}[(Y_j^{\gamma})^{*}\mathcal{G}_{(j)}^{\gamma,0}Y_j^{\gamma}]_{11}=\mathbb{E}\sum_{s,l}y_{sj}[\mathcal{G}_{(j)}^{\gamma,0}]_{sl}y_{lj}= \mathbb{E}\sum_l[\mathcal{G}_{(j)}^{\gamma,0}]_{ll}\mathbb{E}y_{lj}^2+\mathbb{E}\sum_{s\neq l}[\mathcal{G}_{(j)}^{\gamma,0}]_{sl}\mathbb{E}y_{sj}y_{lj}\asymp 1
\end{equation*}
and
\begin{equation*} \mathbb{E}[\mathcal{G}_{(j)}^{\gamma,0}Y_j^{\gamma}(Y_j^{\gamma})^{*}\mathcal{G}_{(j)}^{\gamma,0}]_{ib}
=\mathbb{E}[\mathcal{G}_{(j)}^{\gamma,0}\mathcal{G}_{(j)}^{\gamma,0}]_{ib}\mathbb{E}y_{lj}^2+\mathbb{E}\sum_{l\neq s}[\mathcal{G}_{(j)}^{\gamma,0}]_{is}[\mathcal{G}_{(j)}^{\gamma,0}]_{lb}\mathbb{E}y_{sj}y_{lj}\lesssim \mathrm{o}(1).
\end{equation*}
One may easily see from the above expression that, the typical rate of the first term $[(Y_j^{\gamma})^{*}\mathcal{G}_{(j)}^{\gamma,0}Y_j^{\gamma}]_{11}$ is larger than the second one $[\mathcal{G}_{(j)}^{\gamma,0}Y_j^{\gamma}(Y_j^{\gamma})^{*}\mathcal{G}_{(j)}^{\gamma,0}]_{ib}$ after taking expectation. Then, it is practically simple to just consider the cell $[(Y_j^{\gamma})^{*}\mathcal{G}_{(j)}^{\gamma,0}Y_j^{\gamma}]_{11}$ in \eqref{eq_prf_Greenfunctioncomparison_expansion_f} since the rest ones are smaller after taking expectation.

In the sequel, we consider three cases for $\mathfrak{M}_{(j)}:=\mathfrak{M}_{(j)}^{I}+\mathfrak{M}_{(j)}^{II}+\mathfrak{M}_{(j)}^{III}$.
\begin{itemize}
	\item [\textbf{Case I.}] \textbf{For $\beta^{(K)}_{k_1k_2\dots k_q}$, $k_1\ge 1$ and all $k_2,\dots,k_q$ are even numbers.}

Firstly, notice that in this case, we can directly use
the moment estimation in Lemma \ref{lem_moment_rates}. Let $f^{(k)}_{(j)}(0)=C_k(k!)^{-1}F^{(1+k)}(\operatorname{Im}[\mathcal{G}_{(j)}^{\gamma,0}]_{ab})$ for some non-zero constant $C_k$. According to \eqref{eq_prf_greenfunctioncomparison_deterministicbound_ly}, $a\in\mathrm{T}_r$ or $b\in\mathrm{T}_c$ with $i,j$ in $\mathrm{I}_r/\mathrm{T}_r$ or $\mathrm{I}_c/\mathrm{T}_c$, so one can easily obtain the following estimation,
\begin{equation*}
    f_{(j)}^{(k)}(0)\cdot\mathbbm{1}(\psi_{ij}=0)\cdot\mathbbm{1}(\Omega)\lesssim\frac{n^{(C_0+s+5+s_l)\epsilon}}{t^{s+5+s_l}}.
\end{equation*}

Now, ignoring the notation $\operatorname{Im}$ and invoking the Definition \ref{def_prf_Greenfunctioncomparison_encode_beta}, we may rewrite $(\mathcal{E}_{11})_{ij}$ as
	\begin{equation*}
		\begin{split}
			&\mathbbm{1}(\Psi_{ij}=0)\mathbb{E}_{\Psi}\big[\big|\mathfrak{M}^{I}_{(j)}\big(L_{ij}\rho_j^{-1}-\frac{\gamma t^{1/2}w_{ij}}{(1-\gamma^2)^{1/2}}\big)\big|\big]\\
            \lesssim&\sum_{K=2s+1,o}^{2s\cdot s_l+4s+1}\sum_{k=1,o}^{K}\mathbbm{1}(\psi_{ij}=0)\mathbb{E}_{\Psi}\Big[(ly_{ij})^k\beta_{0k_2\dots k_q}^{(K-k)}\mathfrak{F}_{(j),kk_2\ldots k_q}^{(K)}f_{(j)}^{(k)}(0)\Big(L_{ij}\rho_j^{-1}-\frac{\gamma t^{1/2}w_{ij}}{(1-\gamma^2)^{1/2}}\Big)\Big]\\
			=&:\sum_{K=2s+1,o}^{2s\cdot s_l+4s+1}\sum_{k=1,o}^{K}(\mathcal{E}_{11,k})_{ij},
		\end{split}
	\end{equation*}
where the summation is over all possible odd $K\in\mathbb{N}$ with $2s+1\le K\le 2s(4s+1)\cdot s_l$. For simplification, we fix $K$ in the following discussion.

In the sequel, we consider $(\mathcal{E}_{11,k})_{ij}$ for each $k$.
	
	\textbf{Case $k\ge 5$.} Expanding the term $(ly_{ij})^k$ and using the symmetric condition for $w_{ij}$, we have
	\begin{equation*}
		\begin{split}
			|(\mathcal{E}_{11,k})_{ij}|&\lesssim \sum_{l_1+l_2+l_3=k}\mathbb{E}_{\Psi}\Big[\Big(|L_{ij}\rho_j^{-1}|^{l_1+1}|t^{1/2}w_{ij}|^{l_2}|M_{ij}\rho_j^{-1}|^{l_3}\\
			&+|L_{ij}\rho_j^{-1}|^{l_1}|t^{1/2}w_{ij}|^{l_2+1}|M_{ij}\rho_j^{-1}|^{l_3}\Big)\beta_{0k_2\dots k_q}^{(K-k)}\mathfrak{F}_{(j),kk_2\ldots k_q}^{(K)}f_{(j)}^{(k)}(0)|\Big]\cdot\mathbbm{1}(\psi_{ij}=0)
			\\
			 &\lesssim\frac{t\mathbbm{1}(\psi_{ij}=0)\mathbb{E}_{\Psi}(|f_{(j)}^{(k)}(0)|\big|\mathbbm{1}(\Omega))}{n^{1+\alpha/2}}+\frac{t\mathbbm{1}(\psi_{ij}=0)\mathbb{E}_{\Psi}(|f_{(j)}^{(k)}(0)|\big|\mathbbm{1}(\Omega^c))}{n^{1+\alpha/2}},
		\end{split}
	\end{equation*}
where  we have used the fact that $L_{ij}$ can not happen with $M_{ij}$ together, and we have also deployed the deterministic bound $|L_{ij}\rho_j^{-1}|\lesssim \mathrm{O}(n^{-1/2-\epsilon_l})$. Then, we find the leading term is $\mathbb{E}_{\Psi}\big(|L_{ij}\rho_j^{-1}|^{k}|tw_{ij}^2|\big)$.
	Similarly to control $(\mathcal{E}_{11})_{ij}^{Res}$, we have on the event $\Omega^c$,
	\begin{equation*}
		\frac{t\mathbbm{1}(\psi_{ij}=0)\mathbb{E}_{\Psi}(|f_{(j)}^{(k)}(0)|\big|\mathbbm{1}(\Omega^c))}{n^{1+\alpha/2}}\lesssim n^{K_3}Q_0\mathbbm{1}(\psi_{ij}=0)
	\end{equation*}
	for some large $K_3>0$. On the event $\Omega$, we recover $f_{(j)}^{(k)}(Y_j^{\gamma})$ from $f_{(j)}^{(k)}(0)$ by Taylor's expansion for $F^{(1+k)}(\operatorname{Im}[\mathcal{G}_{(j)}^{\gamma,0}]_{ab})$ around $\operatorname{Im}[\mathcal{G}^{\gamma}]_{ab}$. Thus, we get
	\begin{equation*}
		\begin{split}	
&\frac{t\mathbbm{1}(\psi_{ij}=0)\mathbb{E}_{\Psi}(|f_{(j)}^{(k)}(0)|\big|\mathbbm{1}(\Omega))}{n^{1+\alpha/2}}\\
&\lesssim \frac{n^{(s+4+k(1+s_l))\epsilon}}{t^{s+3+k(1+s_l)}n^{1+\alpha/2}}\sum_{r=1}^{s_r}\mathbb{E}_{\Psi}\big[F^{(1+k+r)}(\operatorname{Im}[\mathcal{G}^{\gamma}]_{ab})\Delta^r([\mathcal{G}_{(j)}^{\gamma,0}]_{ab},[\mathcal{G}^{\gamma}]_{ab})\big],
		\end{split}
	\end{equation*}
	where $\Delta([\mathcal{G}_{(j)}^{\gamma,0}]_{ab},[\mathcal{G}^{\gamma}]_{ab})$ is defined as the difference between $\operatorname{Im}[\mathcal{G}_{(j)}^{\gamma,0}]_{ab}$ and $\operatorname{Im}[\mathcal{G}^{\gamma}]_{ab}$. It is easy to see from the resolvent expansion \eqref{eq_prf_Greenfunction_resolvent_expansion} that the difference between $[\mathcal{G}^{\gamma}]_{ab}$ and $[\mathcal{G}_{(j)}^{\gamma,0}]_{ab}$ is at most $n^{\epsilon}$ under $\Omega$. From \eqref{eq_prf_greenfunctioncomparison_deterministicbound_ly}, we have the upper bound $\Delta^r(Y^{\gamma}_j,[\mathcal{G}^{\gamma}]_{ab})\lesssim n^{r\epsilon}$. Choosing $\epsilon<\epsilon_l/(s_l+s_r)$ with $s=s_l=5$ and $s_r=3$, we conclude that
	\begin{eqnarray*}
		 |(\mathcal{E}_{11,k})_{ij}|\lesssim\frac{\mathbbm{1}(\psi_{ij}=0)}{n^{2+\epsilon_{\alpha}/2}}\sum_{r=1}^{3}\mathbb{E}_{\Psi}\big[|F^{(1+k+r)}(\operatorname{Im}[\mathcal{G}^{\gamma}]_{ab})|\big]+\frac{\mathbbm{1}(\psi_{ij}=0)}{n^3}+n^{K_3}Q_0\mathbbm{1}(\psi_{ij}=0).
	\end{eqnarray*}
	
	\textbf{Case $k=3$.} Firstly, $(\mathcal{E}_{11,3})_{ij}$ can be expanded as
	\begin{equation*}
		\begin{split}
			|(\mathcal{E}_{11,3})_{ij}|\asymp&\mathbb{E}_{\Psi}\Big[\big((L_{ij}\rho_j^{-1})^4+t(L_{ij}\rho_j^{-1})^2w_{ij}^2+t^2w_{ij}^4
			+(L_{ij}\rho_j^{-1})^2(M_{ij}\rho_j^{-1})^2+tw_{ij}^2(M_{ij}\rho_j^{-1})^2\big)\\
			&\qquad\times \beta_{0k_2\dots k_q}^{(K-3)}\mathfrak{F}_{(j),3k_2\ldots k_q}^{(K)}f_{(j)}^{(3)}(0)\Big]\cdot\mathbbm{1}(\psi_{ij}=0).
		\end{split}
	\end{equation*}
	The leading term above is $tw_{ij}^2(M_{ij}\rho_j^{-1})^2$. Again, similarly to the case $k\ge5$, considering both $\Omega$ and $\Omega^c$, we have
	\begin{equation*}
		\begin{split}
			&\big|\mathbb{E}_{\Psi}\big[tw_{ij}^2(M_{ij}(\operatorname{diag}S)^{-1/2})^2\beta_{0k_2\dots k_q}^{(K-3)}\mathfrak{F}_{(j),3k_2\ldots k_q}^{(K)}f_{(j)}^{(3)}(0)\big]\cdot\mathbbm{1}(\psi_{ij}=0)\big|\\
			&\lesssim\frac{\mathbbm{1}(\psi_{ij}=0)\cdot t|f^{(3)}_{(j)}(0)|\cdot\mathbbm{1}(\Omega)}{n^{2}}+n^{K_4}Q_0\mathbbm{1}(\psi_{ij}=0)
		\end{split}
	\end{equation*}
	for some large constant $K_4>0$. Recovering $f^{(3)}_{(j)}(Y_j^{\gamma})$ from $f_{(j)}^{(3)}(0)$ and considering their error, we have
	\begin{equation*}
		\begin{split}
			 |(\mathcal{E}_{11,3})_{ij}|\lesssim&\frac{tn^{12\epsilon}}{n^2}\sum_{r=1}^3\mathbb{E}_{\Psi}\big[|F^{(4+r)}(\operatorname{Im}[\mathcal{G}^{\gamma}]_{ab})|\big]\cdot\mathbbm{1}(\psi_{ij}=0)\cdot\mathbbm{1}(i\in\mathrm{T}_r,j\in\mathrm{T}_c)\\
			 &+\frac{n^{12\epsilon}}{n^{2}t^7}\sum_{r=1}^3\mathbb{E}_{\Psi}\big[|F^{(4+r)}(\operatorname{Im}[\mathcal{G}^{\gamma}]_{ab})|\big]\cdot(1-\mathbbm{1}(i\in\mathrm{T}_r,j\in\mathrm{T}_c))+n^{K_4}Q_0\mathbbm{1}(\psi_{ij}=0).
		\end{split}
	\end{equation*}
Since $\Psi$ is well configured, one has $\#(1-\mathbbm{1}(i\in\mathrm{T}_r,j\in\mathrm{T_c}))\le n^{2/\alpha-\epsilon_0}$. Thus, we conclude  this case.
	
	\textbf{Case $k=1$.} A direct calculation leads to
	\begin{equation*}
		|(\mathcal{E}_{11,1})_{ij}|=\gamma \mathbb{E}_{\Psi}\Big[\beta_{0k_2\dots k_q}^{(K)}\mathfrak{F}_{(j),1k_2\ldots k_q}^{(K)}\big(L_{ij}^2\cdot\rho^{-2}_j-tw_{ij}^2\big)f_{(j)}^{(1)}(0)\Big]\cdot\mathbbm{1}(\psi_{ij}=0).
	\end{equation*}
	For $\mathfrak{F}_{(j),k_1k_2\ldots k_q}^{(K)}$ part, we have
	\begin{equation*}
		\mathfrak{F}_{(j),k_1k_2\ldots k_q}^{(K)}\lesssim C(\sum_{l_2}[\mathcal{G}_{(j)}^{\gamma,0}]_{l_2l_2})\times (\sum_{l_3}[\mathcal{G}_{(j)}^{\gamma,0}]_{l_3l_3})\times\dots\times (\sum_{l_q}[\mathcal{G}_{(j)}^{\gamma,0}]_{l_ql_q})
	\end{equation*}
	for some constant $C$. Then, using the trivial bound of $[\mathcal{G}^{\gamma,0}_{(j)}]_{ab}$'s, we have $|\mathfrak{F}_{(j),k_1k_2\ldots k_q}^{(K)}|=\mathrm{O}(n^{q-1})$. Intuitively, after taking the expectation for $\beta_{0k_2\dots k_q}^{(K)}\big(L_{ij}^2\cdot\rho^{-2}_j-tw_{ij}^2\big)$, the $\beta_{0k_2\dots k_q}^{(K)}$ term will give rates of order at least $\mathrm{O}(n^{-q+1})$ and the moment matching condition on $L_{ij}^2\rho_j^{-1}$ and $tw_{ij}^2$ will give sufficiently fast convergence rate to cancel $\sum_{ij}$. However, one cannot directly use the moment matching condition here due to the correlation between $\beta_{0k_2\dots k_q}^{(K)}$ and $L_{ij}^2\cdot\rho^{-2}_j$. To deal with this issue, we deploy the cumulant expansion trick. We take the cumulant expansion for $L_{ij}^2\rho_j^{-2}$ and $tw_{ij}^2$ respectively as
	\begin{equation*}
		\begin{split}
			&\mathbb{E}_{\Psi}\Big[\beta_{0k_2\dots k_q}^{(K)}(L_{ij}^2\cdot\rho^{-2}_j)\Big]=\sum_{k=0}^{q-1}\frac{\kappa_{k+1,L}}{k!}\mathbb{E}\Big(\frac{\partial^k\beta_{0k_2\dots k_q}^{(K)}}{(\partial (L_{ij}^2\cdot\rho^{-2}_j))^k}\Big)+\mathbb{E}(r_{q-1}(\beta_{0k_2\dots k_q}^{(K)}(L_{ij}^2\cdot\rho^{-2}_j))),\\
			&\mathbb{E}_{\Psi}\Big[\beta_{0k_2\dots k_q}^{(K)}(tw_{ij}^2)\Big]=\sum_{k=0}^{q-1}\frac{\kappa_{k+1,w}}{k!}\mathbb{E}\Big(\frac{\partial^k\beta_{0k_2\dots k_q}^{(K)}}{(\partial (tw_{ij}^2)^k}\Big)+\mathbb{E}(r_{q-1}(\beta_{0k_2\dots k_q}^{(K)}(tw_{ij}^2)))=\kappa_{1,w}\mathbb{E}(\beta_{0k_2\dots k_q}^{(K)}),
		\end{split}
	\end{equation*}
 where $\kappa_{k+1,L}$ and $\kappa_{k+1,w}$ denote the $(k+1)$-th cumulants of $L_{ij}^2\rho_j^{-2}$ and $tw_{ij}^2$, respectively.
	Then, one can easily calculate that $\kappa_{k+1,L}=\mathrm{O}(n^{-(1+\epsilon_l)(k+1)})$ and
	\begin{gather*}
		\mathbb{E}\Big(\frac{\partial^k\beta_{0k_2\dots k_q}^{(K)}}{(\partial (L_{ij}^2\cdot\rho^{-2}_j))^k}\Big)=\mathbb{E}\Big(\frac{\partial^k\beta_{0k_2\dots k_q}^{(K)}}{\partial (L_{ij}^2)^k}\cdot (\frac{\partial L_{ij}^2\cdot\rho^{-2}_j}{\partial L_{ij}^2})^{-k}\Big)=\mathrm{O}(n^{-q+1}),\\
		\mathbb{E}(r_{q-1}(\beta_{0k_2\dots k_q}^{(K)}(L_{ij}^2\cdot\rho^{-2}_j)))\lesssim \mathbb{E}(L_{ij}^{2(q+1)}\rho_j^{-2(q+1)})\sup|(\beta_{0k_2\dots k_q}^{(K)})^{q+1}|\le\mathrm{O}(n^{-(1+\epsilon_l)(q+1)}),
	\end{gather*}
	where in the last inequality, we have used the deterministic bound $|(\beta_{0k_2\dots k_q}^{(K)})^{q+1}|\le 1$.
	
	The moment matching condition and the fact that $t=n\mathbb{E}\big((1-\psi_{ij})^2(1-\chi_{ij})^2l_{ij}^2\rho_j^{-2}\big)$ imply
	\begin{equation*}
		|\kappa_{1,L}-\kappa_{1,w}|=|\mathbb{E}\big((1-\chi_{ij})^2l_{ij}^2\big)-\mathbb{E}(tw_{ij}^2)|\lesssim tn^{-3+2/\alpha-\epsilon_0}.
	\end{equation*}
	Thereafter, following the same procedures as in cases $k\ge5$ and $k=3$, we have
	\begin{equation*}
		 |(\mathcal{E}_{11,1})_{ij}|\lesssim\frac{1}{n^{2+\epsilon_{\alpha}/2}}\sum_{r=1}^{3}\mathbb{E}_{\Psi}\big[|F^{(2+r)}(\operatorname{Im}[\mathcal{G}^{\gamma}]_{ab})|\big]\cdot\mathbbm{1}(\psi_{ij}=0)+n^{-3}\mathbbm{1}(\psi_{ij}=0)+n^{K_5}Q_0\mathbbm{1}(\psi_{ij}=0).
	\end{equation*}
	
	Finally, combining the above results and choosing $\epsilon\le \min\{\epsilon_{\alpha},\epsilon_{l},\epsilon_0\}/100$, we find that after summation over $i,j$,
 \begin{equation*}
     \sum_{ij}\mathbbm{1}(\Psi_{ij}=0)\mathbb{E}_{\Psi}\big[\mathfrak{M}_{(j)}^I(L_{ij}\rho_j^{-1}-\frac{\gamma t^{1/2}w_{ij}}{(1-\gamma^2)^{1/2}})\big]\le\frac{C_{I}}{(1-\gamma^2)^{1/2}}\big(n^{-\epsilon_l+12\epsilon}(\mathfrak{J}_{0,ab}+1)+Q_0n^{K_5})\big),
 \end{equation*}
 which is just \eqref{eq_prf_Greenfunctioncomparison_equivbound}.  We remark that \textbf{Case I} actually corresponds to the setting where $X=(X_{ji})$ is symmetric.
	
	\item [\textbf{Case II.}] \textbf{For $\beta^{(K)}_{k_1k_2\dots k_q}$,  $k_1\ge 1$ and at least one of $k_2,\dots,k_q$ is odd.}
	
	 Firstly in this case, we cannot directly use the moment estimation in Lemma \ref{lem_moment_rates}. However, one may observe that for $k_1\ge3$
	\begin{equation*}
		\mathbb{E}\beta^{(K+1)}_{(k_1+1)k_2\dots k_q}=\mathbb{E}y_{ij}^{k_1+1}y_{1j}^{k_2}\cdots y_{qj}^{k_q}\le\mathbb{E}y_{ij}^{4}y_{1j}^{2}\cdots y_{qj}^{2},
	\end{equation*}
	where we regard $(L_{ij}\rho_j^{-1}-\frac{\gamma t^{1/2}w_{ij}}{(1-\gamma^2)^{1/2}})$ as an extra $y_{ij}$ so that we have the lower script $k_1+1$. Therefore, most of the terms in \textbf{Case II} can be bounded by quantities $\beta^{(2q+1)}_{32\dots 2}$ in \textbf{Case I}. It remains to consider $\beta_{2k_2\dots k_q}^{(K)}$ and $\beta_{1k_2\dots k_q}^{(K)}$. We should be careful with these terms since we cannot easily use the deterministic bound $y_{kj}^2\le 1$ due to the existence of odd moments of $y_{ij}$'s. We first consider $\beta_{2k_2\dots k_q}^{(K)}$. Since $K$ is odd, we can classify it into two categories,
	\begin{itemize}
		\item [(i)] There is at least one $k_l=1$ for $2\le l\le q$.
		\item [(ii)] All $k_l\ge 2$ for $2\le l\le q$, therefore there must exist one odd $k_l$.
	\end{itemize}
	For $(i)$,  the final expectation reads
	\begin{equation*}
		|(\mathcal{E}_{11,2})_{ij}|\sim \mathfrak{F}_{(j),k_1k_2\ldots k_q}^{(K)}\mathbb{E}_{\Psi}[\beta^{(K-3)}_{0k_2\dots k_{l-1}k_{l+1}\dots k_q}y_{lj}(L_{ij}^3\rho_j^{-3}-tL_{ij}\rho_j^{-1}w_{ij}^2)f_{(j)}^{(2)}(0)]\cdot\mathbbm{1}(\psi_{ij}=0).
	\end{equation*}
	Typically, we have the rough estimation that $|\mathfrak{F}_{(j),k_1k_2\ldots k_q}^{(K)}|\le\mathrm{O}(n^{q-1})$. This estimation can be improved since $k_l=1$ results in $\sum_{s}[\mathcal{G}_{(j)}^{\gamma,0}]_{sl}y^{k_s}_{sj}y_{lj}$  in $\mathfrak{F}_{(j),k_1k_2\ldots k_q}^{(K)}$. By the Cauchy-Schwartz inequality and the Wald identity, we obtain  $|\sum_{s}[\mathcal{G}_{(j)}^{\gamma,0}]_{sl}|=\mathrm{O}(\sqrt{n})$. Therefore, we reduce the typical rate of  $|\mathfrak{F}_{(j),k_1k_2\ldots k_q}^{(K)}|$ by at least $\mathrm{O}(\sqrt{n})$, and we have $|\mathfrak{F}_{(j),k_1k_2\ldots k_q}^{(K)}|\le\mathrm{O}(n^{q-3/2})$. Then applying Lemma 3.1 in \cite{gine1997student}, Lemma \ref{lem_moment_rates} and Lemma \ref{lem_oddmoment_est}, we have
	\begin{gather*}
		\begin{split}
			\mathbb{E}_{\Psi}[\beta^{(K-3)}_{0k_2\dots k_{l-1}k_{l+1}\dots k_q}y_{lj}(L_{ij}^3\rho_j^{-3})]&\lesssim \mathbb{E}(y_{lj}L_{ij}\rho_j^{-1})^{1/2}\times \mathbb{E}_{\Psi}[y^3_{lj}L^3_{ij}\rho_j^{-3}(\beta^{(K-3)}_{0k_2\dots k_{l-1}k_{l+1}\dots k_q})^2]^{1/2}\\
			&\lesssim \mathrm{o}(n^{-3/2-(q-1)}),
		\end{split}
	\end{gather*}
	where we have used the deterministic bound $|L_{ij}\rho_j^{-1}|=\mathrm{o}(n^{-1/2})$ and $|y_{kj}|\le 1$. Besides,
	\begin{equation*}
		\begin{split}
			&\mathbb{E}_{\Psi}[\beta^{(K-3)}_{0k_2\dots k_{l-1}k_{l+1}\dots k_q}ty_{lj}(L_{ij}\rho_j^{-1})w_{ij}^2]\\
\lesssim & t\mathbb{E}(w_{ij}^2)\times\mathbb{E}(y_{lj}L_{ij}\rho_j^{-1})^{1/2}\times \mathbb{E}_{\Psi}[\beta^{(2K-4)}_{011k_2k_2\dots k_{l-1}k_{l-1}k_{l+1}k_{l+1}\dots k_qk_q}]^{1/2}\\
			\lesssim & \mathrm{o}(n^{-5/2-(q-1)}).
		\end{split}
	\end{equation*}
	This together with the estimation for $|\mathfrak{F}_{(j),k_1k_2\ldots k_q}^{(K)}|$ implies that $|(\mathcal{E}_{11,2})_{ij}|\lesssim\mathrm{o}(n^{-2})$.
	
	Now, we turn to $(ii)$. A direct calculation shows that
	\begin{equation*}
		|(\mathcal{E}_{11,2})_{ij}|\sim \mathfrak{F}_{(j),k_1k_2\ldots k_q}^{(K)}\mathbb{E}_{\Psi}[\beta^{(K-2)}_{0k_2\dots k_q}(L_{ij}^3\rho_j^{-3}-tL_{ij}\rho_j^{-1}w_{ij}^2)f_{(j)}^{(2)}(0)]\cdot\mathbbm{1}(\psi_{ij}=0).
	\end{equation*}
	For odd $k_l$, there is one summation in $\mathfrak{F}_{(j),k_1k_2\ldots k_q}^{(K)}$ reading $\sum_l[\mathcal{G}_{(j)}^{\gamma,0}]_{li}[\mathcal{G}_{(j)}^{\gamma,0}]_{ll}y_{lj}^{k_l}y_{ij}$. Again, using the Cauchy-Schwartz inequality and the Wald identity, we find that $|\sum_l[\mathcal{G}_{(j)}^{\gamma,0}]_{li}[\mathcal{G}_{(j)}^{\gamma,0}]_{ll}|=\mathrm{O}(\sqrt{n})$. Therefore, we can obtain the improved typical rate $|\mathfrak{F}_{(j),k_1k_2\ldots k_q}^{(K)}|\le\mathrm{O}(n^{q-3/2})$. Thereafter, using Lemma \ref{lem_oddmoment_est} with the deterministic bound $|L_{ij}\rho_j^{-1}|=\mathrm{o}(n^{-1/2})$, one has
	\begin{equation*}
		\begin{split}
			\mathbb{E}_{\Psi}[\beta^{(K-2)}_{0k_2\dots k_q}(L_{ij}^3\rho_j^{-3})]\lesssim \mathbb{E}(L_{ij}\rho_j^{-1})^{2}\times \mathrm{o}(n^{-(q-1)-1/2})\lesssim \mathrm{o}(n^{-3/2-(q-1)}),
		\end{split}
	\end{equation*}
	and
	\begin{equation*}
		\begin{split}
			\mathbb{E}_{\Psi}[\beta^{(K-2)}_{0k_2\dots k_q}ty_{lj}(L_{ij}\rho_j^{-1})w_{ij}^2]\lesssim t\mathbb{E}(w_{ij}^2)\times\mathrm{o}(n^{-(q-1)-1/2})\lesssim \mathrm{o}(n^{-3/2-(q-1)}).
		\end{split}
	\end{equation*}
	Summarising the above calculation, we conclude that $|(\mathcal{E}_{11,2})_{ij}|\lesssim\mathrm{o}(n^{-2})$.
	
	Finally, we consider the remaining case $\beta_{1k_2\dots k_q}^{(K)}$. For this, we have
	\begin{equation*}
		|(\mathcal{E}_{11,2})_{ij}|\sim\gamma \mathbb{E}_{\Psi}\Big[\beta_{0k_2\dots k_q}^{(K)}\big(L_{ij}^2\cdot\rho^{-2}_j-tw_{ij}^2\big)f_{(j)}^{(1)}(0)\Big]\cdot\mathbbm{1}(\psi_{ij}=0).
	\end{equation*}
	The strategy to handle this term is the same as the one in \textbf{Case I}, where we used the cumulant expansion argument combined with the second moment matching condition. The only difference here is that we apply Lemma \ref{lem_oddmoment_est} to approximate $\mathbb{E}_{\Psi}\beta_{0k_2\dots k_q}^{(K-1)}$ if there is at least one $k_l=1$ for $2\le l\le q$ or we use the deterministic bound $|y_{kj}|\le 1$ to ensure that all $k_l$'s are even if there is no $k_l=1$ and then use Lemma \ref{lem_moment_rates}. This gives $|(\mathcal{E}_{11,2})_{ij}|\lesssim\mathrm{o}(n^{-2})$. Then one may follow similar lines in \textbf{Case I} to obtain the desired result. We omit further details here.

 In conclusion, we finally obtain that
 \begin{equation*}
     \sum_{ij}\mathbbm{1}(\Psi_{ij}=0)\mathbb{E}_{\Psi}\big[\mathfrak{M}_{(j)}^{II}(L_{ij}\rho_j^{-1}-\frac{\gamma t^{1/2}w_{ij}}{(1-\gamma^2)^{1/2}})\big]\le\frac{C_{II}}{(1-\gamma^2)^{1/2}}\big(n^{-\epsilon_l}(\mathfrak{J}_{0,ab}+1)+Q_0n^{K_6})\big),
 \end{equation*}
 for some large constant $K_6>0$.
	
	\item [\textbf{Case III.}] \textbf{For $\beta^{(K)}_{k_1k_2\dots k_q}$,  $k_1=0$.}
	
	Notice that at this time we have no extra $y_{ij}$. Therefore, we have
	\begin{equation*}
		(\mathcal{E}_{11,0})_{ij}=\mathfrak{F}_{(j),0k_2\ldots k_q}^{(K)}\mathbb{E}_{\Psi}\Big[\beta_{0k_2\dots k_q}^{(K)}(L_{ij}\rho_j^{-1}-\frac{\gamma t^{1/2}w_{ij}}{(1-\gamma^2)^{1/2}})f_{(j)}(0)\Big]\cdot\mathbbm{1}(\psi_{ij}=0).
	\end{equation*}
	One should also notice that there must exist the term $\sum_{l\neq i}[\mathcal{G}_{(j)}^{\gamma,0}]_{ib}[\mathcal{G}_{(j)}^{\gamma,0}]_{al}$ in $\mathfrak{F}_{(j),k_1k_2\ldots k_q}^{(K)}$, which is accompanied by an additional $y_{lj}$ if $l\neq i$ and $k_l$ is odd for $2\le l\le q$. Then after some calculation, one may find that the leading term in $(\mathcal{E}_{11,0})_{ij}$ is
	\begin{equation*}
		\sum_{l\neq i}[\mathcal{G}_{(j)}^{\gamma,0}]_{ib}[\mathcal{G}_{(j)}^{\gamma,0}]_{al}\mathbb{E}_{\Psi}\big[y_{lj}(L_{ij}\rho_j^{-1}-\frac{\gamma t^{1/2}w_{ij}}{(1-\gamma^2)^{1/2}})f_{(j)}(0)\big]\cdot\mathbbm{1}(\psi_{ij}=0).
	\end{equation*}
	Taking the summation over $i,j$, we have
	\begin{equation*}
		\begin{split}
			&\sum_{i,j}\sum_{l\neq i}[\mathcal{G}_{(j)}^{\gamma,0}]_{ib}[\mathcal{G}_{(j)}^{\gamma,0}]_{al}\mathbb{E}_{\Psi}\big[y_{lj}(L_{ij}\rho_j^{-1}-\frac{\gamma t^{1/2}w_{ij}}{(1-\gamma^2)^{1/2}})f_{(j)}(0)\big]\\
			&\lesssim n\sum_j(\sum_i[\mathcal{G}_{(j)}^{\gamma,0}]^2_{ib})^{1/2}(\sum_l[\mathcal{G}_{(j)}^{\gamma,0}]_{al}^2)^{1/2}\mathbb{E}_{\Psi}(y_{lj}L_{ij}\rho_j^{-1}f_{(j)}(0))\lesssim n^2\mathbb{E}(y_{lj}L_{ij}\rho_j^{-1}) \lesssim \mathrm{o}(1),
		\end{split}
	\end{equation*}
where in the last step we have used Lemma \ref{lem_oddmoment_est} for $\alpha
 \ge 3$. Specifically, following the cumulant expansion for $\mathbb{E}_{\Psi}(y_{lj}L_{ij}\rho_j^{-1})$ w.r.t. $L_{ij}$, one has
 \begin{equation*}
     \begin{split}
         \mathbb{E}_{\Psi}(y_{lj}L_{ij}\rho_j^{-1})\lesssim n^{-\epsilon_l} \mathbb{E}(y_{lj}\rho_j^{-1})+n^{-2\epsilon_l}\mathbb{E}(y_{lj}L_{ij}\rho_j^{-3})\lesssim \mathrm{o}(n^{(1/2-\epsilon_h)(-\alpha+1)-1})
     \end{split}
 \end{equation*}
 by $\mathbb{E}(y_{lj}\rho_j^{-1})\lesssim \kappa_{1,m}\mathbb{E}(\rho_j^{-2})+\mathbb{E}(H_{lj})n^{-1}+\kappa_{2,m}\mathbb{E}(M_{lj}\rho_{j}^{-3})\lesssim n^{(1/2-\epsilon_h)(-\alpha+1)-1}$
which follows from a similar argument in the proof of Lemma \ref{lem_oddmoment_est}. Finally, following the argument in \textbf{Case I}, we obtain that
 \begin{equation*}
     \sum_{ij}\mathbbm{1}(\Psi_{ij}=0)\mathbb{E}_{\Psi}\big[\mathfrak{M}_{(j)}^{III}(L_{ij}\rho_j^{-1}-\frac{\gamma t^{1/2}w_{ij}}{(1-\gamma^2)^{1/2}})\big]\le\frac{C_{II}}{(1-\gamma^2)^{1/2}}\big(n^{-\epsilon_l}(\mathfrak{J}_{0,ab}+1)+Q_0n^{K_7})\big),
 \end{equation*}
 for some large constant $K_7>0$.
\end{itemize}

In the end, combining the results for $\mathfrak{M}_{(j)}^I$-$\mathfrak{M}_{(j)}^{III}$ and $\mathfrak{R}_{(j)}$, we complete the proof of \eqref{eq_prf_Greenfunctioncomparison_equivbound}.   \qed

\subsubsection{Proof for Theorem \ref{thm_greenfuncomp_entrywise_uniformbound}}\label{app_prf_thm_greenfuncomp_entrywise_uniformbound}
In this part, we prove Theorem \ref{thm_greenfuncomp_entrywise_uniformbound} assuming that Theorem \ref{thm_greenfuncomp_entrywise_differror} holds. In the sequel, we only show the case where $z\in\gamma_1\vee\gamma_2$. For any $\delta>0$, we define
\begin{eqnarray*}
&&	\mathfrak{B}_0(\delta_1,\delta_2,z,\Psi):=\mathbb{P}_{\Psi}(\sup_{0\le\gamma\le1}\sup_{a,b\in[n] \atop
		a,b\in\mathrm{T}_{r}}|z^{1/2}\mathfrak{R}_{ab}([\mathcal{G}^{\gamma}(z)]_{ab}-\delta_{ab}b_t(z)m^{(t)}(\zeta))|>n^{-\delta_1},\\
&&\qquad \qquad\qquad \qquad\qquad \qquad	\sup_{0\le\gamma\le1}\sup_{a,b\in[n]\atop a\in\mathrm{I}_{r}\vee b\in\mathrm{I}_{r}}|z^{1/2}\mathfrak{R}_{ab}[\mathcal{G}^{\gamma}(z)]_{ab}|>n^{-\delta_2}),\\
&&	\mathfrak{B}_1(\delta_1,\delta_2,z,\Psi):=\mathbb{P}_{\Psi}(\sup_{0\le\gamma\le1}\sup_{u,v\in[p]\atop
		 u,v\in\mathrm{T}_{c}}|z^{1/2}\mathfrak{C}_{uv}([G^{\gamma}(z)]_{uv}-\delta_{uv}(1+t\underline{m}_t(z))\underline{m}^{(t)}(\zeta))|>n^{-\delta_1},\\
&&\qquad \qquad\qquad \qquad	\qquad \qquad\sup_{0\le\gamma\le1}\sup_{u,v\in[p]\atop u\in\mathrm{I}_{c}\vee v\in\mathrm{I}_{c}}|z^{1/2}\mathfrak{C}_{uv}[G^{\gamma}(z)]_{uv}|>n^{-\delta_2}),\\
&&	\mathfrak{B}_2(\delta_1,\delta_2,z,\Psi):=\mathbb{P}_{\Psi}(\sup_{0\le\gamma\le1}\sup_{a\in[n],u\in[p]\atop
		a\in\mathrm{T}_{r},u\in\mathrm{T}_c}|\mathfrak{M}_{au}[\mathcal{G}^{\gamma}(z)Y^{\gamma}]_{au}|>n^{-\delta_1},\\
&&\qquad \qquad\qquad \qquad\qquad \qquad \sup_{0\le\gamma\le1}\sup_{a\in[n],u\in[p]\atop a\in\mathrm{I}_{r}\vee u\in\mathrm{I}_{c}}|z^{1/2}\mathfrak{M}_{au}[\mathcal{G}^{\gamma}(z)Y^{\gamma}]_{au}|>n^{-\delta_2}).
\end{eqnarray*}
We need the following lemma on monotonicity.
\begin{lemma}\label{lem_Green_monotonicity}
	Suppose that $\Psi$ is well configured. Fix $c_0$ in Theorem \ref{thm_greenfuncomp_entrywise_differror} and adjust $\epsilon$ as $-\epsilon_{\alpha}/2$ in $\Omega_1$, $\Omega_2$ and $\Omega_3$. For all $z\in\gamma_1$, we set $z^{\prime}=E^{\prime}+\mathrm{i}\eta^{\prime}$, where
	\begin{itemize}
		\item[(a)] in case $z\in\bigcup_{i=4}^7\mathcal{C}_i$,
		\begin{equation*}
			E^{\prime}=E,\quad \eta^{\prime}=n^{\epsilon_{0}}\eta;
		\end{equation*}
		\item[(b)]  in case $z\in\mathcal{C}_3$,
		\begin{equation*}
			E^{\prime}=E+\frac{(1-n^{\epsilon_{0}/2})(\sqrt{E^2+\eta^2}-E)}{2},\quad \eta^{\prime}=\eta.
		\end{equation*}
	\end{itemize}
	Similarly, we define  $\bar{z}^{\prime}$ when $\bar{z}\in\gamma_1$. Then for any $\delta_1,\delta_2>0$ and $D>0$, there exists a constant $C>0$ such that
	\begin{equation*}
		\max_{k\in\{0,1,2\}}\mathfrak{B}_k(\delta_1,\delta_2,z,\Psi)\le Cn^{C}\max_{k\in\{0,1,2\}}\mathfrak{B}_{k}(\epsilon_{\alpha}/2-3\epsilon_0,4\epsilon_l-3\epsilon_0,z^{\prime},\Psi)+Cn^{-D}.
	\end{equation*}
	The above estimate also holds for $z\in\gamma_2,\gamma_1^0,\gamma_2^0$.
\end{lemma}
\begin{proof}
	Let $q$ be any fixed integer, and $F_q(x):=|x|^{2q}$. By Theorem \ref{thm_greenfuncomp_entrywise_differror}, we will focus on the case where $(\#_1,\#_2,\#_3)=\big(\mathfrak{R}_{ab}(\operatorname{Im}[\mathcal{G}^{\gamma}(z)]_{ab}-\delta_{ab}b_t(z)m^{(t)}(\zeta)),\mathfrak{R}_{ab}(\operatorname{Im}[\mathcal{G}^0(z)]_{ab}-\delta_{ab}b_t(z)m^{(t)}(\zeta),\mathfrak{J}_{0,ab})\big)$. Theorem \ref{thm_greenfuncomp_entrywise_differror} with $F(x)=F_q(x)$ implies that there exists one constant $C_1>0$ such that,
	\begin{equation*}
		\begin{split}
			 &\mathbb{E}_{\Psi}\big(F_q[\mathfrak{R}_{ab}(\operatorname{Im}[\mathcal{G}^{\gamma}(z)]_{ab}-\delta_{ab}b_t(z)m^{(t)}(\zeta))]\big)-\mathbb{E}_{\Psi}\big(F_q[\mathfrak{R}_{ab}(\operatorname{Im}[\mathcal{G}^{0}(z)]_{ab}-\delta_{ab}b_t(z)m^{(t)}(\zeta))]\big)\\
			&<C_1n^{-c_0}(\mathfrak{J}_{q,0}+1)+C_1Q_0n^{C_1},
		\end{split}
	\end{equation*}
	where $\mathfrak{J}_{q,0}:=\sup_{a,b\in[n]}\mathbb{E}_{\Psi}\big(|F_q[\mathfrak{R}_{ab}(\operatorname{Im}[\mathcal{G}^{\gamma}(z)]_{ab}-\delta_{ab}b_t(z)m^{(t)}(\zeta))]|\big)$. Taking supremum over $i,j\in[n]$ in the above inequality, we obtain that
	\begin{equation*}
		\begin{split}
			 (1-C_1n^{-c_0})\mathfrak{J}_{q,0}\le&\max_{i,j\in[n]}\mathbb{E}_{\Psi}\big(F_q[\mathfrak{R}_{ab}(\operatorname{Im}[\mathcal{G}^{0}(z)]_{ab}-\delta_{ab}b_t(z)m^{(t)}(\zeta))]\big)\\
			&+C_1n^{-c_0}+3C_1n^{C_1}\max_{k\in\{0,1,2\}}\mathfrak{B}_k(\epsilon_{\alpha}/2,4\epsilon_l,z,\Psi).
		\end{split}
	\end{equation*}
	The following lemma from \cite[Lemma B.1]{bao2023smallest} indicates the Lipschitz continuity when the argument is the Green function.
	\begin{lemma}
		For any deterministic matrix $A\in\mathbb{R}^{n\times p}$, let $\mathcal{L}(A,z)$ be defined  in \eqref{eq_def_resolvent_L}. We have for $z=E+\mathrm{i}\eta\in\mathbb{C}^{+}$ and $V>1$,
		\begin{gather*}
			\max_{i,j\in[n+p]}[\mathcal{L}(A,E+\mathrm{i}\eta/V)]_{ij}\vee1\le V\cdot\max_{i,j\in[n+p]}[\mathcal{L}(A,E+\mathrm{i}\eta)]_{ij}\vee1.
		\end{gather*}
	\end{lemma}
	The above lemma implies that for $z\in\gamma_1$,
	\begin{equation*}
		\max_{k\in\{0,1,2\}}\mathfrak{B}_k(\epsilon_{\alpha}/2,4\epsilon_l,z,\Psi)\le \max_{k\in\{0,1,2\}}\mathfrak{B}_k(\epsilon_{\alpha}/2-3\epsilon_0,4\epsilon_l-3\epsilon_0,z^{\prime},\Psi).
	\end{equation*}
	Moreover, for any $z_0=E_0+\mathrm{i}\eta_0\in\gamma_1$, we have that
	\begin{equation*}
		\mathbb{E}_{\Psi}\big(F_q[\mathfrak{R}_{ij}(\operatorname{Im}[\mathcal{G}^{0}(z_0)]_{ij}-\delta_{ij}b_t(z)m^{(t)}(\zeta)])\big)\lesssim n^{-2q\epsilon_{\alpha}}\mathbbm{1}(i,j\in\mathrm{T}_r)+n^{-8q\epsilon_{l}}\mathbbm{1}(i\in\mathrm{I}_r~\text{or}~j\in\mathrm{I}_r).
	\end{equation*}
	It follows that for $i,j\in\mathrm{T}_r$ and some constant $C_{21}>0$,
	\begin{gather*}
		\mathfrak{J}_{q,0}\le C_{21}n^{-2q\epsilon_{\alpha}}+C_{21}n^{C_{21}}\max_{k\in\{0,1,2\}}\mathfrak{B}_k(\epsilon_{\alpha}/2,4\epsilon_l,z_0,\Psi),\\
		\mathfrak{J}_{q,0}\le C_{21}n^{-2q\epsilon_{\alpha}}+C_{21}n^{C_{21}}\max_{k\in\{0,1,2\}}\mathfrak{B}_k(\epsilon_{\alpha}/2-3\epsilon_0,4\epsilon_l-3\epsilon_0,z^{\prime},\Psi).
	\end{gather*}

	Besides, applying Markov's inequality for $4q\epsilon_l>D$, we have that for some constants $C_2,C_3>0$ and $0\le\gamma\le 1$
	\begin{itemize}
		\item [(a)] in case $i,j\in\mathrm{T}_r$,
		\begin{equation*}
			\begin{split}
				 &\mathbb{P}_{\Psi}\big(|z_0^{1/2}\mathfrak{R}_{ij}(\operatorname{Im}[\mathcal{G}^{\gamma}(z_0)]_{ij}-\delta_{ij}b_t(z)m^{(t)}(\zeta))|> n^{-\delta_1}\big)\\
				&\le \frac{|z^{\prime}|^q\mathfrak{J}_{q,0}}{n^{-q\delta_1}}
				\le C_3 n^{-q\epsilon_{\alpha}}+C_3 n^{C_3}\max_{k\in\{0,1,2\}}\mathfrak{B}_k(\epsilon_{\alpha}/2-3\epsilon_0,4\epsilon_l-3\epsilon_0,z^{\prime},\Psi);
			\end{split}
		\end{equation*}
		\item [(b)] in case $i\in\mathrm{I}_r$ or $j\in\mathrm{I}_r$,
		\begin{equation*}
			\begin{split}
				 &\mathbb{P}_{\Psi}\big(|z_0^{1/2}\mathfrak{R}_{ij}(\operatorname{Im}[\mathcal{G}^{\gamma}(z_0)]_{ij}-\delta_{ij}b_t(z)m^{(t)}(\zeta))|> n^{-\delta_2}\big)
				\\
				&\le \frac{|z^{\prime}|^q\mathfrak{J}_{q,0}}{n^{-2q\delta_2}}
				\le C_3 n^{-4q\epsilon_{l}}+C_3 n^{C_3}\max_{k\in\{0,1,2\}}\mathfrak{B}_k(\epsilon_{\alpha}/2-3\epsilon_0,4\epsilon_l-3\epsilon_0,z^{\prime},\Psi),
			\end{split}
		\end{equation*}
	\end{itemize}
	where we choose $\delta_1\le\epsilon_{\alpha}/4$ and $\delta_2\le2\epsilon_l$.
	Finally, taking the union for both $i,j\in\mathrm{T}_r$ and $i\in\mathrm{I}_r$ or $j\in\mathrm{I}_r$ and applying an $\epsilon$-net argument to $\gamma$ with the following deterministic bounds
	\begin{equation*}
		\frac{\partial[\mathcal{G}^{\gamma}(z)]_{ij}}{\partial \gamma}\lesssim\frac{\|L_{ij}\|+\gamma\|t^{1/2}W\|}{\eta^2}
	\end{equation*}
	and the observation that $\|L_{ij}\|\lesssim n^{-\epsilon_l}$ and $\mathbb{P}(\|W\|>2+C)<n^{-D}$, we can obtain that
	\begin{equation*}
		\begin{split}
			\mathfrak{B}_0(\delta_1,\delta_2,z,\Psi)&\le \mathbb{P}_{\Psi}\big(\sup_{0\le\gamma\le1\atop i\in\mathrm{I}_r\vee j\in\mathrm{I}_r}|z_0^{1/2}\mathfrak{R}_{ij}(\operatorname{Im}[\mathcal{G}^{\gamma}(z_0)]_{ij}-\delta_{ab}b_t(z)m^{(t)}(\zeta))|>n^{-2\epsilon_{l}}\big)\\
			&\quad+\mathbb{P}_{\Psi}\big(\sup_{0\le\gamma\le1\atop i,j\in\mathrm{T}_r}|z_0^{1/2}\mathfrak{R}_{ij}(\operatorname{Im}[\mathcal{G}^{\gamma}(z_0)]_{ij}-\delta_{ab}b_t(z)m^{(t)}(\zeta))|> n^{-\epsilon_{\alpha}/2}\big) \\
			&\le C_4 n^{-c_4D}+C_4n^{C_4}\max_{k\in\{0,1,2\}}\mathfrak{B}_k(\epsilon_{\alpha}/2-3\epsilon_0,4\epsilon_l-3\epsilon_0,z^{\prime},\Psi),
		\end{split}
	\end{equation*}
	for some large constant $C_4$ and small constant $c_4>0$. Repeating the above procedure for all $\mathfrak{B}_k(\delta_1,\delta_2,z,\Psi),k=1,2$ and adjusting the constants that have appeared will conclude this lemma.
\end{proof}

Now, we proceed to prove Theorem \ref{thm_greenfuncomp_entrywise_uniformbound}. From Lemma \ref{lem_Green_monotonicity} above, we have the relationship
\begin{equation*}
	\max_{k\in\{0,1,2\}}\mathfrak{B}_k(\delta_1,\delta_2,z,\Psi)\le Cn^{C}\max_{k\in\{0,1,2\}}\mathfrak{B}_{k}(\epsilon_{\alpha}/2-3\epsilon_0,4\epsilon_l-3\epsilon_0,z^{\prime},\Psi)+Cn^{-D}.
\end{equation*}
Let $\delta_1=\epsilon_{\alpha}/4,\delta_2=2\epsilon_l$. Repeating the above procedures $m$ times such that $4\epsilon_l-3m\epsilon_0<\epsilon_l/16$, we have
\begin{equation*}
	\max_{k\in\{0,1,2\}}\mathfrak{B}_k(\epsilon_{\alpha}/4,\delta_2,z,\Psi)\le Cn^{C}\max_{k\in\{0,1,2\}}\mathfrak{B}_{k}(\epsilon_{\alpha}/4-3m\epsilon_0,\epsilon_l/16,z_m,\Psi)+Cn^{-D}.
\end{equation*}
On the other hand, notice that $\eta_m=n^{m\epsilon_0}\eta\gtrsim n^{\epsilon_l}$, which gives the crude bound that $\|\mathcal{G}^{\gamma}\|\lesssim n^{-\epsilon_l}$. Therefore, we have $\max_{k\in\{0,1,2\}}\mathfrak{B}_{k}(\epsilon_{\alpha}/4-3m\epsilon_0,\epsilon_l/16,z_m,\Psi)=0$, which concludes the proof.   \qed

\subsubsection{Proof of Theorem \ref{thm_greenfuncomp_average}}\label{app_prf_thm_greenfuncomp_average}
We begin with some notation. Denote $\widetilde{W}$'s entries as $\widetilde{w}_{ij}$, and define $\widetilde{Y}^{\gamma}$ parallel to $Y^{\gamma}$ with substitution of $W$ by $\widetilde{W}$. We define $\widetilde{\mathfrak{hy}}_j$ through its elements,
\begin{equation*}
	 \widetilde{hy}_{ij}(0,\widetilde{w}_{ij})=(1-\psi_{ij})\chi_{ij}m_{ij}\rho^{-1}_{j}+\psi_{ij}h_{ij}\rho_j^{-1}+(1-\gamma^2)^{1/2}t^{1/2}\widetilde{w}_{ij}.\quad i\in[n],\;j\in[p].
\end{equation*}
And recall the definitions of $\mathfrak{ly}_j$ and $\mathfrak{hy}_j$ in \eqref{eq_def_ly&hy}. Some calculation gives that
\begin{equation*}
	\begin{split}
		&\frac{\partial \mathbb{E}_{\Psi}(|n\eta(\operatorname{Im}m^{\gamma}(z)-\operatorname{Im}\widetilde{m}_{n,0}(z))|^{2q})}{\partial\gamma}\\
		 &=-2q\sum_{i,j}\mathbb{E}_{\Psi}\Big(\sum_{a}\eta\operatorname{Im}\big([\mathcal{G}^{\gamma}]_{ia}[\mathcal{G}^{\gamma}Y^{\gamma}]_{aj}\big)\big(L_{ij}\rho_j^{-1}-\frac{\gamma t^{1/2}w_{ij}}{(1-\gamma^2)^{1/2}}\big)\\
&\qquad \qquad \cdot\big(\eta\sum_a\operatorname{Im}[\mathcal{G}^{\gamma}]_{aa}-\eta\sum_a\operatorname{Im}[\widetilde{\mathcal{G}}^0]_{aa}\big)^{2q-1}\Big)\\
		&=-2q\sum_{i,j}\Big((\mathfrak{I}_1)_{ij}-(\mathfrak{I}_2)_{ij}\Big),
	\end{split}
\end{equation*}
where
\begin{gather*}
	(\mathfrak{I}_1)_{ij}:=\mathbb{E}_{\Psi}\Big(\mathfrak{i}_{(j)}(Y^{\gamma}_j)\cdot\big(L_{ij}\rho_j^{-1}-\frac{\gamma t^{1/2}w_{ij}}{(1-\gamma^2)^{1/2}}\big)\cdot \mathfrak{d}_{2q-1}(Y^{\gamma}_j,\widetilde{\mathfrak{hy}}_j)\Big)\cdot\mathbbm{1}(\psi_{ij}=0),\\
	(\mathfrak{J}_2)_{ij}:=-\frac{\gamma t^{1/2}}{(1-\gamma^2)^{1/2}}\mathbb{E}_{\Psi}\Big(w_{ij}\cdot\mathfrak{i}_{(j)}(Y^{\gamma}_j)\cdot \mathfrak{d}_{2q-1}(Y_j^{\gamma},\widetilde{\mathfrak{hy}}_j)\Big)\cdot\mathbbm{1}(\psi_{ij}=1),\\
	\mathfrak{i}_{(j)}(\varpi):=\eta\operatorname{Im}\big[(\mathcal{G}_{(j)}^{\gamma,\varpi})^2Y_{(j)}^{\gamma,\varpi}\big]_{ij},\quad \mathfrak{d}_{2q-1}(\varpi_1,\varpi_2):=\big(\eta\sum_a\operatorname{Im}[\mathcal{G}_{(j)}^{\gamma,\varpi_1}]_{aa}-\eta\sum_a\operatorname{Im}[\widetilde{\mathcal{G}}_{(j)}^{0,\varpi_2}]_{aa}\big)^{2q-1}.
\end{gather*}
Here we recycle the notation $\mathfrak{i}_{(j)}(\varpi)$ to suggest that it is a parallel quantity to the one in Section \ref{app_prf_thm_greenfuncomp_entrywise_differror}.

We consider $(\mathfrak{I}_2)_{ij}$ first. Applying Gaussian integration by parts, we obtain that
\begin{equation*}
	\begin{split}
		(\mathfrak{I}_2)_{ij}=&-\frac{\gamma t^{1/2}}{(1-\gamma^2)^{1/2}n}\Big(\mathbb{E}_{\Psi}\big[\partial_{w_{ij}}\big(\mathfrak{i}_{(j)}(Y^{\gamma}_j)\big)\cdot \mathfrak{d}_{2q-1}(Y^{\gamma}_j,\widetilde{\mathfrak{hy}}_j)\big]\\
		 &+(2q-1)\mathbb{E}_{\Psi}\big[\mathfrak{i}_{(j)}(Y^{\gamma}_j)\partial_{w_{ij}}\big(\eta\sum_a\operatorname{Im}[\mathcal{G}^{\gamma}]_{aa}\big)\cdot \mathfrak{d}_{2q-1}(Y^{\gamma}_j,\widetilde{\mathfrak{hy}}_j)\big]\Big)\cdot\mathbbm{1}(\psi_{ij}=1).
	\end{split}
\end{equation*}
We calculate the two partial derivatives one by one. Firstly, for $\partial_{w_{ij}}\big(\mathfrak{i}_{(j)}(Y^{\gamma}_j)\big)$, we have
\begin{equation*}
	\begin{split}
		 \partial_{w_{ij}}\big(\mathfrak{i}_{(j)}(Y^{\gamma}_j)\big)=&\Big(\operatorname{Im}\big[(\mathcal{G}^{\gamma})^2\big]_{ii}-2\operatorname{Im}\big(\big[(\mathcal{G}^{\gamma})^2\big]_{ii}\big[(Y_j^{\gamma})^{*}\mathcal{G}^{\gamma}Y^{\gamma}_j\big]_{jj}\big)\\
		 &-2\operatorname{Im}\big(\big[(\mathcal{G}^{\gamma})^2Y^{\gamma}_j\big]_{ij}\big[\mathcal{G}^{\gamma}Y^{\gamma}_j\big]_{ij}\big)\Big)\cdot(1-\gamma^2)^{-1/2}t^{1/2}\eta.
	\end{split}
\end{equation*}
Secondly, for $\partial_{w_{ij}}\big(\eta\sum_a\operatorname{Im}[\mathcal{G}^{\gamma}]_{aa}\big)$, we have
\begin{equation*}
	\partial_{w_{ij}}\big(\eta\sum_a\operatorname{Im}[\mathcal{G}^{\gamma}]_{aa}\big)=-2t^{1/2}(1-\gamma^2)^{-1/2}\mathfrak{i}_{(j)}(Y^{\gamma}_j).
\end{equation*}
Applying Wald's identity, Theorem \ref{thm_greenfuncomp_entrywise_uniformbound} for the case $a=b$ and noticing that $i\in\mathrm{I}_r$ and $j\in\mathrm{I}_c$ when $\psi_{ij}=1$, we obtain that
\begin{equation}\label{eq_prf_greenfunction_average_prior_G^2}
	 \big|\big[(\mathcal{G}^{\gamma})^2\big]_{ii}\big|\le\sum_a\big|\big[\mathcal{G}^{\gamma}\big]_{ai}\big|^2=\frac{\operatorname{Im}[\mathcal{G}^{\gamma}]_{ii}}{\eta}\prec \frac{\eta+t^{-2}n^{-c\epsilon_l}}{\eta},
\end{equation}
where we also used the fact that $\operatorname{Im}m^{(t)}(\zeta)\sim\eta$. Similarly, we have
\begin{equation}\label{eq_prf_greenfunction_average_prior_G^2Y}
	\begin{split}
		 \big|\big[(\mathcal{G}^{\gamma})^2Y_j^{\gamma}\big]_{ij}\big|&\le\sum_a\big|\big[\mathcal{G}^{\gamma}Y_j^{\gamma}\big]_{aj}\big|^2+\sum_a\big|\big[\mathcal{G}^{\gamma}\big]_{ia}\big|^2\\
		&=\big[(Y_j^{\gamma})^{*}|\mathcal{G}^{\gamma}|^2Y_j^{\gamma}\big]_{jj}+\eta^{-1}\operatorname{Im}[\mathcal{G}^{\gamma}]_{ii}\\
		&=\big[\bar{G}^{\gamma}\big]_{jj}+z\big[|G^{\gamma}|^2\big]_{jj}+\eta^{-1}\operatorname{Im}[\mathcal{G}^{\gamma}]_{ii}
		\prec \frac{\eta+t^{-2}n^{-c\epsilon_l}}{\eta},
	\end{split}
\end{equation}
where we used Wald's identity, $(AA^{*}-zI)^{-1}A=A^{*}(AA^{*}-zI)^{-1}$ with $z$ outside the spectrum of $A$.

Using the above estimates, together with Theorem \ref{thm_greenfuncomp_entrywise_uniformbound} and the fact that $\sum_{i,j}\mathbbm{1}(\psi_{ij}=1)\le n^{1-\epsilon_y}$ with  $\epsilon_y=\epsilon_h$, we obtain that
\begin{equation*}
	 |(\mathfrak{I}_2)_{ij}|\lesssim\frac{\gamma}{(1-\gamma^2)^{1/2}n^{1-\epsilon_{\alpha}}t^3}\sum_{k=1}^2\mathbb{E}_{\Psi}\big(\mathrm{O}_{\prec}(n^{-\epsilon_{\alpha}})\cdot|\mathfrak{d}_{2q-k}(Y^{\gamma}_j,\widetilde{\mathfrak{hy}}_j)|\big)\cdot\mathbbm{1}(\psi_{ij}=1).
\end{equation*}
In order to relate $\mathfrak{d}_{2q-k}(Y^{\gamma}_j,\widetilde{\mathfrak{hy}}_j)$ to $\mathfrak{d}_{2q}(Y^{\gamma}_j,\widetilde{\mathfrak{hy}}_j)$, we deploy Young's inequality,
\begin{equation*}
	 \mathbb{E}_{\Psi}\big(\mathrm{O}_{\prec}(n^{-\epsilon_{\alpha}}t^{-3})\cdot|\mathfrak{d}_{2q-1}(Y^{\gamma}_j,\widetilde{\mathrm{hy}}_j)|\big)\lesssim\mathbb{E}_{\Psi}\big(|\mathfrak{d}_{2q}(Y^{\gamma}_j,\widetilde{\mathrm{hy}}_j)|\big)+n^{-2q\epsilon_{\alpha}}.
\end{equation*}
Similar argument can be applied to the case $k=2$. Then we have
\begin{equation*}
	 |(\mathfrak{I}_2)_{ij}|\lesssim\frac{\gamma}{(1-\gamma^2)^{1/2}n^{1-\epsilon_{\alpha}}t^3}\Big(\mathbb{E}_{\Psi}\big(|\mathfrak{d}_{2q}(Y^{\gamma}_j,\widetilde{\mathfrak{hy}}_j)|\big)+n^{-2q\epsilon_{\alpha}}\Big)\cdot\mathbbm{1}(\psi_{ij}=1).
\end{equation*}

Next, we consider $(\mathfrak{I}_1)_{ij}$. Recall the resolvent expansion in \eqref{eq_prf_Greenfunction_resolvent_expansion},
\begin{eqnarray*}
	 \mathcal{G}^{\gamma}&=&\sum_{k=0}^{s}\mathcal{G}_{(j)}^{\gamma,0}(-Y^{\gamma}_j(Y^{\gamma}_j)^{*}\mathcal{G}_{(j)}^{\gamma,0})^k+\mathcal{G}^{\gamma}(-Y^{\gamma}_j(Y^{\gamma}_j)^{*}\mathcal{G}_{(j)}^{\gamma,0})^{s+1}\\
	 \widetilde{\mathcal{G}}_{(j)}^{0,\widetilde{\mathfrak{hy}}_j}&=&\sum_{k=0}^{s}\widetilde{\mathcal{G}}_{(j)}^{0,0}(-\widetilde{\mathfrak{hy}}_j\widetilde{\mathfrak{hy}}_j^{*}\widetilde{\mathcal{G}}_{(j)}^{0,0})^k+\widetilde{\mathcal{G}}_{(j)}^{0,\widetilde{\mathfrak{hy}}_j}(-\widetilde{\mathfrak{hy}}_j\widetilde{\mathfrak{hy}}_j^{*}\widetilde{\mathcal{G}}_{(j)}^{0,0})^{s+1}\\
	 (\mathcal{G}^{\gamma})^2&=&\sum_{k=0}^{2s}(\mathcal{G}_{(j)}^{\gamma,0})^2(-Y^{\gamma}_j(Y^{\gamma}_j)^{*}\mathcal{G}_{(j)}^{\gamma,0})^k+\sum_{k=0}^{s}\mathcal{G}_{(j)}^{\gamma,0}\mathcal{G}^{\gamma}(-Y_j^{\gamma}(Y_j^{\gamma})^{*}\mathcal{G}_{(j)}^{\gamma,0})^{k+s+1}\\
&&\qquad \qquad +(\mathcal{G}^{\gamma})^2(-Y^{\gamma}_j(Y^{\gamma}_j)^{*}\mathcal{G}_{(j)}^{\gamma,0})^{2s+2}.
\end{eqnarray*}
Since $Y_j^{\gamma}=Y_{(j)}^{\gamma,0}+\mathbf{e}_j^{*}Y^{\gamma}_{j}$, we expand $\eta^{-1}\mathfrak{i}_{(j)}(Y^{\gamma}_j)$ as
\begin{equation*}
	\begin{split}
		&\eta^{-1}\mathfrak{i}_{(j)}(Y^{\gamma}_j)\\
		 =&\operatorname{Im}\big[(\mathcal{G}_{(j)}^{\gamma,0})^2\mathbf{e}_j^{*}Y_j^{\gamma}\big]_{ij}+\operatorname{Im}\Big[\sum_{k=0}^{2s}(\mathcal{G}_{(j)}^{\gamma,0})^2(-Y^{\gamma}_j(Y^{\gamma}_j)^{*}\mathcal{G}_{(j)}^{\gamma,0})^k\mathbf{e}_j^{*}Y^{\gamma}_j\Big]_{ij}\\
		 &+\operatorname{Im}\Big[\sum_{k=0}^s\mathcal{G}_{(j)}^{\gamma,0}\mathcal{G}^{\gamma}(-Y^{\gamma}_j(Y^{\gamma}_j)^{*}\mathcal{G}_{(j)}^{\gamma,0})^{k+s+1}\mathbf{e}_j^{*}Y^{\gamma}_{(j)}\Big]_{ij}+\operatorname{Im}\Big[(\mathcal{G}^{\gamma})^2(-Y^{\gamma}_j(Y^{\gamma}_j)^{*}\mathcal{G}_{(j)}^{\gamma,0})^{2s+2}\mathbf{e}_j^{*}Y^{\gamma}_j\Big]_{ij}\\
		 =&\operatorname{Im}\Big[\sum_{k=0}^{2s}(\mathcal{G}_{(j)}^{\gamma,0})^2(-Y^{\gamma}_j(Y^{\gamma}_j)^{*}\mathcal{G}_{(j)}^{\gamma,0})^k\mathbf{e}_j^{*}Y^{\gamma}_j\Big]_{ij}+\eta^{-1}\mathfrak{i}^{r}_{(j)}(0)\\
        =&\eta^{-1}\big(\mathfrak{i}^1_{(j)}(0)+\mathfrak{i}_{(j)}^r(0)\big),
	\end{split}
\end{equation*}
where $\mathfrak{i}_{(j)}^r(0)$ is the residual term. Similarly,  $\mathfrak{d}_{2q-1}(Y^{\gamma}_j,\widetilde{\mathrm{hy}}_j)$ can be expanded as
\begin{equation*}
	\begin{split}
		&\mathfrak{d}_{2q-1}(Y^{\gamma}_j,\widetilde{\mathfrak{hy}}_j)\\ =&\Big(\eta\sum_a\big(\operatorname{Im}\big[\mathcal{G}_{(j)}^{\gamma,0}\big]_{aa}-\operatorname{Im}\big[\widetilde{\mathcal{G}}_{(j)}^{0,0}\big]_{aa}\big)\Big)^{2q-1}+\sum_{r=1}^{2q-1}\binom{2q-1}{r}\Big(\eta\sum_a\big(\operatorname{Im}\big[\mathcal{G}_{(j)}^{\gamma,0}\big]_{aa}-\operatorname{Im}\big[\widetilde{\mathcal{G}}_{(j)}^{0,0}\big]_{aa}\big)\Big)^{2q-1-r}\\
		 &\cdot\Big(\eta\sum_a\big(\operatorname{Im}\big[\sum_{k=1}^{l_1}\mathcal{G}_{(j)}^{\gamma,0}(-Y^{\gamma}_j(Y^{\gamma}_j)^{*}\mathcal{G}_{(j)}^{\gamma,0})^k\big]_{aa}-\operatorname{Im}\big[\sum_{k=1}^{l_1}\widetilde{\mathcal{G}}_{(j)}^{0,0}(-\widetilde{\mathfrak{hy}}_j\widetilde{\mathfrak{hy}}_j^{*}\widetilde{\mathcal{G}}_{(j)}^{0,0})^k\big]_{aa}\big)\Big)^r+\mathfrak{d}^r_{2q-1}(\mathcal{G}^{\gamma},\widetilde{\mathcal{G}}_{(j)}^{0,\widetilde{\mathfrak{hy}}_j})\\
		 =&\mathfrak{d}_{2q-1}(0,0)+\sum_{r=1}^{2q-1}\binom{2q-1}{r}\Big(\eta\sum_a\big(\operatorname{Im}\big[\mathcal{G}_{(j)}^{\gamma,0}\big]_{aa}-\operatorname{Im}\big[\widetilde{\mathcal{G}}_{(j)}^{0,0}\big]_{aa}\big)\Big)^{2q-1-r}\\
		 &\cdot\Big(\eta\sum_a\big(\operatorname{Im}\big[\sum_{k=1}^{l_1}\mathcal{G}_{(j)}^{\gamma,0}(-Y^{\gamma}_j(Y^{\gamma}_j)^{*}\mathcal{G}_{(j)}^{\gamma,0})^k\big]_{aa}-\operatorname{Im}\big[\sum_{k=1}^{l_1}\widetilde{\mathcal{G}}_{(j)}^{0,0}(-\widetilde{\mathfrak{hy}}_j\widetilde{\mathfrak{hy}}_j^{*}\widetilde{\mathcal{G}}_{(j)}^{0,0})^k\big]_{aa}\big)\Big)^r+\mathfrak{d}^r_{2q-1}(\mathcal{G}^{\gamma},\widetilde{\mathcal{G}}_{(j)}^{0,\widetilde{\mathfrak{hy}}_j})\\
        &:=\mathfrak{d}_{2q-1}(0,0)+\mathfrak{d}_{2q-1}^1(0,0)+\mathfrak{d}^r_{2q-1}(\mathcal{G}^{\gamma},\widetilde{\mathcal{G}}_{(j)}^{0,\widetilde{\mathfrak{hy}}_j}).
	\end{split}
\end{equation*}
where $\mathfrak{d}^r_{2q-1}(\mathcal{G}^{\gamma},\widetilde{\mathcal{G}}_{(j)}^{0,\widetilde{\mathfrak{hy}}_j})$ is the residual term containing $\mathcal{G}^{\gamma}$ and $\widetilde{\mathcal{G}}_{(j)}^{0,\widetilde{\mathfrak{hy}}_j}$.

Thereafter, the expansions of $\mathfrak{i}_{(j)}(Y^{\gamma}_j)$ and $\mathfrak{d}_{2q-1}(Y^{\gamma}_j,\widetilde{\mathfrak{hy}}_j)$ give the expansion of $(\mathfrak{I}_1)_{ij}$ as
\begin{equation*}
	\begin{split}
		 (\mathfrak{I}_1)_{ij}=&\mathbb{E}_{\Psi}\Big\{\big(\mathfrak{i}^1_{(j)}(0)+\mathfrak{i}^r_{(j)}(0)\big)\cdot\big(\mathfrak{d}_{2q-1}(0,0)+\mathfrak{d}_{2q-1}^1(0,0)+\mathfrak{d}^r_{2q-1}(\mathcal{G}^{\gamma},\widetilde{\mathcal{G}}_{(j)}^{0,\widetilde{\mathfrak{hy}}_j})\big)\\
&\qquad \qquad \cdot\big(L_{ij}\rho_j^{-1}-\frac{\gamma t^{1/2}w_{ij}}{(1-\gamma^2)^{1/2}}\big)\Big\}\cdot\mathbbm{1}(\psi_{ij}=0)\\
        =&\mathbb{E}_{\Psi}\Big\{\Big[\big(\mathfrak{d}_{2q-1}(0,0)+\mathfrak{d}_{2q-1}^1(0,0)\big)\cdot\mathfrak{i}^1_{(j)}(0)+\mathfrak{i}^1_{(j)}(0)\cdot\mathfrak{d}_{2q-1}^r(\mathcal{G}^{\gamma},\widetilde{\mathcal{G}}^{0,\widetilde{\mathfrak{hy}_j}}_{(j)})\\
        &\qquad\qquad+\big(\mathfrak{d}_{2q-1}(0,0)+\mathfrak{d}_{2q-1}^1(0,0)+\mathfrak{d}^r_{2q-1}(\mathcal{G}^{\gamma},\widetilde{\mathcal{G}}_{(j)}^{0,\widetilde{\mathfrak{hy}}_j})\big)\cdot\mathfrak{i}^r_{(j)}(0)\Big]\\
        &\qquad \qquad\qquad\cdot(L_{ij}\rho_j^{-1}-\frac{\gamma t^{1/2}w_{ij}}{(1-\gamma^2)^{1/2}}\big)\Big\}\cdot\mathbbm{1}(\psi_{ij}=0)\\
		 =&:\mathbb{E}_{\Psi}\Big[\big(\mathfrak{M}_{(j)}(0,0)+\mathfrak{R}_{(j)}(\mathcal{G}^{\gamma},\widetilde{\mathcal{G}}_{(j)}^{0,\widetilde{\mathfrak{hy}}_j})\big)\cdot\big(L_{ij}\rho_j^{-1}-\frac{\gamma t^{1/2}w_{ij}}{(1-\gamma^2)^{1/2}}\big)\Big]\cdot\mathbbm{1}(\psi_{ij}=0),
	\end{split}
\end{equation*}
where $\mathfrak{M}_{(j)}(0,0)$ is the main term, $\mathfrak{M}_{(j)}(0,0):=\big(\mathfrak{d}_{2q-1}(0,0)+\mathfrak{d}_{2q-1}^1(0,0)\big)\cdot\mathfrak{i}^1_{(j)}(0)$. Meanwhile, $\mathfrak{R}_{(j)}(\mathcal{G}^{\gamma},\widetilde{\mathcal{G}}_{(j)}^{0,\widetilde{\mathfrak{hy}}_j})$ is the residual term,
\begin{equation*}
    \begin{split}
        \mathfrak{R}_{(j)}(\mathcal{G}^{\gamma},\widetilde{\mathcal{G}}_{(j)}^{0,\widetilde{\mathfrak{hy}}_j})&:=\mathfrak{i}^1_{(j)}(0)\cdot\mathfrak{d}_{2q-1}^r(\mathcal{G}^{\gamma},\widetilde{\mathcal{G}}^{0,\widetilde{\mathfrak{hy}_j}}_{(j)})\\
        &+\big(\mathfrak{d}_{2q-1}(0,0)+\mathfrak{d}_{2q-1}^1(0,0)+\mathfrak{d}^r_{2q-1}(\mathcal{G}^{\gamma},\widetilde{\mathcal{G}}_{(j)}^{0,\widetilde{\mathfrak{hy}}_j})\big)\cdot\mathfrak{i}^r_{(j)}(0).
    \end{split}
\end{equation*}
Similarly to the proof in Section \ref{app_prf_thm_greenfuncomp_entrywise_differror}, we will use the assemble patterns classification $\beta_{k_1k_2\dots k_r}^{(K)}$ in Definition \ref{def_prf_Greenfunctioncomparison_encode_beta} to trace the moment pattern in $\mathfrak{M}_{(j)}(0,0)$. Before moving ahead, we remark that we only consider $\mathfrak{i}^1_{(j)}(0)\cdot\mathfrak{d}_{2q-1}(0,0)$ of $\mathfrak{M}_{(j)}(0,0)$  below since $\mathfrak{i}^1_{(j)}(0)\cdot\mathfrak{d}^1_{2q-1}(0,0)$ shares the same order as $\mathfrak{i}^1_{(j)}(0)\cdot\mathfrak{d}_{2q-1}(0,0)$ after taking expectation. This claim can be verified by the fact that the quantities $\sum_a\operatorname{Im}[\mathcal{G}_{(j)}^{\gamma,0}(Y_j^{\gamma}(Y_j^{\gamma})^{*}\mathcal{G}_{(j)}^{\gamma,0})^k]_{aa}, k\ge1$ are of constant order after taking expectation. For example, when $k=2$,
\begin{equation*}
	\begin{split}
		 &\sum_a[\mathcal{G}_{(j)}^{\gamma,0}(Y_j^{\gamma}(Y_j^{\gamma})^{*}\mathcal{G}_{(j)}^{\gamma,0})^2]_{aa}=\sum_a(\sum_k[\mathcal{G}_{(j)}^{\gamma,0}]_{ak}y_{kj})^2\times\sum_{s,l}[\mathcal{G}_{(j)}^{\gamma,0}]_{sl}y_{sj}y_{lj}\\
		 &=\sum_{a,k}[\mathcal{G}_{(j)}^{\gamma,0}]^2_{ak}y_{kj}^2\times\sum_{s,l}[\mathcal{G}_{(j)}^{\gamma,0}]_{sl}y_{sj}y_{lj}+\sum_{a}\sum_{t,k}[\mathcal{G}_{(j)}^{\gamma,0}]_{ak}[\mathcal{G}_{(j)}^{\gamma,0}]_{at}y_{kj}y_{tj}\times \sum_{s,l}[\mathcal{G}_{(j)}^{\gamma,0}]_{sl}y_{sj}y_{lj}\\
		&=\sum_{a,k}[\mathcal{G}_{(j)}^{\gamma,0}]^3_{ak}y_{kj}^4+\sum_{a,k}\sum_{l\neq k}[\mathcal{G}_{(j)}^{\gamma,0}]^2_{ak}[\mathcal{G}_{(j)}^{\gamma,0}]_{kl}y_{kj}^3y_{lj}+\sum_{a,k}\sum_{s,l\neq k}[\mathcal{G}_{(j)}^{\gamma,0}]^2_{ak}[\mathcal{G}_{(j)}^{\gamma,0}]_{sl}y_{kj}^2y_{sj}y_{lj}\\
		 &\quad+\sum_{a}\sum_{t,k}[\mathcal{G}_{(j)}^{\gamma,0}]_{ak}[\mathcal{G}_{(j)}^{\gamma,0}]_{at}[\mathcal{G}_{(j)}^{\gamma,0}]_{kt}y^2_{kj}y^2_{tj}+2\sum_{a}\sum_{t,k}\sum_{l\neq t}[\mathcal{G}_{(j)}^{\gamma,0}]_{ak}[\mathcal{G}_{(j)}^{\gamma,0}]_{at}[\mathcal{G}_{(j)}^{\gamma,0}]_{kl}y^2_{kj}y_{tj}y_{lj}\\
		&\quad+\sum_{a}\sum_{t,k}\sum_{s,l\neq t,k}[\mathcal{G}_{(j)}^{\gamma,0}]_{ak}[\mathcal{G}_{(j)}^{\gamma,0}]_{at}[\mathcal{G}_{(j)}^{\gamma,0}]_{sl}y_{kj}y_{tj}y_{sj}y_{lj}.
	\end{split}
\end{equation*}
Then, using the Cauchy-Schwartz inequality and the Wald identity, we obtain that the expectation of each term has the order of at most $\mathrm{O}(1)$.

We first consider $\mathfrak{R}_{(j)}(\mathcal{G}^{\gamma},\widetilde{\mathcal{G}}_{(j)}^{0,\widetilde{\mathfrak{hy}}_j})\cdot(L_{ij}\rho_j^{-1}-\gamma t^{1/2}w_{ij}/(1-\gamma^2)^{1/2})$. The strategy to control this residual term is the same as the procedure in the proof of Theorem \ref{thm_greenfuncomp_entrywise_differror}, where we can split $\mathfrak{R}_{(j)}(\mathcal{G}^{\gamma},\widetilde{\mathcal{G}}_{(j)}^{0,\widetilde{\mathfrak{hy}}_j})$ into $\mathfrak{R}_{(j),good}(\mathcal{G}^{\gamma},\widetilde{\mathcal{G}}_{(j)}^{0,\widetilde{\mathfrak{hy}}_j})+\mathfrak{R}_{(j),bad}(\mathcal{G}^{\gamma},\widetilde{\mathcal{G}}_{(j)}^{0,\widetilde{\mathfrak{hy}}_j})$ depending on whether there are enough terms $\mathfrak{ly}_j^{*}\mathcal{G}_{(j)}^{\gamma,0}\mathfrak{ly}_j$ in the cell $(Y_j^{\gamma})^{*}\mathcal{G}_{(j)}^{\gamma,0}Y_j^{\gamma}.$ And in $\mathfrak{R}_{(j),bad}(\mathcal{G}^{\gamma},\widetilde{\mathcal{G}}_{(j)}^{0,\widetilde{\mathfrak{hy}}_j})$, we find that the leading term results from the worse case such as $h_{u_1j}^{k_1}h_{u_2j}^{k_2}\dots h_{u_rj}^{k_r}\rho_j^{-(k_1+\dots+k_r)}[\mathcal{G}^{\gamma}]_{u_1u_2}\dots[\mathcal{G}^{\gamma,0}_{(j)}]_{u_ru_r}$. Using the same argument as in the proof of Theorem \ref{thm_greenfuncomp_entrywise_differror}, we can obtain that
\begin{eqnarray*}
   && \sum_{ij}\mathbb{E}_{\Psi}\Big[\mathfrak{R}_{(j)}(\mathcal{G}^{\gamma},\widetilde{\mathcal{G}}_{(j)}^{0,\widetilde{\mathfrak{hy}}_j})\cdot(L_{ij}\rho_j^{-1}-\gamma t^{1/2}w_{ij}/(1-\gamma^2)^{1/2})\Big]\cdot\mathbbm{1}(\psi_{ij}=0)\\
    &\lesssim& n^{-\alpha/2+3/2-\epsilon_y}\mathbb{E}_{\Psi}\big(\mathfrak{d}_{2q-1}(Y_j^{\gamma},\widetilde{\mathfrak{hy}}_j)\big).
\end{eqnarray*}

Now, we turn to $\mathfrak{M}_{(j)}(0,0)$. Recalling $\beta_{k_1k_2\ldots k_r}^{(K)}$ and $\mathbb{Q}^{\mathrm{permutation}}_{k_1k_2\ldots k_r}$ in Definition \ref{def_prf_Greenfunctioncomparison_encode_beta}, we define the following assemble patterns to simplify the expression of $(\mathfrak{I}_1)_{ij}$,
\begin{equation*}
\begin{split}
	\mathfrak{P}_{(j),k_1k_2\ldots k_r}^{(K)}=&\sum_{\stackrel{u_r\ne \cdots\ne u_2\ne i}{(v_1,\ldots,v_{K})\in \mathbb{Q}^{\mathrm{permutation}}_{k_1k_2\ldots k_r}}}[(\mathcal{G}_{(j)}^{\gamma,0})^2]_{iv_1}[\mathcal{G}_{(j)}^{\gamma,0}]_{v_{2}v_{3}}\cdots[\mathcal{G}_{(j)}^{\gamma,0}]_{v_{\ell-1}v_{K}}.
	\end{split}
\end{equation*}
Then, we can write
\begin{equation*}
	\begin{split}
		&\mathbb{E}_{\Psi}\Big[\mathfrak{M}_{(j)}(0,0)\cdot\big(L_{ij}\rho_j^{-1}-\frac{\gamma t^{1/2}w_{ij}}{(1-\gamma^2)^{1/2}}\big)\Big]\cdot\mathbbm{1}(\psi_{ij}=0)\\
		&\simeq\eta\sum_{K=1,o}^{4s+1}\sum_{k=1}^{K}\mathbb{E}_{\Psi}\Big[\beta_{kk_2\dots k_r}^{(K)}\mathfrak{P}_{(j),kk_2\ldots k_r}^{(K)}\cdot \mathfrak{d}_{2q-1}(0,0) \big(L_{ij}\rho_j^{-1}-\frac{\gamma t^{1/2}w_{ij}}{(1-\gamma^2)^{1/2}}\big)\Big]\cdot\mathbbm{1}(\psi_{ij}=0),
	\end{split}
\end{equation*}
where the summation is over all the possible odd $K\in\mathbb{N}$ with $1\le K\le 4s+1$. For simplicity, we will fix $K$ in the following arguemnt.

 As in the proof of Theorem \ref{thm_greenfuncomp_entrywise_differror}, we split $\mathfrak{M}_{(j)}(0,0)$ into $\mathfrak{M}^I_{(j)}(0,0)+\mathfrak{M}_{(j)}^{II}(0,0)+\mathfrak{M}_{(j)}^{III}(0,0)$.
\begin{itemize}
	\item [\textbf{Case I.}] {\bf For $\beta^{(K)}_{kk_2\dots k_r}$, $k\ge 1$ and all $k_2,\dots,k_r$ are even numbers.}

For this case, we have
\begin{equation*}
    \begin{split}
        &\eta\mathbb{E}_{\Psi}\big[\mathfrak{M}^{I}_{(j)}(0,0)\cdot\big(L_{ij}\rho_j^{-1}-\frac{\gamma t^{1/2}w_{ij}}{(1-\gamma^2)^{1/2}}\big)\big]\cdot\mathbbm{1}(\psi_{ij}=0)\\
        &=\eta \sum_{K=1,o}^{4s+1}\sum_{k=1,o}^K\mathbb{E}_{\Psi}\Big[ly_{ij}^k\beta_{0k_2\dots k_r}^{(K-k)}\mathfrak{P}_{(j),kk_2\ldots k_r}^{(K)}\cdot \mathfrak{d}_{2q-1}(0,0) \big(L_{ij}\rho_j^{-1}-\frac{\gamma t^{1/2}w_{ij}}{(1-\gamma^2)^{1/2}}\big)\Big]\cdot\mathbbm{1}(\psi_{ij}=0)\\
        &:=\sum_{K=1,o}^{4s+1}\sum_{k=1,o}^K(\mathfrak{I}_{11,k})_{ij}
    \end{split}
\end{equation*}
Notice that $K$ and $k$ are odd. In the following we  consider different cases for $k$.
	
	\textbf{Case $k\ge5$.} For this case, we have
	\begin{equation*}
		\begin{split}
		(\mathfrak{I}_{11,k})_{ij}=&\eta\mathbb{E}_{\Psi}\Big[ly_{ij}^k\beta_{0k_2\dots k_r}^{(K-k)}\mathfrak{P}_{(j),kk_2\ldots k_r}^{(K)}\cdot \mathfrak{d}_{2q-1}(0,0) \big(L_{ij}\rho_j^{-1}-\frac{\gamma t^{1/2}w_{ij}}{(1-\gamma^2)^{1/2}}\big)\Big]\cdot\mathbbm{1}(\psi_{ij}=0)\\
			&\lesssim\eta\sum_{u_1+u_2=k}\mathbb{E}_{\Psi}\Big[|L_{ij}\rho_j^{-1}|^{u_1+1}|t^{1/2}w_{ij}|^{u_2}\cdot \mathfrak{P}_{(j),kk_2\ldots k_r}^{(K)}\cdot \mathfrak{d}_{2q-1}(0,0)\Big]\cdot\mathbbm{1}(\psi_{ij}=0)\\
			&\lesssim\frac{\eta t}{n^{1+\alpha/2}}\mathbb{E}_{\Psi}\big[\mathfrak{d}_{2q-1}(0,0)\big]\cdot\mathbbm{1}(\psi_{ij}=0),
		\end{split}
	\end{equation*}
	where we have used the entrywise bounds in Theorem \ref{thm_greenfuncomp_entrywise_uniformbound} and the deterministic bound $\sum_ly_{lj}^{2k}\le 1$. It should be noticed that $[\widetilde{\mathcal{G}}_{(j)}^{0,0}]_{ab}$ shares the same entrywise bound as $[\mathcal{G}_{(j)}^{\gamma,0}]_{ab}$. To see this, we first observe that $[\widetilde{\mathcal{G}}_{(j)}^{0,0}]_{ab}$ has the same moment behavior as $[\mathcal{G}_{(j)}^{0,0}]_{ab}$. And then, since
        \begin{equation*}
		\frac{\partial\mathbb{E}_{\Psi}(\operatorname{Im}[\mathcal{G}_{(j)}^{\gamma,0}]_{ab})}{\partial\gamma}=-\sum_{r,s\neq j}\mathbb{E}_{\Psi}\Big[\operatorname{Im}\big([\mathcal{G}_{(j)}^{\gamma,0}]_{rb}[\mathcal{G}_{(j)}^{\gamma,0}Y_{(j)}^{\gamma,0}]_{as}\big)\cdot\big((1-\chi_{rs})l_{rs}\rho_s^{-1}-\frac{\gamma t^{1/2}w_{rs}}{(1-\gamma^2)^{1/2}}\big)\Big],
	\end{equation*}
 we conclude
        \begin{equation*}
		\frac{\partial\mathbb{E}_{\Psi}(\operatorname{Im}[\mathcal{G}_{(j)}^{\gamma,0}]_{ab})}{\partial\gamma}=\mathrm{o}(1),
	\end{equation*}
 where we used  \eqref{eq_prf_Greenfunctioncomparison_equivbound} with $\mathfrak{J}_{0,ab}\lesssim t^{-2}n^{-c_l\epsilon_l}$ and $Q_0\lesssim n^{-D}$ for sufficiently large $D>0$. Also we used Theorem \ref{thm_greenfuncomp_entrywise_uniformbound}. Then, we have
\begin{equation}\label{eq_prf_greenfunctioncomparison_relation_Ggamma&G0}
		 \mathbb{E}_{\Psi}(\operatorname{Im}[\mathcal{G}_{(j)}^{\gamma,0}]_{ab})=\mathbb{E}_{\Psi}(\operatorname{Im}[\mathcal{G}_{(j)}^{0,0}]_{ab})+\mathrm{o}(1).
	\end{equation}
	So, we can get the parallel entrywise bounds for $[\widetilde{\mathcal{G}}_{(j)}^{0,0}]_{ab}$ from $[\mathcal{G}_{(j)}^{\gamma,0}]_{ab}$.
	
	Now, It remains to relate $\mathfrak{d}_{2q-1}(0,0)$ to $\mathfrak{d}_{2q-1}(Y_j^{\gamma},\widetilde{\mathfrak{hy}}_j)$. Firstly, observe that
	\begin{equation*}
		\begin{split}
			 &\sum_a\operatorname{Im}\big[\sum_{k=1}^{l}\mathcal{G}_{(j)}^{\gamma,0}(-Y^{\gamma}_j(Y^{\gamma}_j)^{*}\mathcal{G}_{(j)}^{\gamma,0})^k\big]_{aa}=\sum_{k=1}^{l}\operatorname{Im}\big[\operatorname{tr}\big(\mathcal{G}_{(j)}^{\gamma,0}(-Y^{\gamma}_j(Y^{\gamma}_j)^{*}\mathcal{G}_{(j)}^{\gamma,0})^k\big)\big]\\
			 &=\sum_{k=1}^{l}\operatorname{Im}\big[((Y^{\gamma}_j)^{*}\mathcal{G}_{(j)}^{\gamma,0}Y^{\gamma}_j)^{k-1}\cdot((Y^{\gamma}_j)^{*}(\mathcal{G}_{(j)}^{\gamma,0})^2Y^{\gamma}_j)\big]
			\\
			 &\lesssim\sum_{k=1}^{l}\operatorname{Im}\big[((Y^{\gamma}_j)^{*}\mathcal{G}_{(j)}^{\gamma,0}Y^{\gamma}_j)^{k-1}\big]+\operatorname{Im}\big[(Y^{\gamma}_j)^{*}(\mathcal{G}_{(j)}^{\gamma,0})^2Y^{\gamma}_j\big]
			\prec 1,
		\end{split}
	\end{equation*}
	where we deployed the trivial bound $\|\mathcal{G}_{(j)}^{\gamma,0}\|\sim\mathrm{O}(1)$, $(Y^{\gamma}_j)^{*}Y^{\gamma}_j\lesssim1$. Thereafter, we have
	\begin{equation}\label{eq_prf_greenfunctioncomparison_Fdiffer_prior}
		\begin{split}
			 &\sum_a\big(\operatorname{Im}\big[\sum_{k=1}^{l}\mathcal{G}_{(j)}^{\gamma,0}(-Y^{\gamma}_j(Y^{\gamma}_j)^{*}\mathcal{G}_{(j)}^{\gamma,0})^k\big]_{aa}-\operatorname{Im}\big[\sum_{k=1}^{l}\widetilde{\mathcal{G}}_{(j)}^{0,0}(-\widetilde{\mathfrak{hy}}_j\widetilde{\mathfrak{hy}}_j^{*}\widetilde{\mathcal{G}}_{(j)}^{0,0})^k\big]_{aa}\big)\\
			 &\lesssim\sum_{k=1}^{l}\operatorname{Im}\Big[\big|((Y^{\gamma}_j)^{*}\mathcal{G}_{(j)}^{\gamma,0}Y^{\gamma}_j)^{k-1}-(\widetilde{\mathfrak{hy}}_j^{*}\widetilde{\mathcal{G}}_{(j)}^{0,0}\widetilde{\mathfrak{hy}}_j)^{k-1}\big|\cdot |(\widetilde{\mathfrak{hy}}_j^{*}(\widetilde{\mathcal{G}}_{(j)}^{0,0})^2\widetilde{\mathfrak{hy}}_j)|+\big|((Y^{\gamma}_j)^{*}(\mathcal{G}_{(j)}^{\gamma,0})^2Y^{\gamma}_j)\\
			&\quad-(\widetilde{\mathfrak{hy}}_j^{*}(\widetilde{\mathcal{G}}_{(j)}^{0,0})^2\widetilde{\mathfrak{hy}}_j)\big|\cdot |(\widetilde{\mathfrak{hy}}_j^{*}\widetilde{\mathcal{G}}_{(j)}^{0,0}\widetilde{\mathfrak{hy}}_j)^{k-1}|\Big]\\
			 &\lesssim\sum_{k=1}^{l_1}\operatorname{Im}\Big[\big|((Y^{\gamma}_j)^{*}\mathcal{G}_{(j)}^{\gamma,0}Y^{\gamma}_j)^{k-1}-(\widetilde{\mathfrak{hy}}_j^{*}\widetilde{\mathcal{G}}_{(j)}^{0,0}\widetilde{\mathfrak{hy}}_j)^{k-1}\big|+\big|((Y^{\gamma}_j)^{*}(\mathcal{G}_{(j)}^{\gamma,0})^2Y^{\gamma}_j)-(\widetilde{\mathfrak{hy}}_j^{*}(\widetilde{\mathcal{G}}_{(j)}^{0,0})^2\widetilde{\mathfrak{hy}}_j)\big|\Big],
		\end{split}
	\end{equation}
	where in the second step we used the basic inequality $|ab-cd|\le|a-c|\cdot|d|+|b-d|\cdot|c|$, and in the third step we applied the trivial bound $\|\widetilde{\mathcal{G}}_{(j)}^{0,0}\|+\|(\mathcal{G}_{(j)}^{\gamma,0})^2\|\le C$ with $\widetilde{\mathfrak{hy}}_j^{*}\widetilde{\mathfrak{hy}}_j\prec 1$. Notice that the elements $Y_{ij}^{\gamma}$ and $\widetilde{hy}_{ij}$ only differ in $\psi_{ij}h_{ij}\rho_j^{-1}$ despite the Gaussian random variables $w_{ij}$ and $\widetilde{w}_{ij}$. Then by the large deviation bounds in Lemma \ref{lem_largedeviation_H} and the large deviation bounds for Gaussian random vectors, together with the fact that there are at most finite $\psi_{ij}=1$ for $1\le i\le n$, one has
	\begin{eqnarray*} &&\operatorname{Im}\Big[\big|((Y^{\gamma}_j)^{*}\mathcal{G}_{(j)}^{\gamma,0}Y^{\gamma}_j)^{k-1}-(\widetilde{\mathfrak{hy}}_j^{*}\widetilde{\mathcal{G}}_{(j)}^{0,0}\widetilde{\mathfrak{hy}}_j)^{k-1}\big|\Big]\\
&\lesssim&\operatorname{Im}\big|\sum_ih^{2k-2}_{ij}\rho_j^{-2k+2}([\mathcal{G}_{(j)}^{\gamma,0}]^{k-1}_{ii}-[\widetilde{\mathcal{G}}_{(j)}^{0,0}]^{k-1}_{ii})\big|+\mathrm{o}_{\prec}(n^{-\epsilon_l}).
	\end{eqnarray*}
	Also notice from \eqref{eq_prf_greenfunctioncomparison_relation_Ggamma&G0} and Theorem \ref{thm_greenfuncomp_entrywise_uniformbound} that $|[\mathcal{G}_{(j)}^{\gamma,0}]_{ii}-[\widetilde{\mathcal{G}}_{(j)}^{0,0}]_{ii}|\prec n^{-c_l\epsilon_{\alpha}}$. Then $$\operatorname{Im}\Big[\big|((Y^{\gamma}_j)^{*}\mathcal{G}_{(j)}^{\gamma,0}Y^{\gamma}_j)^{k-1}-(\widetilde{\mathfrak{hy}}_j^{*}\widetilde{\mathcal{G}}_{(j)}^{0,0}\widetilde{\mathfrak{hy}}_j)^{k-1}\big|\Big]\prec n^{-\epsilon_l}.$$ Similar bound can be obtained for $\big|((Y^{\gamma}_j)^{*}(\mathcal{G}_{(j)}^{\gamma,0})^2Y^{\gamma}_j)-(\widetilde{\mathfrak{hy}}_j^{*}(\widetilde{\mathcal{G}}_{(j)}^{0,0})^2\widetilde{\mathfrak{hy}}_j)\big|$ via a parallel argument. So,
	\begin{equation*}
		 \sum_a\big(\operatorname{Im}\big[\sum_{k=1}^{l}\mathcal{G}_{(j)}^{\gamma,0}(-Y^{\gamma}_j(Y^{\gamma}_j)^{*}\mathcal{G}_{(j)}^{\gamma,0})^k\big]_{aa}-\operatorname{Im}\big[\sum_{k=1}^{l}\widetilde{\mathcal{G}}_{(j)}^{0,0}(-\widetilde{\mathfrak{hy}}_j\widetilde{\mathfrak{hy}}_j^{*}\widetilde{\mathcal{G}}_{(j)}^{0,0})^k\big]_{aa}\big)=\mathrm{o}_{\prec}(n^{-\epsilon_l}).
	\end{equation*}
	Then, we find that
	\begin{equation*}
		\begin{split}
			 &\mathbb{E}_{\Psi}\big(\mathfrak{d}_{2q-1}(Y^{\gamma}_j,\widetilde{\mathfrak{hy}}_j)\big)=\mathbb{E}_{\Psi}\big(\mathfrak{d}_{2q-1}(0,0)\big)+\sum_{r=1}^{2q-1}C_q\mathbb{E}_{\Psi}\Big[\eta\sum_a\big(\operatorname{Im}\big[\mathcal{G}_{(j)}^{\gamma,0}\big]_{aa}-\operatorname{Im}\big[\widetilde{\mathcal{G}}_{(j)}^{0,0}\big]_{aa}\big)\Big)^{2q-1-r}\\
			 &\cdot\Big(\eta\sum_a\big(\operatorname{Im}\big[\sum_{k=1}^{l}\mathcal{G}_{(j)}^{\gamma,0}(-Y^{\gamma}_j(Y^{\gamma}_j)^{*}\mathcal{G}_{(j)}^{\gamma,0})^k\big]_{aa}-\operatorname{Im}\big[\sum_{k=1}^{l}\widetilde{\mathcal{G}}_{(j)}^{0,0}(-\widetilde{\mathfrak{hy}}_j\widetilde{\mathfrak{hy}}_j^{*}\widetilde{\mathcal{G}}_{(j)}^{0,0})^k\big]_{aa}\big)\Big)^r\Big]+\mathrm{o}\big(\mathfrak{d}_{2q-1}(Y^{\gamma}_j,\widetilde{\mathfrak{hy}}_j)\big)\\
			 &\lesssim\mathbb{E}_{\Psi}\big(\mathfrak{d}_{2q-1}(0,0)\big)+\sum_{r=1}^{2q-1}C_q\mathbb{E}_{\Psi}\big(\mathfrak{d}_{2q-1}(0,0)\big)^{\frac{2q-1-r}{2q-1}}\cdot\mathbb{E}_{\Psi}\Big[\Big(\eta\sum_a\big(\operatorname{Im}\big[\sum_{k=1}^{l}\mathcal{G}_{(j)}^{\gamma,0}(-Y^{\gamma}_j(Y^{\gamma}_j)^{*}\mathcal{G}_{(j)}^{\gamma,0})^k\big]_{aa}\\
			 &-\operatorname{Im}\big[\sum_{k=1}^{l}\widetilde{\mathcal{G}}_{(j)}^{0,0}(-\widetilde{\mathfrak{hy}}_j\widetilde{\mathfrak{hy}}_j^{*}\widetilde{\mathcal{G}}_{(j)}^{0,0})^k\big]_{aa}\big)\Big)^{2q-1}\Big]^{\frac{r}{2q-1}}+\mathrm{o}\big(\mathfrak{d}_{2q-1}(Y^{\gamma}_j,\widetilde{\mathfrak{hy}}_j)\big)\\
			 &\lesssim\mathbb{E}_{\Psi}\big(\mathfrak{d}_{2q-1}(0,0)\big)+\sum_{r=1}^{2q-1}\big(C^{\prime}_q\mathbb{E}_{\Psi}\big(\mathfrak{d}_{2q-1}(0,0)\big)+\mathrm{o}_{\prec}(n^{-r\epsilon_l})\big)+\mathrm{o}\big(\mathfrak{d}_{2q-1}(Y^{\gamma}_j,\widetilde{\mathfrak{hy}}_j)\big),
		\end{split}
	\end{equation*}
	where we deployed the result of $\mathfrak{R}_{(j)}(\mathcal{G}^{\gamma},\widetilde{\mathcal{G}}_{(j)}^{0,\widetilde{\mathfrak{hy}}_j})$ in the first equality. In the second step, we applied Holder's inequality. In the third step, we utilized \eqref{eq_prf_greenfunctioncomparison_Fdiffer_prior}, and in the fourth step we employed Young's inequality. Hence we obtain a lower bound that
	\begin{equation*}
		(1-\mathrm{o}(1))\mathbb{E}_{\Psi}\big(\mathfrak{d}_{2q-1}(Y^{\gamma}_j,\widetilde{\mathfrak{hy}}_j)\big)\lesssim C\mathbb{E}_{\Psi}\big(\mathfrak{d}_{2q-1}(0,0)\big)+\mathrm{o}_{\prec}(n^{-\epsilon_l}),
	\end{equation*}
	for some constant $C>0$. Plugging this bound again in the expansion of $\mathbb{E}_{\Psi}\big(\mathfrak{d}_{2q-1}(Y^{\gamma}_j,\widetilde{\mathfrak{hy}}_j)\big)$, and using \eqref{eq_prf_greenfunctioncomparison_Fdiffer_prior} in the second step, we have
	\begin{equation*}
		\begin{split}
&			\mathbb{E}_{\Psi}\big(\mathfrak{d}_{2q-1}(Y^{\gamma}_j,\widetilde{\mathfrak{hy}}_j)\big)\\
&\gtrsim\mathbb{E}_{\Psi}\big(\mathfrak{d}_{2q-1}(0,0)\big)+C_1\sum_{r=1}^{2q-1}\mathbb{E}_{\Psi}\big(\mathfrak{d}_{2q-1-r}(Y^{\gamma}_j,\widetilde{\mathfrak{hy}}_j)\big)\cdot\mathrm{o}_{\prec}(n^{-r\epsilon_l})+\mathrm{o}\big(\mathfrak{d}_{2q-1}(Y^{\gamma}_j,\widetilde{\mathfrak{hy}}_j)\big).
		\end{split}
	\end{equation*}
	Ignoring the positive factor before the summation, moving the negative quantities to the left hand side and applying Young's inequality, we obtain that
	\begin{equation*}
		 C_2(1-\mathrm{o}(1))\mathbb{E}_{\Psi}\big(\mathfrak{d}_{2q-1}(Y^{\gamma}_j,\widetilde{\mathfrak{hy}}_j)\big)+\mathrm{o}_{\prec}(n^{-(2q-1)\epsilon_l})\gtrsim\mathbb{E}_{\Psi}\big(\mathfrak{d}_{2q-1}(0,0)\big).
	\end{equation*}
	Finally, after recovering $\mathfrak{d}_{2q-1}(Y^{\gamma}_j,\widetilde{\mathfrak{hy}}_j)$, we have for $k\ge5$,
	\begin{equation*}
		 (\mathfrak{I}_{11,k})_{ij}\lesssim\frac{\eta}{n^{1+\alpha/2}}\mathbb{E}_{\Psi}\big[\mathfrak{d}_{2q-1}(Y^{\gamma}_j,\widetilde{\mathfrak{hy}}_j)\big]\cdot\mathbbm{1}(\psi_{ij}=0)+\mathrm{o}_{\prec}(n^{-1-\alpha/2-(2q-1)\epsilon_1}).
	\end{equation*}
	
	\textbf{Case $k=3$.} We have,
	\begin{equation*}
		\begin{split}
			(\mathfrak{I}_{11,3})_{ij}=&\eta\mathbb{E}_{\Psi}\Big[ly_{ij}^3\beta_{0k_2\dots k_r}^{(K-3)}\mathfrak{P}_{(j),3k_2\ldots k_r}^{(K)}\cdot \mathfrak{d}_{2q-1}(0,0)(L_{ij}\rho_j^{-1}-\frac{\gamma t^{1/2}w_{ij}}{(1-\gamma^2)^{1/2}})\Big]\cdot\mathbbm{1}(\psi_{ij}=0)\\
			 &\lesssim\eta\mathbb{E}_{\Psi}\Big[\big((L_{ij}\rho_j^{-1})^4+t(L_{ij}\rho_j^{-1})^2w_{ij}^2+t^2w_{ij}^4+2tw_{ij}^2(M_{ij}\rho_j^{-1})^2\big)\beta_{0k_2\dots k_r}^{(K-3)}\\
			&\quad\cdot\mathfrak{P}_{(j),3k_2\ldots k_r}^{(K)}\cdot \mathfrak{d}_{2q-1}(0,0)\Big]\cdot\mathbbm{1}(\psi_{ij}=0).
		\end{split}
	\end{equation*}
	It is sufficient to consider $tw_{ij}^2(M_{ij}\rho_j^{-1})^2$, which is the leading term in the above expansion. Before proceeding, we take a look at $\mathfrak{P}_{(j),3k_2\ldots k_r}^{(K)}$. It should be noticed that from the pattern $3k_2\ldots k_r$, there must exist one cell $\sum_{k,j}[\mathcal{G}_{(j)}^{\gamma,0}]_{kl}[(\mathcal{G}_{(j)}^{\gamma,0})^2]_{ii}[\mathcal{G}_{(j)}^{\gamma,0}]_{ii}y_{kj}y_{lj}y_{ij}^3$, where $\sum_{k,j}[\mathcal{G}_{(j)}^{\gamma,0}]_{kl}[(\mathcal{G}_{(j)}^{\gamma,0})^2]_{ii}[\mathcal{G}_{(j)}^{\gamma,0}]_{ii}$ provides the typical order $\mathrm{O}(n^{3/2})$. On the other hand, the pattern $y_{ki}y_{lj}$ with even power will contribute at least $\mathrm{O}(n^{-2})$ after taking expectation. Therefore, using Theorem \ref{thm_greenfuncomp_entrywise_uniformbound}, \eqref{eq_prf_greenfunction_average_prior_G^2}, \eqref{eq_prf_greenfunction_average_prior_G^2Y} and \eqref{eq_prf_greenfunctioncomparison_relation_Ggamma&G0}, we have
	\begin{equation*}
		\begin{split}
		&	 (\mathfrak{I}_{11,3})_{ij}\lesssim\frac{\eta^2t}{n^2}\mathbb{E}_{\Psi}\big[\mathfrak{d}_{2q-1}(0,0)\big]\cdot\mathbbm{1}(\psi_{ij}=0)\\
&\lesssim\frac{\eta^2t}{n^2}\mathbb{E}_{\Psi}\big[\mathfrak{d}_{2q-1}(Y^{\gamma}_j,\widetilde{\mathfrak{hy}}_j)\big]\cdot\mathbbm{1}(\psi_{ij}=0)+\mathrm{o}_{\prec}(n^{-2-2\epsilon_{l}-(2q-1)\epsilon_1}),
		\end{split}
	\end{equation*}
	noting the fact that $\operatorname{Im}m^{(t)}(\zeta)\sim\eta$ for $z\in \gamma_1\lor \gamma_2$.
	
	\textbf{Case $k=1$.} We observe for this case that
	\begin{equation*}
		\begin{split}
			(\mathfrak{I}_{11,1})_{ij}=&\eta\mathbb{E}_{\Psi}\big[\big(L^2_{ij}\rho_j^{-2}-tw_{ij}^2\big)\beta_{0k_2\dots k_r}^{(K-1)}\cdot\mathfrak{P}_{(j),1k_2\ldots k_r}^{(K)}\cdot \mathfrak{d}_{2q-1}(0,0)\big]\cdot\mathbbm{1}(\psi_{ij}=0)\\
			&\lesssim\eta\mathbb{E}_{\Psi}\Big[\big(L^2_{ij}\rho_j^{-2}-tw_{ij}^2\big)\beta_{0k_2\dots k_r}^{(K-1)}\cdot\mathfrak{P}_{(j),1k_2\ldots k_r}^{(K)}\Big]\cdot\mathbb{E}_{\Psi}\big[\mathfrak{d}_{2q-1}(0,0)\big]\cdot\mathbbm{1}(\psi_{ij}=0).
		\end{split}
	\end{equation*}
	The strategy to handle this case is the same as the one in the proof of Theorem \ref{thm_greenfuncomp_entrywise_differror}. In other words we take the cumulant expansion for $L^2_{ij}\rho_j^{-2}$ in $\mathbb{E}_{\Psi}\big(L^2_{ij}\rho_j^{-2}\beta_{0k_2\dots k_r}^{(K-1)}\big)$ and take the cumulant expansion for $tw_{ij}^2$ in $\mathbb{E}_{\Psi}\big(tw_{ij}^2\beta_{0k_2\dots k_r}^{(K-1)}\big)$, respectively. The first cumulants  match under the second moment matching condition and the higher cumulants are of the order $\mathrm{O}(n^{-2-4\epsilon_l})$. Besides, the extra $L_{ij}\rho_j^{-1}$ indicates that there will be a cell with the form either $\sum_{l,s}[\mathcal{G}_{(j)}^{\gamma,0}]_{li}[\mathcal{G}_{(j)}^{\gamma,0}]_{ls}y^2_{lj}y_{sj}y_{ij}$ or $\sum_{l}[\mathcal{G}_{(j)}^{\gamma,0}]_{li}[(\mathcal{G}_{(j)}^{\gamma,0})^2]_{il}y_{lj}^2y_{ij}$. Then one may check that at most we have $|\mathfrak{P}_{(j),1k_2\ldots k_r}^{(K)}|=\mathrm{O}(n^{r-1})$ which can be canceled by the expectation of $\beta_{0k_2\dots k_r}^{(K-1)}$ or its derivatives. More details can be found in the proof of Theorem \ref{thm_greenfuncomp_entrywise_differror}. Finally, a standard argument as in cases $k\ge 5$ and $k=3$ gives that
	\begin{equation*}
		(\mathfrak{I}_{11,1})_{ij}\lesssim\frac{\eta t}{n^{3-2/\alpha+\epsilon_0}}\cdot \mathbb{E}_{\Psi}\big[\mathfrak{d}_{2q-1}(Y^{\gamma}_j,\widetilde{\mathfrak{hy}}_j)\big]\cdot\mathbbm{1}(\psi_{ij}=0)+\mathrm{o}_{\prec}(n^{-3+2/\alpha-\epsilon_0-2\epsilon_l-(2q-1)\epsilon_1}).
	\end{equation*}
	
	\item [\textbf{Case II.}] \textbf{For $\beta^{(K)}_{k_1k_2\dots k_r}$,  $k_1\ge 1$ and at least one of $k_2,\dots,k_r$ is odd.}\\
        For this case, we have
        \begin{equation*}
    \begin{split}
        &\eta\mathbb{E}_{\Psi}\big[\mathfrak{M}^{II}_{(j)}(0,0)\cdot\big(L_{ij}\rho_j^{-1}-\frac{\gamma t^{1/2}w_{ij}}{(1-\gamma^2)^{1/2}}\big)\big]\cdot\mathbbm{1}(\psi_{ij}=0)\\
        &=\eta \sum_{K=1,o}^{4s+1}\sum_{k=1}^K\mathbb{E}_{\Psi}\Big[\beta_{kk_2\dots k_r}^{(K-k)}\mathfrak{P}_{(j),kk_2\ldots k_r}^{(K)}\cdot \mathfrak{d}_{2q-1}(0,0) \big(L_{ij}\rho_j^{-1}-\frac{\gamma t^{1/2}w_{ij}}{(1-\gamma^2)^{1/2}}\big)\Big]\cdot\mathbbm{1}(\psi_{ij}=0)\\
        &:=\sum_{K=1,o}^{4s+1}\sum_{k=1}^K(\mathfrak{I}_{12,k})_{ij}.
    \end{split}
\end{equation*}

	We only consider $\beta_{2k_2\dots k_r}^{(K)}$ and $\beta_{1k_2\dots k_r}^{(K)}$ since the others  can be handled similarly. Firstly, for $\beta_{2k_2\dots k_r}^{(K)}$, we need to consider the quantity
	\begin{equation*}
		(\mathfrak{I}_{12,2})_{ij}\sim \eta\mathbb{E}_{\Psi}\big[\big(L_{ij}^3\rho_j^{-3}-tL_{ij}\rho_j^{-1}w_{ij}^2\big)\beta_{0k_2\dots k_r}^{(K-2)}\cdot\mathfrak{P}_{(j),2k_2\ldots k_r}^{(K)}\cdot \mathfrak{d}_{2q-1}(0,0)\big]\cdot\mathbbm{1}(\psi_{ij}=0).
	\end{equation*}
	The proof of Theorem \ref{thm_greenfuncomp_entrywise_differror} shows that each odd $k_l,2\le l\le r$ could reduce the order of $\mathfrak{P}_{(j),2k_2\ldots k_r}^{(K)}$ from $\mathrm{O}(n^{r-1})$ to $\mathrm{O}(n^{r-3/2})$. Together with the fact that we can gain at least $\mathrm{o}(n^{-3/2-(r-1)})$ from the expectations $\mathbb{E}_{\Psi}\big[L_{ij}^3\rho_j^{-3}\beta_{0k_2\dots k_r}^{(K-2)}\big]$ or $\mathbb{E}_{\Psi}\big[tL_{ij}\rho_j^{-1}w_{ij}^2\beta_{0k_2\dots k_r}^{(K-2)}\big]$, we can obtain that
	\begin{equation*}
		(\mathfrak{I}_{12,2})_{ij}\lesssim \mathrm{o}(n^{-2-\epsilon_l})\mathbb{E}_{\Psi}\big[\mathfrak{d}_{2q-1}(Y^{\gamma}_j,\widetilde{\mathfrak{hy}}_j)\big]\cdot\mathbbm{1}(\psi_{ij}=0)+\mathrm{o}_{\prec}(n^{-2-\epsilon_l-(2q-1)\epsilon_1}).
	\end{equation*}
	
	Secondly, for $\beta_{1k_2\dots k_r}^{(K)}$, we have
	\begin{equation*}
		(\mathfrak{I}_{12,1})_{ij}\sim \eta\mathbb{E}_{\Psi}\big[\big(L_{ij}^2\rho_j^{-2}-tw_{ij}^2\big)\beta_{0k_2\dots k_r}^{(K-1)}\cdot\mathfrak{P}_{(j),1k_2\ldots k_r}^{(K)}\cdot \mathfrak{d}_{2q-1}(0,0)\big]\cdot\mathbbm{1}(\psi_{ij}=0).
	\end{equation*}
	Similarly to the strategy to handle \textbf{Case $k=1$} in \textbf{Case I}, we apply the cumulant expansion. Also notice that we use Lemma \ref{lem_oddmoment_est} to approximate $\mathbb{E}_{\Psi}\beta_{0k_2\dots k_r}^{(K-1)}$ if there is at least one $k_l=1$ for $2\le l\le r$, or apply the deterministic bound $|y_{kj}|\le 1$ to ensure all $k_l$'s are even if there is no $k_l=1$ and sequentially employ Lemma \ref{lem_moment_rates}. Thereafter, one may follow a standard routine to obtain that
	\begin{equation*}
		(\mathfrak{I}_{12,1})_{ij}\lesssim \mathrm{o}(n^{-2-\epsilon_l})\mathbb{E}_{\Psi}\big[\mathfrak{d}_{2q-1}(Y^{\gamma}_j,\widetilde{\mathfrak{hy}}_j)\big]\cdot\mathbbm{1}(\psi_{ij}=0)+\mathrm{o}_{\prec}(n^{-2-\epsilon_l-(2q-1)\epsilon_1}).
	\end{equation*}
	
	\item [\textbf{Case III.}] \textbf{For $\beta^{(K)}_{k_1k_2\dots k_r}$,  $k_1= 0$.}\\
	At this time, we have
\begin{equation*}
    \begin{split}
        &\eta\mathbb{E}_{\Psi}\big[\mathfrak{M}^{II}_{(j)}(0,0)\cdot\big(L_{ij}\rho_j^{-1}-\frac{\gamma t^{1/2}w_{ij}}{(1-\gamma^2)^{1/2}}\big)\big]\cdot\mathbbm{1}(\psi_{ij}=0)\\
        &=\eta \sum_{K=1,o}^{4s+1}\mathbb{E}_{\Psi}\Big[\beta_{0k_2\dots k_r}^{(K)}\mathfrak{P}_{(j),kk_2\ldots k_r}^{(K)}\cdot \mathfrak{d}_{2q-1}(0,0) \big(L_{ij}\rho_j^{-1}-\frac{\gamma t^{1/2}w_{ij}}{(1-\gamma^2)^{1/2}}\big)\Big]\cdot\mathbbm{1}(\psi_{ij}=0)\\
        &:=\sum_{K=1,o}^{4s+1}(\mathfrak{I}_{13,0})_{ij}.
    \end{split}
\end{equation*}
We only need to consider
	\begin{equation*}
		(\mathfrak{I}_{13,0})_{ij}\sim \eta\mathbb{E}_{\Psi}\big[\big(L_{ij}\rho_j^{-1}\big)\beta_{0k_2\dots k_r}^{(K)}\cdot\mathfrak{P}_{(j),0k_2\ldots k_r}^{(K)}\cdot \mathfrak{d}_{2q-1}(0,0)\big]\cdot\mathbbm{1}(\psi_{ij}=0).
	\end{equation*}
	Some calculations show that the leading quantity in $\beta_{0k_2\dots k_r}^{(K)}\mathfrak{P}_{(j),0k_2\ldots k_r}^{(K)}$ is $\sum_{k\neq i}[(\mathcal{G}_{(j)}^{\gamma,0})^2]_{ik}y_{kj}$. So the typical order of $\sum_{i,j} (\mathfrak{I}_{13,0})_{ij}$ is
    \begin{equation*}
		\begin{split}
			&\sum_{i,j}\mathbb{E}_{\Psi}\big[\sum_{k\neq i}[(\mathcal{G}_{(j)}^{\gamma,0})^2]_{ik}y_{kj}L_{ij}\rho_j^{-1}\big]\cdot\mathbbm{1}(\psi_{ij}=0)\\
   \lesssim &\sum_{i,j}(\sum_k[(\mathcal{G}_{(j)}^{\gamma,0})^2]^2_{ik})^{1/2}n^{1/2}\mathbb{E}(y_{kj}L_{ij}\rho_j^{-1})\cdot\mathbbm{1}(\psi_{ij}=0)\lesssim \mathrm{o}(n^{3/2-\alpha/2}).
		\end{split}
	\end{equation*}
by a similar argument in the proof of Theorem \ref{thm_greenfuncomp_entrywise_differror}. This concludes \textbf{Case III}.
	
\end{itemize}

Finally, combining the estimates for three cases $\mathfrak{M}_{(j)}^{I}(0,0)$ to $\mathfrak{M}_{(j)}^{III}(0,0)$, recalling the result for $\mathfrak{R}_{(j)}(\mathcal{G}^{\gamma},\widetilde{\mathcal{G}}^{0,\widetilde{\mathfrak{hy}}_j}_{(j)})$, we obtain that for $\alpha\ge 3$
\begin{equation*}
	\begin{split}
		(\mathfrak{J}_1)_{ij}
		\lesssim \frac{\eta}{n^{2+\epsilon_l}}\mathbb{E}_{\Psi}\big[\mathfrak{d}_{2q}(Y^{\gamma}_j,\widetilde{\mathfrak{hy}}_j)\big]+\mathrm{o}_{\prec}(n^{-\alpha/2-1/2-2\epsilon_l-(2q-1)\epsilon_1}).
	\end{split}
\end{equation*}
Together with the result for $(\mathfrak{J}_2)_{ij}$, we have
\begin{equation*}
	\frac{\partial \mathbb{E}_{\Psi}(|n\eta(\operatorname{Im}m^{\gamma}(z)-\operatorname{Im}\widetilde{m}_{n,0}(z))|^{2q})}{\partial\gamma}=-\frac{2q\eta}{n^{\epsilon_l}}\Big(\mathbb{E}_{\Psi}\big[\mathfrak{d}_{2q}(Y^{\gamma}_j,\widetilde{\mathfrak{hy}}_j)\big]+\mathrm{O}(n^{-2q\epsilon_l})\Big)+\mathrm{o}_{\prec}(n^{-2\epsilon_l-(2q-1)\epsilon_1}).
\end{equation*}
Therefore, for any $0\le\gamma\le1$,
\begin{equation*}
	\begin{split}
		 &\mathbb{E}_{\Psi}\big(|n\eta(\operatorname{Im}m^{\gamma}(z)-\operatorname{Im}\widetilde{m}_{n,0}(z))|^{2q}\big)-\mathbb{E}_{\Psi}\big(|n\eta(\operatorname{Im}m^{0}(z)-\operatorname{Im}\widetilde{m}_{n,0}(z))|^{2q}\big)\\
		&=\int_{0}^{\gamma}\frac{\partial \mathbb{E}_{\Psi}(|n\eta(\operatorname{Im}m^{\gamma^{\prime}}(z)-\operatorname{Im}\widetilde{m}_{n,0}(z))|^{2q})}{\partial\gamma^{\prime}}\mathrm{d}\gamma^{\prime}.
	\end{split}
\end{equation*}
Taking supremum over $\gamma$, and collecting the above estimates, we have
\begin{equation*}
	\begin{split}
		 &\sup_{0\le\gamma\le1}\mathbb{E}_{\Psi}\big(|n\eta(\operatorname{Im}m^{\gamma}(z)-\operatorname{Im}\widetilde{m}_{n,0}(z))|^{2q}\big)-\mathbb{E}_{\Psi}\big(|n\eta(\operatorname{Im}m^{0}(z)-\operatorname{Im}\widetilde{m}_{n,0}(z))|^{2q}\big)\\
		&\lesssim \mathrm{O}(n^{-\epsilon_l})\sup_{0\le\gamma\le1}\mathbb{E}_{\Psi}\big(|n\eta(\operatorname{Im}m^{\gamma}(z)-\operatorname{Im}\widetilde{m}_{n,0}(z))|^{2q}\big)+\mathrm{O}(n^{-(2q+1)\epsilon_l}).
	\end{split}
\end{equation*}

And for $\alpha\in (2,3)$, combining all the estimates for $(\mathcal{J}_{1,k})_{ij}$ as well as the result for $(\mathcal{J}_{2,k})_{ij}$, we actually obtain that
\begin{eqnarray*}
&&\frac{\partial \mathbb{E}_{\Psi}(|n^{3/2-\alpha/2}\eta(\operatorname{Im}m^{\gamma}(z)-\operatorname{Im}\widetilde{m}_{n,0}(z))|^{2q})}{\partial\gamma}\\
&=&-\frac{2q\eta}{n^{\epsilon_l}}\Big(\mathbb{E}_{\Psi}\big[\mathfrak{d}_{2q}(Y^{\gamma}_j,\widetilde{\mathfrak{hy}}_j)\big]+\mathrm{O}(n^{-2q\epsilon_l})\Big)+\mathrm{o}_{\prec}(n^{-2\epsilon_l-(2q-1)\epsilon_1}).
\end{eqnarray*}
This gives that
\begin{equation*}
	\begin{split}
&	 \sup_{0\le\gamma\le1}(1+\mathrm{o}(1))\mathbb{E}_{\Psi}\big(|n^{3/2-\alpha/2}\eta(\operatorname{Im}m^{\gamma}(z)-\operatorname{Im}\widetilde{m}_{n,0}(z))|^{2q}\big)\\
&\qquad \qquad -\mathbb{E}_{\Psi}\big(|n^{3/2-(3-\alpha)_{+}}\eta(\operatorname{Im}m^{0}(z)-\operatorname{Im}\widetilde{m}_{n,0}(z))|^{2q}\big)
	\lesssim\mathrm{O}(n^{-(2q+1)\epsilon_l}).\\
	\end{split}
\end{equation*}
Putting together the above two cases, we can conclude this theorem.  \qed

\subsection{Proofs for characteristic function estimation}\label{app_sec_charafun}
\subsubsection{Collection of derivatives}\label{subsubsec_derivative}
In this subsection, we collect some derivatives that appear in the cumulant expansion, which can be obtained by repeatedly applying the chain rule. Recalling the definitions of $\rho_j, L_{ij},M_{ij},H_{ij}$, and noting the fact that $\psi_{ij},\chi_{ij}$ are indicators, we have
\begin{gather*}
	\partial_{m,ij} Y_{ij}=\rho_j^{-1}(1-M^2_{ij}\rho_j^{-2}),
	\partial_{m,ij} Y_{kj}=-Y_{kj}\rho_{j}^{-2}M_{ij},~\text{for}~k\ne i.
\end{gather*}
Further, for $a\in [n],b\in [p]$, we denote by $E_{ab}$  the $n\times p$ matrix with the entries $E_{ab}=\delta_{ac}\delta_{bd}$. Let $F_{ab}=Y(E_{ab})^*E_{ab}$ and
\begin{equation*}
	\begin{split}
		\mathscr{P}_{0}^{ab}=E_{ab}(E_{ab})^*, ~\mathscr{P}_{1}^{ab}=E_{ab}Y^*,~\mathscr{P}_{2}^{ab}=Y(E_{ab})^*,\mathscr{P}_{3}^{ab}=E_{ab}^*Y,~\mathscr{P}_{4}^{ab}=Y^*(E_{ab}),\\
		\mathscr{Q}_{0}^{ab}=F_{ab}(F_{ab})^*, ~\mathscr{Q}_{1}^{ab}=F_{ab}Y^*,~\mathscr{Q}_{2}^{ab}=Y(F_{ab})^*,\mathscr{Q}_{3}^{ab}=F_{ab}^*Y,~\mathscr{Q}_{4}^{ab}=Y^*(F_{ab}).
	\end{split}
\end{equation*}
Notice $\mathscr{Q}_{1}^{ab}=\mathscr{Q}_{2}^{ab}=YE_{ab}^*E_{ab}Y^*$. With the notation above, we get that
\begin{equation*}
	\partial_{m,ij} Y=E_{ij}\rho_{j}^{-1}(1-M_{ij}^2\rho_{j}^{-2})-M_{ij}\rho_{j}^{-2}Y^{(ij)}(E_{ij})^*E_{ij}=:E_{ij}\rho_{j}^{-1}+F_{ij}(-M_{ij}\rho_{j}^{-2}),
\end{equation*}
where $Y^{(ij)}$ denotes the matrix $Y$ with $Y_{ij}$ being replaced by $0$, which satisfies that the $j$-th column is non-zero and others are zero. Thereafter,
\begin{equation*}
	\begin{split}
		\partial_{m,ij} (\mathcal{G}(z))^{q}=-\sum_{a=1}^2\sum_{\substack{q_1,q_2\ge 1\\ q_1+q_2=q+1}}(\mathcal{G}(z))^{q_1}(\rho_{j}^{-1}\mathscr{P}_{a
		}^{ij}+(-M_{ij}\rho_{j}^{-2})\mathscr{Q}_{a
		}^{ij})(\mathcal{G}(z))^{q_2},\\
		\partial_{m,ij} (G(z))^{q}=-\sum_{a=3}^4\sum_{\substack{q_1,q_2\ge 1\\ q_1+q_2=q+1}}({G}(z))^{q_1}(\rho_{j}^{-1}\mathscr{P}_{a
		}^{ij}+(-M_{ij}\rho_{j}^{-2})\mathscr{Q}_{a
		}^{ij})({G}(z))^{q_2},
	\end{split}
\end{equation*}
which further imply that
\begin{equation*}
	\begin{split}
		\frac{\partial (Y^*\mathcal{G}(z))}{\partial M_{ij}}=(E_{ij}\rho_{j}^{-1}+F_{ij}(-M_{ij}\rho_{j}^{-2}))^*\mathcal{G}(z)
		-\sum_{a=1}^{2}[Y^*\mathcal{G}(z)(\rho_{j}^{-1}\mathscr{P}_{a
		}^{ij}+(-M_{ij}\rho_{j}^{-2})\mathscr{Q}_{a
		}^{ij})\mathcal{G}(z)],\\
		\frac{\partial (G(z)Y^*)}{\partial M_{ij}}=G(z)(E_{ij}\rho_{j}^{-1}+F_{ij}(-M_{ij}\rho_{j}^{-2}))^*
		-\sum_{a=3}^{4}[{G}(z)(\rho_{j}^{-1}\mathscr{P}_{a
		}^{ij}+(-M_{ij}\rho_{j}^{-2})\mathscr{Q}_{a
		}^{ij}){G}(z)Y^*].
	\end{split}
\end{equation*}
Similarly, for integer $q\ge 1$, one can get
\begin{equation*}
	\frac{\partial^q \rho_{j}^{-1}}{\partial M_{ij}^{q}}\sim \begin{cases}
		\rho_{j}^{-q-1}, ~q=2q_1,\\
		\rho_{j}^{-q-2}M_{ij}, ~q=2q_1+1,
	\end{cases}
\end{equation*}
and
\begin{equation*}
	\frac{\partial^q Y_{ij}}{\partial M_{ij}^{q}}\sim \begin{cases}
		\rho_{j}^{-q-1}M_{ij}, ~q=2q_1,\\
		\rho_{j}^{-q}, ~q=2q_1+1,
	\end{cases}, \frac{\partial^q Y_{kj}}{\partial M_{ij}^{q}}\sim \begin{cases}
		Y_{kj}\rho_{j}^{-q}, ~q=2q_1,\\
		Y_{kj}\rho_{j}^{-q-1}M_{ij}, ~q=2q_1+1,
	\end{cases}
\end{equation*}
where we remark here that the symbol $\sim$ means the right hand side of $\sim$ is  the main term of the left hand side up to a multiplicative constant. The negligible terms such as $M_{ij}^2\rho_{j}^{-2}$ are omitted. Thus, we have
\begin{equation*}
	\frac{\partial^q Y}{\partial M_{ij}^{q}}\sim \begin{cases}
		\rho_{j}^{-q-1}M_{ij}E_{ij}+\rho_{j}^{-q}F_{ij}, ~q=2q_1,\\
		\rho_{j}^{-q}E_{ij}+M_{ij}\rho_{j}^{-q-1}F_{ij}, ~q=2q_1+1,
	\end{cases}~\frac{\partial^q F_{ij}}{\partial M_{ij}^{q}}\sim \begin{cases}
		\rho_{j}^{-q}F_{ij}, ~q=2q_1,\\
		M_{ij}\rho_{j}^{-q-1}F_{ij}, ~q=2q_1+1.
	\end{cases}
\end{equation*}
With the notation above, recalling the high probability bounds $\|\mathcal{G}(z)\|\prec 1$ and $\|YY^*\|\prec 1$, we have the estimates $\partial_{m,ij} \mathcal{G}(z)\sim\sum_{a=1}^2\mathcal{G}(z)(\rho_{j}^{-1}\mathscr{P}_{a}^{ij}+M_{ij}\rho_{j}^{-2}\mathscr{Q}_{a}^{ij})\mathcal{G}(z)$. For the high-order derivatives,  the induction gives $ \partial_{m,ij}^q \mathcal{G}(z)\sim \rho_{j}^{-q}\mathcal{G}(z)\cdots \mathcal{G}(z)$, where the factors $E_{ij},F_{ij},Y$ are hidden in the product with appropriate orders. This further implies
\begin{equation*}
	\frac{\partial^q (Y^*\mathcal{G}(z))}{\partial M_{ij}^{q}}\sim \rho_{j}^{-q}(E_{ij}^*+F_{ij}^*+Y^*)\mathcal{G}(z)\cdots \mathcal{G}(z)
\end{equation*}
for $q\ge 2$, where the terms with $M_{ij}\rho_{j}^{-1}$ are absorbed into the main term since $\rho_{j}^{-1}\lesssim n^{-1/2}$ with high probability. Moreover, similarly to the estimate of $\mathrm{I}_{t3}$ in \cite[Lemma 5.4]{bao2022spectral}, we have
\begin{equation*}
	\begin{split}
		\frac{\partial \mathrm{e}^{\mathrm{i}x\langle L_{1}(f)\rangle}}{\partial M_{ij}}=& -\frac{x}{2\pi}\oint_{\bar{\gamma}_2^0}\frac{\partial \operatorname{tr}\mathcal{G}(z_2)}{\partial M_{ij}}\mathrm{e}^{\mathrm{i}x\langle L_{1}(f)\rangle} f(z_2)\mathrm{d}z_2+\mathrm{O}_{\prec}(n^{-K/2}),\\
		\frac{\partial^2 \mathrm{e}^{\mathrm{i}x\langle L_{1}(f)\rangle}}{\partial M_{ij}^2}=&
		-\frac{x}{2\pi}\oint_{\bar{\gamma}_2^0}\frac{\partial^2 \operatorname{tr}\mathcal{G}(z_2)}{\partial M_{ij}^2}\mathrm{e}^{\mathrm{i}x\langle L_{1}(f)\rangle}f(z_2)\mathrm{d}z_2\\
  &+\frac{x^2}{4\pi^2}\oint_{\bar{\gamma}_2^0}\oint_{\bar{\gamma}_3^0}\frac{\partial \operatorname{tr}\mathcal{G}(z_2)}{\partial M_{ij}}\frac{\partial \operatorname{tr}\mathcal{G}(z_3)}{\partial M_{ij}} \mathrm{e}^{\mathrm{i}x\langle L_{1}(f)\rangle} f(z_2)f(z_3)\mathrm{d}z_2\mathrm{d}z_3+\mathrm{O}_{\prec}(n^{-K/2})
	\end{split}
\end{equation*}
with
\begin{equation*}
	\begin{split}
		\partial_{m,ij}\operatorname{tr}\mathcal{G}(z_2) = &-2\rho_{j}^{-1}[\mathcal{G}(z_2)\mathcal{G}(z_2)Y]_{ij}+2M_{ij}\rho_{j}^{-2}[Y^*\mathcal{G}(z_2)\mathcal{G}(z_2)Y]_{jj},\\
		\partial_{m,ij}^2\operatorname{tr}\mathcal{G}(z_2) \sim &2\rho_{j}^{-2}z_2[\mathcal{G}(z_2)\mathcal{G}(z_2)]_{ii}[G(z_2)]_{jj}+2\rho_{j}^{-2}[\mathcal{G}(z_2)]_{ii}[Y^*\mathcal{G}(z_2)\mathcal{G}(z_2)Y]_{jj}\\
  &+2\rho_{j}^{-2}[Y^*\mathcal{G}(z_2)\mathcal{G}(z_2)Y]_{jj}+\text{off-diagonals},
	\end{split}
\end{equation*}
where we keep the diagonal elementts which are main terms. By induction, we have $\partial_{m,ij}^{q} \mathrm{e}^{\mathrm{i}x\langle L_{1}(f)\rangle}\lesssim \rho_{j}^{-q}$ with the error terms absorbed due to the entrywise local law. Throughout this section, for simplicity, we denote by $A_{ii}, B_{ii}$ and  $A_{ij}, B_{ij}$ the diagonal and off-diagonal elements of the matrices (cf. $[\mathcal{G}(z)]_{ii}, [\mathcal{G}(z)Y]_{ij}, [\mathcal{G}^2(z)Y]_{ij}$) without abuse of notation.

\subsubsection{Proof of Lemma \ref{lem_cumulant_estM}}\label{subsubsec_prf_cumulantexpansionL}
\textbf{I. Terms with $\kappa_{1,m}$}\\
Recalling \eqref{eq_rhoj_highpro}, the estimates of $\kappa_{1,m}\lesssim n^{(1/2-\epsilon_h)(-\alpha+1)}$ in Section \ref{sec_characteristicfunctionestimation} and the Cauchy-Schwarz inequality $\sum_{ij}[Y^*\mathcal{G}(z_1)]_{ji}^2=\operatorname{tr}(Y^*\mathcal{G}(z_1)\mathcal{G}(z_1)Y)\asymp n$,  one has for $\alpha>3$,
\begin{equation*}
	\begin{split}
		\sum_{ij}\kappa_{1,m}\mathscr{M}_{0,ij}
		\lesssim &\kappa_{1,m}(\sum_{ij}\rho_j^{-2}\sum_{ij}[Y^*\mathcal{G}(z_1)]_{ji}^2)^{1/2}\lesssim n^{1+(1/2-\epsilon_h)(-\alpha+1)}\lesssim n^{-c}.
	\end{split}
\end{equation*}
The case where $\alpha\in (2,3]$ can be handled by the arguments for $\kappa_{3,m}$ below, and thus we omit the details here.

\

\noindent\textbf{II. Terms with $\kappa_{2,m}$}\\
Invoking the derivatives in Section \ref{subsubsec_derivative}, we have
\begin{equation*}
	\begin{split}
		\sum_{ij}\mathscr{M}_{1,ij}&=-\sum_{ij}\mathbb{E}_{\Psi}^{\chi}\left[\rho_{j}^{-3}M_{ij}[Y^*\mathcal{G}(z_1)]_{ji}\langle \mathrm{e}^{\mathrm{i}x\langle L_{1}(f)\rangle}\rangle\right]\\
		&=-\mathbb{E}_{\Psi}^{\chi}\left(\operatorname{tr}(Y^*\mathcal{G}(z_1)M(\operatorname{diag}(S))^{-1/2})(\operatorname{diag}(S))^{-1})\langle \mathrm{e}^{\mathrm{i}x\langle L_{1}(f)\rangle}\rangle\right)\\
		 &=-\mathbb{E}_{\Psi}^{\chi}\left(\operatorname{tr}(Y^*\mathcal{G}(z_1)M(\operatorname{diag}(S))^{-1/2})[n^{-1}I+(\operatorname{diag}(S))^{-1}-n^{-1}I])\langle \mathrm{e}^{\mathrm{i}x\langle L_{1}(f)\rangle}\rangle\right)\\
		&= -n^{-1} \mathbb{E}_{\Psi}^{\chi}\left(\operatorname{tr}(Y^*\mathcal{G}(z_1)M(\operatorname{diag}(S))^{-1/2})\langle \mathrm{e}^{\mathrm{i}x\langle L_{1}(f)\rangle}\rangle\right)+\mathrm{O}_{\prec}(n^{-\epsilon_{\alpha}}),
	\end{split}
\end{equation*}
where we have used \eqref{eq_tr_SS} and  $\|Y^*\mathcal{G}(z_1)M(\operatorname{diag}(S))^{-1/2})\|\prec 1$. Recalling the rate
\begin{equation*}
	\mathbb{E}_{\Psi}^{\chi}\left(\operatorname{tr}(Y^*\mathcal{G}(z_1)H(\operatorname{diag}(S))^{-1/2})\langle \mathrm{e}^{\mathrm{i}x\langle L_{1}(f)\rangle}\rangle\right)\lesssim \sum_{ij}\mathbb{E}_{\Psi}^{\chi}[H_{ij}\rho_j^{-1}[Y^*\mathcal{G}(z_1)]_{ji}]\lesssim n^{3/2+(1/2-\epsilon_h)(-\alpha+1)},
\end{equation*}
and the relation $Y^*\mathcal{G}(z)Y=Y^*\mathcal{G}(z_1)(M+H)(\operatorname{diag}(S))^{-1/2}$, we have
\begin{equation*}
	\sum_{ij}\mathscr{M}_{1,ij}=-\mathbb{E}_{\Psi}^{\chi}((1+z_1m(z_1))\langle \mathrm{e}^{\mathrm{i}x\langle L_{1}(f)\rangle}\rangle)+\mathrm{O}_{\prec}(n^{-\epsilon_{\alpha}}).
\end{equation*}
Similarly, with the notation from Section \ref{subsubsec_derivative}, we have
\begin{equation*}
	\begin{split}
		\mathscr{M}_{2,ij}=
		 &\mathbb{E}_{\Psi}^{\chi}\left[\rho_{j}^{-2}\left([(E_{ij})^*\mathcal{G}(z_1)]_{ji}-\sum_{a=1}^{2}[Y^*(\mathcal{G}(z_1))\mathscr{P}_{a}^{ij}(\mathcal{G}(z_1))]_{ji}\right)\langle \mathrm{e}^{\mathrm{i}x\langle L_{1}(f)\rangle}\rangle\right]\\
		 &-\mathbb{E}_{\Psi}^{\chi}\left[M_{ij}\rho_{j}^{-3}\left([(F_{ij})^*\mathcal{G}(z_1)]_{ji}-\sum_{a=1}^{2}[Y^*(\mathcal{G}(z_1))\mathscr{Q}_{a}^{ij}(\mathcal{G}(z_1))]_{ji}\right)\langle \mathrm{e}^{\mathrm{i}x\langle L_{1}(f)\rangle}\rangle\right].
	\end{split}
\end{equation*}
By direct calculation, we have
\begin{equation*}
	\begin{split}
		[(E_{ij})^*\mathcal{G}(z_1)]_{ji}-\sum_{a=1}^{2}[Y^*\mathcal{G}(z_1)\mathscr{P}_{a}^{ij}\mathcal{G}(z_1)]_{ji}
		&=-[Y^*\mathcal{G}(z_1)E_{ij}Y^*\mathcal{G}(z_1)]_{ji}-z_1[G(z_1)(E_{ij})^*\mathcal{G}(z_1)]_{ji},\\
		[(F_{ij})^*\mathcal{G}(z_1)]_{ji}-\sum_{a=1}^{2}[Y^*\mathcal{G}(z_1)\mathscr{Q}_{a}^{ij}\mathcal{G}(z_1)]_{ji}
		&=-[Y^*\mathcal{G}(z_1)F_{ij}Y^*\mathcal{G}(z_1)]_{ji}-z_1[G(z_1)(F_{ij})^*\mathcal{G}(z_1)]_{ji},
	\end{split}
\end{equation*}
where we have applied the identities
\begin{equation*}
	\begin{split}
		Y^*(\mathcal{G}(z))^lY=(G(z))^lY^*Y=(G(z))^{l-1}+z(G(z))^{l},~
		(\mathcal{G}(z))^lYY^*=(\mathcal{G}(z))^{l-1}+z(\mathcal{G}(z))^{l},
	\end{split}
\end{equation*}
with ${G}(z)=(Y^*Y-zI)^{-1}$. Thus, we can rewrite $\mathscr{M}_{2,ij}$ as
\begin{equation*}
	\begin{split}
		\mathscr{M}_{2,ij}
		 =&\mathbb{E}_{\Psi}^{\chi}\left[\rho_{j}^{-2}\left(-[Y^*\mathcal{G}(z_1)]_{ji}[Y^*\mathcal{G}(z_1)]_{ji}-z_1[G(z_1)]_{jj}[\mathcal{G}(z_1)]_{ii}\right)\langle \mathrm{e}^{\mathrm{i}x\langle L_{1}(f)\rangle}\rangle\right]\\
		 &+\mathbb{E}_{\Psi}^{\chi}\left[M_{ij}\rho_{j}^{-3}\left([Y^*\mathcal{G}(z_1)Y]_{jj}[Y^*\mathcal{G}(z_1)]_{ji}+z_1[G(z_1)]_{jj}[Y^*\mathcal{G}(z_1)]_{ji}\right)\langle \mathrm{e}^{\mathrm{i}x\langle L_{1}(f)\rangle}\rangle\right]\\
	\end{split}
\end{equation*}
by the identity $Y^*(\mathcal{G}(z_1))^{l}Y=(G(z_1))^{l-1}+z_1(G(z_1))^l$.
For the first term, one has
\begin{equation*}
	\begin{split}
		&\sum_{ij}\mathbb{E}_{\Psi}^{\chi}\left[\rho_{j}^{-2}[Y^*\mathcal{G}(z_1)]_{ji}[Y^*\mathcal{G}(z_1)]_{ji}\langle \mathrm{e}^{\mathrm{i}x\langle L_{1}(f)\rangle}\rangle\right]
		=\mathbb{E}_{\Psi}^{\chi}\left[\operatorname{tr}(Y^*\mathcal{G}(z_1)\mathcal{G}(z_1)Y(\operatorname{diag}(S))^{-1})\langle \mathrm{e}^{\mathrm{i}x\langle L_{1}(f)\rangle}\rangle\right]\\
		=&\mathbb{E}_{\Psi}^{\chi}\left[n^{-1}\operatorname{tr}(YY^*\mathcal{G}(z_1)\mathcal{G}(z_1))\langle \mathrm{e}^{\mathrm{i}x\langle L_{1}(f)\rangle}\rangle\right]+\mathrm{O}_{\prec}(n^{-\epsilon_{\alpha}})=\mathbb{E}_{\Psi}^{\chi}\left[\partial_{z_1}(z_1m(z_1))\langle \mathrm{e}^{\mathrm{i}x\langle L_{1}(f)\rangle}\rangle\right]+\mathrm{O}_{\prec}(n^{-\epsilon_{\alpha}}),
	\end{split}
\end{equation*}
and
\begin{equation*}
	\begin{split}
		\sum_{ij}\mathbb{E}_{\Psi}^{\chi}\left[\rho_{j}^{-2}\left(z_1[G(z_1)]_{jj}[\mathcal{G}(z_1)]_{ii}\right)\langle \mathrm{e}^{\mathrm{i}x\langle L_{1}(f)\rangle}\rangle\right]=\mathbb{E}_{\Psi}^{\chi}[(\phi^{-1}z_1\underline{m}(z_1)+\mathrm{O}_{\prec}(n^{-\epsilon_{\alpha}}))\operatorname{tr}\mathcal{G}(z_1)\langle \mathrm{e}^{\mathrm{i}x\langle L_{1}(f)\rangle}\rangle]
	\end{split}
\end{equation*}
by \eqref{eq_tr_SS}.
For the second one, similarly to $\mathscr{M}_{1,ij}$, we can using the cumulant expansion to obtain
\begin{equation*}
	\begin{split}
		&\mathbb{E}_{\Psi}^{\chi}\left[M_{ij}\rho_{j}^{-3}[G(z_1)]_{jj}[Y^*\mathcal{G}(z_1)]_{ji}\langle \mathrm{e}^{\mathrm{i}x\langle L_{1}(f)\rangle}\rangle\right]\\
		=&\kappa_{2,m}\mathbb{E}_{\Psi}^{\chi}\left[\partial_{m,ij}(\rho_{j}^{-3}[G(z_1)]_{jj}[Y^*\mathcal{G}(z_1)]_{ji}\langle \mathrm{e}^{\mathrm{i}x\langle L_{1}(f)\rangle}\rangle)\right]+\mathrm{O}_{\prec}(n^{-2-2\epsilon_{\alpha}}),
	\end{split}
\end{equation*}
for $\alpha>2$ where the error follows from \eqref{eq_rhoj_highpro} and the facts that
$\kappa_{1,m}\lesssim n^{(1/2-\epsilon_h)(-\alpha+1)}$, $\kappa_{q,m}\lesssim n^{(1/2-\epsilon_h)(q-\alpha)_+}$ for $q\ge 3$ and $\rho_{j}^{-q-2}\lesssim n^{-q/2-1}$. We notice that only the diagonal entries contribute since the summation over $i,j$ will absorb the off-diagonal entries into the trace up to negligible terms by \eqref{eq_tr_SS} since $\sum_{ij}\rho_j^{-4}A_{ij}B_{ji}=\operatorname{tr}(BA(\operatorname{diag}S)^{-2})\prec n^{-2}\operatorname{tr}(AB)+n^{-1-\epsilon_{\alpha}}\prec n^{-1}$ for any matrices $A,B$ with bounded spectral norm. Thus, the main term is
\begin{equation*}
	\begin{split}
		&\sum_{ij}\mathbb{E}_{\Psi}^{\chi}\left[\rho_{j}^{-3}[G(z_1)]_{jj}\partial_{m,ij}([Y^*\mathcal{G}(z_1)]_{ji})\langle \mathrm{e}^{\mathrm{i}x\langle L_{1}(f)\rangle}\rangle\right]\\
		 =&\sum_{ij}\mathbb{E}_{\Psi}^{\chi}\left[\rho_{j}^{-4}\left(-[G(z_1)]_{jj}[Y^*\mathcal{G}(z_1)]_{ji}[Y^*\mathcal{G}(z_1)]_{ji}-z_1[G(z_1)]_{jj}^2[\mathcal{G}(z_1)]_{ii}\right)\langle \mathrm{e}^{\mathrm{i}x\langle L_{1}(f)\rangle}\rangle\right]\\
		 &+\sum_{ij}\mathbb{E}_{\Psi}^{\chi}\left[M_{ij}\rho_{j}^{-5}[G(z_1)]_{jj}\left([Y^*\mathcal{G}(z_1)Y]_{jj}[Y^*\mathcal{G}(z_1)]_{ji}+z_1[G(z_1)]_{jj}[Y^*\mathcal{G}(z_1)]_{ji}\right)\langle \mathrm{e}^{\mathrm{i}x\langle L_{1}(f)\rangle}\rangle\right]\\
		 =&-z_1\mathbb{E}_{\Psi}^{\chi}\left[\left(n^{-2}\operatorname{tr}(\mathcal{G}(z_1))\operatorname{tr}((\operatorname{diag}G(z_1))^2)\right)\langle \mathrm{e}^{\mathrm{i}x\langle L_{1}(f)\rangle}\rangle\right]+\mathrm{O}_{\prec}(n^{-\epsilon_{\alpha}})\\
		=&-\phi^{-1}\mathbb{E}_{\Psi}^{\chi}[z_1m(z_1)[\underline{m}(z_1)]^2\langle \mathrm{e}^{\mathrm{i}x\langle L_{1}(f)\rangle}\rangle]+\mathrm{O}_{\prec}(n^{-\epsilon_{\alpha}})
	\end{split}
\end{equation*}
by \eqref{eq_tr_SS}. Thereafter, we get
\begin{equation*}
	\sum_{ij}\mathbb{E}_{\Psi}^{\chi}\left[M_{ij}\rho_{j}^{-3}[G(z_1)]_{jj}[Y^*\mathcal{G}(z_1)]_{ji}\langle \mathrm{e}^{\mathrm{i}x\langle L_{1}(f)\rangle}\rangle\right]=-\phi^{-1}\mathbb{E}_{\Psi}^{\chi}[z_1m(z_1)[\underline{m}(z_1)]^2\langle \mathrm{e}^{\mathrm{i}x\langle L_{1}(f)\rangle}\rangle]+\mathrm{O}_{\prec}(n^{-\epsilon_{\alpha}}).
\end{equation*}

For $\mathscr{M}_{3,ij}$, similar to Lemma 5.4 of \cite{bao2022spectral}, the difference between $L_{1}(f)$ and $L_{2}(f)$ is $\mathrm{O}_{\prec}(n^{-K+1})$, which also has deterministic upper bound $n^{\mathrm{O}(K)}$. Thus, we replace $L_1(f)$ by $L_{2}(f)$ in the integrand, which avoids possible singularities of the integrand. It follows that
\begin{equation*}
	\begin{split}
		 \mathscr{M}_{3,ij}&=\mathrm{i}x\mathbb{E}^{\chi}_{\Psi}\left[\rho_{j}^{-1}([Y^*\mathcal{G}(z_1)]_{ji})(\partial_{m,ij}L_{1}(f))\cdot\mathrm{e}^{\mathrm{i}x\langle L_{1}(f)\rangle}\right]\\
		 &=\mathrm{i}x\mathbb{E}^{\chi}_{\Psi}\left[\rho_{j}^{-1}([Y^*\mathcal{G}(z_1)]_{ji})(\partial_{m,ij}L_{2}(f))\cdot\mathrm{e}^{\mathrm{i}x\langle L_{1}(f)\rangle}\right]+\mathrm{O}_{\prec}(n^{-K/2})\\
		 &=\frac{-x}{2\pi}\oint_{\bar{\gamma}_2^0}\mathbb{E}^{\chi}_{\Psi}\left[\rho_{j}^{-1}([Y^*\mathcal{G}(z_1)]_{ji})(\partial_{m,ij}\operatorname{tr}\mathcal{G}(z_2))\cdot\mathrm{e}^{\mathrm{i}x\langle L_{1}(f)\rangle}\right]f(z_2)\mathrm{d}z_2+\mathrm{O}_{\prec}(n^{-K/2}).
	\end{split}
\end{equation*}
Recalling $\partial_{m,ij}\operatorname{tr}\mathcal{G}(z_2)=-2\rho_{j}^{-1}[\mathcal{G}(z_2)\mathcal{G}(z_2)Y]_{ij}+2M_{ij}\rho_{j}^{-2}[Y^*\mathcal{G}(z_2)\mathcal{G}(z_2)Y]_{jj}$, we have
\begin{equation*}
	\begin{split}
		 \sum_{ij}\mathscr{M}_{3,ij}&=\frac{x}{\pi}\oint_{\bar{\gamma}_2^0}\sum_{ij}\mathbb{E}^{\chi}_{\Psi}\left[\rho_{j}^{-2}[Y^*\mathcal{G}(z_1)]_{ji}[\mathcal{G}(z_2)\mathcal{G}(z_2)Y]_{ij}\cdot\mathrm{e}^{\mathrm{i}x\langle L_{1}(f)\rangle}\right]f(z_2)\mathrm{d}z_2\\
		 &-\frac{x}{\pi}\oint_{\bar{\gamma}_2^0}\sum_{ij}\mathbb{E}_{\Psi}^{\chi}\left(M_{ij}\rho_{j}^{-3}[Y^*\mathcal{G}(z_1)]_{ji}[Y^*\mathcal{G}(z_2)\mathcal{G}(z_2)Y]_{jj}\cdot\mathrm{e}^{\mathrm{i}x\langle L_{1}(f)\rangle}\right)f(z_2)\mathrm{d}z_2+\mathrm{O}_{\prec}(n^{-K/2}).
	\end{split}
\end{equation*}
The first term can be estimated as
\begin{equation*}
	\begin{split}
		 &\frac{x}{\pi}\oint_{\bar{\gamma}_2^0}\mathbb{E}^{\chi}_{\Psi}\left[\operatorname{tr}\left(Y^*\mathcal{G}(z_1)\mathcal{G}(z_2)\mathcal{G}(z_2)Y(\operatorname{diag}(S))^{-1}\right)\cdot\mathrm{e}^{\mathrm{i}x\langle L_{1}(f)\rangle}\right]f(z_2)\mathrm{d}z_2\\
		 &=\frac{x}{\pi}\oint_{\bar{\gamma}_2^0}\mathbb{E}^{\chi}_{\Psi}\left[\partial_{z_2}[n^{-1}\operatorname{tr}\left(YY^*\mathcal{G}(z_1)\mathcal{G}(z_2)\right)]\cdot\mathrm{e}^{\mathrm{i}x\langle L_{1}(f)\rangle}\right]f(z_2)\mathrm{d}z_2+\mathrm{O}_{\prec}(n^{-\epsilon_{\alpha}})\\
		 &=\frac{x}{\pi}\oint_{\bar{\gamma}_2^0}\mathbb{E}^{\chi}_{\Psi}\left[\partial_{z_2}[n^{-1}\operatorname{tr}\left(\frac{z_1\mathcal{G}(z_1)-z_2\mathcal{G}(z_2)}{z_1-z_2}\right)]\cdot\mathrm{e}^{\mathrm{i}x\langle L_{1}(f)\rangle}\right]f(z_2)\mathrm{d}z_2+\mathrm{O}_{\prec}(n^{-\epsilon_{\alpha}})\\
		 &=\frac{x}{\pi}\oint_{\bar{\gamma}_2^0}\left[{\partial_{z_2}\left(\frac{z_1m(z_1)-z_2m(z_2)}{z_1-z_2}\right)}\right]f(z_2)\mathrm{d}z_2\cdot\mathbb{E}_{\Psi}^{\chi}[\mathrm{e}^{\mathrm{i}x\langle L_{1}(f)\rangle}]+\mathrm{O}_{\prec}(n^{-c}),
	\end{split}
\end{equation*}
where the second equality uses the identity $\operatorname{tr}(F(\mathcal{G}(z))^2)=\partial_{z}(\operatorname{tr}F\mathcal{G}(z))$ for any function $F$ which is independent of $z$, the third line follows from the identity
\begin{equation*}
	YY^*\mathcal{G}(z_1)\mathcal{G}(z_2)=(I_{n}+z_1\mathcal{G}(z_1))\mathcal{G}(z_2)=\frac{z_1\mathcal{G}(z_1)-z_2\mathcal{G}(z_2)}{z_1-z_2},
\end{equation*}
due to the fact $\mathcal{G}(z_1)\mathcal{G}(z_2)=[\mathcal{G}(z_1)-\mathcal{G}(z_2)]/(z_1-z_2)$. For the second term, similar to $\mathscr{M}_{2,ij}$, applying the cumulant expansion gives
\begin{equation*}
	\begin{split}
		 &\sum_{ij}\mathbb{E}_{\Psi}^{\chi}\left(M_{ij}\rho_{j}^{-3}[Y^*\mathcal{G}(z_1)]_{ji}[Y^*\mathcal{G}(z_2)\mathcal{G}(z_2)Y]_{jj}\cdot\mathrm{e}^{\mathrm{i}x\langle L_{1}(f)\rangle}\right)\\
		 &=\sum_{ij}\kappa_{2,m}\mathbb{E}_{\Psi}^{\chi}\left[\partial_{m,ij}(\rho_{j}^{-3}[Y^*\mathcal{G}(z_2)\mathcal{G}(z_2)Y]_{jj}[Y^*\mathcal{G}(z_1)]_{ji} \cdot\mathrm{e}^{\mathrm{i}x\langle L_{1}(f)\rangle})\right]+\mathrm{O}_{\prec}(n^{-2\epsilon_{\alpha
		}}),
	\end{split}
\end{equation*}
where the error term follows from the same argument as $\mathscr{M}_{2,ij}$. The main term can be estimated analogously. Note $Y^*\mathcal{G}(z_2)\mathcal{G}(z_2)Y=G(z_2)+z_2G(z_2)^2=\partial_{z_2}(z_2G(z_2))$.
Thus, the main term is
\begin{equation*}
	\begin{split}
		&\sum_{ij}\mathbb{E}_{\Psi}^{\chi}\left[\rho_{j}^{-3}\partial_{z_2}(z_2[G(z_2)]_{jj})\partial_{m,ij}([Y^*\mathcal{G}(z_1)]_{ji})\cdot \mathrm{e}^{\mathrm{i}x\langle L_{1}(f)\rangle}\right]\\
		 =&\sum_{ij}\mathbb{E}_{\Psi}^{\chi}\left[\rho_{j}^{-4}\partial_{z_2}(z_2[G(z_2)]_{jj})\left(-z_1[G(z_1)]_{jj}[\mathcal{G}(z_1)]_{ii}\right)\cdot \mathrm{e}^{\mathrm{i}x\langle L_{1}(f)\rangle}\right]+\mathrm{O}_{\prec}(n^{-\epsilon_{\alpha}})\\
		 =&-z_1\mathbb{E}_{\Psi}^{\chi}\left[\left(n^{-2}\operatorname{tr}(\mathcal{G}(z_1))\operatorname{tr}((\operatorname{diag}G(z_1))[\operatorname{diag}\partial_{z_2}(z_2G(z_2))])\right)\cdot \mathrm{e}^{\mathrm{i}x\langle L_{1}(f)\rangle}\right]+\mathrm{O}_{\prec}(n^{-\epsilon_{\alpha}})\\
		=&-\phi^{-1}\mathbb{E}_{\Psi}^{\chi}[z_1m(z_1)\underline{m}(z_1)[\partial_{z_2}(z_2\underline{m}(z_2))] \mathrm{e}^{\mathrm{i}x\langle L_{1}(f)\rangle}]+\mathrm{O}_{\prec}(n^{-\epsilon_{\alpha}}).
	\end{split}
\end{equation*}
Thereafter, we have
\begin{equation*}
	\begin{split}
	&	\sum_{ij}\mathscr{M}_{3,ij}=\\
&\frac{x}{\pi}\oint_{\bar{\gamma}_2^0}\left[{\partial_{z_2}\left(\frac{z_1m(z_1)-z_2m(z_2)}{z_1-z_2}+\phi^{-1}z_1m(z_1)\underline{m}(z_1)z_2\underline{m}(z_2)\right)}\right]f(z_2)\mathrm{d}z_2\cdot\mathbb{E}_{\Psi}^{\chi}[\mathrm{e}^{\mathrm{i}x\langle L_{1}(f)\rangle}]+\mathrm{O}_{\prec}(n^{-c}).
	\end{split}
\end{equation*}

\

\noindent\textbf{III. Terms with $\kappa_{3,m}$}\\ For $\alpha>3$, notice $\kappa_{3,m}=\mathbb{E}(\xi^3)-\mathrm{o}(n^{(1/2-\epsilon_h)(-\alpha+3)})$, which might not be negligible under the non-symmetric condition. Denote
\begin{equation*}
	\begin{split}
		 \sum_{a=0}^{2}\binom{2}{a}\mathbb{E}_{\Psi}^{\chi}\left[\partial_{m,ij}^{a}\rho_j^{-1}\partial_{m,ij}^{2-a}\left([Y^*\mathcal{G}(z_1)]_{ji}\langle \mathrm{e}^{\mathrm{i}x\langle L_{1}(f)\rangle}\rangle\right)\right]=:\sum_{a=0}^{2}E_{1a,m}.
	\end{split}
\end{equation*}
For $a=2$, noting that $\partial^2_{m,ij}\rho_j^{-1}=-\rho_j^{-3}+3\rho_j^{-5}M_{ij}^2\sim -\rho_j^{-3}$ by \eqref{eq_rhoj_highpro}, we have
\begin{equation*}
	\begin{split}
		\sum_{ij}E_{12,m}
		=- \sum_{ijk}\mathbb{E}_{\Psi}^{\chi}\left[\rho_{j}^{-4}M_{kj}[\mathcal{G}(z_1)]_{ki}\langle \mathrm{e}^{\mathrm{i}x\langle L_{1}(f)\rangle}\rangle\right]+\mathrm{O}_{\prec}(n^{-c}),
	\end{split}
\end{equation*}
where the error is implied by $\sum_{jk}\mathbb{E}(\rho_j^{-4}\sum_{i}H_{kj}[\mathcal{G}(z_1)]_{ki})\lesssim n^{1/2}\mathbb{E}|H_{kj}|\lesssim n^{-c}$ since $\sum_{i}[\mathcal{G}(z_1)]_{ki}\lesssim n^{1/2}(\sum_{i}[\mathcal{G}(z_1)]_{ki}^2)^{1/2}\lesssim n^{1/2}$. For the main part, applying the cumulant expansion w.r.t. $M_{kj}$ gives
\begin{equation*}
	\begin{split}
		&\sum_{ijk}\mathbb{E}_{\Psi}^{\chi}\left[\rho_{j}^{-4}M_{kj}[\mathcal{G}(z_1)]_{ki}\langle \mathrm{e}^{\mathrm{i}x\langle L_{1}(f)\rangle}\rangle\right]\\
		=&\kappa_{1,m}\sum_{ijk}\mathbb{E}_{\Psi}^{\chi}\left[\rho_{j}^{-4}[\mathcal{G}(z_1)]_{ki}\langle \mathrm{e}^{\mathrm{i}x\langle L_{1}(f)\rangle}\rangle\right]+\kappa_{2,m}\sum_{ijk}\mathbb{E}_{\Psi}^{\chi}\left[\partial_{m,kj}\left(\rho_{j}^{-4}[\mathcal{G}(z_1)]_{ki}\langle \mathrm{e}^{\mathrm{i}x\langle L_{1}(f)\rangle}\rangle\right)\right]\\
		&+\frac{\kappa_{3,m}}{2!}\sum_{ijk}\mathbb{E}_{\Psi}^{\chi}\left[\partial_{m,kj}^2\left(\rho_{j}^{-4}[\mathcal{G}(z_1)]_{ki}\langle \mathrm{e}^{\mathrm{i}x\langle L_{1}(f)\rangle}\rangle\right)\right]+\mathrm{O}_{\prec}(n^{-\epsilon_{\alpha}}),
	\end{split}
\end{equation*}
where the error follows from $\kappa_{q,m}\lesssim n^{(1/2-\epsilon_h)(q-\alpha)}$ for $q\ge 4$ and $\partial_{m,ij}^{q-1}\left(\rho_{j}^{-4}[\mathcal{G}(z_1)]_{ki}\langle \mathrm{e}^{\mathrm{i}x\langle L_{1}(f)\rangle}\rangle\right)\lesssim \rho_j^{-3-q}\lesssim n^{-(3+q)/2}$ by \eqref{eq_rhoj_highpro}. Similarly, we have
\begin{equation*}
	\kappa_{1,m}\sum_{ijk}\mathbb{E}_{\Psi}^{\chi}\left[\rho_{j}^{-4}[\mathcal{G}(z_1)]_{ki}\langle \mathrm{e}^{\mathrm{i}x\langle L_{1}(f)\rangle}\rangle\right]\lesssim \kappa_{1,m}\sum_{j}n^{-2}|\sum_{ik}[\mathcal{G}(z_1)]_{ki}|\lesssim n^{(1/2-\epsilon_h)(-\alpha+1)+1/2}\lesssim n^{-c}.
\end{equation*}
For the second term, all the terms with at least two off-diagonal elements can be bounded  as $\sum_{ij}A_{ij}B_{ij}\lesssim \sum_{ij}(A_{ij}^2+B_{ij}^2)\lesssim \operatorname{tr}(AA^*+BB^*)$ via the Cauchy-Schwarz inequality. The remaining term is
\begin{equation*}
	\begin{split}
		&\sum_{k}\mathbb{E}_{\Psi}^{\chi}\left[\rho_{j}^{-4}\partial_{m,kj}\left([\mathcal{G}(z_1)]_{ki}\right)\langle \mathrm{e}^{\mathrm{i}x\langle L_{1}(f)\rangle}\rangle\right]\\
		=&-\mathbb{E}_{\Psi}^{\chi}\left[\rho_{j}^{-5}[Y^*\mathcal{G}(z_1)]_{ji}\operatorname{tr}\mathcal{G}(z_1)\langle \mathrm{e}^{\mathrm{i}x\langle L_{1}(f)\rangle}\rangle\right]-\mathbb{E}_{\Psi}^{\chi}\left[\rho_{j}^{-5}[Y^*\mathcal{G}(z_1)\mathcal{G}(z_1)]_{ji}\langle \mathrm{e}^{\mathrm{i}x\langle L_{1}(f)\rangle}\rangle\right]\\
		&+\mathbb{E}_{\Psi}^{\chi}\left[\rho_{j}^{-5}[Y^*\mathcal{G}(z_1)]_{ji}[Y^*\mathcal{G}(z_1)M(\operatorname{diag}S)^{-1/2}]_{jj}\langle \mathrm{e}^{\mathrm{i}x\langle L_{1}(f)\rangle}\rangle\right]\\
		=&-\mathbb{E}_{\Psi}^{\chi}\left[\rho_{j}^{-5}[Y^*\mathcal{G}(z_1)]_{ji}\operatorname{tr}\mathcal{G}(z_1)\langle \mathrm{e}^{\mathrm{i}x\langle L_{1}(f)\rangle}\rangle\right]+\mathrm{O}_{\prec}(n^{-5/2})\\
		=&\mathbb{E}_{\Psi}^{\chi}\left[\mathrm{O}_{\prec}(n^{-5/2})\operatorname{tr}\mathcal{G}(z_1)\langle \mathrm{e}^{\mathrm{i}x\langle L_{1}(f)\rangle}\rangle\right]+\mathrm{O}_{\prec}(n^{-5/2}).
	\end{split}
\end{equation*}
For the third term, each term in the expansion has at least two off-diagonal elements. So it is also negligible by the Cauchy-Schwarz inequality.

For $a=1$, we have
\begin{equation*}
	\begin{split}
		E_{11,m}=-\mathbb{E}_{\Psi}^{\chi}\left[\rho_j^{-3}M_{ij}\partial_{m,ij}\left([Y^*\mathcal{G}(z_1)]_{ji}\right)\langle \mathrm{e}^{\mathrm{i}x\langle L_{1}(f)\rangle}\rangle\right]-\mathbb{E}_{\Psi}^{\chi}\left[\rho_j^{-3}M_{ij}[Y^*\mathcal{G}(z_1)]_{ji}\partial_{m,ij}\langle \mathrm{e}^{\mathrm{i}x\langle L_{1}(f)\rangle}\rangle\right],
	\end{split}
\end{equation*}
where the second part is negligible since each term in the expansion has at least one off-diagonal element with the factor $\rho_j^{-4}$. For the first summand in the above identity,
\begin{equation*}
	\begin{split}
		&\sum_{ij}\mathbb{E}_{\Psi}^{\chi}\left[\rho_j^{-3}M_{ij}\partial_{m,ij}\left([Y^*\mathcal{G}(z_1)]_{ji}\right)\langle \mathrm{e}^{\mathrm{i}x\langle L_{1}(f)\rangle}\rangle\right]\\
		 =&\sum_{ij}\mathbb{E}_{\Psi}^{\chi}\left[\rho_{j}^{-4}M_{ij}\left(-[Y^*\mathcal{G}(z_1)]_{ji}[Y^*\mathcal{G}(z_1)]_{ji}-z_1[G(z_1)]_{jj}[\mathcal{G}(z_1)]_{ii}\right)\langle \mathrm{e}^{\mathrm{i}x\langle L_{1}(f)\rangle}\rangle\right]\\
		&+\sum_{ij}\mathbb{E}_{\Psi}^{\chi}\left(M_{ij}^2\rho_{j}^{-5}[Y^*\mathcal{G}(z_1)]_{ji}(1+2z_1[G(z_1)]_{jj})\langle \mathrm{e}^{\mathrm{i}x\langle L_{1}(f)\rangle}\rangle\right)\\
		=&\sum_{ij}\mathbb{E}_{\Psi}^{\chi}\left[\rho_{j}^{-4}M_{ij}\left(-z_1[G(z_1)]_{jj}[\mathcal{G}(z_1)]_{ii}\right)\langle \mathrm{e}^{\mathrm{i}x\langle L_{1}(f)\rangle}\rangle\right]+\mathrm{O}_{\prec}(n^{-1/2}),
	\end{split}
\end{equation*}
since $M_{ij}\rho_j^{-4}\lesssim n^{-\epsilon_h}\rho_j^{-3}$, $\mathbb{E}|M_{ij}^2\rho_j^{-5}|\lesssim n^{-5/2}$ and $\sum_{ij}A_{ij}B_{ij}\lesssim \operatorname{tr}(AA^*+BB^*)$. The remaining part can be expanded w.r.t. $M_{ij}$ as
\begin{equation*}
	\begin{split}
		&\sum_{ij}\mathbb{E}_{\Psi}^{\chi}\left[\rho_{j}^{-4}M_{ij}\left([G(z_1)]_{jj}[\mathcal{G}(z_1)]_{ii}\right)\langle \mathrm{e}^{\mathrm{i}x\langle L_{1}(f)\rangle}\rangle\right]\\
		=&\kappa_{2,m}\sum_{ij}\mathbb{E}_{\Psi}^{\chi}\left[\partial_{m,ij}(\rho_{j}^{-4}[G(z_1)]_{jj}[\mathcal{G}(z_1)]_{ii}\langle \mathrm{e}^{\mathrm{i}x\langle L_{1}(f)\rangle}\rangle)\right]+\mathrm{O}_{\prec}(n^{-1/2}),
	\end{split}
\end{equation*}
where the error term follows from $\kappa_{1,m}\lesssim n^{(1/2-\epsilon_h)(-\alpha+1)}$, $\kappa_{q,m}\lesssim n^{(1/2-\epsilon_h)(q-\alpha)_{+}}$ for $q\ge 3$ and  \eqref{eq_rhoj_highpro}. For the above double sum, there is at least one off-diagonal element, which can be bounded as
\begin{equation*}
	\sum_{ij}\mathbb{E}_{\Psi}^{\chi}\left[\partial_{m,ij}(\rho_{j}^{-4}[G(z_1)]_{jj}[\mathcal{G}(z_1)]_{ii}\langle \mathrm{e}^{\mathrm{i}x\langle L_{1}(f)\rangle}\rangle)\right]\lesssim \sum_{ij}\mathbb{E}_{\Psi}^{\chi}[\rho_j^{-5}|M_{ij}|]+\sum_{ij}\mathbb{E}_{\Psi}^{\chi}[\rho_j^{-5}A_{ij}]\lesssim n^{-1/2}.
\end{equation*}

For $a=0$, notice
\begin{equation*}
	\begin{split}
		\sum_{ij}E_{10,m}=\sum_{ij}\mathbb{E}_{\Psi}^{\chi}\left[\rho_j^{-1}\partial_{m,ij}^2\left([Y^*\mathcal{G}(z_1)]_{ji}\langle \mathrm{e}^{\mathrm{i}x\langle L_{1}(f)\rangle}\rangle\right)\right].
	\end{split}
\end{equation*}
Invoking the derivatives in Section \ref{subsubsec_derivative}, we get $\partial_{m,ij}^2 [Y^*\mathcal{G}(z)]_{ji}\sim \rho_j^{-2}A_{ji}, \partial_{m,ij}^2\operatorname{tr}\mathcal{G}(z_2) \sim \rho_{j}^{-2}B_{ii}+\rho_j^{-2}C_{ij}$ for matrices $A,B,C$ formed by $E_{ij}, F_{ij}, \mathcal{G}(z)$. Recall that the terms with at least two off-diagonal elements are negligible due to the Cauchy-Schwarz inequality. It suffices to consider
\begin{equation*}
	\begin{split}
		&\sum_{ij}\mathbb{E}_{\Psi}^{\chi}[\rho_j^{-1}\partial_{m,ij}[Y^*\mathcal{G}(z_1)]_{ji}\partial_{m,ij}(\langle \mathrm{e}^{\mathrm{i}x\langle L_{1}(f)\rangle}\rangle)]\\
		=&\frac{x}{\pi}\oint_{\bar{\gamma}_2^{0}} \sum_{ij}\mathbb{E}_{\Psi}^{\chi}\big(\rho_j^{-3}[E_{ij}^*\mathcal{G}(z_1)-\sum_{a=1}^{2}Y^*\mathcal{G}(z_1)\mathscr{P}_{a}^{ij}\mathcal{G}(z_1)]_{ji}[\mathcal{G}(z_2)\mathcal{G}(z_2)Y]_{ij}\cdot \mathrm{e}^{\mathrm{i}x\langle L_{1}(f)\rangle}\big)f(z_2)\mathrm{d}z_2+\mathrm{O}_{\prec}(n^{-c})\\
		= &\frac{x}{\pi}\oint_{\bar{\gamma}_2^{0}} \sum_{ij}\mathbb{E}_{\Psi}^{\chi}\left(\rho_j^{-3}[\mathcal{G}(z_1)]_{ii}[\mathcal{G}(z_2)\mathcal{G}(z_2)Y]_{ij}\cdot \mathrm{e}^{\mathrm{i}x\langle L_{1}(f)\rangle}\right)f(z_2)\mathrm{d}z_2+\mathrm{O}_{\prec}(n^{-c})\\
		= &\frac{x}{\pi}\oint_{\bar{\gamma}_2^{0}} m(z_1)\partial_{z_2}\mathbb{E}_{\Psi}^{\chi}\big(\sum_{ij}\rho_j^{-3}[\mathcal{G}(z_2)Y]_{ij}\cdot \mathrm{e}^{\mathrm{i}x\langle L_{1}(f)\rangle}\big)f(z_2)\mathrm{d}z_2+\mathrm{O}_{\prec}(n^{-c}),
	\end{split}
\end{equation*}
where the error comes from the terms with at least two off-diagonal elements and the last line follows from the entrywise local law,
\begin{equation*}
	\begin{split}
		\sum_{ij}\mathbb{E}_{\Psi}^{\chi}\left(\rho_j^{-3}[\mathcal{G}(z_1)-m(z_1)]_{ii}[\mathcal{G}(z_2)\mathcal{G}(z_2)Y]_{ij}\cdot \mathrm{e}^{\mathrm{i}x\langle L_{1}(f)\rangle}\right)
		\lesssim  n^{-c}.
	\end{split}
\end{equation*}
For the first part, recalling the identity $\mathbb{E}(\xi_1\xi_2)=\mathbb{E}(\xi_1\langle\xi_2\rangle)+\mathbb{E}(\xi_1)\mathbb{E}(\xi_2)$ and the estimate of $E_{12,m}$ above, we have
\begin{equation*}
	\begin{split}
		\mathbb{E}_{\Psi}^{\chi}\big(\sum_{ij}\rho_j^{-3}[\mathcal{G}(z_2)Y]_{ij}\cdot \mathrm{e}^{\mathrm{i}x\langle L_{1}(f)\rangle}\big)=\mathbb{E}_{\Psi}^{\chi}\big(\sum_{ij}\rho_j^{-3}[\mathcal{G}(z_2)Y]_{ij}\big)\mathbb{E}_{\Psi}^{\chi}(\mathrm{e}^{\mathrm{i}x\langle L_{1}(f)\rangle})+\mathrm{O}_{\prec
		}(n^{-c}).
	\end{split}
\end{equation*}
It suffices to estimate $\mathbb{E}_{\Psi}^{\chi}\big(\sum_{ij}\rho_j^{-3}[\mathcal{G}(z_2)Y]_{ij}\big)$. Similarly to $E_{12,m}$, one has the main part
\begin{equation*}
	\begin{split}
		 \mathbb{E}_{\Psi}^{\chi}\big(\sum_{ij}\rho_j^{-3}[\mathcal{G}(z_2)Y]_{ij}\big)=&-\mathbb{E}_{\Psi}^{\chi}\big[\sum_{ij}\rho_{j}^{-5}[Y^*\mathcal{G}(z_2)]_{ji}\operatorname{tr}\mathcal{G}(z_2)\big]+\mathrm{O}_{\prec}(n^{-c})\\
		=&-nm(z_2)\mathbb{E}_{\Psi}^{\chi}\big(\sum_{ij}\rho_j^{-5}[\mathcal{G}(z_2)Y]_{ij}\big)+\mathrm{O}_{\prec}(n^{-c}),
	\end{split}
\end{equation*}
where we have applied the averaged local law in the last line. Define the matrix $\mathbf{E}_{n\times p}=\sum_{ij}E_{ij}\rho_{j}^{-3}$ where $E_{ab}=\delta_{ac}\delta_{bd}$. We have $\sum_{ij}\rho_j^{-3}[\mathcal{G}(z_2)Y]_{ij}=\operatorname{tr}(Y^*\mathcal{G}(z_2)\mathbf{E}_{n\times p})$ and
\begin{equation*}
	\begin{split}
		&\mathbb{E}_{\Psi}^{\chi}\big(\sum_{ij}\rho_j^{-5}[\mathcal{G}(z_2)Y]_{ij}\big)= \mathbb{E}_{\Psi}^{\chi}[\operatorname{tr}(Y^*\mathcal{G}(z_2)\mathbf{E}_{n\times p}(\operatorname{diag}S)^{-1})]\\
		=&n^{-1}\mathbb{E}_{\Psi}^{\chi}[\operatorname{tr}(Y^*\mathcal{G}(z_2)\mathbf{E}_{n\times p})]+\mathbb{E}_{\Psi}^{\chi}[\operatorname{tr}(Y^*\mathcal{G}(z_2)\mathbf{E}_{n\times p}((\operatorname{diag}S)^{-1}-n^{-1}I))]\\
		=&n^{-1}\mathbb{E}_{\Psi}^{\chi}[\operatorname{tr}(Y^*\mathcal{G}(z_2)\mathbf{E}_{n\times p})]+\mathrm{O}_{\prec}(\max_{j}|\sum_{i}\rho_j^{-3}[Y^*\mathcal{G}(z_2)]_{ji}|\operatorname{tr}|(\operatorname{diag}S)^{-1}-n^{-1}I)|)\\
		=&n^{-1}\mathbb{E}_{\Psi}^{\chi}[\operatorname{tr}(Y^*\mathcal{G}(z_2)\mathbf{E}_{n\times p})]+\mathrm{O}_{\prec}(n^{-1-\epsilon_{\alpha}})
	\end{split}
\end{equation*}
by the elementary inequality $\operatorname{tr}(A\operatorname{diag}(B))\lesssim \max_{j}|A_{jj}|\operatorname{tr}|B|$, the trace inequality \eqref{eq_tr_SS} and the Cauchy-Schwarz inequality $\sum_{i}\rho_j^{-3}[Y^*\mathcal{G}(z_2)]_{ji}\lesssim (\sum_{i}\rho_j^{-6})^{1/2}(\sum_{i}[Y^*\mathcal{G}(z_2)]_{ji}[Y^*\mathcal{G}(z_2)]_{ji})^{1/2}\lesssim n^{-1}$ since $\sum_{i}[Y^*\mathcal{G}(z_2)]_{ji}[Y^*\mathcal{G}(z_2)]_{ji}=[Y^*\mathcal{G}(z_2)\mathcal{G}(z_2)Y]_{jj}\lesssim 1$.
Thereafter, we have $$\mathbb{E}_{\Psi}^{\chi}(\sum_{ij}\rho_j^{-3}[\mathcal{G}(z_2)Y]_{ij})=\mathrm{O}_{\prec}(n^{-c}).$$ The remaining term $\mathbb{E}_{\Psi}^{\chi}[\rho_j^{-1}[Y^*\mathcal{G}(z_1)]_{ji}\partial_{m,ij}^2(\langle \mathrm{e}^{\mathrm{i}x\langle L_{1}(f)\rangle}\rangle)]$ can be estimated similarly, and thus omitted. To this end, the terms with $\kappa_{3,m}$ for $\alpha>3$ are negligible with high probability.

\

\noindent\textbf{IV. Terms with $\kappa_{4,m}$}\\
Invoke the fact that the terms with diagonal elements are dominant since the terms with off-diagonal elements are negligible due to the trace estimation \eqref{eq_tr_SS}. Denote
\begin{equation*}
	 E_{2,m}=\sum_{a=0}^{3}\binom{3}{a}\mathbb{E}_{\Psi}^{\chi}\left[\partial_{m,ij}^{a}\rho_j^{-1}\left(\partial_{m,ij}^{3-a}[Y^*\mathcal{G}(z_1)]_{ji}\langle \mathrm{e}^{\mathrm{i}x\langle L_{1}(f)\rangle}\rangle\right)\right]=:\sum_{a=0}^{3}E_{2a,m}.
\end{equation*}
For $a=3$, noting \eqref{eq_tr_SS}, $\kappa_{4,m}\le n^{2(1/2-\epsilon_h)}$ and $\partial_{m,ij}^{3}\rho_j^{-1}\sim \rho_{j}^{-5}M_{ij}$, we have
\begin{equation*}
	\begin{split}
		\sum_{ij}\kappa_{4,m}E_{23,m}&\lesssim \kappa_{4,m}\sum_{ij}\mathbb{E}_{\Psi}^{\chi}\left[\rho_{j}^{-5}M_{ij}[Y^*\mathcal{G}(z_1)]_{ji}\langle \mathrm{e}^{\mathrm{i}x\langle L_{1}(f)\rangle}\rangle\right]\\
		 &=\kappa_{4,m}\mathbb{E}_{\Psi}^{\chi}\left(\operatorname{tr}(Y^*\mathcal{G}(z_1)M(\operatorname{diag}(S))^{-1/2}(\operatorname{diag}(S))^{-2})\langle \mathrm{e}^{\mathrm{i}x\langle L_{1}(f)\rangle}\rangle\right)\\
		&\lesssim n^{-2} \kappa_{4,m}\mathbb{E}_{\Psi}^{\chi}|\operatorname{tr}(Y^*\mathcal{G}(z_1)M(\operatorname{diag}(S))^{-1/2})|+\kappa_{4,m}\mathrm{O}_{\prec}(n^{-1-\epsilon_{\alpha}})\lesssim \mathrm{O}_{\prec}(n^{-2\epsilon_{h}}).
	\end{split}
\end{equation*}
The case where $a=1$ is the same, say, $\sum_{ij}\kappa_{4,m}E_{21,m}\prec n^{-2\epsilon_h}$, where we borrow similar arguments from $E_{1,m}$ above.

For $a=2$, by $\partial_{m,ij}^{2}\rho_j^{-1}=- \rho_{j}^{-3}+3M_{ij}^2\rho_{j}^{-5}$ and the derivatives in Section \ref{subsubsec_derivative}, we have
\begin{equation*}
	\begin{split}
		E_{22,m}= &\mathbb{E}_{\Psi}^{\chi}\left[\rho_{j}^{-4}(1-3M_{ij}^2\rho_{j}^{-2})\left([Y^*\mathcal{G}(z_1)]_{ji}[Y^*\mathcal{G}(z_1)]_{ji}+z_1[G(z_1)]_{jj}[\mathcal{G}(z_1)]_{ii}\right)\langle \mathrm{e}^{\mathrm{i}x\langle L_{1}(f)\rangle}\rangle\right]\\
		 &+\mathbb{E}_{\Psi}^{\chi}\left[M_{ij}\rho_{j}^{-5}(1-3M_{ij}^2\rho_{j}^{-2})\left([Y^*\mathcal{G}(z_1)Y]_{jj}[Y^*\mathcal{G}(z_1)]_{ji}+z_1[G(z_1)]_{jj}[Y^*\mathcal{G}(z_1)]_{ji}\right)\langle \mathrm{e}^{\mathrm{i}x\langle L_{1}(f)\rangle}\rangle\right]\\
		 &-\frac{x}{\pi}\oint_{\bar{\gamma}_2^0}\mathbb{E}^{\chi}_{\Psi}\left[\rho_{j}^{-4}(1-3M_{ij}^2\rho_{j}^{-2})([Y^*\mathcal{G}(z_1)]_{ji})(\mathcal{G}(z_2)\mathcal{G}(z_2)Y)_{ij}\cdot\mathrm{e}^{\mathrm{i}x\langle L_{1}(f)\rangle}\right]f(z_2)\mathrm{d}z_2\\
		 &+\frac{x}{\pi}\oint_{\bar{\gamma}_2^0}\mathbb{E}^{\chi}_{\Psi}\left[M_{ij}\rho_{j}^{-5}(1-3M_{ij}^2\rho_{j}^{-2})([Y^*\mathcal{G}(z_1)]_{ji})[Y^*\mathcal{G}(z_2)\mathcal{G}(z_2)Y]_{jj}\cdot\mathrm{e}^{\mathrm{i}x\langle L_{1}(f)\rangle}\right]f(z_2)\mathrm{d}z_2\\
&\qquad\qquad\qquad+\mathrm{O}_{\prec}(n^{-K/2}).
	\end{split}
\end{equation*}
Similarly to $\mathscr{M}_{2,ij}$,  noticing \eqref{eq_tr_SS}, $\mathbb{E}_{\Psi}^{\chi}[3M_{ij}^2\rho_{j}^{-6}]\lesssim n^{-3}$, and the terms with off-diagonal elements are negligible, we can see that the main term is
\begin{equation*}
	\begin{split}
		\sum_{ij}\kappa_{4,m}E_{22,m}&= \kappa_{4,m}\sum_{ij}\mathbb{E}_{\Psi}^{\chi}\left[\rho_{j}^{-4}z_1[G(z_1)]_{jj}[\mathcal{G}(z_1)]_{ii}\langle \mathrm{e}^{\mathrm{i}x\langle L_{1}(f)\rangle}\rangle\right]+\mathrm{O}_{\prec}(n^{-2\epsilon_h})\\
		&=\kappa_{4,m}\mathbb{E}_{\Psi}^{\chi}\left[z_1\operatorname{tr}\mathcal{G}(z_1)\operatorname{tr}(G(z_1)(\operatorname{diag}(S))^{-2})\langle \mathrm{e}^{\mathrm{i}x\langle L_{1}(f)\rangle}\rangle\right]+\mathrm{O}_{\prec}(n^{-2\epsilon_h})\\
		&=\mathrm{O}_{\prec}(n^{-2\epsilon_h})\mathbb{E}_{\Psi}^{\chi}\left[z_1\operatorname{tr}\mathcal{G}(z_1)\langle \mathrm{e}^{\mathrm{i}x\langle L_{1}(f)\rangle}\rangle\right]+\mathrm{O}_{\prec}(n^{-2\epsilon_h}).
	\end{split}
\end{equation*}
For $a=0$, similarly to the case of $E_{22,m}$ and $E_{10,m}$, we have
\begin{equation*}
	\begin{split}
		\partial_{m,ij}^2\operatorname{tr}\mathcal{G}(z_2) &\sim 2\rho_{j}^{-2}\partial_{z_2}\left(z_2[G(z_2)]_{jj}[\mathcal{G}(z_2)]_{ii}+[Y^*\mathcal{G}(z_2)Y]_{jj}\right)+\text{off-diagonals}.
	\end{split}
\end{equation*}
The part with off-diagonal elements is negligible since the sum with $\rho_{j}^{-4}$ can be absorbed into the trace of some matrices, which can be bounded by $\mathrm{O}_{\prec}(n^{-c})$ due to \eqref{eq_tr_SS}. Moreover, notice the relationship that
\begin{equation*}
	\begin{split}
		&n^{-2}\sum_{ij}(z_2[G(z_2)]_{jj}[\mathcal{G}(z_2)]_{ii}+[Y^*\mathcal{G}(z_2)Y]_{jj})
		=n^{-2}\sum_{ij}(z_2[G(z_2)]_{jj}[\mathcal{G}(z_2)]_{ii}+1+z_2[G(z_2)]_{jj})\\
		=&\phi^{-1}z_2m(z_2)\underline{m}(z_2)+z_2\underline{m}(z_2)+1+\mathrm{O}_{\prec}(n^{-1})=\mathrm{O}_{\prec}(n^{-1}),
	\end{split}
\end{equation*}
where we have used the fact that $\phi^{-1}zm(z)\underline{m}(z)+zm(z)+1=0$, which follows from \eqref{eq_mp} and $\underline{m}(z)=\phi m(z)-(1-\phi)/z$. Thus, we have
\begin{equation*}
	\sum_{ij}\kappa_{4,m}E_{20,m}\lesssim \mathrm{O}_{\prec}(n^{-2\epsilon_h})\mathbb{E}_{\Psi}^{\chi}\left[z_1\operatorname{tr}\mathcal{G}(z_1)\langle \mathrm{e}^{\mathrm{i}x\langle L_{1}(f)\rangle}\rangle\right]+\mathrm{O}_{\prec}(n^{-2\epsilon_h}).
\end{equation*}
Combining the estimates above gives
\begin{equation*}
	\sum_{ij}\kappa_{4,m}E_{2,m}=\mathrm{O}_{\prec}(n^{-2\epsilon_h})\mathbb{E}_{\Psi}^{\chi}\left[z_1\operatorname{tr}\mathcal{G}(z_1)\langle \mathrm{e}^{\mathrm{i}x\langle L_{1}(f)\rangle}\rangle\right]+\mathrm{O}_{\prec}(n^{-2\epsilon_h}).
\end{equation*}

\

\noindent\textbf{V. Other terms}\\
Following the arguments to handle $E_{1,m}$ and $E_{2,m}$, one can show that
\begin{equation*}
	\sum_{ij}E_{3,m}=\mathrm{O}_{\prec}(n^{-4\epsilon_h})\mathbb{E}_{\Psi}^{\chi}\left[z_1\operatorname{tr}\mathcal{G}(z_1)\langle \mathrm{e}^{\mathrm{i}x\langle L_{1}(f)\rangle}\rangle\right]+\mathrm{O}_{\prec}(n^{-4\epsilon_h}),
\end{equation*}
which relies on the estimate \eqref{eq_tr_SS} and $\kappa_{2q,m}\le n^{(2q-2)(1/2-\epsilon_h)}$. For $E_{4,m}$, note the fact that $\Xi=1$ with high probability, which further implies that  $\partial_{m,ij}^{q_1}\Xi=0$ for any fixed $q_1\ge 1$ with high probability. Thus we have $|E_{2,m}|\prec n^{-D}$ for any fixed $D>0$ since $\partial_{m,ij}^{q_0}\left(\rho_j^{-1}[Y^*\mathcal{G}(z_1)]_{ji}\langle \mathrm{e}^{\mathrm{i}x\langle L_{1}(f)\rangle}\rangle\right) \partial_{m,ij}^{q_1}\Xi$ has deterministic upper bound for $\mathrm{Im}z\ge n^{-K}$ for some large and fixed $K>0$.

It remains to handle the error term $R_{m,ij}(q)$, which can be bounded similarly after taking the sum over $i,j$. To this end, for the first part, recalling the derivatives in Section \ref{subsubsec_derivative}, we have
\begin{equation*}
	\mathbb{E}(|M_{ij}^{2q+2}|\mathbbm{1}(|M_{ij}|>n^{-c}))\lesssim n^{2q(1/2-\epsilon_{h})}
\end{equation*}
and
\begin{equation*}
	\sup_{M_{ij}\in\mathbb{R}}\left|\partial_{m,ij}^{2q+1}\left(\rho_j^{-1}[Y^*\mathcal{G}(z_1)]_{ji}\langle \mathrm{e}^{\mathrm{i}x\langle L_{1}(f)\rangle}\rangle\Xi\right)\right|\lesssim \rho_{j}^{-(2q+2)}n^{\mathrm{O}(K)},
\end{equation*}
due to the factor $\Xi$ and the truncation $\mathrm{Im}z\ge n^{-K}$. Thus, for large enough $q$, the first term is negligible because of the high probability bound $\rho_{j}^{-1}\prec n^{-1/2}$. While for the second term, noting the high probability bound in Lemma \ref{lem_pre_estnorm} and applying the argument to derive Lemma C.5 in \cite{bao2022spectral}, one has
\begin{equation*}
	\sup_{|M_{ij}|\le n^{-\epsilon}}\left|\partial_{m,ij}^{2q+1}\left(\rho_j^{-1}[Y^*\mathcal{G}(z_1)]_{ji}\langle \mathrm{e}^{\mathrm{i}x\langle L_{1}(f)\rangle}\rangle\Xi\right)\right|\lesssim \rho_{j}^{-(2q+2)},
\end{equation*}
which together with the crude bound $ \mathbb{E}|M_{ij}^{2q+2}|\lesssim n^{2q(1/2-\epsilon_{h})}$ and the high probability bound $\rho_{j}^{-1}\prec n^{-1/2}$ gives the desired result $\mathbb{E}(R_{m,ij}(q))\lesssim n^{-D}$ for any fixed $D>0$. Thereafter we complete the proof of this proposition.   \qed

\subsubsection{Proof of Lemma \ref{lem_cumulant_estH}}\label{subsec_cumulant_estH}
In this subsection we consider the part
\begin{equation*}
	II=\sum_{j\in\mathrm{I}_c}\mathbb{E}_{\Psi}^{\chi}(Y_j^*\mathcal{G}(z_1)Y_j\langle \mathrm{e}^{\mathrm{i}x\langle L_{1}(f)\rangle}\rangle).
\end{equation*}
Firstly, the resolvent expansion gives
\begin{equation*}
	Y_j^*\mathcal{G}(z_1)Y_j=Y_j^*\mathcal{G}^{(j)}(z_1)Y_j-Y_j^*\mathcal{G}(z_1)Y_jY_j^*\mathcal{G}^{(j)}(z_1)Y_j,
\end{equation*}
where $\mathcal{G}^{(j)}(z_1)=(YY^*-Y_jY_j^*-z_1 I)^{-1}$ is independent of the $j$-th column $Y_j$. Applying $|n^{-1}\operatorname{tr}\mathcal{G}(z_1)-n^{-1}\operatorname{tr}\mathcal{G}^{(j)}(z_1)|\lesssim n^{-1}$ and the averaged local law in Theorem \ref{thm_greenfuncomp_average}, we have
\begin{equation*}
	\begin{split}
		Y_j^*\mathcal{G}(z_1)Y_j=&\frac{Y_j^*\mathcal{G}^{(j)}(z_1)Y_j}{1+Y_j^*\mathcal{G}^{(j)}(z_1)Y_j}
		 =\frac{Y_j^*\mathcal{G}^{(j)}(z_1)Y_j-n^{-1}\operatorname{tr}\mathcal{G}^{(j)}(z_1)+n^{-1}\operatorname{tr}\mathcal{G}^{(j)}(z_1)}{1+n^{-1}\operatorname{tr}\mathcal{G}^{(j)}(z_1)+Y_j^*\mathcal{G}^{(j)}(z_1)Y_j-n^{-1}\operatorname{tr}\mathcal{G}^{(j)}(z_1)}\\
		=&\frac{\Delta_j(z_1)+m(z_1)+\mathrm{O}_{\prec}(n^{-1})}{1+m(z_1)+\Delta_j(z_1)+\mathrm{O}_{\prec}(n^{-1})},
	\end{split}
\end{equation*}
where $\Delta_j(z_1):=Y_j^*\mathcal{G}^{(j)}(z_1)Y_j-n^{-1}\operatorname{tr}\mathcal{G}^{(j)}(z_1)$.  Then, by the Taylor expansion for the last equation, one obtains that
\begin{equation*}
	\begin{split}
		 Y_j^*\mathcal{G}(z_1)Y_j=\sum_{k=1}^{K_0}(-1)^{k-1}\frac{(\Delta_j(z_1)+m(z_1))\Delta_j^{k-1}(z_1)}{(1+m(z_1))^k}+\operatorname{Res}(\Delta_j^{K_0-1}(z_1))+\mathrm{O}_{\prec}(n^{-1}),
	\end{split}
\end{equation*}
where the term $\operatorname{Res}(\Delta_j^{K_0-1}(z_1))$ contains the leading term $\Delta_j^{K_0-1}(z_1)$ with possibly deterministic coefficients. Plugging the expansion into $II$ and applying a conditional trick, we have
\begin{equation*}
	\begin{split}
		 =\sum_{j\in\mathrm{I}_c}\sum_{k=1}^{K_0}\mathbb{E}_{\Psi}^{\chi}\mathbb{E}_{Y_j}\big[\big((-1)^{k-1}\frac{\Delta_j^{k}(z_1)}{(1+m(z_1))^{k+1}}+\operatorname{Res}(\Delta_j^{K_0-1}(z_1))+\mathrm{O}_{\prec}(n^{-1})\big)\langle \mathrm{e}^{\mathrm{i}x\langle L_{1}(f)\rangle}\rangle\big].
	\end{split}
\end{equation*}

So we need to establish a good concentration estimation for the conditional expectations $\mathbb{E}_{Y_j}\Delta_j^k(z_1)$, The random variable $Y_j^*\mathcal{G}^{(j)}(z)Y_j$ can be decomposed into two parts, the diagonal part $\sum_{i}Y_{ij}^2[\mathcal{G}^{(j)}(z)]_{ii}$ and the off-diagonal part $\sum_{k\ne l}Y_{kj}Y_{lj}[\mathcal{G}^{(j)}(z)]_{kl}$. We believe that the main contribution comes from the diagonal part. The following lemma provides the concentration of $Y_j^*\mathcal{G}^{(j)}(z)Y_j$ for $\alpha\in (2,4]$.
\begin{lemma}[Concentration of the quadratic form]\label{lemma_momentquadratic}
	For $\alpha\in (2,4]$, let $A$ be one $n\times n$ deterministic symmetric matrix with $\|A\|\prec 1$. Define $\beta_{k_1,\ldots,k_r}=\mathbb{E}(y_1^{k_1}\cdots y_r^{k_r})$ and $S_r=\sum_{i}A_{ii}^{r}$ for $r\ge 1$. Then we have, for $k\ge 2$,
	\begin{equation*}
		\begin{split}
			\mathbb{E}(y^*Ay-\beta_{2}S_1)^k\lesssim \begin{cases}
				n\beta_{4}+\mathrm{O}(n^{-1}),~&\text{if $\xi$ is symmetric around 0},\\
				n\beta_{4}+\mathrm{o}(n^{-1/2}),~&\text{otherwise}.
			\end{cases}
		\end{split}
	\end{equation*}
\end{lemma}

\begin{proof}
	Consider the decomposition $y^*Ay=\sum_{i}y_i^2A_{ii}+\sum_{j\ne k}y_jy_kA_{jk}$. Notice $\|A\|\prec 1$, which implies $S_k\lesssim n$.
	Invoking the moment estimates in Lemmas \ref{lem_moment_rates} and \ref{lem_oddmoment_est}, we have
	\begin{equation*}
		\mathbb{E}(y^*Ay)=n^{-1}S_1+\beta_{1,1}\sum_{j\ne k}A_{jk}=n^{-1}S_1+\mathrm{o}(n^{-1/2}),
	\end{equation*}
	where the error term follows from the Cauchy-Schwarz inequality,
	\begin{equation*}
		\sum_{j\ne k}A_{jk}\le (\sum_{j\ne k}1^2\sum_{j\ne k}A_{jk}^2)^{1/2}\lesssim (n^2(\operatorname{tr}(AA^*)-S_2))^{1/2}\lesssim n^{3/2}.
	\end{equation*}
	Notice $\beta_{k_1,\ldots,k_r}=\mathbb{E}(y_1^{k_1}\cdots y_r^{k_r})$, which will vanish if there is at least one odd exponent under the symmetry condition. For $k=2$, similarly to Lemma 3.1 of \cite{heiny2023logdet}, we have
	\begin{equation*}
		\begin{split}
			\mathbb{E}(y^*Ay)^2=\mathbb{E}(\sum_{i}y_i^2A_{ii})^2+ \mathbb{E}(\sum_{j\ne k}y_jy_kA_{jk})^2+2\mathbb{E}(\sum_{i=1,j\ne k}y_i^2y_jy_kA_{ii}A_{jk}).
		\end{split}
	\end{equation*}
	By Lemma \ref{lem_moment_rates}, the first summand in the right hand side of the above identity can be rewritten as
	\begin{equation*}
		\mathbb{E}(\sum_{i}y_i^2A_{ii})^2=\beta_4\sum_{i}A_{ii}^2+\beta_{2,2}\sum_{i\ne j}A_{ii}A_{jj}=\beta_4S_2+\beta_{2,2}(S_1^2-S_2).
	\end{equation*}
	For the second summand, we have
	\begin{equation*}
		\mathbb{E}(\sum_{j\ne k}y_jy_kA_{jk})^2=\beta_{2,2}\sum_{j\ne k}A_{jk}^2+\beta_{2,1,1}\sum_{j\ne k\ne l}A_{jk}A_{jl}+\beta_{1,1,1,1}\sum_{j\ne k\ne l\ne s}A_{jk}A_{ls},
	\end{equation*}
	where the main part $\beta_{2,2}\sum_{j\ne k}A_{jk}^2\lesssim n^{-1}$ since $\sum_{j\ne k}A_{jk}^2=\operatorname{tr}(AA^*)-S_2\lesssim n\|A\|^2-S_2\lesssim n$. The last summand can be expressed as
	\begin{equation*}
		\mathbb{E}(\sum_{i=1,j\ne k}y_i^2y_jy_kA_{ii}A_{jk})=\beta_{2,1,1}\sum_{i=1,j\ne k}A_{ii}A_{jk}+\beta_{3,1}\sum_{i\ne j}A_{ii}A_{ij}.
	\end{equation*}
	Under the symmetry condition, we have $\mathbb{E}(y^*Ay)^2=\beta_{4}S_2+\beta_{2,2}(\operatorname{tr}(AA^*)+S_1^2-2S_2)$. Noting the estimates of Lemma \ref{lem_moment_rates}, one can find that the main term is $\beta_{2,2}S_1^2+\beta_{4}S_2$, where the other two terms can be bounded as $\beta_{2,2}(\operatorname{tr}(AA^*)-2S_2)\lesssim n^{-1}$. Therefore, we have
	\begin{equation*}
		\mathbb{E}(y^*Ay-\beta_2S_1)^2=\beta_{4}S_2+(\beta_{2,2}-\beta_2^2)S_1^2+\mathrm{O}(n^{-1})=\beta_{4}(S_2-n^{-1}S_1^2)+\mathrm{O}(n^{-1})
	\end{equation*}
	by $\beta_{2,2}-\beta_2^2=(n^{-1}-\beta_4)/(n-1)-n^{-2}=-n^{-1}\beta_{4}+\mathrm{O}(n^{-3})$. For general $k\ge 3$, by the basic bound $|y^*Ay-\beta_2S_1|\lesssim 1$, one has
	\begin{equation*}
		\mathbb{E}(y^*Ay-\beta_2S_1)^k\lesssim \mathbb{E}(y^*Ay-\beta_2S_1)^2\lesssim n\beta_{4}+\mathrm{O}(n^{-1}),
	\end{equation*}
	which confirms the claim of the symmetry case.
	
	Without the symmetry condition, the identity $y_{1}y_2=y_{1}y_2(\sum_{i}y_i^2)$ gives $\beta_{1,1}=2\beta_{3,1}+(n-2)\beta_{2,1,1}$, which further implies that
	\begin{equation*}
		\beta_{2,1,1}=(n-2)^{-1}(\beta_{1,1}-2\beta_{3,1})\lesssim \mathrm{o}(n^{-3}),
	\end{equation*}
	where we have used the estimate $\beta_{3,1}\lesssim \beta_{1,1}^{1/2}\beta_{3,3}^{1/2}\lesssim \beta_{1,1}^{1/2}[\beta_{2,2}\beta_{4,4}]^{1/4}\le \mathrm{o}(n^{-3/2-\alpha/4})$ by Lemma 3.1 of \cite{gine1997student} and Lemma \ref{lem_oddmoment_est}. Combined with the Cauchy-Schwarz inequality
	\begin{equation*}
		\sum_{j\ne k\ne l}A_{jk}A_{jl}\le \sum_{j}(\sum_{k}A_{jk})^2\le \sum_{j}(n\sum_{k}A_{jk}^2)\le n\operatorname{tr}(AA^*)\lesssim n^2,
	\end{equation*}
	one has $\beta_{2,1,1}\sum_{j\ne k\ne l}A_{jk}A_{jl}\le \mathrm{o}(n^{-1})$. Moreover, by $\beta_{1,1,1,1}\le \mathrm{o}(n^{-4})$ in Lemma \ref{lem_oddmoment_est}, we have
	\begin{equation*}
		\begin{split}
			\beta_{1,1,1,1}\sum_{j\ne k\ne l\ne s}A_{jk}A_{ls}\le \mathrm{o}(n^{-4})[(\sum_{j\ne k}A_{jk})^2-\sum_{j\ne k\ne l}A_{jk}A_{jl}-\sum_{j\ne k}A_{jk}^2]\le \mathrm{o}(n^{-1}),
		\end{split}
	\end{equation*}
	since $\sum_{j\ne k}A_{jk}\le \sum_{j}(n\sum_{k}A_{jk}^2)^{1/2}\lesssim n^{3/2}$. Moreover, by the Cauchy-Schwarz inequality
	\begin{equation*}
		\sum_{i=1,j\ne k}A_{ii}A_{jk}\lesssim n\sum_{j\ne k}A_{jk}\lesssim n^{5/2},~\sum_{i\ne j}A_{ii}A_{ij}\lesssim n^{3/2},
	\end{equation*}
	we have
	\begin{equation*}
		\mathbb{E}(\sum_{i=1,j\ne k}y_i^2y_jy_kA_{ii}A_{jk})=\beta_{2,1,1}\sum_{i=1,j\ne k}A_{ii}A_{jk}+\beta_{3,1}\sum_{i\ne j}A_{ii}A_{ij}\lesssim \mathrm{o}(n^{-1/2}).
	\end{equation*}
	Thus, without the symmetry condition, we have
	\begin{equation*}
		\mathbb{E}(y^*Ay-\beta_2S_1)^2=\beta_{4}(S_2-n^{-1}S_1^2)+\mathrm{o}(n^{-1/2}).
	\end{equation*}
	For general $k\ge 3$, by the basic bound $|y^*Ay-\beta_2S_1|\lesssim 1$ again, one has
	\begin{equation*}
		\mathbb{E}(y^*Ay-\beta_2S_1)^k\lesssim \mathbb{E}(y^*Ay-\beta_2S_1)^2\lesssim n\beta_{4}+\mathrm{o}(n^{-1/2}),
	\end{equation*}
	completing the proof of this proposition.
\end{proof}

Applying the above lemma to $\mathbb{E}|\Delta_j(z_1)|^k$ for $k\ge2$ and $\alpha\in(3,4]$ gives
\begin{equation*}
	\begin{split}
		 &\sum_{j\in\mathrm{I}_c}\sum_{k=2}^{K_0}\mathbb{E}_{Y_j}\big[(\Delta_j^k(z_1)+\operatorname{Res}(\Delta_j^{K_0-1}(z_1))+\mathrm{O}_{\prec}(n^{-1}))\langle \mathrm{e}^{\mathrm{i}x\langle L_{1}(f)\rangle}\rangle\big]\lesssim \sum_{j\in\mathrm{I}_c}\mathbb{E}\Delta_j^2(z_1)\lesssim |\mathrm{I}_c| n^{1-\alpha/2}\lesssim n^{3-\alpha},
	\end{split}
\end{equation*}
where in the first step we have used the fact that $|\Delta_j(z_1)|\sim 1$ and $|\langle \mathrm{e}^{\mathrm{i}x\langle L_{1}(f)\rangle}\rangle|\sim 1$, and in the last step we have applied the fact that $|\mathrm{I}_c|=\mathrm{O}(n^{2-\alpha/2})$.

It remains to considerr $\mathbb{E}_{\Psi}^{\chi}\big[\sum_{j\in\mathrm{I}_c}\frac{\Delta_j(z_1)}{(1+m(z_1))^2}\langle \mathrm{e}^{\mathrm{i}x\langle L_{1}(f)\rangle}\rangle\big]$. Typically, since $\mathbb{E}_{Y_j}\Delta_j(z_1)=0$, we believe that this term still vanishes after taking expectation. To this end, we take a further derivative of $x$ and let $x=0$. Then we have
\begin{equation*}
	\begin{split}
		\mathbb{E}_{\Psi}^{\chi}\Big[\big(\sum_{j\in\mathrm{I}_c}\frac{\Delta_j(z_1)}{(1+m(z_1))^2}\big)\cdot \frac{1}{2\pi}\oint_{\bar{\gamma}_1^0}\langle\operatorname{tr}R\mathcal{G}(z_1)\rangle\frac{f(z_1)}{z_1}\mathrm{d}z_1\Big].
	\end{split}
\end{equation*}
We only consider the random part in the above expression. Applying the resolvent representation, one has
\begin{equation*}
	Y^{*}_j\mathcal{G}(z_1)Y_j=\frac{\Delta_j(z_1)+m(z_1)+\widetilde{\Delta}_j(z_1)}{1+m(z_1)+\Delta_j(z_1)+\widetilde{\Delta}_j(z_1)},
\end{equation*}
where we have introduced an additional term $\widetilde{\Delta}_j(z_1)=n^{-1}\operatorname{tr}\mathcal{G}^{(j)}(z_1)-m(z_1)$. Since we have $|\widetilde{\Delta}_j(z_1)|\prec n^{-1}$ by the averaged local law, it suffices to consider
\begin{equation*}
	\begin{split}
		 &\mathbb{E}_{\Psi}^{\chi}\Big[\langle\sum_{j\in\mathrm{I}_c}\Delta_j(z_1)\rangle\cdot\Big(\sum_{j=1}^p\big(\sum_{k=1}^{K_1}\frac{(-1)^{k-1}\Delta^k_j(z_1)(1+\widetilde{\Delta}_j(z_1))}{(1+m(z_1))^k}+\operatorname{Res}(\Delta_j^{K_1-1}(z_1))\big)\Big)\Big]\\
		 &=\mathbb{E}_{\Psi}^{\chi}\Big[\langle\sum_{j\in\mathrm{I}_c}\Delta_j(z_1)\rangle\cdot\Big((\sum_{j\in\mathrm{I}_c}+\sum_{j\in\mathrm{T}_c})\big(\sum_{k=1}^{K_1}\frac{(-1)^{k-1}\Delta^k_j(z_1)(1+\widetilde{\Delta}_j(z_1))}{(1+m(z_1))^k}+\operatorname{Res}(\Delta_j^{K_1-1}(z_1))\big)\Big)\Big].
	\end{split}
\end{equation*}
We first consider the case where $j\in\mathrm{I}_c$. By Lemma \ref{lemma_momentquadratic}, we only need to handle $\mathbb{E}_{\Psi}^{\chi}\big[(\sum_{j\in\mathrm{I}_c}\Delta_j(z_1))^2\big]$. It turns out that
\begin{equation*}
	\begin{split}
		 &\mathbb{E}_{\Psi}^{\chi}\big[(\sum_{j\in\mathrm{I}_c}\Delta_j(z_1))^2\big]=\mathbb{E}_{\Psi}^{\chi}\Big[\sum_{j\in\mathrm{I}_c}\Delta_j^2(z_1)+\sum_{j,l\in\mathrm{I}_c\atop j\neq l}\Delta_j\Delta_l\Big]\\
		~=~&\mathbb{E}_{\Psi}^{\chi}\Big[\sum_{i\in\mathrm{I}_c}\mathbb{E}_{Y_j}\big(\Delta_j^2(z_1)+\sum_{j,l\in\mathrm{I}_c\atop j\neq l}\Delta_j\big(Y_l^{*}\mathcal{G}^{(lj)}(z_1)Y_l-Y_l^{*}\mathcal{G}^{(l)}Y_jY_j^{*}\mathcal{G}^{(lj)}Y_l-n^{-1}\operatorname{tr}\mathcal{G}^{(lj)}(z_1)+\mathrm{O}_{\prec}(n^{-1})\big)\big)\Big]\\
		~\asymp~& \mathbb{E}_{\Psi}^{\chi}\Big[\sum_{j,l\in\mathrm{I}_c\atop j\neq l}\Delta_j(z_1)\cdot \big(Y_l^{*}\mathcal{G}^{(l)}Y_jY_j^{*}\mathcal{G}^{(lj)}Y_l\big)\Big]+\mathrm{o}(1).
	\end{split}
\end{equation*}
Observing that $Y_l^{*}\mathcal{G}^{(l)}Y_jY_j^{*}\mathcal{G}^{(lj)}Y_l=Y_l^{*}\mathcal{G}^{(lj)}Y_jY_j^{*}\mathcal{G}^{(lj)}Y_l-Y_l^{*}\mathcal{G}^{(l)}Y_jY_j^{*}\mathcal{G}^{(lj)}(z_1)Y_jY_j^{*}\mathcal{G}^{(lj)}Y_l$, we may obtain that
\begin{equation*}
	\begin{split}
		 Y_l^{*}\mathcal{G}^{(l)}Y_jY_j^{*}\mathcal{G}^{(lj)}Y_l&=\frac{Y_l^{*}\mathcal{G}^{(lj)}Y_jY_j^{*}\mathcal{G}^{(lj)}Y_l}{1+Y_j^{*}\mathcal{G}^{(lj)}Y_j}=\frac{Y_l^{*}\mathcal{G}^{(lj)}Y_jY_j^{*}\mathcal{G}^{(lj)}Y_l}{1+m(z_1)+\Delta_{j(l)}(z_1)+\mathrm{O}_{\prec}(n^{-1})},
	\end{split}
\end{equation*}
where $\Delta_{j(l)}=Y_j^{*}\mathcal{G}^{(jl)}Y_j-n^{-1}\operatorname{tr}\mathcal{G}^{(jl)}$. Similarly, we can conduct the expansion for $\Delta_j(z_1)$, so
\begin{equation*}
	\begin{split}
		\Delta_j(z_1)
		 &=\Delta_{j(l)}(z_1)-\frac{Y_j^{*}\mathcal{G}^{(jl)}Y_lY_l^{*}\mathcal{G}^{(jl)}Y_j}{1+Y_l^{*}\mathcal{G}^{(jl)}Y_l}+\mathrm{O}_{\prec}(n^{-1}).
	\end{split}
\end{equation*}
Then, one has
\begin{equation*}
	\begin{split}
		&\mathbb{E}_{\Psi}^{\chi}\Big[\sum_{j,l\in\mathrm{I}_c\atop j\neq l}\Delta_j(z_1)\cdot \big(Y_l^{*}\mathcal{G}^{(l)}Y_jY_j^{*}\mathcal{G}^{(lj)}Y_l\big)\Big]\\
		~\asymp~&\mathbb{E}_{\Psi}^{\chi}\Big[\sum_{j,l\in\mathrm{I}_c\atop j\neq l}\Big(\mathbb{E}_{Y_l}\big(\sum_{k=1}^{K_2}\Delta^k_{j(l)}(z_1)\cdot (Y_l^{*}\mathcal{G}^{(lj)}Y_jY_j^{*}\mathcal{G}^{(lj)}Y_l)\big)+\frac{Y_j^{*}\mathcal{G}^{(jl)}Y_lY_l^{*}\mathcal{G}^{(jl)}Y_j}{1+Y_l^{*}\mathcal{G}^{(jl)}Y_l}\cdot\frac{Y_l^{*}\mathcal{G}^{(lj)}Y_jY_j^{*}\mathcal{G}^{(lj)}Y_l}{1+m(z_1)}\Big]\\
		~\asymp~&\mathbb{E}_{\Psi}^{\chi}\Big[\sum_{j,l\in\mathrm{I}_c\atop j\neq l}\big(n^{-1}\sum_{k=1}^{K_2}\Delta_{j(l)}^kY_j^{*}\mathcal{G}^{(lj)}\mathcal{G}^{(lj)}Y_j+\frac{(Y_l^{*}\mathcal{G}^{(jl)}Y_j)^4}{(1+m(z_1))^2}\big)\Big]+\mathrm{o}(1)~\asymp~\mathrm{o}(1).
	\end{split}
\end{equation*}
This concludes that
\begin{equation*} \mathbb{E}_{\Psi}^{\chi}\Big[\langle\sum_{j\in\mathrm{I}_c}\Delta_j(z_1)\rangle\cdot\Big(\sum_{j\in\mathrm{I}_c}\big(\sum_{k=1}^{K_1}\frac{(-1)^{k-1}\Delta_j^k(z_1)}{(1+m(z_1))^k}+\operatorname{Res}(\Delta_j^{K_1-1}(z_1))\big)\Big)\Big]=\mathrm{o}(1).
\end{equation*}
It remains to study
\begin{equation*}
\mathbb{E}_{\Psi}^{\chi}\Big[\sum_{j\in\mathrm{I}_c}\Big(\sum_{l\in\mathrm{T}_c}\langle\Delta_j(z_1)\rangle\cdot\big(\sum_{k=1}^{K_1}\frac{(-1)^{k-1}\Delta_l^k(z_1)(1+\widetilde{\Delta}_l(z_1))}{(1+m(z_1))^k}\big)+\operatorname{Res}(\Delta_l^{K_1-1}(z_1))\Big)\Big].
\end{equation*}
In the following, we only consider $\mathbb{E}_{\Psi}^{\chi}\Big[\sum_{j\in\mathrm{I}_c}\sum_{l\in\mathrm{T}_c}\langle\Delta_j(z_1)\rangle\cdot\Delta_l(z_1)\Big]$ while others can be handled similarly. The key observation here is that for indexes $l\in\mathrm{T}_c$, all $X_{li}$'s have deterministic bounds, which is suitable to the cumulant expansion. Before moving forward, we need firstly handle the term $\Delta_j(z_1)\cdot (n^{-1}\operatorname{tr}\mathcal{G}^{(l)})$ in $\Delta_j(z_1)\Delta_l(z_1)$. One may find that
\begin{equation*}
	\begin{split}
		 &\mathbb{E}_{\Psi}^{\chi}\big(\Delta_j(z_1)\cdot(n^{-1}\operatorname{tr}\mathcal{G}^{(l)})\big)=\mathbb{E}_{\Psi}^{\chi}\big[\Delta_j(z_1)\cdot(n^{-1}\operatorname{tr}\mathcal{G}^{(lj)}-n^{-1}Y_j^{*}\mathcal{G}^{(lj)}\mathcal{G}^{(l)}Y_j)\big]\\
		~=~&\mathbb{E}_{\Psi}^{\chi}\big[\mathbb{E}_{Y_j}\big(\Delta_j\cdot n^{-1}\operatorname{tr}\mathcal{G}^{(lj)}\big)-n^{-1}\Delta_j(z_1)\frac{Y_j^{*}(\mathcal{G}^{(lj)})^2Y_j}{1+Y_j^{*}\mathcal{G}^{(lj)}Y_j}\big]\\
		 ~=~&n^{-1}\mathbb{E}_{\Psi}^{\chi}\big[\Delta_j(z_1)\frac{Y_j^{*}(\mathcal{G}^{(lj)})^2Y_j-n^{-1}\operatorname{tr}(\mathcal{G}^{(lj)})^2+n^{-1}\operatorname{tr}(\mathcal{G}^{(lj)})^2}{1+Y_j^{*}\mathcal{G}^{(lj)}Y_j-n^{-1}\operatorname{tr}\mathcal{G}^{(lj)}+n^{-1}\operatorname{tr}\mathcal{G}^{(lj)}}\big]\\
		 ~\asymp~&n^{-1}\mathbb{E}_{\Psi}^{\chi}\big[\Delta_j(z_1)\frac{n^{-1}\operatorname{tr}(\mathcal{G}^{(lj)})^2+\mathrm{O}_{\prec}(n^{-1})}{1+n^{-1}\operatorname{tr}\mathcal{G}^{(lj)}+\mathrm{O}_{\prec}(n^{-1})}\big]=n^{-1}\mathbb{E}_{\Psi}^{\chi}\big[\mathbb{E}_{Y_j}\big(\Delta_j(z_1)\frac{n^{-1}\operatorname{tr}(\mathcal{G}^{(lj)})^2}{1+m(z_1)}\big)+\Delta_j\cdot\mathrm{O}_{\prec}(n^{-1})\big]\\
		~=~&\mathrm{O}(n^{-2}),
	\end{split}
\end{equation*}
where we also have used the deterministic bound $|\Delta_j(z_1)|\prec 1$.

Now we consider the remaining term and apply the cumulant expansion. To this end, we have
\begin{equation*}
	\begin{split}
		 &\mathbb{E}_{\Psi}^{\chi}\big(\Delta_j(z_1)\cdot(Y_l^{*}\mathcal{G}^{(l)}Y_l)\big)=\sum_i\mathbb{E}_{\Psi}^{\chi}\big[\Delta_j(z_1)\big(Y_{il}^2\mathcal{G}_{ii}^{(l)}(z_1)+\sum_{k\neq i}Y_{il}Y_{kl}\mathcal{G}_{ik}^{(l)}\big)\big]\\
		=&\sum_i\kappa_{1,m}\mathbb{E}_{\Psi}^{\chi}\big(\Delta_j(z_1)\cdot(\rho_l^{-1}Y_{il}\mathcal{G}_{ii}^{(l)}+\sum_{k\neq i}\rho_l^{-1}Y_{kl}\mathcal{G}^{(l)}_{ik})\big)\\
		 &+\sum_i\kappa_{2,m}\mathbb{E}_{\Psi}^{\chi}\big[\big(\partial_{m,il}\Delta_j(z_1)\big)\cdot(\rho_l^{-1}Y_{il}\mathcal{G}_{ii}^{(l)}+\sum_{k\neq i}\rho_l^{-1}Y_{kl}\mathcal{G}^{(l)}_{ik})+\Delta_j(z_1)\big(\partial_{m,il}(\rho_l^{-1}Y_{il}\mathcal{G}_{ii}^{(l)}+\sum_{k\neq i}\rho_l^{-1}Y_{kl}\mathcal{G}^{(l)}_{ik})\big)\big]\\
		 &+\sum_i\frac{\kappa_{3,m}}{2!}\mathbb{E}_{\Psi}^{\chi}\big[(\partial^2_{m,il}\Delta_j(z_1)\big)\cdot(\rho_l^{-1}Y_{il}\mathcal{G}_{ii}^{(l)}+\sum_{k\neq i}\rho_l^{-1}Y_{kl}\mathcal{G}^{(l)}_{ik})+\Delta_j(z_1)\big(\partial^2_{m,il}(\rho_l^{-1}Y_{il}\mathcal{G}_{ii}^{(l)}+\sum_{k\neq i}\rho_l^{-1}Y_{kl}\mathcal{G}^{(l)}_{ik})\big)\big]\\
		&+\sum_{t\ge 4}^q\sum_i\mathrm{O}(\kappa_{t,m})\mathbb{E}_{\Psi}^{\chi}\big[(\partial^{s-1}_{m,il}\Delta_j(z_1))\cdot(\rho_l^{-1}Y_{il}\mathcal{G}_{ii}^{(l)}+\sum_{k\neq i}\rho_l^{-1}Y_{kl}\mathcal{G}^{(l)}_{ik})+\Delta_j(z_1)\big(\partial^{s-1}_{m,il}(\rho_l^{-1}Y_{il}\mathcal{G}_{ii}^{(l)}+\sum_{k\neq i}\rho_l^{-1}Y_{kl}\mathcal{G}^{(l)}_{ik})\big)\big]\\
		&+\sum_i\mathbb{E}(R_{m,il}(q)),
	\end{split}
\end{equation*}
where the error term $R_{m,il}$ is given by
\begin{equation*}
	\begin{split}
		R_{m,il}(q)=&C\sup_{M_{il}\in\mathbb{R}}|\partial^{q+1}_{m,il}\big(\Delta_j(z_1)\cdot(\rho_l^{-1}Y_{il}\mathcal{G}_{ii}^{(l)}+\sum_{k\neq i}\rho_l^{-1}Y_{kl}\mathcal{G}^{(l)}_{ik})\big)|\cdot\mathbb{E}\big[|M_{il}^{q+2}|\mathbbm{1}(|M_{il}|>n^{-\epsilon})\big]\\
		&+C\mathbb{E}|M_{ij}^{q+2}|\cdot\sup_{|M_{il}|\le n^{-\epsilon}}|\partial_{m,il}^{q+1}\big(\Delta_j(z_1)\cdot(\rho_l^{-1}Y_{il}\mathcal{G}_{ii}^{(l)}+\sum_{k\neq i}\rho_l^{-1}Y_{kl}\mathcal{G}^{(l)}_{ik})\big)|
	\end{split}
\end{equation*}
for some constant $C$ and $\epsilon>0$.

We calculate these terms one by one. Firstly, observe that $\kappa_{1,m}\lesssim n^{(1/2-\epsilon_h)(-\alpha+1)}$. This gives
\begin{equation*}
	\sum_i\kappa_{1,m}\mathbb{E}_{\Psi}^{\chi}\big(\Delta_j(z_1)\cdot(\rho_l^{-1}Y_{il}\mathcal{G}_{ii}^{(l)}+\sum_{k\neq i}\rho_l^{-1}Y_{kl}\mathcal{G}^{(l)}_{ik})\big)\lesssim\mathrm{O}(n^{(1-\epsilon_h)(-\alpha+1)+1/2}),
\end{equation*}
where we have used Wald's identity to estimate the summation involving the off-diagonal entries of $\mathcal{G}^{(l)}$.

Secondly, from Section \ref{subsubsec_derivative}, we have
\begin{equation*}
	 \partial_{m,il}\Delta_j(z_1)=Y_j^{*}\mathcal{G}^{(j)}(\rho_j^{-1}Y_{il}-\rho_j^{-1}Y_{il}^3)\mathcal{G}^{(j)}Y_j-n^{-1}\operatorname{tr}(\mathcal{G}^{(j)})^2\cdot(\rho_j^{-1}Y_{il}-\rho_j^{-1}Y_{il}^3)
\end{equation*}
and
\begin{equation*}
	\begin{split}
		\partial_{m,il}(\rho_l^{-1}Y_{il}\mathcal{G}_{ii}^{(l)}+\sum_{k\neq i}\rho_l^{-1}Y_{kl}\mathcal{G}_{ik}^{(l)})=M_{il}\rho_l^{-3}Y_{il}\mathcal{G}_{ii}^{(l)}+\rho_l^{-2}\mathcal{G}_{ii}^{(l)}+M_{il}^2\rho_l^{-4}\mathcal{G}_{ii}^{(l)}+\sum_{k\neq i}M_{il}\rho_j^{-3}Y_{kl}\mathcal{G}_{ik}^{(l)}.
	\end{split}
\end{equation*}
The above two identities provide the estimate for the terms following $\kappa_{2,m}$,
\begin{equation*}
	\begin{split}
		 &\sum_i\kappa_{2,m}\mathbb{E}_{\Psi}^{\chi}\big[\big(\partial_{m,il}\Delta_j(z_1)\big)\cdot(\rho_l^{-1}Y_{il}\mathcal{G}_{ii}^{(l)}+\sum_{k\neq i}\rho_l^{-1}Y_{kl}\mathcal{G}^{(l)}_{ik})+\Delta_j(z_1)\big(\partial_{m,il}(\rho_l^{-1}Y_{il}\mathcal{G}_{ii}^{(l)}+\sum_{k\neq i}\rho_l^{-1}Y_{kl}\mathcal{G}^{(l)}_{ik})\big)\big]\\
		 ~=~&\sum_i\kappa_{2,m}\mathbb{E}_{\Psi}^{\chi}\big[\Delta^{\prime}_j(z_1)\cdot(\rho_l^{-1}Y_{il}-\rho_j^{-1}Y_{il}^3)\cdot(\rho_l^{-1}Y_{il}\mathcal{G}_{ii}^{(l)}+\sum_{k\neq i}\rho_l^{-1}Y_{kl}\mathcal{G}^{(l)}_{ik})\\
		 &+\Delta_j(z_1)\cdot(M_{il}\rho_l^{-3}Y_{il}\mathcal{G}_{ii}^{(l)}+\rho_l^{-2}\mathcal{G}_{ii}^{(l)}+M_{il}^2\rho_l^{-4}\mathcal{G}_{ii}^{(l)}+\sum_{k\neq i}M_{il}\rho_j^{-3}Y_{kl}\mathcal{G}_{ik}^{(l)})\big]\\
		 ~\asymp~&\kappa_{2,m}\mathbb{E}_{\Psi}^{\chi}\big[\Delta_j^{\prime}(z_1)\cdot(\rho_l^{-2}Y_l^{*}\mathcal{G}^{(l)}Y_l)+\Delta_j(z_1)\cdot(\rho_l^{-2}\operatorname{tr}\mathcal{G}^{(l)}+\rho_l^{-2}Y_l^{*}\mathcal{G}^{(l)}Y_l)\big]+\mathrm{O}_{\prec}(n^{-\alpha/2}),
	\end{split}
\end{equation*}
where we denote $\Delta_j^{\prime}(z_1)=Y_j^{*}(\mathcal{G}^{(j)})^2Y_j-n^{-1}\operatorname{tr}(\mathcal{G}^{(j)})^2$ and use the moment condition $\mathbb{E}Y_{il}^4\asymp n^{-\alpha/2}$. Furthermore, it is easy to find that the leading term within the expectation from the above expansion is $\rho_j^{-2}\Delta_j(z_1)\operatorname{tr}\mathcal{G}^{(l)}$. By the exactly same procedure as the calculation of $\mathbb{E}_{\Psi}^{\chi}\big[\Delta_j(z_1)\cdot(n^{-1}\operatorname{tr}\mathcal{G}^{(l)})\big]$, we actually can show that $\mathbb{E}_{\Psi}^{\chi}\big[\rho_j^{-2}\Delta_j(z_1)\operatorname{tr}\mathcal{G}^{(l)}\big]=\mathrm{O}_{\prec}(n^{-2})$. The only difference here is that the factor $n^{-1}$ is replaced by $\rho_l^{-2}$. However, we find that $\rho_l^{-2}\lesssim n^{-1}$ and it is independent of $Y_j$. We omit further details for simplification here.

We remark that for the rest terms, each additional partial derivative for $M_{il}$ provides one additional $\rho_l^{-1}$ factor. And for $\alpha\in(3,4]$, $\kappa_1$, $\kappa_2$ and $\kappa_3$  are bounded. Besides, for $\kappa_k$ with $k\ge 4$, we may use the deterministic bound $|M_{il}\rho_l^{-1}|=\mathrm{o}_{\prec}(1)$. Then we may conclude that the rest terms in the expansion of $\mathbb{E}_{\Psi}^{\chi}\big[\Delta_j(z_1)\cdot(Y_l^{*}\mathcal{G}^{(l)}Y_l)\big]$ are much smaller than the first two terms involving $\kappa_{1,m}$ and $\kappa_{2,m}$. Combining all the above results with the fact that $|\mathrm{I}_c|=\mathrm{o}(n^{1/2})$ and $|\mathrm{T}_c|=\mathrm{O}(n)$, we finally conclude that for $\alpha\in(3,4]$,
\begin{equation*}
	 \mathbb{E}_{\Psi}^{\chi}\Big[\big(\sum_{j\in\mathrm{I}_c}\frac{\Delta_j(z_1)}{(1+m(z_1))^2}\big)\cdot\frac{1}{2\pi}\oint_{\bar{\gamma}_1^0}\langle\operatorname{tr}R\mathcal{G}(z_1)\rangle\frac{f(z_1)}{z_1}\mathrm{d}z_1\Big]=\mathrm{o}(1).
\end{equation*}
Therefore, when $\alpha\in(3,4]$, we have $II=\mathrm{o}(1)$.

\subsubsection{Proof for the critical case $\alpha=3$}\label{sec_alpha3}
 In this section, we show that the critical condition
\begin{equation*}
	\lim_{x\rightarrow \infty} x^{3}\mathbb{P}(|\xi|>x)=0,
\end{equation*}
is necessary and sufficient for the universal CLT with the same asymptotic mean and variance as demonstrated in the case of $\alpha\in (3,4]$, as well as the finite fourth-moment case. Similar to the case of $\alpha\in (3,4]$, the main difficulty to extend the result to $\alpha=3$ is the concentration of the quadratic form $Y_j^*\mathcal{G}^{(j)}(z)Y_j$. Thanks to Lemma \ref{lemma_momentquadratic}, the asymptotic behavior of $Y_j^*\mathcal{G}^{(j)}(z)Y_j$ is driven by the parameter $\beta_4$, which illustrates the explicit relation between the slowly varying function $l(x)$ and the concentration of $Y_j^*\mathcal{G}^{(j)}(z)Y_j$ for $\alpha \in (2,4]$.

\

\noindent\textbf{I. The sufficient condition for the universal CLT}\\
Let's consider a further resampling of $X_{ji}$ at the level of $\delta_n n^{1/2}$ such that $\delta_n\rightarrow 0$ and
\begin{equation*}
	\lim_{n\rightarrow \infty}\delta_{n}^{-3}l(\delta_n n^{1/2})=0.
\end{equation*}
Firstly, by the definition of $H_{ij}$, one has for general $0<s<3$,
\begin{equation*}
	\begin{split}
		\mathbb{E}|H_{ij}|^{s}= \mathbb{E}(|h_{ij}|^s)\mathbb{P}(|X_{ji}|>\delta_n n^{1/2})\lesssim \int_{\delta_n n^{1/2}}^{\infty} x^{-4+s}l(x)\mathrm{d}x \lesssim \mathrm{o}(n^{(-3+s)/2}).
	\end{split}
\end{equation*}
We define
\begin{equation*}
	\begin{split}
		\mathbb{P}(\psi_{ij}=1)=\mathbb{P}(|X_{ji}|>\delta_n n^{1/2}).
	\end{split}
\end{equation*}
One can check that $\mathbb{P}(|\xi|>\delta_n n^{1/2})\lesssim l(\delta_n n^{1/2})(\delta_n n^{1/2})^{-3}$, which implies, for large enough $n$,
\begin{equation*}
	\begin{split}
		&\mathbb{P}(\#\{(i,j):|X_{ji}|>\delta_n n^{1/2}\}>n^{1/2})\lesssim \mathbb{P}(|\xi|_{(n^{1/2})}>\delta_n n^{1/2})\\
		\lesssim &\sum_{k=n^{1/2}}^{n^2}\binom{n^2}{k} [l(\delta_n n^{1/2})(\delta_n n^{1/2})^{-3}]^{k}(1-n^{-1})^{n^2-k}\\
		\lesssim &\sum_{k=n^{1/2}}^{n^2}\left(e\frac{n^2}{k}\right)^k n^{-3k/2}[\delta_n^{-3} l(\delta_n n^{1/2})]^k\lesssim e^{-n^{1/2}} \lesssim n^{-D},
	\end{split}
\end{equation*}
for any fixed $D>0$ since $\lim_{n\rightarrow \infty}\delta_{n}^{-3}l(\delta_n n^{1/2})=0$. This implies that the following event
\begin{equation*}
	\Omega_{\Psi}=\{\Psi~\text{has at most}~n^{1/2}~\text{entries equal to one}\}
\end{equation*}
holds with high probability for $\alpha=3$, which shows $|\mathrm{I}_c|=\mathrm{O}(n^{1/2})$ given $\Omega_{\Psi}$. In the following, we follow the arguments leading to the estimate of the characteristic function of $\operatorname{tr}f(R)$, which relies on estimating two parts of \eqref{eq_prf_characteristicfunctioncomparison_expansion_phi}.

Under the critical condition, we have $\lim_{n\rightarrow \infty}l(n^{1/2})=0$, which implies
\begin{equation*}
	n\beta_4\sim \frac{\alpha\Gamma(\alpha/2)\Gamma(2-\alpha/2)}{2\Gamma(2)}n^{-1/2}l(n^{1/2})=\mathrm{o}(n^{-1/2})
\end{equation*}
by Lemma \ref{lem_moment_rates}. Recycling the notation in the case of $\alpha\in (3,4]$ and applying Lemma \ref{lemma_momentquadratic} to $\Delta_j(z_1)$ result in
\begin{equation*}
	\mathbb{E}\Delta_j^2(z_1)\lesssim n\beta_4+\mathrm{o}(n^{-1/2})=\mathrm{o}(n^{-1/2}).
\end{equation*}
For the heavy-tailed part, we have
\begin{equation*}
	\begin{split}
		 &\sum_{j\in\mathrm{I}_c}\sum_{k=2}^{K_0}\mathbb{E}_{Y_j}\big[(\Delta_j^k(z_1)+\operatorname{Res}(\Delta_j^{K_0-1}(z_1))+\mathrm{O}_{\prec}(n^{-1}))\langle \mathrm{e}^{\mathrm{i}x\langle L_{1}(f)\rangle}\rangle\big]\lesssim \sum_{j\in\mathrm{I}_c}\mathbb{E}\Delta_j^2(z_1)=\mathrm{o}(1),
	\end{split}
\end{equation*}
where in the first step we have used the fact that $|\Delta_j(z_1)|\sim 1$ and $|\langle \mathrm{e}^{\mathrm{i}x\langle L_{1}(f)\rangle}\rangle|\sim 1$, and in the last step, we have applied the fact that $|\mathrm{I}_c|=\mathrm{O}(n^{1/2})$ with high probability. Following the same arguments and noticing the estimate $\sum_{j\in\mathrm{I}_c}\mathbb{E}\Delta_j^2(z_1)=\mathrm{o}(1)$ under the critical condition, with the decomposition at $\delta_n n^{1/2}$, we also have
\begin{equation*}
	\mathbb{E}_{\Psi}^{\chi}\big[\sum_{j\in\mathrm{I}_c}\frac{\Delta_j(z_1)}{(1+m(z_1))^2}\langle \mathrm{e}^{\mathrm{i}x\langle L_{1}(f)\rangle}\rangle\big]=\mathrm{o}(1).
\end{equation*}

Now we turn to the $M$ part, which can be estimated by the cumulant expansion trick similar to the case of $\alpha\in (3,4]$, except the truncation level. Recycling the notation, we have
\begin{equation*}
	\kappa_{1,m}=\mathrm{o}(n^{-1}),~\kappa_{2,m}=1-\mathrm{o}(n^{-1/2}),~\kappa_{3,m}=\mathrm{o}(\log n).
\end{equation*}
Following the arguments leading to Lemma \ref{lem_cumulant_estM}, one can show that the terms with $\kappa_{1,m}$ and $\kappa_{3,m}$ are negligible, and the main part is the term with $\kappa_{2,m}$, which has the same estimate as $\alpha>3$ due to $\kappa_{2,m}=1-\mathrm{o}(n^{-1/2})$. For the error terms, due to the decomposition, one can get the decay rate of $\delta_n$ for high-order terms, which will be negligible for large enough $q$. Thus, we conclude that the universal CLT holds under the critical condition.

\

\noindent\textbf{II. The necessary condition for the universal CLT}\\
In this paragraph, we argue that \eqref{eq_criticalcase} is the necessary condition for the established CLT. To see this, we assume that
\begin{equation}\label{eq_oppnecessarycondition}
	\limsup_{x\rightarrow\infty}x^3\mathbb{P}(|\xi|>x)\neq0.
\end{equation}
Indeed, it is sufficient to show that under condition \eqref{eq_oppnecessarycondition}, there will be a change in the limiting distribution of the LSS. In particular, we assume that the LSS is asymptotically normal and show that the limiting variance will change.
 More precisely we will show that the heavy-tailed part $II$ is not negligible anymore given \eqref{eq_oppnecessarycondition}, which further implies that the approximated ODE of $\mathcal{T}_f(x)$ does not converge to the universal case and thus demonstrates the necessity of \eqref{eq_criticalcase}. In the following, we only discuss the details of the symmetric case while the asymmetric case will be even worse and inherit the same key observation.

Assuming \eqref{eq_oppnecessarycondition}, there is a small constant $0<c_0$ and a sequence $\{l_N\}$ such that $l_N\rightarrow\infty$ as $N\rightarrow\infty$ and
\begin{equation*}
	\mathbb{P}(|\xi|>l_N)\ge\frac{c_0}{l_N^3},\quad N\rightarrow\infty.
\end{equation*}
By Lemma 4.1 in \cite{heiny2023logdet} for $\alpha=3$, we have
\begin{equation}\label{eq_momentest_necessity}
	\mathbb{E}(Y_{11}^{2k_1}\dots Y_{q1}^{2k_q})\asymp \frac{C(3)}{n^{-N_1/2+3q/2}}
\end{equation}
for some constant $C(3)>0$, where $N_1=\#\{1\le i\le q:k_i=1\}$ and $k_1,\dots,k_q\ge 1$.
Moreover, one has
\begin{equation*}
	\begin{split}
		&\mathbb{P}(\#\{(i,j):|X_{ji}|>n^{1/2-\epsilon_h}\}\le C_0n^{1/2})\lesssim\mathbb{P}(|\xi|_{(C_0n^{1/2})}\le n^{1/2-\epsilon_h})=1-\mathbb{P}(|\xi|_{(C_0n^{1/2})}>n^{1/2-\epsilon_h}),
	\end{split}
\end{equation*}
where we choose $C_0$ such that $C_0n^{1/2}\in\mathbb{Z}^{+}$. Taking a sequence of $n_N$ such that $l_{N-1}\le \lfloor n_N^{1/2-\epsilon_h}\rfloor$ and $\lceil n_N^{1/2-\epsilon_h}\rceil\le l_{N}$, we have
\begin{equation*}
	\begin{split}
		 &1-\mathbb{P}(|\xi|_{(C_0n_N^{1/2})}>n_N^{1/2-\epsilon_h})\lesssim1-\sum_{k=C_0n_N^{1/2}}^{n_N^2}\binom{n_N^2}{k}(n_N+1)^{(-3/2+3\epsilon_h)k}(1-(n_N+1)^{-3/2+3\epsilon_h})^{n_N^2-k}\\
		 &~\lesssim~1-\sum_{k=C_0n_N^{1/2}}^{n_N^2}\frac{n_N}{k^{1/2}(n_N^2-k)^{1/2}}\big(\frac{n_N^{1/2+3\epsilon_h}}{k}\big)^k(\frac{n_N^2-n_N^{1/2+3\epsilon_h}}{n_N^2-k})^{n_N^2-k}\\
		&~\lesssim~ 1-\sum_{k=C_0n_N^{1/2}}^{C_1n_N^{1/2+3\epsilon_h}}\frac{1}{k^{1/2}}\big(\frac{n_N^{1/2+3\epsilon_h}}{k}\big)^k-\sum_{k=C_1n_N^{1/2+3\epsilon_h}}^{n_N^2}\frac{n_N}{k^{1/2}(n_N^2-k)^{1/2}}\big(\frac{n_N^{1/2+3\epsilon_h}}{k}\big)^k(\frac{n_N^2-n_N^{1/2+3\epsilon_h}}{n_N^2-k})^{n_N^2-k}\\
		&\lesssim1-\sum_{k=C_0n_N^{1/2}}^{C_1n_N^{1/2+3\epsilon_h}}\frac{1}{k^{1/2}}\big(\frac{n_N^{1/2+3\epsilon_h}}{k}\big)^k\lesssim n_N^{-D},
	\end{split}
\end{equation*}
where $C_1$ satisfies that $C_1n_N^{1/2+3\epsilon_h}\in\mathbb{Z}^{+}$, in the fourth step we have employed Stirling's formula, and in the last step we have invoked the fact that a well-defined probability is nonnegative so that we can choose $k$ in the range $[C_0n_N^{1/2}, C_1n_N^{1/2+3\epsilon_h}]$ to ensure the last inequality. Finally, summarising the condition on $n_N$, we can update Lemma \ref{lem_wellconfigured} as follows.
\begin{lemma}\label{lem_wellconfigured_alpha=3}
	Under the condition \eqref{eq_oppnecessarycondition}, $\Psi$ is well configured as there are at most $n_N^{1-\epsilon_y}$ but at least $n_N^{1/2}$ entries equal to one with high probability.
\end{lemma}

Now, we reconsider the characteristic function. Note that at this time, we denote the matrix size of $R$ as $n_N\times n_N$. Recall that we have the decomposition of $(\mathcal{T}_{f}(x))^{\prime}$ as
\begin{equation*}
	\begin{split}
		 (\mathcal{T}_{f}(x))^{\prime}&=-\frac{1}{2\pi}\oint_{\bar{\gamma}_1^0}\Big(\sum_{i}\sum_{j\in\mathrm{T}_c}\mathbb{E}^{\chi}_{\Psi}[ly_{ij}[Y^{*}\mathcal{G}(z_1)]_{ji}\langle e^{ix\langle L_1(f)\rangle}\rangle]\cdot\mathbbm{1}(\psi_{ij}=0)\\
		&+\sum_{j\in\mathrm{I}_c}\mathbb{E}_{\Psi}^{\chi}[Y_j^{*}\mathcal{G}(z_1)Y_j\langle e^{ix\langle L_1(f)\rangle}\rangle]\cdot\mathbbm{1}(\psi_{ij}=1)\Big)\cdot f(z_1)\mathrm{d}z_1.
	\end{split}
\end{equation*}
After carefully checking the proof of Lemma \ref{lem_cumulant_estM}, we find that it still holds under the condition \eqref{eq_oppnecessarycondition} (Indeed, it holds for all $\alpha>2$ under the symmetry). Then the first term $\sum_{i}\sum_{j\in\mathrm{T}_c}I_{ij}$ in the integral establishes the same limiting variance as $\alpha\in(3,4]$. It remains the second term $II$. To this end, one may repeat the procedures in the proof of Lemma \ref{lem_cumulant_estH} (taking derivative of $x$ and setting $x=0$).
It turns out that we have to consider
\begin{equation*}
	 \mathbb{E}_{\Psi}^{\chi}\Big[(\sum_{j\in\mathrm{I}_c}\frac{\Delta_j(z_1)}{(1+m(z_1))^2})\cdot\frac{1}{2\pi}\oint_{\bar{\gamma}_1^0}\langle\operatorname{tr}R\mathcal{G}(z_1)\rangle\frac{f(z_1)}{z_1}\mathrm{d}z_1\Big].
\end{equation*}
Now using the resolvent expansion for $\operatorname{tr}R\mathcal{G}(z_1)$, we may conclude the following inequality for properly chosen $C_1>0$,
\begin{equation*}
	\begin{split}
		 &\mathbb{E}_{\Psi}^{\chi}\Big[(\sum_{j\in\mathrm{I}_c}\frac{\Delta_j(z_1)}{(1+m(z_1))^2})\cdot\frac{1}{2\pi}\oint_{\bar{\gamma}_1^0}\langle\operatorname{tr}R\mathcal{G}(z_1)\rangle\frac{f(z_1)}{z_1}\mathrm{d}z_1\Big]\ge \mathbb{E}_{\Psi}^{\chi}\big[C_1\sum_{j\in\mathrm{I}_c}\Delta_j^2(z_1)\big]\\
		 &=C_1\sum_{j\in\mathrm{I}_c}\mathbb{E}_{\Psi}^{\chi}\big(Y_j^{*}\mathcal{G}^{(j)}(z_1)Y_jY_j^{*}\mathcal{G}^{(j)}(z_1)Y_j-(\frac{1}{n_N}\operatorname{tr}\mathcal{G}^{(j)}(z_1))^2\big)\\
		&=C_1\sum_{j\in\mathrm{I}_c}\mathbb{E}_{\Psi}^{\chi}\big[\sum_iY_{ij}^4(\mathcal{G}_{ii}^{(j)})^2+\sum_{i\neq k}Y_{ij}^2Y_{kj}^2\mathcal{G}_{ii}^{(j)}\mathcal{G}_{kk}^{(j)}+\sum_{i\neq k}Y_{ij}^2Y_{kj}^2(\mathcal{G}_{ik}^{(j)})^2-(\frac{1}{n_N}\operatorname{tr}\mathcal{G}^{(j)})^2\big]\\
		&=C_1\sum_{j\in\mathrm{I}_c}\Big(\frac{3\pi}{4n_N^{3/2}}\sum_i(\mathcal{G}_{ii}^{(j)})^2+\frac{1}{n_N^2}\sum_{i\neq k}\mathcal{G}_{ii}^{(j)}\mathcal{G}_{kk}^{(j)}+\frac{1}{n_N^2}\sum_{i\neq k}(\mathcal{G}_{ik}^{(j)})^2-\frac{1}{n_N^2}\sum_i(\mathcal{G}_{ii}^{(j)})^2-\frac{1}{n_N^2}\sum_{i\neq k}\mathcal{G}_{ii}^{(j)}\mathcal{G}_{kk}^{(j)}\Big)\\
		&\ge C_2\sum_{j\in\mathrm{I}_c}\big(\frac{1}{n_N^{3/2}}\sum_i(\mathcal{G}_{ii}^{(j)})^2-\frac{1}{n_N^2}\sum_i(\mathcal{G}_{ii}^{(j)})^2\big)\\
		&\ge C_2\sum_{j\in\mathrm{I}_c}\frac{C_3}{n_N^{1/2}}-\mathrm{o}(1)\gtrsim\frac{|\mathrm{I}_c|}{n_N^{1/2}}\gtrsim C_4>0,
	\end{split}
\end{equation*}
where in the fourth step we have employed \eqref{eq_momentest_necessity}, in the fifth step we just have dropped away $n_N^{-2}\sum_{i\neq k}(\mathcal{G}_{ik}^{(j)})^2>0$, and in the last step we have applied Lemma \ref{lem_wellconfigured_alpha=3}. Thereafter, the second term $II$ is not negligible given \eqref{eq_oppnecessarycondition} with high probability, which further implies the necessity of \eqref{eq_criticalcase}.   \qed

\

\noindent\textbf{III. An example $f(x)=x^2$}\\
In this part, we consider Schott's statistics $f(x)=x^2$ in Example \ref{example}. By the definition of $R$ with entry $R_{ik}=\sum_{j=1}^{p}Y_{ij}Y_{kj}$, we have
\begin{equation*}
	\begin{split}
		\operatorname{tr}(R^2)=\sum_{i=1}^{n}(\sum_{j=1}^{p}y_{ij}^2)^2+\sum_{i\ne k}^{n}(\sum_{j=1}^{p}y_{ij}y_{kj})^2=p+\sum_{i}\sum_{j\ne l}y_{ij}^2y_{il}^2+\sum_{i\ne k}\sum_{j\ne l}y_{ij}y_{kj}y_{il}y_{kl},
	\end{split}
\end{equation*}
where we have used the relation $\sum_{j}\sum_{i, k}y_{ij}^2y_{kj}^2=\sum_{j}(\sum_{i}y_{ij}^2)^2=p$.
Hence
\begin{equation*}
	\begin{split}
		\mathbb{E}(\operatorname{tr}(R^2))=p+np(p-1)\beta_{2}^2+n(n-1)p(p-1)\beta_{1,1}^2,
	\end{split}
\end{equation*}
where the last term satisfies $n^2p^2\beta_{1,1}^2\lesssim n^{-2}\mathbbm{1}(\alpha\ge 3)+n^{4-2\alpha}\mathbbm{1}(2<\alpha<3)$ by Lemma \ref{lem_oddmoment_est}. It suffices to consider the variance of the last two terms. By direct calculations, we have
\begin{equation*}
	\begin{split}
		\mathbb{E}\big(\sum_{i}\sum_{j\ne l}y_{ij}^2y_{il}^2\big)^2=&np(p-1)(2\beta_{4}^2+4(p-2)\beta_{4}\beta_{2}^2+(p-2)(p-3)\beta_{2}^4)\\
		&+n(n-1)p(p-1)(2\beta_{2,2}^2+4(p-2)\beta_{2,2}\beta_{2}^2+(p-2)(p-3)\beta_{2}^4) \\
		=&2np^2\beta_{4}^2+n^{-2}(p^4-6p^3+11p^2)+4n^{-2}p(p-1)(p-2)+2n^{-2}p^2+\mathrm{O}_{\prec}(n^{1-\alpha/2})\\
		=&2np^2\beta_{4}^2+n^{-2}[p^4-2p^3+p^2]+\mathrm{O}_{\prec}(n^{1-\alpha/2}),
	\end{split}
\end{equation*}
noting that $\beta_2=1/n, \beta_{4}\asymp n^{-\alpha/2}l(n^{1/2})$ and $\beta_{2,2}=(n^2-n)^{-1}(1-n\beta_{4})=n^{-2}-\mathrm{O}_{\prec}(n^{-1-\alpha/2})$. The second moment of the sum of crossed terms in $\operatorname{tr}(R^2)$  is (we only keep the coefficient of the main term since others are negligible)
\begin{equation*}
	\begin{split}
		&\mathbb{E}\big(\sum_{i\ne k}\sum_{j\ne l}y_{ij}y_{kj}y_{il}y_{kl}\big)^2\\
		=&n(n-1)p(p-1)[4\beta_{2,2}^2+(p-2)\beta_{2,2}\beta_{1,1}^2+(p-2)(p-3)\beta_{1,1}^4]\\
		&+n(n-1)(n-2)p(p-1)[\beta_{2,1,1}^2+(n-3)\beta_{1,1,1,1}^2+(p-2)\beta_{2,1,1}\beta_{1,1}^2+(p-2)(p-3)\beta_{1,1}^4]\\
		&+n(n-1)(n-2)(n-3)p(p-1)(p-2)[\beta_{1,1,1,1}\beta_{1,1}^2+(p-3)\beta_{1,1}^4]\\
		=&4n^2p^2\beta_{2,2}^2+n^4p^4\beta_{1,1}^4+\mathrm{o}(1)
	\end{split}
\end{equation*}
since $\beta_{1,1}=\mathrm{o}(n^{-2}), \beta_{2,1,1}=\mathrm{o}(n^{-3}), \beta_{1,1,1,1}=\mathrm{o}(n^{-4})$ and $\beta_{2,2}=(n^2-n)^{-1}(1-n\beta_{4})$.
The covariance of the respective two sums in $\operatorname{tr}(R^2)$ is
\begin{equation*}
	\begin{split}
		&\mathbb{E}\big(\sum_{i}\sum_{j\ne l}y_{ij}^2y_{il}^2\big)\big(\sum_{i\ne k}\sum_{j\ne l}y_{ij}y_{kj}y_{il}y_{kl}\big)\\
		=&n(n-1)p(p-1)[\beta_{3,1}^2+(p-2)\beta_{3,1}\beta_{2}\beta_{1,1}+(p-2)(p-3)\beta_{2}^2\beta_{1,1}^2]\\
		&+n(n-1)(n-2)p(p-1)[\beta_{2,1,1}^2+(p-2)\beta_{2,1,1}\beta_{2}\beta_{1,1}+(p-2)(p-3)\beta_{2}^2\beta_{1,1}^2]\\
		=&np^4\beta_{1,1}^2+\mathrm{o}(1),
	\end{split}
\end{equation*}
since $\beta_{3,1}\le \beta_{1,1}= \mathrm{o}(n^{-2})$ and $\beta
_{2,1,1}=\mathrm{o}(n^{-3})$.
Thereafter, the variance of $\operatorname{tr}R^2$ is given by
\begin{equation*}
	\begin{split}
		\operatorname{Var}(\operatorname{tr}R^2)=2np^2\beta_{4}^2+4p^2n^{-2}+\mathrm{o}(1),
	\end{split}
\end{equation*}
finishing the proof of Example \ref{example}.
\qed






\bibliographystyle{plain}
\bibliography{reference}

\end{document}